\documentclass[a4paper, 12pt]{amsart}
\usepackage{amsfonts, amsthm, amssymb, amsmath, stmaryrd}
\usepackage{mathrsfs,array}
\usepackage{xy}
\usepackage{verbatim}
\usepackage{amscd} 
\usepackage{tikz}
\usepackage{tikz-cd}
\usetikzlibrary{matrix,arrows,decorations.pathmorphing}
\input xy
\xyoption{all}
\usepackage[unicode]{hyperref}

\newtheorem{thm}{Theorem}[subsection]

\newtheorem{cor}[thm]{Corollary}

\newtheorem{lem}[thm]{Lemma} 
\newtheorem{prop}[thm]{Proposition}
\newtheorem{conj}[thm]{Conjecture} \theoremstyle{definition}
 \theoremstyle{definition}
\newtheorem{defn}[thm]{Definition} \theoremstyle{remark}
\newtheorem{rem}[thm]{\bf Remark}

\newtheorem{para}[thm]{\bf}

\DeclareMathOperator{\Hom}{Hom}
\DeclareMathOperator{\End}{End}
\DeclareMathOperator{\Ext}{Ext}

\DeclareMathOperator{\im}{im}
\DeclareMathOperator{\coker}{coker}
\DeclareMathOperator{\modd}{mod}

\DeclareMathOperator{\Spec}{Spec}

\DeclareMathOperator{\Fil}{Fil}
\DeclareMathOperator{\gr}{gr}
\DeclareMathOperator{\Gal}{Gal}

\DeclareMathOperator{\Frob}{Frob}

\DeclareMathOperator{\Sym}{Sym}

\DeclareMathOperator{\hht}{ht}
\DeclareMathOperator{\Aut}{Aut}
\DeclareMathOperator{\ad}{ad}
\DeclareMathOperator{\Supp}{Supp}

\def\cont{\mathrm{cont}}
\def\cris{\mathrm{cris}}

\def\dR{\mathrm{dR}}

\def\Ind{\mathrm{Ind}}

\def\patch{\mathrm{p}}
\def\ps{\mathrm{ps}}
\def\loc{\mathrm{loc}}
\def\ord{\mathrm{ord}}
\def\depth{\mathrm{depth}}

\newcommand{\Z}{\mathbb{Z}}
\newcommand{\F}{\mathbb{F}}
\newcommand{\Q}{\mathbb{Q}}
\newcommand{\R}{\mathbb{R}}
\newcommand{\T}{\mathbb{T}}
\newcommand{\A}{\mathbb{A}}
\newcommand{\bC}{\mathbb{C}}
\newcommand{\fR}{\mathbf{R}}

\newcommand{\tr}{\mathrm{tr}}
\newcommand{\GL}{\mathrm{GL}}
\newcommand{\SL}{\mathrm{SL}}
\newcommand{\PGL}{\mathrm{PGL}}

\newcommand{\cC}{\mathcal{C}}
\newcommand{\cI}{\mathcal{I}}
\newcommand{\cO}{\mathcal{O}}
\newcommand{\cL}{\mathcal{L}}
\newcommand{\defeq}{\stackrel{\mathrm{def}}{=}}

\newcommand{\cOf}{C_\cO^\mathrm{f}}
\newcommand{\kq}{\mathfrak{q}}
\newcommand{\kp}{\mathfrak{p}}
\newcommand{\kkp}{k(\kp)}
\newcommand{\kQ}{\mathfrak{Q}}
\newcommand{\kP}{\mathfrak{P}}
\newcommand{\kC}{\mathfrak{C}}
\newcommand{\kB}{\mathfrak{B}}
\newcommand{\kF}{\mathfrak{F}}
\newcommand{\ka}{\mathfrak{a}}
\newcommand{\kb}{\mathfrak{b}}
\newcommand{\km}{\mathfrak{m}}
\newcommand{\mm}{\mathrm{m}}
\newcommand{\DAi}{(D\otimes_F \A_F^\infty)^\times}
\newcommand{\AFi}{(\A_F^\infty)^\times}

\newcommand{\RNum}[1]{\uppercase\expandafter{\romannumeral #1\relax}}

\newcommand{\overbar}[1]{\mkern 1.5mu\overline{\mkern-1.5mu#1\mkern-1.5mu}\mkern 1.5mu}

\begin{document}

\title{The Fontaine-Mazur conjecture in the residually reducible case}
\author{Lue Pan}
\address{Fine Hall, Washington Road\\
Princeton NJ 08544-1000 USA}
\email{lpan@princeton.edu}
\begin{abstract} 
We prove new cases of Fontaine-Mazur conjecture on two-dimensional Galois representations over $\Q$ when the residual representation is reducible. Our approach is via a semi-simple local-global compatibility of the completed cohomology and a Taylor-Wiles patching argument for the completed homology in this case. As a key input, we generalize the work of Skinner-Wiles in the ordinary case. In addition, we also treat the residually irreducible case at the end of the paper. Combining with people's earlier work, we can prove the Fontaine-Mazur conjecture completely in the regular case when $p\geq 5$.
\end{abstract}

\maketitle

\tableofcontents

\section{Introduction}
First we recall the following remarkable conjecture of Fontaine and Mazur made in \cite{FM93}.
\begin{conj}[Fontaine-Mazur] \label{FMconj}
Let $p$ be a prime number and 
\[\rho:\Gal(\bar{\Q}/\Q)\to\GL_2(\overbar{\Q_p})\]
be a continuous, irreducible representation such that
\begin{itemize}
\item $\rho$ is only ramified at finitely many places,
\item the restriction of $\rho$ at the decomposition group at $p$ is potentially semi-stable in the sense of Fontaine,
\item $\rho$ is odd, i.e. $\det\rho(c)=-1$ for any complex conjugation $c\in \Gal(\bar{\Q}/\Q)$.
\end{itemize}
Then $\rho$ arises from a cuspidal eigenform up to twist.
\end{conj}

In \cite{Kis09a}, Kisin has proved this conjecture in many cases via the so-called Breuil-M\'ezard conjecture. Under slightly more restrictively conditions, Emerton in \cite{Eme1} gave another proof using his completed cohomology and local-global compatibility results. A common ingredient for both works is the $p$-adic local Langlands correspondence for $\GL_2(\Q_p)$, which was established by the work of Breuil, Colmez, Berger, Kisin, Pa\v{s}k\={u}nas.

It should be pointed out that both works need assume the residual representation $\bar{\rho}$ to be irreducible when restricted to ${\Gal(\bar{\Q}/\Q(\zeta_p))}$. In particular, $\bar\rho$ itself has to be irreducible. Results without this assumption are mostly known in the ordinary case (i.e. $\rho|_{\Gal(\overbar{\Q_p}/\Q_p)}$ is reducible) by the work of Skinner-Wiles \cite{SW99}, \cite{SW01}. They work with the $p$-adic family of ordinary forms (Hida family), which gives them more freedom and allows them to handle these bad cases.

A natural generalization of Hida family in the non-ordinary case is the completed cohomology. It is natural to ask whether we can establish new cases of Fontaine-Mazur conjecture by working with completed cohomology. One main result of this paper is:

\begin{thm}
Let $p$ be an odd prime and $\rho$ be as in the conjecture. Assume furthermore that
\begin{itemize}
\item $\rho|_{\Gal(\overbar{\Q_p}/\Q_p)}$ has distinct Hodge-Tate weights,
\item the semi-simplification of $\bar{\rho}$ is a sum of two characters $\bar{\chi}_1\oplus \bar{\chi}_2$. If $p=3$, we require that $\frac{\bar{\chi}_1}{\bar{\chi}_2}|_{\Gal(\overbar{\Q_p}/\Q_p)}\neq\omega$. Here $\omega$ denotes the mod $p$ cyclotomic character.
\end{itemize} 
Then $\rho$ comes from a cuspidal eigenform up to twist.
\end{thm}

As we just mentioned, when $\rho|_{\Gal(\overbar{\Q_p}/\Q_p)}$ is reducible and $\frac{\bar{\chi}_1}{\bar{\chi}_2}|_{\Gal(\overbar{\Q_p}/\Q_p)}\neq\mathbf{1}$, this is the main result of \cite{SW99}. The missing case $\frac{\bar{\chi}_1}{\bar{\chi}_2}|_{\Gal(\overbar{\Q_p}/\Q_p)}=\mathbf{1}$ in the ordinary case will be proved in section \ref{AgoroS-W2}. One main contribution here (Theorem \ref{mainthm}) is our result in the non-ordinary case (i.e.  $\rho|_{\Gal(\overbar{\Q_p}/\Q_p)}$ is irreducible). Our method will not yield a new proof in the ordinary case. In fact, our proof in the non-ordinary case uses the work of Skinner-Wiles (and its generalization \ref{thmB}, \ref{thmC}) as a key ingredient. In the discussion below, we will assume $\rho|_{\Gal(\overbar{\Q_p}/\Q_p)}$ is irreducible.

The assumption that $\frac{\bar{\chi}_1}{\bar{\chi}_2}|_{\Gal(\overbar{\Q_p}/\Q_p)}\neq\omega$ when $p=3$ is used almost everywhere throughout the paper. The essential difficulty (at least for the author) is our lack of knowledge of the $p$-adic local Langlands correspondence in this case. See Theorem \ref{pa1}, \ref{pa2} for more details.

To prove the main result, we follow the strategy of Emerton. There are two steps:
\begin{enumerate}
\item Prove a `big' $R=\T$ result. Here $R$ denotes some global Galois deformation ring of pseudo-characters with no condition at $p$ and $\T$ is a localization of some Hecke algebra for completed cohomology.
\item Prove a classicality result: if $\rho$ arises from some completed cohomology and $\rho|_{\Gal(\overbar{\Q_p}/\Q_p)}$ is irreducible, de Rham of distinct Hodge-Tate weights, then up to twist it comes from a classical eigenform. 
\end{enumerate}

In reality, we do not quite prove an $R=\T$ result in the first step. What we actually will prove is identifying some irreducible components of $\Spec R$ as in $\Spec \T$, which is enough for our purpose. For simplicity, we will only discuss the generic case ($\frac{\bar{\chi}_1}{\bar{\chi}_2}|_{\Gal(\overbar{\Q_p}/\Q_p)}\neq\mathbf{1},\omega^{\pm 1}$) in this introduction. The basic idea is as follows. Let $\kp\in\Spec R$ be the prime ideal given by $\rho$. Then a standard calculation of Galois cohomology shows that we may choose an irreducible component $C$ containing $\kp$ with large dimension (lemma \ref{dimrpskp}). Under our generic assumption on $\frac{\bar{\chi}_1}{\bar{\chi}_2}|_{\Gal(\overbar{\Q_p}/\Q_p)}$, we can show that the ordinary locus $C^{\ord}$ in $C$ corresponding to deformations reducible at $p$ also has large dimension (lemma \ref{Corddim}). Using the result of Skinner-Wiles and its generalization (Theorem \ref{thmB}), we are able to show that $C^{\ord}$ is contained in $\Spec \T$. Now choose a `nice' one-dimensional prime $\kq\in C^{\ord}$ and apply a Taylor-Wiles type patching argument to this prime. We conclude that $C$ hence $\kp$ is contained in $\Spec\T$. Some of the arguments here are inspired by the work of B\"ockle \cite{Boc01}.

Here we need a version of patching argument for completed cohomology. For this part of the discussion, we will focus on patching at the maximal ideal as in the classical case. Although the patching will be at a one-dimensional prime (following \cite{SW99}) in our applications, there is no essential difference except several technical complications. A patching argument in this setting has already been considered by Gee and Newton in \cite{GN16}. As far as I understand, there are two essential difficulties:
\begin{enumerate}
\item We need a correct lower bound on $\dim\T$.
\item The patched completed homology is not finitely generated over $R_\infty$, the product of some local deformation rings with several formal variables. This is equivalent with saying that the completed homology is not finitely generated over $\T$.
\end{enumerate}

We remark that in the classical case, since the Hecke algebra is a finite $\Z_p$-module, we have a natural and correct bound for its dimension (which is $1$!). In the ordinary case, we get a bound on $\dim\T$ from the Iwasawa algebra $\Lambda$ (see corollary \ref{orddim} below). For the second difficulty, recall that  the classical patching argument involves supports, which do not behave well once we are not in the finitely generated world. The failure of the completed homology to be finitely generated over $\T$ is because there is an action of $\GL_2(\Q_p)$ on it. How to translate the information coming from this action is crucial. 

The solution to both problems is our local-global compatibility result and an application of the work of Pa\v{s}k\={u}nas \cite{Pas13}, which is the main technical innovation in this paper. Let $R^{\ps}_p$ be the local deformation ring of some two-dimensional pseudo-representation at $p$ with fixed determinant. It will play the similar role here as the Iwasawa algebra in Hida's theory. There are two actions of $R^{\ps}_p$ on the completed homology: one comes from the associated Galois pseudo-representation over the Hecke algebra, the Galois side; the other one comes from the action of $\GL_2(\Q_p)$, the spectral side. More precisely, by the work of Pa\v{s}k\={u}nas, we know that $R^{\ps}_p$ is in fact a component of the Bernstein centre of certain category of representations of $\GL_2(\Q_p)$ (Theorem 1.5 of \cite{Pas13}). Our main result is that both actions are the \textit{same} (Theorem \ref{lgc}). We note that this result is also obtained independently by Pa\v{s}k\={u}nas in \cite{Pas18} by similar methods.

This simple equality has surprisingly many important corollaries. For example, 
\begin{cor}
$\T$ has dimension at least $1+2$.
\end{cor}
In general, under certain assumptions, this will be $1+2[F:\Q]$ if we are working with some totally real field $F$ in which $p$ completely splits. We refer to Theorem \ref{dh} for a precise statement. Also we get the classicality result (corollary \ref{classicality}) as we need in the second step of the whole argument. For the patching argument: the local-global compatibility result implies that a construction in the work of Pa\v{s}k\={u}nas turns the completed homology into a finitely generated faithful $\T$-module (corollary \ref{mclgc}). Thus we can apply the usual patching argument to this module and everything works well again.

A local-global compatibility conjecture for completed cohomology first appeared in Emerton's work \cite{Eme1}, \cite{Eme06} attacking the Fontaine-Mazur conjecture. Using our description of the action of the Bernstein centre on the completed cohomology, we can prove his conjecture when the Galois representation at $p$ is irreducible. To prove our local-global compatibility result, we follow the strategy of Emerton: the density result for crystalline points reduces ourselves to the crystalline case, which follows from the result of Berger-Breuil \cite{BB10} and classical local-global compatibility.

In view of the main results of \cite{SW01}, \cite{Kis09a}, \cite{HT15} and this paper, one missing case of conjecture \ref{FMconj} is when $\rho|_{\Gal(\overbar{\Q_p}/\Q_p)}$ is irreducible and $\bar\rho$ is irreducible but becomes reducible when restricted to ${\Gal(\bar{\Q}/\Q(\zeta_p))}$. Following \cite{SW01}, we also include a proof of this case at the end of this paper (section \ref{FMcitric2}). Hence combining the work of Skinner-Wiles, Kisin, Hu-Tan, we can prove the Fontaine-Mazur conjecture completely in the regular case when $p\geq 5$. More precisely, we have

\begin{thm}
Let $p$ be an odd prime and $\rho$ be as in conjecture \ref{FMconj}. Assume that
\begin{itemize}
\item $\rho|_{\Gal(\overbar{\Q_p}/\Q_p)}$ has distinct Hodge-Tate weights,
\item If $p=3$, the semi-simplification of $\bar{\rho}|_{\Gal(\overbar{\Q_p}/\Q_p)}$ is not of the form $\eta\oplus\eta\omega$ for some character $\eta$. 
\end{itemize} 
Then $\rho$ comes from a cuspidal eigenform up to twist.
\end{thm}

\begin{rem}
In his thesis \cite{Tung18}, Shen-Ning Tung proved this result under the usual Taylor-Wiles condition \textit{without} our assumption when $p=3$. In a subsequent paper  \cite{Tung19}, he established similar results when $p=2$. More precisely, he independently proved some local-global compatibility results similar to our results in \ref{sblgc} and deduced some finiteness results after applying Colmez's functors to the completed cohomology. Using these results, he established new cases of Breuil-M\'{e}zard conjectures and reduced to the framework in Kisin's work.
\end{rem}


The paper is organized as follows. We first recall some properties of deformation rings of pseudo-representations in section \ref{Srop-r}. This will only be used in our patching argument at a one-dimensional prime. In section \ref{L-gc}, we introduce completed cohomology and state and prove our local-global compatibility result. Section \ref{Paao-dp} is rather technical. We give in detail our patching argument for the completed homology at a one-dimensional prime. In section \ref{AgoroS-W1}, we follow the strategy of Skinner-Wiles and generalize their work into the form we are using. The missing case in the ordinary case will be treated in section \ref{AgoroS-W2}. In section \ref{Tmt}, we put all these results together and prove our main result in the residually reducible case. In the last two sections \ref{FMcitric1}, \ref{FMcitric2}, we will use our patching argument to obtain results in the residually irreducible, non-ordinary case.

Even for people who are only interested in the residually reducible case, we strongly recommend reading section \ref{FMcitric1} first to get an idea of how our patching argument works for the completed homology (under the usual Taylor-Wiles condition). Note that this will give a new proof of Kisin and Emerton's results in the non-ordinary case.

\subsection*{Acknowledgement} 
Part of this paper is the author's Ph.D thesis at Princeton University. I would like to thank my advisor Richard Taylor for introducing me to this area, especially explaining the work of Skinner-Wiles, and for his constant encouragement to work on this topic. I would also like to thank Christopher Skinner for sharing his ideas on his thesis work, Ziquan Zhuang for his help on commutative algebra, and Yongquan Hu, Bao V. Le Hung for some useful discussions. Finally, I would like to thank the anonymous referees for their careful reading and many helpful comments. 

\subsection*{Notation} 
Throughout the paper, we fix an odd prime $p$, a finite extension $E$ of $\Q_p$ and a uniformizer $\varpi$ of $E$. We denote its ring of integer by $\cO$ and residue field by $\F$. Also we fix an embedding of $E$ into $\overbar{\Q_p}$, some algebraic closure of $\Q_p$ and an isomorphism $\iota_p:\overbar{\Q_p}\simeq \bC$.

We use $\cOf$ to denote the category of Artinian local $\cO$-algebras with residue field $\F$ and use $C_\cO$ to denote the category of topological local $\cO$-algebras which are isomorphic to inverse limits of objects of $\cOf$.

Following \cite{Pas13}, we introduce some categories. See  \S 2 of \cite{Pas13} for more precise definitions. Let $G$ be a topological group which is locally pro-$p$ and $(A,\km)$ be a complete local Noetherian $\cO$-algebra with residue field $\F$. We denote by $\mathrm{Mod}_G(A)$ the category of $A[G]$-modules, $\mathrm{Mod}_G^{\mathrm{sm}}(A)$ the full subcategory of $\mathrm{Mod}_G(A)$ consisting of smooth $G$-representations, i.e. any element in the representation is killed by some power of $\km$ and fixed by some open subgroup of $G$. We also denote by $\mathrm{Mod}^{\mathrm{l\,fin}}_G(A)$ the full subcategory of $\mathrm{Mod}_G^{\mathrm{sm}}(A)$ consisting of representations locally of finite length and $\mathrm{Mod}^{\mathrm{l\, adm}}_G(A)$ the full subcategory consisting of locally admissible representations. Let $Z$ be the centre of $G$ and $\zeta:Z\to A^\times$ be a continuous character. Then we can define $\mathrm{Mod}_{G,\zeta}^{\mathrm{sm}}(A),\mathrm{Mod}^{\mathrm{l\,fin}}_{G,\zeta}(A),\mathrm{Mod}^{\mathrm{l\, adm}}_{G,\zeta}(A)$ similarly as subcategories with central character $\zeta$. All representations in these categories are considered to have discrete topology. 

Let $H$ be a compact open subgroup of $G$ and $A[[H]]$ be the completed group algebra of $H$. Let $\mathrm{Mod}_G^{\mathrm{pro\,aug}}(A)$ be the category of profinite linearly topological $A[[H]]$-modules with action of $A[G]$ such that both agree on $A[H]$, with morphisms $G$-equivariant continuous homomorphisms of topological
$A[[H]]$-modules. This is independent of the choice of $H$. Taking Pontryagin duals (with the discrete topology on $E/\cO$):
\[M\mapsto M^\vee:=\Hom^\mathrm{cont}_\cO(M,E/\cO)\mbox{ with the compact-open topology}\]
induces an anti-equivalence of categories between $\mathrm{Mod}_G^{\mathrm{sm}}(A)$ and $\mathrm{Mod}_G^{\mathrm{pro\,aug}}(A)$. There is a natural isomorphism between $M^{\vee\vee}$ and $M$. Under this anti-equivalence, we define $\kC_{G,\zeta}(\cO)$ to be the full subcategory of $\mathrm{Mod}_G^{\mathrm{pro\,aug}}(A)$ with objects isomorphic to $\pi^\vee$ for some object $\pi$ in $\mathrm{Mod}^{\mathrm{l\, adm}}_{G,\zeta}(A)$. Note that this is the $\kC(\cO)$ used in \S5 \cite{Pas13}.

For a prime ideal $\kp$ of a commutative ring $R$, we denote its residue field by $k(\kp)$. Let $R_\kp$ be the localization at $\kp$. We write $\hat{R_\kp}$ as its $\kp$-adic completion. For an ideal $I$, we denote its height by $\hht(I)$ and dimension of $R/I$ by $\dim R/I$.

For a representation $\rho:\Gamma\to\GL_n(R)$, we use $\tr(\rho)$ to denote its trace.

Suppose $F$ is a number field with maximal order $O_F$. For any finite place $v$, we will write $F_v$ (resp. $O_{F_v}$) for the completion of $F$ (resp. $O_F$) at $v$, $\varpi_v$ for a uniformizer of $F_v$, $k(v)$ for the residue field, $N(v)$ for the norm of $v$ (in $\Q$),  $G_{F_v}$ for a decomposition group above $v$, $I_{F_v}$ for its inertia group and $\Frob_v$ for a geometric Frobenius element in $G_{k(v)}:=G_{F_v}/I_{F_v}$. By abuse of notation, we will also fix a lifting of $\Frob_v$ in $G_{F_v}$. If $l$ is a rational prime, then we denote $O_F\otimes_\Z \Z_l$ by $O_{F,l}$. The adele ring of $F$ will be denoted by $\A_F$ and $|\cdot|_{\A_F}:\A_F\to\R$ denotes the adelic norm map. Suppose $S$ is a finite set of places of $F$. We use $\A_F^S$ to denote the set of adeles away from $S$ and $G_{F,S}$ for the Galois group of the maximal extension of $F$ unramified outside $S$ and all infinite places. We will view $(\A_F^S)^\times$ as a subgroup of $\A_F^\times$ with elements having $1$ at $S$. The absolute Galois group of $F$ is denoted by $G_F=\Gal(\overbar{F}/F)$. 

We use $\varepsilon$ to denote the $p$-adic cyclotomic character and $\omega$ to denote the mod $p$ cyclotomic character. Our convention for the Hodge-Tate weight of $\varepsilon$ is $-1$. We normalize local class field theory by sending uniformizers to \textit{geometric} Frobenius elements.

For a finite set $Q$, we denote the number of elements in $Q$ by $|Q|$.

\section{Some results on pseudo-representations} \label{Srop-r}
In this section, we collect some results on the $2$-dimensional pseudo-representation and its deformations. These results will be useful in our study of the patching argument at a one-dimensional prime.

\subsection{Pseudo-representations}
\begin{para} \label{defnPr}
Suppose $\Gamma$ is a profinite group and $R$ is a topological commutative ring with $1$ in which $2$ is invertible. A $2$-dimensional \textit{pseudo-representation} (or \textit{pseudo-character}) of $\Gamma$ over $R$ is a continuous function $T:\Gamma\to R$ which behaves like a `trace' in the sense that (see \cite{Ta91})
\begin{enumerate}
\item $T(1)=2$, where $1\in \Gamma$ is the identity element. 
\item $T(\sigma\tau)=T(\tau\sigma)$ for any $\sigma,\tau\in\Gamma$.
\item $T(\gamma\delta\eta)+T(\gamma\eta\delta)-T(\gamma\eta)T(\delta)-T(\eta\delta)T(\gamma)-T(\delta\gamma)T(\eta)+T(\gamma)T(\delta)T(\eta)=0$, for any $\delta,\gamma,\eta\in\Gamma$.
\end{enumerate}

We define the determinant $\det(T)$ of a pseudo-representation $T$ by
\[\det(T)(\sigma)=\frac{1}{2}(T(\sigma)^2-T(\sigma^2)),~\sigma\in\Gamma.\]
\end{para}

\begin{para}[\bf Assumption 1: Mazur's finiteness condition $\Phi_p$] \label{Phip}
We assume that $\Gamma$ satisfies Mazur's finiteness condition $\Phi_p$ in \cite{Maz87}, i.e. for every open subgroup $\Gamma_0$ of $\Gamma$, the maximal pro-$p$ quotient of $\Gamma_0$ is topologically finitely generated.
\end{para}

\begin{para}[\bf Assumption 2] \label{assco}
Though the following assumption on $\Gamma$ and $T$ can be weakened, it will simplify a lot of discussions below. We will keep this assumption in this section.
\begin{itemize}
\item There exists an order $2$ element $\sigma^*\in\Gamma$ such that $T(\sigma^*)=0$.
\end{itemize}

In our case, $\Gamma$ will be certain Galois group of a totally real field and $\sigma^*$ can be chosen as any complex conjugation in $\Gamma$. 
\end{para}

\begin{para} \label{adx}
Given a pseudo-representation $T$ which satisfies \ref{assco} and $\sigma,\tau\in\Gamma$, we can define:
\begin{itemize}
\item $a(\sigma)=\frac{1}{2}(T(\sigma^*\sigma)+T(\sigma))$.
\item $d(\sigma)=\frac{1}{2}(-T(\sigma^*\sigma)+T(\sigma))=T(\sigma)-a(\sigma)$.
\item $x(\sigma,\tau)=a(\sigma\tau)-a(\sigma)a(\tau)$.
\end{itemize}

They satisfy the following identities appeared in the original definition of Wiles \cite{Wi88}.
\begin{enumerate}
\item $x(\sigma,\tau)=d(\tau\sigma)-d(\tau)d(\sigma)$.
\item $x(\sigma\tau,\delta)=a(\sigma)x(\tau,\delta)+x(\sigma,\delta)d(\tau)$.
\item $x(\sigma,\tau\delta)=a(\delta)x(\sigma,\tau)+x(\sigma,\delta)d(\tau)$.
\item $x(\alpha,\beta)x(\sigma,\tau)=x(\alpha,\tau)x(\sigma,\beta)$.
\end{enumerate}
Since there seems to be no good reference for these identities, we sketch a proof here.
\end{para}

\begin{proof}
Take $\gamma=\sigma,\delta=\tau,\eta=\sigma^*$ in the last axiom of the definition of pseudo-representation. This will give us the first identity.

Write $A(\gamma,\delta,\eta)=T(\gamma\delta\eta)+T(\gamma\eta\delta)-T(\gamma\eta)T(\delta)-T(\eta\delta)T(\gamma)-T(\delta\gamma)T(\eta)+T(\gamma)T(\delta)T(\eta)$. Consider 
\[A(\sigma,\tau,\delta\sigma^*)+A(\tau,\delta,\sigma^*\sigma)-A(\tau,\sigma\delta,\sigma^*)=0.\] 
A direct computation shows that this gives
\begin{eqnarray*}2T(\sigma\tau\delta\sigma^*)=T(\sigma\tau)T(\delta\sigma^*)+T(\tau\delta)T(\sigma\sigma^*)+T(\sigma\delta)T(\tau\sigma^*)+T(\tau)T(\sigma\delta\sigma^*)+\\T(\sigma)T(\tau\delta\sigma^*)+T(\delta)T(\tau\delta\sigma^*)-T(\sigma)T(\tau)T(\delta\sigma^*)-T(\tau)T(\delta)T(\sigma\sigma^*).
\end{eqnarray*}
Replacing $\delta$ by $\delta\sigma^*$ and taking the sum of this equality with the equality above, we obtain the second identity. The third identity can be proved similarly.

The last identity follows by applying the second and third identities repeatedly:
\begin{eqnarray*}
x(\alpha,\beta)x(\sigma,\tau)&=&(a(\alpha\beta)-a(\alpha)a(\beta))x(\sigma,\tau) \\
&=& (x(\alpha\beta\sigma,\tau)-x(\alpha\beta,\tau)d(\sigma))-a(\alpha)a(\beta)x(\sigma,\tau)\\
&=& (x(\alpha,\tau)d(\beta\sigma)+a(\alpha)x(\beta\sigma,\tau))-x(\alpha\beta,\tau)d(\sigma)-a(\alpha)a(\beta)x(\sigma,\tau)\\
&=& (x(\alpha,\tau)x(\sigma,\beta)+x(\alpha,\tau)d(\beta)d(\sigma))+a(\alpha)x(\beta\sigma,\tau)-x(\alpha\beta,\tau)d(\sigma)\\&&-a(\alpha)a(\beta)x(\sigma,\tau)
\end{eqnarray*}
Applying the second identity to $x(\beta\sigma,\tau),x(\alpha\beta,\tau)$, we get the desired identity.
\end{proof}

\begin{rem} \label{adrep}
If $\rho:\Gamma\to\GL_2(R)$ is a continuous representation, then $\tr(\rho)$ is a pseudo-representation. If $\rho(\sigma^*)=\begin{pmatrix} 1 & 0\\ 0 & -1\end{pmatrix}$, then $\rho(\sigma)=\begin{pmatrix} a(\sigma) & b(\sigma) \\ c(\sigma) & d(\sigma)\end{pmatrix}$ with $a(\sigma),d(\sigma)$ defined above and $x(\sigma,\tau)=b(\sigma)c(\tau)$.
\end{rem}

\begin{para}\label{tar}
Now assume $R$ is either a field or a discrete valuation ring. We can attach a $2$-dimensional representation $\rho$ of $\Gamma$ over $R$ with trace $T$. There are two cases:
\begin{enumerate}
\item All $x(\sigma,\tau)=0$. Then $a,d:\Gamma\to R^\times$ are two characters. Define 
\[\rho(\sigma)=\begin{pmatrix} a(\sigma) &0 \\ 0 & d(\sigma) \end{pmatrix}.\]
We call this case \textbf{reducible}. Note that $R$ can be a general ring in this case.
\item $x(\sigma,\tau)\neq 0$ for some $\sigma,\tau$. Choose $\sigma_0,\tau_0$ such that $
\frac{x(\sigma,\tau)}{x(\sigma_0,\tau_0)}\in R$ for any $\sigma,\tau$. Define
\[\rho(\sigma)=\begin{pmatrix} a(\sigma) & \frac{x(\sigma,\tau_0)}{x(\sigma_0,\tau_0)} \\ x(\sigma_0,\sigma) & d(\sigma) \end{pmatrix}.\]
One can check that this really defines a representation of $\Gamma$ using the identities above. It is clear that $\rho$ is absolutely irreducible. We call this case \textbf{irreducible}. Note that in this case, if $R$ is a field, then $\rho$ is unique up to conjugation.
\end{enumerate}
The determinant of $\rho$ is nothing but $\mathrm{det}(T)$ we defined before.
\end{para}

\subsection{Deformation rings of pseudo-representations I}
\begin{para}
Let $T_\F:\Gamma\to\F$ be a $2$-dimensional pseudo-representation. Then the functor sending each object $R$ of $\cOf$ to the set of pseudo-representations $T:\Gamma\to R$ which lift $T_\F$ is prorepresented by a complete Noetherian local $\cO$-algebra $R^{\ps}_{T_\F}$ (lemma 1.4.2 of \cite{Kis09a}). Write $T^{univ}:\Gamma\to R^{\ps}_{T_\F}$ as the universal pseudo-character. A standard argument shows that $R^{\ps}_{T_\F}$ is topologically generated by $T^{univ}(\Gamma)$. By Theorem 1 of \cite{Ta91}, there exists a finite set $S\subset\Gamma$ such that the values of $T^{univ}$ on $S$ determine $T^{univ}$. Hence $T^{univ}(S)$ topologically generates $R^{\ps}_{T_\F}$. Although both cited results require $\Gamma$ to be topologically finitely generated, it is enough to only assume the finiteness condition $\Phi_p$ by Corollary A.3 of \cite{Pas10}.

If $T_\F$ is irreducible, write $\bar{\rho}$ as its associated $2$-dimensional representation. Then it is well-known that the natural map $R^{\ps}_{T_\F}\to R_{\bar{\rho}}$ by evaluating the trace is an isomorphism. Here $R_{\bar{\rho}}$ is the universal deformation ring of $\bar{\rho}$. 

From now on, we will assume our $T_\F$ is reducible.

Suppose $\kp$ is an one-dimensional prime of $R^{\ps}_{T_\F}$. Consider $T(\kp)\defeq T^{univ}\otimes \kkp:\Gamma\to \kkp$. Let $\rho(\kp):\Gamma\to \GL_2(k(\kp))$ be a representation with trace $T(\kp)$ as we discussed in the previous subsection. In this subsection, we assume  
\begin{itemize}
\item$\kkp=E$. 
\item $\rho(\kp)$ is irreducible (hence absolutely irreducible).
\end{itemize}
We want to give a moduli interpretation of $\widehat{(R^{\ps}_{T_\F})_\kp}$, the $\kp$-adic completion of $(R^{\ps}_{T_\F})_\kp$.

Consider the functor $D_\kp$ from the category of Artinian local $E$-algebras with residue field $E$ (equipped with $p$-adic topology) to the category of sets which sends $A$ to the set of $2$-dimensional pseudo-representations over $A$ lifting $T(\kp)$. The main result here is
\end{para}

\begin{prop} \label{Dkp}
The deformation problem $D_\kp$ is pro-represented by $\widehat{(R^{\ps}_{T_\F})_\kp}$ with universal pseudo-representation $\Gamma \stackrel{T^{univ}}{\longrightarrow} R^{\ps}_{T_\F}\to \widehat{(R^{\ps}_{T_\F})_\kp}$. In fact, the same result holds without assuming $\rho(\kp)$ is irreducible.
\end{prop}

\begin{proof}
Given an Artinian local $E$-algebra $A$ with residue field $E$ and a lifting $T_A$ of $T(\kp)$ to $A$, we need to define a map from $\widehat{(R^{\ps}_{T_\F})_\kp}$ to $A$.  Let $S$ be a finite subset of $\Gamma$ such that $T^{univ}(S)$ topologically generates $R^{\ps}_{T_\F}$. Consider the $\cO$-algebra $A_0$ generated by $T_A(S)$ in $A$. It is easy to see that $A_0$ is a finite local $\cO$-algebra. Moreover $T_A$ factors through $A_0$. Hence there exists a map $R^{\ps}_{T_\F}\to A_0$ by the universal property. This map can be extended to a map $\widehat{(R^{\ps}_{T_\F})_\kp}\to A$. It is unique since the values of $T_A$ on $S$ determine the map. Therefore $\widehat{(R^{\ps}_{T_\F})_\kp}$ pro-represents $D_\kp$. It is clear that we don't need the assumption that $\rho(\kp)$ is irreducible.
\end{proof}

\begin{cor} \label{dr0}
Let $D_{\rho(\kp)}$ be the functor from the category of Artinian local $E$-algebras with residue field $E$ to the category of sets which sends $A$ to the set of deformations of $\rho(\kp)$ to $A$. Then $D_{\rho(\kp)}$ is pro-represented by $\widehat{(R^{\ps}_{T_\F})_\kp}$.
\end{cor}

\begin{proof}
The natural transformation from $D_{\rho(\kp)}$ to $D_\kp$ by evaluating the trace is an isomorphism since we assume $\rho(\kp)$ is absolutely irreducible.
\end{proof}

\subsection{Deformation rings of pseudo-representations II}

In this subsection, we are going to prove a similar but more precise result of corollary \ref{dr0} for primes containing $p$. The setup is a bit complicated, but useful for our applications.

\begin{para}[\textbf{Setup}]
We start with the following data:
\begin{itemize}
\item $A=\F[[T]]$, the formal power series ring over $\F$.
\item $\Gamma_0$ is a profinite group which satisfies $\Phi_p$ in \ref{Phip}.
\item $\rho_0:\Gamma_0\to\GL_2(A), \sigma\mapsto \begin{pmatrix}a_0(\sigma) & b_0(\sigma) \\ c_0(\sigma) & d_0(\sigma) \end{pmatrix}$ is a continuous irreducible representation.
\end{itemize}
We assume that
\begin{itemize}
\item the reduction $\bar{\rho_0}$ of $\rho_0$ modulo $T$ has the form  $\begin{pmatrix}\bar{a_0} & \bar{b_0} \\ 0 & \bar{d_0} \end{pmatrix}$.
\item there exist $\sigma_0,\tau_0\in \Gamma_0$ such that $b_0(\sigma_0)=1$ and $c_0(\tau_0)\neq 0$.
\item there exist an order $2$ element $\sigma_0^*\in\Gamma_0$ such that $\rho_0(\sigma_0^*)=\begin{pmatrix}1 & 0 \\ 0 & -1 \end{pmatrix}$.
\end{itemize}

Let $\Gamma$ be another profinite group satisfying $\Phi_p$ with a fixed surjective map 
\[\pi:\Gamma \to \Gamma_0.\] 
Moreover we assume there exists an order $2$ element $\sigma^*\in \Gamma$ mapping to $\sigma_0^*$. In practice, $\Gamma_0$ will be $G_{F,S}$ for some totally real number field $F$ and some finite set $S$ of places of $F$. The group $\Gamma$ will be $G_{F,S\cup T}$ for some set $T$ of Taylor-Wiles primes and $\sigma^*$ will be some fixed complex conjugation.

Write $\rho$ (resp. $\bar{\rho}$) as the composite maps of $\pi$ and $\rho_0$ (resp. $\bar{\rho_0}$). Consider the following two universal deformation rings (pro-representing functors from $\cOf$ to the category of sets):
\begin{itemize}
\item $R^{\ps}=R^{\ps}_{tr(\bar{\rho})}$, the universal deformation ring of pseudo-representation $\tr\bar{\rho}:\Gamma\to \F$.
\item $R_b$: the universal deformation ring of the representation $\bar{\rho}:\Gamma\to\GL_2(\F)$.
\end{itemize}
Note that the trace of $\rho$ and the representation $\rho$ give rise to prime ideals $\kq$ and $\kq_b$ of $R^{\ps}$ and $R_b$. Both prime ideals contain $p$ and have to be one-dimensional since $\rho$ is irreducible while $\bar{\rho}$ is reducible. There exists a natural map:
\[\varphi:R^{\ps}\to R_b\]
by evaluating the trace and $\varphi^*(\kq_b)=\kq$. It is natural to compare $\widehat{(R^{\ps})_\kq}$ and $\widehat{(R_b)_{\kq_b}}$.
\end{para}

\begin{prop} \label{rpsrc}
There exists an element $c\in R^{\ps}\setminus \kq$ such that
\begin{enumerate}
\item the image of $c$ in $R^{\ps}/\kq\to A$ is $c_0(\tau_0)\neq 0$;
\item for any positive integer $n$, there exists an integer $N=N(n)\ge 0$, such that $c^N$ kills the kernel and the cokernel of the map $R^{\ps}/\kq^n\to R_b/\kq_b^n$. Moreover $N$ only depends on $n$ and $\rho_0:\Gamma_0\to\GL_2(A)$ (hence is independent of $\Gamma$).
\end{enumerate}
\end{prop}

\begin{proof}
Denote $R^{\ps}/\kq$ by $B$. We may assume $B$ and $A$ have the same fraction fields by replacing $A$ with the normal closure of $B$, cf. \ref{tar}. Note that $B$ is topologically generated by the traces of $\rho_0(\Gamma)$, hence independent of $\Gamma$.

Choose $\sigma_*,\tau_*\in\Gamma$ to be some liftings of $\sigma_0,\tau_0$. Recall that $\sigma^*$ is an order $2$ element mapping to $\sigma_0^*$ in $\Gamma_0$. Since $\tr(\bar{\rho})(\sigma_0^*)=0$, by taking $\gamma=\delta=\eta=\sigma^*$ in the third condition of \ref{defnPr}, we see that $T(\sigma^*)=0$ for any lifting $T$ of pseudo-representation $\tr(\bar{\rho})$. In particular, take $T$ to be $T^{univ}:\Gamma\to R^{\ps}$, the universal pseudo-representation, and apply the construction in \ref{adx} for $\sigma^*$. We can define $a(\sigma),d(\sigma),x(\sigma,\tau)\in R^{\ps}$ for $\sigma,\tau\in\Gamma$. We also define:
\[c(\tau)=x(\sigma_*,\tau)\in R^{\ps},~\tau\in \Gamma\]
\[c=x(\sigma_*,\tau_*).\]
Note that the reduction of $c(\tau)$ modulo $\kq$ is $c_0(\pi(\tau))\neq 0$ as $b_0(\pi(\sigma_*))=b_0(\sigma_0)=1$. We claim that $c$ actually works for the proposition.

\noindent\textbf{\underline{Special Case}}: $R^{\ps}/\kq=R_b/{\kq_b}$. First we need some auxiliary construction.

\begin{defn} \label{pscon}
Let R be a commutative ring with $1$ and $I\in\Spec(R)$ such that $I^n=0$ for some integer $n$. Let $c\in R\setminus I$. We can define two new commutative rings:
\[R''\defeq (R\oplus I^1\oplus\cdots \oplus I^{n-1})/J,\]
\[R'\defeq R''/R''[c],\]
where 
\begin{itemize}
\item $R\oplus I^1\oplus\cdots \oplus I^{n-1}$ is viewed as a graded ring with $\gr^i=I^i$.
\item $J$ is an ideal of $R\oplus I^1\oplus\cdots \oplus I^{n-1}$ generated by elements of the form 
\[(a_0,\cdots,a_{n-1})-(0,a_0c,\cdots,a_{n-2}c)\] 
for $a_i\in I^i$ such that $a_ic\in I^{i+1}$.
\item $c\in R''$ by abuse of the notation is the image of $(c,0,\cdots,0)$ in $R''$.
\item $R''[c]$ denotes the $c$-torsion part of $R''$.
\end{itemize}
\end{defn}

\begin{lem} \label{conlem}
\hspace{2em}
\begin{enumerate}
\item The kernel of the natural map $R\to R'$ is killed by $c^{n+1}$.
\item The images of elements of the form $(0,a_1,\cdots,a_{n-1})$ in $R'$ define a prime ideal $I'$. We have $R/I=R'/I'$ and $(I')^n=0$.
\item The cokernel of $R\to R'$ is killed by $c^{n-1}$.
\end{enumerate}
\end{lem}
\begin{proof}
Let $x\in \ker(R\to R'')$. By definition, there exists $a_i\in I^i$ with $a_{n-1}c=0$ such that
\[(x,0,\cdots,0)=(a_0,a_1-a_0c,\cdots,a_{n-1}-a_{n-2}c)\]
in $R\oplus I^1\oplus\cdots \oplus I^{n-1}$. Hence $c^nx=ca_{n-1}=0$. This proves the first part.

The second part is easy. We omit the proof here. The last part follows from 
\[[c^{n-1}(a_0,\cdots,a_{n-1})]=[(a_0c^{n-1}+a_1c^{n-2}+\cdots+a_{n-1},0,\cdots,0)]\]
in $R''$.
\end{proof}

Apply this construction to $R^{\ps}_n\defeq R^{\ps}/\kq^n$ with $I=\kq/\kq^n$ and $c=c(\tau_*)$. We get $(R^{\ps}_n)'$ with an ideal $I'$ and $(R^{\ps}_n)''$. Consider the composition of the following maps:
\[T'_n:\Gamma \stackrel{T^{univ}}{\longrightarrow} R^{\ps}\to R^{\ps}/\kq^n\to(R^{\ps}_n)'.\]
Note that $(R^{\ps}_n)'/I'=R^{\ps}/\kq=R_b/{\kq_b}$. The whole point of introducing $(R^{\ps}_n)'$ is 

\begin{lem} \label{Rpsn'lift}
There exists a lifting $\rho'_n$ of $\bar{\rho}$ from $(R^{\ps}_n)'/I'$ to $(R^{\ps}_n)'$ with trace $T'_n$.
\end{lem}

\begin{proof}
By abuse of notation, we view $a(\sigma),d(\sigma),x(\sigma,\tau)$ as elements in $R^{\ps}_n,(R^{\ps}_n)',(R^{\ps}_n)''$. Recall that $\rho:\Gamma\to\GL_2(A)$ has images in $\GL_2(R_b/\kq_b)=\GL_2(R^{\ps}/\kq)$. Hence for any $\sigma$,
\[x(\sigma,\tau_*)\equiv b_0(\pi(\sigma))c_0(\tau_0)~\mod~\kq.\]
in $R^{\ps}$. That is to say we can find $\tilde{b}(\sigma)\in R^{\ps}_n$ and $y\in\kq/\kq^n$ such that
\[x(\sigma,\tau_*)=\tilde{b}(\sigma)c(\tau_*)+y\]
in $R^{\ps}_n$. Now let $b''(\sigma)=[(\tilde{b}(\sigma),y,0,\cdots,0)]$ be an element in $(R^{\ps}_n)''$. It is easy to check that this is independent of the choice of $\tilde{b}(\sigma)$. Note that we have 
\[b''(\sigma)c(\tau_*)=[(\tilde{b}(\sigma)c,yc,0,\cdots,0)]=[(\tilde{b}(\sigma)c+y,0,\cdots,0)]=x(\sigma,\tau_*)\in (R^{\ps}_n)''. \]
Similarly,
\[b''(\sigma)c(\tau)c(\tau_*)=x(\sigma,\tau)c(\tau_*)\]
Let $b(\sigma)$ be the image of $b''(\sigma)$ in $(R^{\ps}_n)'$. Using these identities and identities in \ref{adx}, one can check that
\[\rho':\Gamma\to\GL_2((R^{\ps}_n)'),~\sigma\mapsto \begin{pmatrix}a(\sigma) &b(\sigma) \\ c(\sigma) & d(\sigma)\end{pmatrix}\]
really defines a representation with trace $T'_n$ which lifts $\rho$.
\end{proof}

Let $\rho'$ be the representation constructed in the proof of the previous lemma. By the universal property, we get a map:
\begin{eqnarray} \label{mapj}
j:R_b\to(R^{\ps}_n)'
\end{eqnarray}
which necessarily factors through $R_b/\kq_b^n$ by the second part of lemma \ref{conlem}.

On the other hand, we may apply the construction \ref{pscon} to $R_b/\kq_b^n$ with ideal $\kq_b/\kq_b^n$ and $c=c(\tau_*)$ and get a ring $(R_b/\kq_b^n)'$. It is easy to see that this construction is functorial  and gives us a diagram (not necessarily commutative):
\[\begin{tikzcd} 
R^{\ps}/\kq^n \arrow[r,"i^{\ps}"] \arrow[d,"\varphi_n"] & (R^{\ps}_n)' \arrow[d, "\varphi'_n"]\\
R_b/\kq_b^n \arrow[ru,"j"] \arrow[r,"i_b"] & (R_b/\kq_b^n)'\\
\end{tikzcd}\]
The square and the upper left triangle are commutative by construction. As for the lower right triangle, we have
\begin{lem}
For any $x\in R_b/\kq_b^n,$
\[\varphi'_n\circ j(x)-i_b(x)\in (R_b/\kq_b^n)'[c].\]
\end{lem}
\begin{proof}
Let $\rho'_b:\Gamma\to\GL_2((R_b/\kq_b^n)')$ be a deformation of $\bar{\rho}$ induced by $i_b$. It is enough to show that modulo $(R_b/\kq_b^n)'[c]$, this representation $\rho'_b$ is conjugate to $\varphi'_n\circ \rho'$ by an element in $1+M_2(m')$, where $m'$ is the maximal ideal of $(R_b/\kq_b^n)'$.

Note that $\tr(\rho'_b)=\varphi'_n\circ \tr(\rho')$. By conjugating with some element in $1+M_2(m')$, we may assume
\[\rho'_b(\sigma^*)=\begin{pmatrix}1 & 0\\ 0& -1\end{pmatrix},~\rho'_b(\sigma_*)=\begin{pmatrix}*&1\\ *&*\end{pmatrix}.\]
Write $\rho'_b(\sigma)=\begin{pmatrix}a'(\sigma)&b'(\sigma)\\ c'(\sigma) & d'(\sigma)\end{pmatrix}$. Then 
\[\varphi'_n (a(\sigma))=a'(\sigma),~\varphi'_n( d(\sigma))=d'(\sigma),\]
\[\varphi'_n(b(\sigma)c(\tau))=\varphi'_n(x(\sigma,\tau))=b'(\sigma)c'(\tau),\]
since they can be defined by the same formulae using the pseudo-representation. Thus
\[c'(\tau)=b'(\sigma_*)c'(\tau)=\varphi'_n(b(\sigma_*)c(\tau))=\varphi'_n(c(\tau)),\]
\[c(b'(\sigma)-\varphi'_n(b(\sigma)))=0.\]
In other words, $\rho'_b\equiv \varphi'_n \circ \rho'\mod (R_b/\kq_b^n)'[c]$. This is exactly what we want.
\end{proof}

Now we can prove the proposition under the assumption $R^{\ps}/\kq=R_b/\kq_b$. Consider the upper left triangle of the previous diagram:
\[\ker(\varphi_n)\subset \ker (i^{\ps})\subset (R^{\ps}/\kq^n)[c^{n+1}],\]
where the last inclusion follows from the first part of lemma \ref{conlem}.

Suppose $x\in R_b/\kq_b^n$. By the last part of lemma \ref{conlem}, we may write $c^{n-1}j(x)=i^{\ps}(y)=j\circ \varphi_n(y)$ for some $y\in R^{\ps}/\kq^n$. Thus $z=\varphi_n(y)-c^{n-1}x\in\ker(j)$. Apply the previous lemma to $z$. We get $cz\in \ker(i_b)$. Now use lemma \ref{conlem} again. We have $c^{n+2}z=0$. Hence
\[c^{2n+1}x=\varphi_n(c^{n+2}y)\in \varphi_n(R^{\ps}_n).\]
This proves the proposition with $N=2n+1$.

\noindent\textbf{\underline{General Cases}}. In general, we can use the following trick to reduce to the special case we treated before. Since $B=R^{\ps}/\kq$ have the same fraction field as $A=\F[[T]]$, there exists a positive integer $r$ such that $T^rA\subset B$. We may assume $r=p^m$. Hence $R^{\ps}/\kq$ can be viewed as an $\F[[T^{p^m}]]$-algebra. Choose a lifting $\tilde{T}$ of $T^{p^m}$ in $R^{\ps}$. This gives $\cO[[T^{p^m}]]$-algebra structures to $R^{\ps}$ and $R_b$. Define:
\[\widetilde{R^{\ps}}=R^{\ps}\otimes_{\cO[[T^{p^m}]]}\cO[[T]],\]
\[\widetilde{R_b}=R_b\otimes_{\cO[[T^{p^m}]]}\cO[[T]]\]
via the natural faithfully flat map $\cO[[T^{p^m}]]\to \cO[[T]]$. Thus it suffices to prove the kernel and cokernel of
\[\widetilde{\varphi_n}:\widetilde{R^{\ps}}/\kq^n\widetilde{R^{\ps}}\to\widetilde{R_b}/\kq_b^n\widetilde{R_b}\]
are killed by some power, depending only on $n$, of $c$. Note that $\widetilde{R^{\ps}}$ and $\widetilde{R_b}$ have natural surjective maps to $A$:
\[\widetilde{R^{\ps}}/\kq \widetilde{R^{\ps}}=B\otimes_{\F[[T^{p^m}]]}\F[[T]]\to A\]
given by $b\otimes a\mapsto ab$. Let $\tilde{\kq}$ (resp. $\tilde{\kq_b}$) be the kernel of $\widetilde{R^{\ps}}\to A$ (resp. $\widetilde{R_b}\to A$). It is clear that 
\[\tilde{\kq}^{p^m}\subseteq\kq\widetilde{R^{\ps}}\subseteq\tilde{\kq}.\]

Now one can check that the previous argument of the special case works here for $(\widetilde{R^{\ps}},\tilde{\kq})$ and $(\widetilde{R_b},\tilde{\kq_b})$. For example, lemma \ref{Rpsn'lift} still holds and  we get a $\cO[[T]]$-algebra homomorphism $\widetilde{R_b}\to(\widetilde{R^{\ps}}/\tilde{\kq}^n)'$, similar to the map $j$ in \eqref{mapj}. As a consequence, the kernel and cokernel of 
\[\theta_l:\widetilde{R^{\ps}}/\tilde{\kq}^l\to \widetilde{R_b}/\tilde{\kq_b}^l\]
are killed by $c^{2l+1}$. Take $l=np^m$ and write $\tilde{\varphi}:\widetilde{R^{\ps}}\to\widetilde{R_b}$. We have
\begin{eqnarray}\label{inc}
c^{2l+1}\widetilde{R_b}\subseteq\tilde{\varphi}(\widetilde{R^{\ps}})+\tilde{\kq_b}^l\subseteq \tilde{\varphi}(\widetilde{R^{\ps}})+\kq_b^n\widetilde{R_b}.
\end{eqnarray}
This proves that $c^{2np^m+1}$ kills the cokernel of $\widetilde{\varphi_n}$.

Let $x$ be an element in $\widetilde{R^{\ps}}$ such that $\tilde{\varphi}(x)\in\kq_b^n\widetilde{R_b}$. It is clear from \eqref{inc} that
\[c^{2np^m+1}\kq_b\widetilde{R_b}\subseteq \tilde{\varphi}(\kq\widetilde{R^{\ps}})+\kq_b^n\widetilde{R_b}.\]
For any positive integer $k$, a simple induction on $k$ will give us
\[c^{(2np^m+1)(n+(2n-1)+\cdots+((n-1)k+1))}\tilde{\varphi}(x)\in \tilde{\varphi}(\kq^n\widetilde{R^{\ps}})+\kq_b^{n+k(n-1)}\widetilde{R_b}.\]
Take $k=2p^m-2$. Then $n+k(n-1)\ge np^m$, hence
\[c^M\tilde{\varphi}(x)\in \tilde{\varphi}(\kq^n\widetilde{R^{\ps}})+\kq_b^{np^m}\widetilde{R_b}\subseteq \tilde{\varphi}(\kq^n\widetilde{R^{\ps}})+\tilde{\kq_b}^{np^m},\]
where $M=(2np^m+1)(n+(2n-1)+\cdots+((n-1)k+1))$ only depends on $p^m,n$. This means that  we may find $y\in \kq^n\widetilde{R^{\ps}}$ such that
\[\tilde{\varphi}(c^Mx-y)\in \tilde{\kq_b}^{np^m}.\]
But kernel of $\theta_{np^m}:\widetilde{R^{\ps}}/\tilde{\kq}^{np^m}\to \widetilde{R_b}/\tilde{\kq_b}^{np^m}$ is killed by $c^{2np^m+1}$. Thus
\[c^{2np^m+1+M}x-c^{2np^m+1}y\in \tilde{\kq}^{np^m}\subseteq\kq^n\widetilde{R^{\ps}}.\]
In other words, $c^{2np^m+M+1}$ kills the kernel of $\widetilde{\varphi_n}$.
\end{proof}

The following corollary generalizes Proposition 2.11 of \cite{SW99}.

\begin{cor} \label{psccomp}
\hspace{2em}
\begin{enumerate}
\item For any one-dimensional prime ideal $\kp_b\in \Spec R_b$ such that $\rho^{univ}\mod \kp_b$ is irreducible, the natural map $\widehat{(R^{\ps})_\kp}\to\widehat{(R_b)_{\kp_b}}$ is an isomorphism and $\Spec (R_b)_{\kp_b}\to \Spec (R^{\ps})_{\kp}$ is surjective. Here $\rho^{univ}:\Gamma\to \GL_2(R_b)$ is a universal deformation and $\kp=\kp_b\cap R^{\ps}$. 
\item For any $\kQ\in \Spec R_b$ such that $\rho^{univ}\mod \kQ$ is irreducible, we have
\[\dim R_b/\kQ\leq\dim R^{\ps}/\kQ\cap R^{\ps}.\]
Moreover, if $\kQ$ is a minimal prime, so is $\kQ\cap R^{\ps}$. 
\end{enumerate}
\end{cor}
\begin{proof}
Note that it is clear from the discussion in \ref{tar} that $k(\kp_b)=k(\kp)$. If $k(\kp)=E$, then the isomorphism between $\widehat{(R^{\ps})_\kp}$ and $\widehat{(R_b)_{\kp_b}}$ is a direct consequence of corollary \ref{dr0}. If $p\in \kp$ and the residue field of the normalization of $R_b/\kp_b$ is $\F$, then the isomorphism follows from proposition \ref{rpsrc}.

In general, we can reduce the problem to these situations by the following trick: Let $L$ be $k(\kp)$ if $p\notin \kp$ and a finite unramified extension of $E$ with residue field same as the integral closure of $\F$ in $k(\kp)$ if $p\in \kp$. Let $O_L$ be the ring of integers of $L$. Then for any prime ideal $\kp'_b\in\Spec R_b\otimes_\cO O_L$ above $\kp$, we may apply the results in the previous paragraph to deformation rings $(R^{\ps})'=R^{\ps}\otimes_\cO O_L, R_b'=R_b\otimes_\cO O_L$ and prime ideals $\kp'=\kp'_b\cap (R^{\ps})',\kp'_b$ and conclude 
\[\widehat{(R^{\ps})'_{\kp'}}\cong\widehat{(R'_b)_{\kp'_b}}.\]
On the other hand, as $(R'_b)_{\kp_b}=(R_b)_{\kp_b}\otimes_\cO O_L$ is a finite $(R_b)_{\kp_b}$-algebra, we have
\[\widehat{(R_b)_{\kp_b}}\otimes_{\cO}O_L\cong \prod_{\kp'_b} \widehat{(R'_b)_{\kp'_b}},\]
where the product is taken over all primes $\kp'_b \in\Spec R'_b$ above $\kp_b$. Similarly $\widehat{(R^{\ps})_{\kp}}\otimes_{\cO}O_L\cong \prod_{\kp'} \widehat{(R^{\ps})'_{\kp'}}$. Here the product is taken over all $\kp'\in\Spec (R^{\ps})'$ above $\kp$. Note that there is a natural bijection between $\{\kp'\}$ and $\{\kp'_b\}$ as $k(\kp)=k(\kp_b)$, which can be seen using the construction in \ref{tar}. Thus 
\[\widehat{(R^{\ps})_{\kp}}\otimes_{\cO}O_L\cong \prod_{\kp'} \widehat{(R^{\ps})'_{\kp'}}\cong \prod_{\kp'_b} \widehat{(R'_b)_{\kp'_b}}\cong \widehat{(R_b)_{\kp_b}}\otimes_{\cO}O_L.\]
From this, we see that $\widehat{(R^{\ps})_\kp}\cong\widehat{(R_b)_{\kp_b}}$. Since $\widehat{(R^{\ps})_\kp}$ is faithfully flat over $(R^{\ps})_\kp$ (and similar result holds for $(R_b)_{\kp_b}$), it is clear that $\Spec (R_b)_{\kp_b}\to \Spec (R^{\ps})_{\kp}$ is surjective.
 
As for the second part, we may find a prime ideal $\kp_b$ containing $\kQ$ such that $R_b/\kp_b$ is one-dimensional and $\rho^{univ}\mod \kp_b$ is irreducible. To see the existence of $\kp_b$, note that 
\[\{\kp\in\Spec R_b, ~\rho^{univ}\mod \kp \mbox{ is reducible}\}\] 
is defined by all $x(\sigma,\tau)$ (see the beginning of proof of proposition \ref{rpsrc} for notations here), hence closed in $\Spec R_b$. Let $f$ be an element of the form $x(\sigma,\tau)$ not in $\kQ$. We can take $\kp_b$ to be any maximal ideal of $(R_b/\kQ)_f$. Let $\kp=R^{\ps}\cap \kp_b$. Our assertion follows from
\[\dim (R_b/\kQ)_{\kp_b}=\dim \widehat{(R_b)_{\kp_b}}/(\kQ)\leq  \dim \widehat{(R^{\ps})_\kp}/(\kQ\cap R^{\ps})=\dim (R^{\ps}/\kQ\cap R^{\ps})_\kp.\]
If $\kQ$ is minimal, we can choose a minimal prime $\kQ'$ of $\widehat{(R_b)_{\kp_b}}$ such that $\kQ'\cap R_b=\kQ$. This is possible as $(R_b)_{\kp_b}\to \widehat{(R_b)_{\kp_b}}$ is faithfully flat. We may regard $\kQ'$ as a prime of $\widehat{(R^{\ps})_\kp}$. By the going-down property of the flat morphism $R^{\ps}\to\widehat{(R^{\ps})_\kp}$, we see that $\kQ\cap R^{\ps}=\kQ'\cap R^{\ps}$ is also a minimal prime.
\end{proof}

In application, we need the following version of proposition \ref{rpsrc}.

\begin{cor} \label{pscrem}
Let $I$ be an ideal of $R^{\ps}$ contained in $\kq$ and $t$ be some positive integer. Write $R_1$ (resp. $R_2$) for $(R^{\ps}/I)[[y_1,\cdots,y_t]]$ (resp. $(R_b/IR_b)[[y_1,\cdots,y_t]]$) and $\kq_1$ (resp. $\kq_2$) for the ideal of $(R^{\ps}/I)[[y_1,\cdots,y_t]]$ (resp. $(R_b/IR_b)[[y_1,\cdots,y_t]]$) generated by $y_1,\cdots,y_t$ and $\kq$ (resp. $y_1,\cdots,y_t$ and $\kq_b$). Let $c$ be the element in proposition \ref{rpsrc}. Then the second part of proposition \ref{rpsrc} holds for $R_1/\kq_1^n\to R_2/\kq_2^n$.
\end{cor}

\begin{proof}
This is easy. We omit the proof here.
\end{proof}

\subsection{Some miscellaneous results}
We keep the same notations and setup as in the previous subsections.

\begin{lem}
Let $\kp$ be a one-dimensional prime ideal of $R^{\ps}$ such that the associated representation $\rho(\kp)$ is irreducible. Then the fibre $\Spec (R_b\otimes_{R^{\ps}}k(\kp))$ has at most one point.
\end{lem}
\begin{proof}
For any $\kp'\in \Spec R_b$ mapping to $\kp\in\Spec R^{\ps}$, we have a natural map: $R^{\ps}/\kp\to R_b/\kp'\to k(\kp')$ and a representation
\[\rho(\kp'):\Gamma\to\GL_2(R_b/\kp')\]
that lifts  $\bar{\rho}$. Note that $\kp'$  is one-dimensional as it cannot be the maximal ideal $\km_b$ of $R_b$. From our previous discussion, we may conjugate $\rho(\kp')$ by some element in $1+M_2(\km_b)$ and assume all the entries of $\rho(\kp')$ are in $k(\kp)$ and $\rho(\kp')(\sigma^*)=\begin{pmatrix} 1 & 0 \\ 0 & -1 \end{pmatrix}$ and $\rho(\kp')(\sigma_*)$ of the form $\begin{pmatrix} * & 1 \\ * & * \end{pmatrix}$. Note that under these assumptions,  since we assume $\rho(\kp')$ is irreducible, all the entries of $\rho(\kp')$ now are determined by $\tr(\rho(\kp'))$. Thus this proves the uniqueness of $\kp'$.
\end{proof}

The following result is due to Skinner-Wiles (cf. the last paragraph of \cite{SW99} page 68).

\begin{prop} \label{ht1def}
Let $a(\sigma),d(\sigma),x(\sigma,\tau),c(\sigma)\in R^{\ps}$ be the elements as in the beginning of the proof of proposition \ref{rpsrc}. Denote by $\km$ the maximal ideal of $R^{\ps}$. Suppose $\kq_b$ is a prime of $\Spec R_b$  such that $R_b/(\kq_b+\km R_b)$ is not Artinian. Then
\[\hht(I_c)\le 1,\]
where $I_c\subseteq R^{\ps}/\kQ$ is the ideal generated by $c(\sigma),\sigma\in\Gamma$ and $\kQ=\kq_b\cap R^{\ps}$.
\end{prop}
\begin{proof}
Let $\rho(\kq_b):\Gamma\to\GL_2(R_b/\kq_b)$ be a deformation induced from the universal deformation. As before, we may assume that it has the form $\sigma\mapsto \begin{pmatrix} a(\sigma)& b(\sigma)\\ c(\sigma)& d(\sigma)\end{pmatrix}$, where $a(\sigma),c(\sigma),d(\sigma)\in R^{\ps}/\kQ$ are defined in the proposition.

Let $B_1$ be the integral closure of $R^{\ps}/\kQ$ in its fraction field. Then by page 237 of \cite{Mat2} (since $R^{\ps}/\kQ$ is complete and reduced), $B_1$ is finite over $R^{\ps}/\kQ$. Denote the intersection of $B_1$ and $R_b/\kq_b$ (in the fraction field of $R^{\ps}/\kQ$) by $B'_1$. Then $B'_1$ is also a finite $R^{\ps}/\kQ$-algebra hence $B'_1/\km B'_1$ is Artinian.

We claim that $c(\sigma),\sigma\in\Gamma$ generate an ideal of height at most one in $B_1$. Note that since $B_1$ is finite over $R^{\ps}/\kQ$, this will imply the proposition. Suppose this is not true. Then for any height one prime ideal $\kp$ of $B_1$, there exists $\sigma\in\Gamma$ such that $c(\sigma)\notin \kp$, hence
\[b(\tau)= \frac{x(\tau,\sigma)}{c(\sigma)}\in (B_1)_{\kp}~\mbox{for any }\tau\in\Gamma.\]
Thus 
\[b(\tau)\in\bigcap_{\kp} (B_1)_{\kp}=B_1,\] 
where $\kp$ runs over all the height one prime ideals of $B_1$ and the equality is valid since $B_1$ is normal. In other words, all the entries of $\rho(\kq_b)$ lie in $B_1$ hence also in $B'_1$. Consider the natural map given by the inclusion $B'_1\to R_b/\kq_b$:
\[B'_1/\km B'_1\to R_b/(\kq_b+\km R_b).\]
Recall that $B'_1/\km B'_1$ is Artinian. The above map is surjective as the image contains all the entries of a universal deformation which topologically generate $R_b$. But this contradicts our assumption that $R_b/(\kq_b+\km R_b)$ is not Artinian. Therefore $c(\sigma),\sigma\in\Gamma$ generates an ideal of height at most one in $B_1$, which implies our assertion in the proposition.
\end{proof}

\begin{rem} \label{ht1Rem}
This remark will only be used in \ref{case3}. Proposition \ref{ht1def} can be extended to the following situation: Let $\psi_i:\Gamma\to E^\times,i=1,2$ be two distinct continuous characters. Then $T=\psi_1+\psi_2$ is a two-dimensional pseudo-representation of $\Gamma$ over $E$. As before, we assume $T(\sigma^*_0)=0$ for some order $2$ element $\sigma^*_0$. We can consider the universal deformation ring $R^{\ps}_T$ of $T$ over $E$ (see the paragraph above \ref{Dkp}). On the other hand, let 
\[\rho_B:\Gamma\to\GL_2(E),~\gamma\mapsto\begin{pmatrix} \psi_2(\gamma) & B(\gamma) \\ 0 & \psi_1(\gamma)\end{pmatrix}\] 
be a non-split extension of $\psi_1$ by $\psi_2$ and $R_B$ be its universal deformation ring (\ref{dr0}). Fix an element $\gamma_0\in\Gamma$ with $B(\gamma_0)=1$. We have a natural map $R^{\ps}_T\to R_B$. Denote the maximal ideal of $R^{\ps}_T$ by $\km_T$. Let $\kQ_B\in\Spec R_B$ such that $R_B/(\km_T+\kQ_B)$ is not Artinian. Then 
\[\hht (I_C)\leq 1,\]
where $I_C$ is the ideal of $R^{\ps}_T/(\kQ_B\cap R^{\ps}_T)$ generated by the images of $x(\gamma_0,\sigma),\sigma\in\Gamma$. Here $x(\cdot,\cdot)$ are defined using $\sigma^*_0$ as before.

The proof is almost the same as proposition \ref{ht1def}. We leave the details to the readers.
\end{rem}

\section{Local-global compatibility} \label{L-gc}
In this section, we will state and prove our local-global compatibility result and deduce some useful corollaries from it (\ref{sblgc}, \ref{edha}). Most results here are also obtained by Pa\v{s}k\={u}nas in \cite{Pas18} \S 5. According to Pa\v{s}k\={u}nas (Remark 5.13. ibid.), these results were discovered in 2011 and remained unpublished. I was unaware of his work when I worked on this project.

For notations, we fix a totally field $F$ in which $p$ completely splits and a quaternion algebra $D$ with centre $F$ which is ramified at all infinite places of $F$ and unramified at all places above $p$. Also we fix isomorphisms $D\otimes F_v\simeq M_2(F_v)$ for any $v$ where $D$ is unramified. Under these isomorphisms, we may view $K_p=\prod_{v|p}\GL_2(O_{F_v}),D_p^\times=\prod_{v|p}\GL_2(F_v)$ as subgroups of $(D\otimes_F\A_F)^\times$. We also write $N_{D/F}:(D\otimes_F\A_F)^\times\to \A_F^\times$ as the reduced norm.

\subsection{Quaternionic forms} \label{quaform}
\begin{para} \label{Quaformssm}
We first recall some results on quaternionic forms. Reference is \cite{Ta06} and \cite{Kis09a}. 

Let $A$ be a topological $\Z_p$-algebra and $U=\prod_{v} U_v$ be an open compact subgroup of $\DAi$ such that $U_v\subseteq \GL_2(O_{F_v})$ for $v|p$. Let $\psi:(\A^\infty_F)^\times/F^\times_{>>0}\to A^\times$ be a continuous character, where $F_{>>0}$ is the set of totally positive elements in $F$. Also let $\tau:\prod_{v|p}U_v\to\Aut(W_\tau)$ be a continuous representation on a finite $A$-module $W_\tau$. By abuse of notation, we also view $\tau$ as a representation of $U$ by projecting to $\prod_{v|p}U_v$.

Let $S_{\tau,\psi}(U,A)$ be the space of continuous functions:
\[f:D^{\times}\setminus \DAi\to W_{\tau}\]
such that for any $g\in \DAi,u\in U,z\in \AFi$, we have
\begin{itemize}
\item $f(gu)=\tau(u^{-1})(f(g))$,
\item $f(gz)=\psi(z)f(g)$.
\end{itemize}
Write $\DAi=\bigsqcup_{i\in I}D^\times t_iU\AFi$ for some finite set $I$ and $t_i\in \DAi$. If $\tau^{-1}|_{U\cap \AFi}$ acts as $\psi|_{U\cap \AFi}$, then there is an isomorphism:
\begin{eqnarray} \label{dcp}
S_{\tau,\psi}(U,A)\simeq \bigoplus_{i\in I}W_{\tau}^{(t_i^{-1}D^\times t_i\cap U\AFi)/F^{\times}},
\end{eqnarray}
by sending $f$ to $(f(t_i))_i$. We say $U$ is sufficiently small if $(t_i^{-1}D^\times t_i\cap U\AFi)/F^{\times}$ is trivial for all $i\in I$. This can be achieved by shrinking $U_v$ for some $v$:
\end{para}
\begin{lem}
Suppose that $D$ is unramified at $v$ and $\zeta+\zeta^{-1}\neq 2 \mod \varpi_v^n$ for some n and any root of unity $\zeta\neq 1$ in a quadratic extension of $F$. Then $U$ is sufficiently small if $U_v$ is contained in the subgroup of $\GL_2(O_{F_v})$, whose elements are unipotent upper triangular modulo $\varpi_v^n$.
\end{lem}
\begin{proof}
Let $\gamma\in D^\times\cap t_iU\AFi t_i^{-1}$ and $l$ be the order of $t_i^{-1}\gamma t_i$ in $(t_i^{-1}D^\times t_i\cap U\AFi)/F^{\times}$. Let $\iota:D\to D$ be the main involution. Then $\gamma^l=\iota(\gamma)^l$ and $\frac{\gamma}{\iota(\gamma)}$ is a root of unity in a quadratic extension of $F$. Now since $t_iU_vF_v^\times t_i^{-1}$ contains $\gamma$, it also contains $\frac{\gamma}{\iota(\gamma)}$. Taking the reduced trace of $\frac{\gamma}{\iota(\gamma)}$, we see that the assumption on $v$ implies $l=1$.
\end{proof}

\begin{cor} \label{sfsmf}
Assume $\tau^{-1}|_{U\cap \AFi}=\psi|_{U\cap \AFi}$ and $U$ is sufficiently small. Then $S_{\tau,\psi}(U,A)$ is a finite free $A$-module and $S_{\tau\otimes B,\psi}(U,B)\simeq S_{\tau,\psi}(U,A)\otimes B$ for any $A$-algebra $B$.
\end{cor}

\begin{para} \label{pcaf}
The relationship of $S_{\tau,\psi}(U,A)$ with classical automorphic forms on $D^\times$ is as follows (see lemma 1.3 of \cite{Ta06}). Suppose $A=E$ and $(\vec{k},\vec{w})\in \Z_{>1}^{\Hom(F,\overbar{\Q_p})}\times  \Z^{\Hom(F,\overbar{\Q_p})}$ such that $k_\sigma+2w_\sigma$ is independent of $\sigma:F\to \overbar{\Q_p}$. Write $w=k_{\sigma}+2w_{\sigma}-1$. We can define the following algebraic representation $\tau_{(\vec{k},\vec{w})}$ of $D_p^\times=(D\otimes \Q_p)^\times$ on
\[W_{(\vec{k},\vec{w}),E}=\bigotimes_{\sigma:F\to E}(\Sym^{k_\sigma-2}(E^2)\otimes \det{}^{w_\sigma}),\]
where $\Sym^{k_\sigma-2}$ denotes the space of homogeneous polynomials of degree $k_\sigma-2$ in two variables with an action of $\GL_2(F_{v(\sigma)})$ given by
\[\begin{pmatrix} a&b\\c&d\end{pmatrix}(f)(X,Y)=f(\sigma(a)X+\sigma(c)Y,\sigma(b)X+\sigma(d)Y).\] 
Here $v(\sigma)$ is the place above $p$ given by $\sigma$.  Let $\psi:(\A^\infty_F)^\times/F^\times_{>>0}\to E^\times$ be a continuous character such that $\tau_{(\vec{k},\vec{w})}^{-1}|_{U\cap \AFi}=\psi|_{U\cap \AFi}$ and write $\tau=\tau_{(\vec{k},\vec{w})}$. Then there is an isomorphism:
\begin{eqnarray*}
S_{\tau_{(\vec{k},\vec{w})},\psi}(U,E)\otimes_{E,\iota_p}\bC&\stackrel{\sim}{\longrightarrow} &\Hom_{D_\infty^\times}(W_{\iota_p(\vec{k},\vec{w}),\bC}^*,C^\infty(D^\times\setminus (D\otimes\A_F)^\times/U,\psi_\bC)),\\
f\otimes 1&\longmapsto& ``w^*\mapsto (g\mapsto w^*(\tau_\bC(g_\infty)^{-1}\tau(g_p)f(g^\infty)))"
\end{eqnarray*}
where 
\begin{itemize}
\item $D_\infty^\times=(D\otimes_\Q \R)^\times$, 
\item $W_{\iota_p(\vec{k},\vec{w}),\bC}=W_{(\vec{k},\vec{w}),E}\otimes_{E,\iota_p}\bC$ is viewed as an algebraic representation $\tau_\bC$ of $D_\infty^\times$ (induced by $\iota_p$), and $W_{\iota_p(\vec{k},\vec{w}),\bC}^*$ is its $\bC$-linear dual.
\item $\psi_\bC:\A_F^\times/F^\times\to\bC^\times$ sends $g$ to $N_{F/\Q}(g_{\infty})^{1-w}\iota_p(N_{F/\Q}(g_p)^{w-1}\psi(g^\infty))$.
\item $C^\infty(D^\times\setminus (D\otimes\A_F)^\times/U,\psi_\bC)$ is the space of smooth $\bC$-valued functions on $D^\times\setminus(D\otimes\A_F)^\times$, right invariant by $U$ and with central character $\psi_\bC$.
\end{itemize}
Note that the right hand side is a subspace of automorphic forms on $(D\otimes\A_F)^\times$.

\begin{rem} \label{pact}
This isomorphism is functorial in $U$. Fix $v|p$ and take the direct limit over all open compact subgroups $U_v$ in $\GL_2(O_{F_v})$. We get 
\[\varinjlim_{U_v} S_{\tau_{(\vec{k},\vec{w})},\psi}(U,E)\otimes\bC\simeq \Hom_{D_\infty^\times}(W_{\iota_p(\vec{k},\vec{w}),\bC}^*,C^\infty(D^\times\setminus (D\otimes\A_F)^\times/U^v,\psi_\bC))\]
where $U^v=\prod_{w\neq v} U_w$. Clearly there is an action of $\GL_2(F_v)$ on the right hand side by right translation. On the left hand side, this is given by 
\[g(f)(h)=\tau(g)(f(hg)),~g\in\GL_2(F_v),f\in S_{\tau_{(\vec{k},\vec{w})},\psi}(U,E), h\in \DAi.\]
\end{rem}
For simplicity, we will write $S_{(\vec{k},\vec{w}),\psi}(U,E)$ for $S_{\tau_{(\vec{k},\vec{w})},\psi}(U,E)$ from now on.
\end{para}

\subsection{Completed homology and cohomology} \label{chac}
\begin{para}
Now we introduce completed homology (and cohomology). We denote by $S_\psi(U,A)$ when $W_{\tau}=A$ with trivial $U_v$-actions. Note that the definition of $S_\psi(U,A)$ makes sense for any topological $\cO$-\textit{module} $A$. In the below we will use $S_\psi(U,A)$ for any topological $\cO$-module $A$ by abuse of notation. Given $U^p=\prod_{v\nmid p}U_v$ (a tame level) and a torsion $\cO$-algebra $A$ with discrete topology, we define
\end{para}

\begin{defn} 
\[S_{\psi}(U^p,A):=\varinjlim_{U_p} S_\psi(U^pU_p,A),\]
with discrete topology, where $U_p=\prod_{v|p}U_v$ runs over all open compact subgroups $U_v$ of $\GL_2(F_v)$. The \textit{completed cohomology} of tame level $U^p$ is defined to be
\[S_\psi(U^p):=\Hom_\cO(E/\cO,S_{\psi}(U^p,E/\cO))\]
equipped with $p$-adic topology, and the \textit{completed homology} is defined to be
\[M_\psi(U^p):=S_\psi(U^p,E/\cO)^\vee=\Hom_\cO(S_{\psi}(U^p,E/\cO),E/\cO)\]
equipped with compact-open topology.
\end{defn}

\begin{rem}
It is easy to see that $S_\psi(U^p,A)$ is naturally isomorphic to the space of locally constant $A$-valued functions on $D^\times\setminus\DAi$ right invariant by $U^p$ with central character $\psi$. Using this equivalent definition, there is a natural action of $D_p^\times=\prod_{v|p}\GL_2(F_v)$ on all spaces defined above by right translation. It is almost by definition that for any open compact subgroup $U_p=\prod_{v|p}U_v\subseteq K_p=\prod_{v|p}\GL_2(O_{F_v})$,
\[S_{\psi}(U^p,A)^{U_p}=S_{\psi}(U^pU_p,A).\]
Hence $S_\psi(U^p,E/\cO)$ is a smooth admissible $\cO$-representation of $D_p^\times$ in the sense of \S2 of \cite{Pas13}. It is also clear that $S_\psi(U^p)$ is $p$-torsion free and
\begin{eqnarray*}
S_\psi(U^p)\cong \varprojlim_n S_\psi(U^p,\cO/p^n)\cong \Hom^{\mathrm{cont}}_{\cO}(M_\psi(U^p),\cO),~M_\psi(U^p)\cong \Hom_\cO(S_\psi(U^p),\cO).
\end{eqnarray*}
\end{rem}

\begin{prop} \label{chproj}
Suppose $U^p$ is small enough such that $U^pK_p$ is sufficiently small and $\psi|_{N_{D/F}(U^p)}$ is trivial. Then $S_\psi(U^p,E/\cO)$ is an injective object in $\mathrm{Mod}_{K_p,\psi}^{\mathrm{sm}}(\cO)$. 
\end{prop}
\begin{proof}
This is a variant of proposition 4.4.3 of \cite{BH15}. We recall their proof here.
\begin{lem} \label{lcc}
Let $M$ be a finite torsion $\cO$-module with a smooth $K_p$-action. Assume $U^pK_p$ is sufficiently small. Then there is a natural isomorphism
\[S_{M,\psi}(U^pK_p,\cO)\stackrel{\sim}{\longrightarrow}\Hom_{\cO[K_p]}(M^\vee,S_\psi(U^p,E/\cO))\]
\end{lem}
\begin{proof}
The map is given by sending $f\in S_{M,\psi}(U^pK_p,\cO)$ to $\ell^\vee \mapsto (g\mapsto \ell^\vee(f(g)))$. One can easily construct an inverse of this map. We omit the details here.
\end{proof}
Now given $0\to\pi'\to\pi$ in $\mathrm{Mod}_{K_p,\psi}^{\mathrm{sm}}(\cO)$, one needs to show
\[\Hom_{\cO[K_p]}(\pi,S_\psi(U^p,E/\cO))\to\Hom_{\cO[K_p]}(\pi',S_\psi(U^p,E/\cO))\]
is surjective. By proposition 2.1.9. of \cite{Eme10a}, we may assume $\pi$ is admissible. If $\pi$ is a finite $\cO$-module, we may apply the previous lemma to $\pi^\vee,(\pi')^\vee$ and conclude the surjectivity from corollary \ref{sfsmf}. In general, we may write $\pi=\bigcup_{n} \pi_n$ as an increasing union of representations of finite $\cO$-length. Write $\pi'_n=\pi_n\cap\pi',\pi''_n=\pi_n/\pi'_n$. The result follows from taking inverse limit of 
\[\Hom_{\cO[K_p]}(\pi''_n,S_\psi(U^p,E/\cO))\to\Hom_{\cO[K_p]}(\pi_n,S_\psi(U^p,E/\cO))\to\Hom_{\cO[K_p]}(\pi'_n,S_\psi(U^p,E/\cO)).\]
Note that the first term is of finite $\cO$-length, so it satisfies Mittag-Leffler conditions.
\end{proof}

\begin{cor} \label{dsc}
Let $S_\psi(U^p)_E$ be $S_\psi(U^p)\otimes_\cO E$, which is a Banach space with unit ball $S_\psi(U^p)$. Then $S_{\psi}(U^p)_E$ is a topological direct summand of $\cC_\psi(K_p,E)^{\oplus d}$ as an $E[K_p]$-module for some $d$, where $\cC_\psi(K_p,E)$ is the space of continuous $E$-valued functions on $K_p$ with central character $\psi|_{O_{F,p}^\times}$.
\end{cor}
\begin{proof}
Taking the Pontryagin dual of the result in the proposition, we know that $M_\psi(U^p)$ is a projective object in the category of profinite linearly topological $\cO[[K_p]]$-modules with central character $\psi^{-1}|_{O_{F,p}^\times}$. Note that $S_\psi(U^p,E/\cO)$ is an admissible representation of $K_p$. Hence $M_\psi(U^p)$ is a finitely generated module over $\cO[[K_p]]$ and therefore a direct summand of 
\[(\cO[[K_p]]\otimes_{\cO[[O_{F,p}^\times]],\psi^{-1}}\cO)^{\oplus d}\]
for some $d$. Here we view $\cO$ as a $\cO[[O_{F,p}^\times]]$-module by $\psi^{-1}$. We get the corollary by taking the continuous $\Hom$ to $E$ since  
\[\Hom^{\cont}_E(\cO[[K_p]]\otimes_{\cO[[O_{F,p}^\times]],\psi^{-1}}\cO,E)\cong \cC_\psi(K_p,E).\]
\end{proof}

\begin{para}[Twisting of a character] \label{Tofac}
Given a continuous character $\theta:(\A^{\infty}_F)^\times/N_{D/F}(U^p)F_{>>0}^\times\to\cO^\times$,  we can define a natural isomorphism by twisting with $\theta$:
\[S_{\psi}(U^p)\stackrel{\sim}{\longrightarrow}S_{\psi\theta^2}(U^p)\]
by sending $f$ to $g\mapsto \theta(N_{D/F}(g))f(g)$. Here we identify $S_{\psi}(U^p)$ with functions on $D^\times\setminus \DAi$. Using this isomorphism, we may sometimes assume $\psi$ is of finite order (not now).
\end{para}

\begin{para}[A density result]
We need a density result like \cite{Eme1} \S5.4. This will reduce the local-global compatibility problem to the compatibility at \textit{crystalline points}. In this section, we fix a place $v$ above $p$ and we assume $\psi|_{N_{D/F}(U^p)}$ is trivial and $\psi|_{O_{F_v}^\times}$ is an \textit{algebraic character}. In particular, there exists an integer $m$ such that,
\[\psi(a_v)=\sigma_v(a_v)^m,~a_v\in O_{F_v}^\times,\]
where $\sigma_v:F\to E$ is the embedding induced by $v$. Consider the subspace $S_\psi(U^p)_E^{v-\mathrm{a},v'-\mathrm{la}}$  (see the definition below) of $\GL_2(O_{F_v})$-algebraic, $\prod_{w\neq v,w|p}\GL_2(O_{F_w})$-locally algebraic vectors of $S_\psi(U^p)_E$. The main result is
\end{para}

\begin{prop} \label{density}
Assume that $\psi|_{N_{D/F}(U^p)}$ is trivial and $\psi|_{O_{F_v}^\times}$ is an algebraic character. Then $S_\psi(U^p)_E^{v-\mathrm{a},v'-\mathrm{la}}$ is dense in $S_\psi(U^p)_E$.
\end{prop}

\begin{proof}
We first recall the construction of $S_\psi(U^p)_E^{v-\mathrm{a},v'-\mathrm{la}}$. For $(\vec{k},\vec{w})\in \Z_{>1}^{\Hom(F,\overbar{\Q_p})}\times  \Z^{\Hom(F,\overbar{\Q_p})}$ such that $k_\sigma+2w_\sigma=m+2$ for any $\sigma$, we can associate an algebraic representation of $K_p$ (see \ref{pcaf}):
\[W_{(\vec{k},\vec{w}),E}=\bigotimes_{\sigma:F\to E}(\Sym^{k_\sigma-2}(E^2)\otimes \det{}^{w_\sigma}).\]
Note that these $W_{(\vec{k},\vec{w}),E}$ exhaust all algebraic representations of $K_p$ with central character locally same as $\psi|_{O_{F,p}^\times}$. Let $U^v$ be an open compact subgroup of $\prod_{w\neq v,w|p}\GL_2(O_{F_w})$. Then the subspace of $\GL_2(O_{F_v})U^v$-algebraic vectors of $S_\psi(U^p)_E$ is defined to be the image of the evaluation map:
\[\bigoplus_{(\vec{k},\vec{w})}\Hom_{E[\GL_2(O_{F_v})U^v]}(W_{(\vec{k},\vec{w}),E},S_\psi(U^p)_E)\otimes_E W_{(\vec{k},\vec{w}),E}\to S_\psi(U^p)_E,\]
where the sum is taken over all $(\vec{k},\vec{w})$ with $k_\sigma+2w_\sigma=m+2$ for any $\sigma$. The subspace $S_\psi(U^p)_E^{v-\mathrm{a},v'-\mathrm{la}}$ of $\GL_2(O_{F_v})$-algebraic, $\prod_{w\neq v,w|p}\GL_2(O_{F_w})$-locally algebraic vectors of $S_\psi(U^p)_E$ is defined to be the union of all $\GL_2(O_{F_v})U^v$-algebraic vectors where $U^v$ runs through all open compact subgroups of $\prod_{w\neq v,w|p}\GL_2(O_{F_w})$.

Let $\cC_{\psi}^{v-\mathrm{a},v'-\mathrm{la}}(K_p,E)$ be the subspace of $\GL_2(O_{F_v})$-algebraic, $\prod_{w\neq v,w|p}\GL_2(O_{F_w})$-locally algebraic vectors in $\cC_{\psi}(K_p,E)$. By corollary \ref{dsc}, it suffices to prove that $\cC_{\psi}^{v-\mathrm{a},v'-\mathrm{la}}(K_p,E)$ is dense in $\cC_{\psi}(K_p,E)$. Note that
\[\cC_{\psi}(K_p,E)\cong \widehat{\bigotimes}_{w|p} \cC_{\psi|_{O_{F_w}^\times}}(\GL_2(O_{F_w}),E),\]
\[\cC_{\psi}^{v-\mathrm{a},v'-\mathrm{la}}(K_p,E)\cong \cC_{\psi|_{O_{F_v}^\times}}^{\mathrm{a}}(\GL_2(O_{F_w}),E)\otimes_E{\bigotimes}_{w|p,w\neq v} \cC_{\psi|_{O_{F_w}^\times}}^{\mathrm{la}}(\GL_2(O_{F_w}),E),\]
where $\cC_{\psi|_{O_{F_v}^\times}}^{\mathrm{a}}(\GL_2(O_{F_v}),E)\subseteq\cC_{\psi|_{O_{F_v}^\times}}(\GL_2(O_{F_v}),E)$ denotes the subspace of $\GL_2(O_{F_v})$-algebraic vectors and $ \cC_{\psi|_{O_{F_w}^\times}}^{\mathrm{la}}(\GL_2(O_{F_w}),E)\subseteq  \cC_{\psi|_{O_{F_w}^\times}}(\GL_2(O_{F_w}),E)$ denotes the subspace of $\GL_2(O_{F_w})$-locally algebraic vectors, $w\neq v$. It's enough to show that
$\cC_{\psi|_{O_{F_v}^\times}}^{\mathrm{a}}(\GL_2(O_{F_v}),E)$ (resp.  $ \cC_{\psi|_{O_{F_w}^\times}}^{\mathrm{la}}(\GL_2(O_{F_w}),E)$) is dense  inside $\cC_{\psi|_{O_{F_v}^\times}}(\GL_2(O_{F_v}),E)$ (resp. $\cC_{\psi|_{O_{F_w}^\times}}(\GL_2(O_{F_w}),E)$). We will only give a proof  for $\cC_{\psi|_{O_{F_v}^\times}}^{\mathrm{a}}(\GL_2(O_{F_v}),E)$ here. The case of $ \cC_{\psi|_{O_{F_w}^\times}}^{\mathrm{la}}(\GL_2(O_{F_w}),E)$ can be proved in a similar way. 

Let $\psi_v=\psi|_{O_{F_v}^\times}$. If $\psi_v$ is trivial, then $\cC_{\psi_v}(\GL_2(O_{F_v}),E)= \cC(\mathrm{PGL}_2(O_{F_v}),E)$, the space of continuous $E$-valued functions on $\mathrm{PGL}_2(O_{F_v})$. The density of $\mathrm{GL}_2(O_{F_v})$-algebraic vectors  in $\cC(\mathrm{PGL}_2(O_{F_v}),E)$ follows from Proposition 6.A.17 of \cite{Pas14} with $G=\PGL_2$.

In general, note that  $\cC(\mathrm{PGL}_2(O_{F_v}),E)$ is a commutative ring and $\cC_{\psi_v}(\GL_2(O_{F_v}),E)$ is a $\cC(\PGL_2(O_{F_v}),E)$-module. It suffices to prove
\[\cC_{\psi_v}(\GL_2(O_{F_v}),E)=\cC(\PGL_2(O_{F_v}),E)\cdot \cC_{\psi_v}^{\mathrm{a}}(\GL_2(O_{F_v}),E)\]
i.e. $\cC_{\psi_v}(\GL_2(O_{F_v}),E)$ is generated by its $\GL_2(O_{F_v})$-algebraic vectors as a $\cC(\PGL_2(O_{F_v}),E)$-module. Suppose $\psi_v(a_v)=\sigma_v(a_v)^m,~a_v\in O_{F_v}^\times$ for some integer $m$.  Write 
\[\GL_2(O_{F_v})\subseteq M_2(O_{F_v})=\{\begin{pmatrix}a& b\\ c&d\end{pmatrix},a,b,c,d\in O_{F_v}\}.\]
Then $\frac{a^{|m|}}{\det{}^{\frac{|m|-m}{2}}},\frac{b^{|m|}}{\det{}^{\frac{|m|-m}{2}}}$ can be viewed as  elements in $ \cC_{\psi_v}^{\mathrm{a}}(\GL_2(O_{F_v}),E)$. Let
\begin{eqnarray*}
U_a&=&\{\begin{pmatrix}a& b\\ c&d\end{pmatrix}\in\GL_2(O_{F_v}),a\notin p O_{F_v}\}\\
U_b&=&\{\begin{pmatrix}a& b\\ c&d\end{pmatrix}\in\GL_2(O_{F_v}),a\in pO_{F_v},b\notin p O_{F_v}\}
\end{eqnarray*}
It is clear that $U_a,U_b$ are disjoint open sets of $\GL_2(O_{F_v})$ and cover the whole group since $F_v=\Q_p$. Let $\mathbf{1}_a,\mathbf{1}_b$ be the characteristic functions of $U_a,U_b$. For any $f\in \cC_{\psi_v}(\GL_2(O_{F_v}),E)$, we can write
\[f=\frac{f\mathbf{1}_a\det{}^{\frac{|m|-m}{2}}}{a^{|m|}}\cdot \frac{a^{|m|}}{\det{}^{\frac{|m|-m}{2}}}+\frac{f\mathbf{1}_b\det{}^{\frac{|m|-m}{2}}}{b^{|m|}}\cdot \frac{b^{|m|}}{\det{}^{\frac{|m|-m}{2}}}.\]
Note that $\frac{f\mathbf{1}_a\det{}^{\frac{|m|-m}{2}}}{a^{|m|}},\frac{f\mathbf{1}_b\det{}^{\frac{|m|-m}{2}}}{b^{|m|}}$ are well-defined functions on $\PGL_2(O_{F_v})$. Hence we have expressed $f$ as an element in $\cC(\PGL_2(O_{F_v}),E)\cdot \cC_{\psi_v}^{\mathrm{a}}(\GL_2(O_{F_v}),E)$ and this is exactly what we want.
\end{proof}

\begin{para}[Relation with classical automorphic forms] \label{rlcc}
We can recover classical automorphic forms on $D^\times$ from the completed cohomology in the following  way. Suppose $(\vec{k},\vec{w})\in \Z_{>1}^{\Hom(F,\overbar{\Q_p})}\times  \Z^{\Hom(F,\overbar{\Q_p})}$ such that $k_\sigma+2w_\sigma$ is independent of $\sigma$ and $U_p$ is an open subgroup of $K_p$ such that $\psi|_{U^pU_p\cap (\A_F^\infty)^\times}=\tau^{-1}_{(\vec{k},\vec{w})}|_{U^pU_p\cap (\A_F^\infty)^\times}$. Then it is not too hard to deduce from lemma \ref{lcc} that
\[S_{(\vec{k},\vec{w}),\psi}(U^pU_p,E)\simeq \Hom_{E[U_p]}(W^{*}_{(\vec{k},\vec{w}),E},S_\psi(U^p)_E).\]
In other words, locally algebraic vectors in $S_{\psi}(U^p)_E$ can be identified with automorphic forms on $(D\otimes \A_F^\infty)^\times$ (see \ref{pcaf}).
\end{para}

\subsection{Hecke algebras and Pseudo-representations} \label{haapr}
\begin{para}
First we introduce Hecke algebras on finite levels. Fix a topological $\Z_p$-algebra $A$, a level $U$, character $\psi$ and a representation $\tau:U_p\to\Aut(W_{\tau})$. Let $S$ be a finite set of primes of $F$ containing all places above $p$ and places $v$ where either $D$ is ramified or $U_v$ is not a maximal open subgroup. For any place $v\notin S$, we define the Hecke operator $T_v\in \End(S_{\tau,\psi}(U,A))$ to be the double coset action $[U_v\begin{pmatrix}\varpi_v&0\\0&1\end{pmatrix}U_v]$. More precisely, write $U_v\begin{pmatrix}\varpi_v&0\\0&1\end{pmatrix}U_v=\bigsqcup_i \gamma_iU_v$, then
\[(T_v\cdot f)(g)=\sum_i f(g\gamma_i),~f\in S_{\tau,\psi}(U,A).\]

We define Hecke algebra $\T^S_{\tau,\psi}(U,A)\subseteq \End_A(S_{\tau,\psi}(U,A))$ to be the $A$-subalgebra generated by all $T_v,v\notin S$. This is a finite commutative $A$-algebra. By definition, $S_{\tau,\psi}(U,A)$ is a finitely-generated faithful $\T^S_{\tau,\psi}(U,A)$-module.

If $U'\subseteq U$ is a smaller open subgroup, then we have a natural surjective map $\T^S_{\tau,\psi}(U',A)\to \T^S_{\tau,\psi}(U,A)$ induced by $S_{\tau,\psi}(U,A)\subseteq S_{\tau,\psi}(U',A)$. 

If $U$ is small enough and $\psi|_{U\cap (\A_F^\infty)^\times}=\tau^{-1}|_{U\cap (\A_F^\infty)^\times}$, then for any ideal $I$ of $A$, there is a natural surjective map $\T^S_{\tau,\psi}(U,A)\to \T^S_{\tau\otimes A/I,\psi}(U,A/I)$ induced by $S_{\tau\otimes A/I,\psi}(U,A/I)\simeq S_{\tau,\psi}(U,A)\otimes A/I$ (see \ref{sfsmf}). This also implies that $\T^S_{\tau,\psi}(U,A)\otimes A/I\to \T^S_{\tau\otimes A/I,\psi}(U,A/I)$ has nilpotent kernel since $S_{\tau,\psi}(U,A)\otimes A/I$ has full supports on both rings.
\end{para}

\begin{para} \label{hac}
Now take $\tau=\tau_{(\vec{k},\vec{w})}$ (see \ref{pcaf}) and $A=\cO$. Assume $\psi|_{U\cap (\A_F^\infty)^\times}=\tau^{-1}_{(\vec{k},\vec{w})}|_{U\cap (\A_F^\infty)^\times}$. It is well-known that there exists a two-dimensional pseudo-representation:
\[T_{(\vec{k},\vec{w}),\psi}(U):G_{F,S}\to\T^S_{\tau,\psi}(U,\cO)=:\T^S_{(\vec{k},\vec{w}),\psi}(U,\cO)\]
sending $\Frob_v$ to $T_v$, $v\notin S$ . This pseudo-representation has determinant $\psi\varepsilon^{-1}$. Here by abuse of notation, we identify $\psi$ with a character of $G_{F,S}$ by the class field theory. Note that the image of this map contains all $T_v,v\notin S$ hence generates the whole Hecke algebra as an $\cO$-module. By Chebatorev density Theorem, $\T^S_{(\vec{k},\vec{w}),\psi}(U,\cO)$ is independent of $S$, so we may simply write $\T_{(\vec{k},\vec{w}),\psi}(U,\cO)$.

Under the same assumption plus $U$ is sufficiently small, we have a natural surjective map $\T_{(\vec{k},\vec{w}),\psi}(U,\cO)/\varpi^n\to \T^S_{\tau_{(\vec{k},\vec{w})}\otimes \cO/\varpi^n,\psi}(U,\cO/\varpi^n)$ for any $n$. Hence there also exists a two-dimensional pseudo-representation over $\T^S_{\tau_{(\vec{k},\vec{w})}\otimes \cO/\varpi^n,\psi}(U,\cO/\varpi^n)$ which is independent of $S$. For simplicity, we write it as $\T_{(\vec{k},\vec{w}),\psi}(U,\cO/\varpi^n)$.
\end{para}

\begin{para} \label{exgal}
We will write $\T^S_\psi(U,A)$ for the Hecke algebra if $\tau$ is the trivial action on $A$. Suppose $A=\cO/\varpi^n$, $U$ is sufficiently small and $\psi|_{U\cap (\A_F^\infty)^\times}$ is trivial modulo $\varpi^n$, i.e. $\psi(U\cap (\A_F^\infty)^\times)\subseteq 1+\varpi^n\cO$. We are going to show the existence of pseudo-representation $G_{F,S}\to\T^S_\psi(U,\cO/\varpi^n)$ with determinant $\psi\varepsilon^{-1}$ sending $\Frob_v$ to $T_v$ for $v\notin S$ in this case. 

It suffices to treat the case where $U$ is small enough such that $\psi|_{N_{D/F}(U)}$ is trivial modulo $\varpi^n$. In general, we may shrink $U$ to $U'$ small enough and conclude from the surjective map $\T^S_\psi(U',\cO/\varpi^n)\to \T^S_\psi(U,\cO/\varpi^n)$. Consider the Teichm\"uller lifting $\tilde\psi$ of $\psi$ modulo $\varpi$. Then $\psi^{-1}\tilde\psi$ is of pro-$p$ order, hence can be written as $\theta^2$ for some character $\theta:(\A_F^{\infty})^\times/F^\times_{>>0}\to\cO^\times$. It is easy to see that we may choose $\theta$ to be trivial on ${N_{D/F}(U)}$. Then twisting with $\theta$ induces an isomorphism (\ref{Tofac}):
\[\varphi_\theta:S_{\psi}(U,\cO/\varpi^n)\stackrel{\sim}{\longrightarrow}S_{\tilde\psi}(U,\cO/\varpi^n)\]
sending $f$ to $g\mapsto \theta(N_{D/F}(g))f(g)$. Note that this map also induces an isomorphism between Hecke algebras with
\[T_v\circ\varphi_\theta=\theta(\Frob_v)\varphi_{\theta}\circ T_v.\]
Now since $\tilde\psi$ is of finite order, we may apply the results in \ref{hac} with $\vec{k}=\vec{w}=\vec{0}$ and get the pseudo-representation over $\T^S_{\tilde\psi}(U,\cO/\varpi^n)$ hence on $\T^S_{\psi}(U,\cO/\varpi^n)$. The determinant is also easy to determine under this map. One direct corollary is that $\T^S_{\psi}(U,\cO/\varpi^n)$ is independent of $S$. So we may drop $S$ in it from now on.
\end{para}

\begin{defn}[Big Hecke algebra]
Let $U^p$ be a tame level and $\psi:(\A_F^\infty)^\times/(U^p\cap (\A_F^\infty)^\times)F^\times_{>>0}\to\cO^\times$ be a continuous character. Then we define the Hecke algebra 
\[\T_\psi(U^p)=\varprojlim_{(n,U_p)\in\cI}\T_{\psi}(U^pU_p,\cO/\varpi^n),\]
where $\cI$ is the set of pairs $(n,U_p)$ with $U_p\subseteq K_p$ and $n$ a positive integer such that $\psi|_{U_p\cap O_{F,p}^\times}\equiv 1\mod \varpi^n$. Equivalently, this is also
\[\varprojlim_{n}\varprojlim_{U_p\subseteq K_p}\T_{\psi}(U^pU_p,\cO/\varpi^n).\]
\end{defn}

\begin{para} \label{hagal}
It is almost by definition that $\T_\psi(U^p)$ acts faithfully on the completed cohomology $S_\psi(U^p)$ and  commutes with the action of $D_p^\times$. By our previous discussion, there exists a two-dimensional pseudo-representation $T_\psi(U^p):G_{F,S}\to \T_\psi(U^p)$ sending $\Frob_v$ to $T_v$ with determinant $\psi\varepsilon^{-1}$. 

From now on, we will always assume $\psi$ is trivial on $U^p\cap(\A^\infty_F)^\times$.
\end{para}

\begin{prop}\label{semiloc}
Let $U_p$ be a pro-$p$ subgroup of $K_p$. Hence $\psi|_{U_p\cap O_{F,p}^\times}$ is trivial modulo $\varpi$. Then the natural map
\[\T_{\psi}(U^p)\to\T_{\psi}(U^pU_p,\F)\]
induces a bijection between the sets of maximal ideals. In particular, $\T_{\psi}(U^p)$ is semi-local. 
\end{prop}
\begin{proof}
This is lemma 2.1.14 of \cite{GN16}. We give a sketch here. It suffices to prove the map
\[\T_\psi(U^pU'_p,\cO/\varpi^n)\to\T_{\psi}(U^pU_p,\F)\]
induces a bijection of maximal ideals, where $U'_p\subseteq U_p$ is an open normal subgroup small enough so that $U^pU'_p$ is sufficiently small and $\psi|_{U'_p\cap O_{F,p}^\times}$ is trivial modulo $\varpi^n$. 

Let $\km$ be a maximal ideal of the Artinian ring $\T_{\psi}(U^pU'_p,\cO/\varpi^n)$. Since $\T_{\psi}(U^pU'_p,\cO/\varpi^n)$ acts faithfully on $S_{\psi}(U^pU'_p,\cO/\varpi^n)$, we know that
\[S_{\psi}(U^pU'_p,\cO/\varpi^n)[\km]\neq 0.\]
The $p$-group $U_p/U'_p$ acts naturally on this $\F$-vector space, hence has a non-zero fixed vector, which by definition belongs to $S_{\psi}(U^pU_p,\F)$. Thus $S_{\psi}(U^pU_p,\F)[\km]\neq 0$ and $\km$ is also a maximal ideal of $\T_{\psi}(U^pU_p,\F)$.
\end{proof}

Let $U'_p,n$ be as in the proof. It follows from Theorem 8.15 of \cite{Mat1} that 
\[\T_\psi(U^pU'_p,\cO/\varpi^n)\cong \T_\psi(U^pU'_p,\cO/\varpi^n)_{\km_1}\times \cdots \T_\psi(U^pU'_p,\cO/\varpi^n)_{\km_r}\]
where $\km_1,\cdots,\km_r$ are the maximal ideals of $\T_\psi(U^pU'_p,\cO/\varpi^n)$. Each $\T_\psi(U^pU'_p,\cO/\varpi^n)_{\km_i}$ is an Artinian local $\cO$-algebra. Passing to the limit, we get

\begin{cor}
Let $\km_1,\cdots,\km_r$ be the maximal ideals of $\T_\psi(U^p)$. Then
\begin{enumerate}
\item $\T_\psi(U^p)\cong \T_\psi(U^p)_{\km_1}\times\cdots\times \T_\psi(U^p)_{\km_r}$ and each $\T_\psi(U^p)_{\km_i}$ is $\km_i$-adically complete and separated. 
\item $S_{\psi}(U^p)\cong S_{\psi}(U^p)_{\km_1}\oplus\cdots\oplus S_{\psi}(U^p)_{\km_r}$ and each $S_\psi(U^p)_{\km_i}$ is $\km_i$-adically complete and separated. 
\end{enumerate}
\end{cor}

\begin{para}
One easily checks that $\T_\psi(U^p)$ commutes with base change. To be precise, let $E'$ be a finite extension of $E$ with ring of integers $\cO'$. Replacing $\cO$ by $\cO'$ in all the definitions, we can define a Hecke algebra $\T'_{\psi}(U^p)$.  Then $\T'_\psi(U^p)\cong \T_\psi(U^p)\otimes_\cO \cO'$.

Suppose $\T_\psi(U^p)$ is non-zero and $\km\in\Spec \T_\psi(U^p)$ is a maximal ideal. Enlarge $\cO$ if necessary, we may assume $\km$ has residue field $\F$. Denote by $T_\km$ the two-dimensional pseudo-representation (over $\F$):
\[G_{F,S}\stackrel{T_\psi(U^p)}{\longrightarrow}\T^\psi(U^p)\to \T^\psi(U^p)/\km=\F.\]
For any $v|p$, let $R^{\ps}_v$ be the universal deformation ring which parametrizes all two-dimensional pseudo-representations of $G_{F_v}$ lifting $T_\km|_{G_{F_v}}$ with determinant $\psi\varepsilon^{-1}|_{G_{F_v}}$. By the universal property, we have a natural map $R^{\ps,\psi\varepsilon^{-1}}_v\to\T_\psi(U^p)_\km$. Taking tensor products over all $v|p$, we get
\[R^{\ps,\psi\varepsilon^{-1}}_p:=\widehat{\bigotimes}_{v|p}R^{\ps,\psi\varepsilon^{-1}}_v\to\T_\psi(U^p)_\km.\]
Therefore, we have defined an action of $R^{\ps,\psi\varepsilon^{-1}}_p$ on the completed cohomology $S_\psi(U^p)_\km$. We denote this action by $\tau_{\Gal}$. It will appear in one side of our local-global compatibility result.
\end{para}

\subsection{Pa\v{s}k\={u}nas' theory} \label{Paskunas}
In this subsection, we recall the main results of Pa\v{s}k\={u}nas in \cite{Pas13} \cite{Pas16} in the case of $\GL_2(\Q_p)$ and some generalizations in \cite{GN16} in the case of products of $\GL_2(\Q_p)$. The main application of his theory is to define another action of $R^{\ps,\psi\varepsilon^{-1}}_p$ on the completed cohomology. Reference is \S 3 and \S 5 of \cite{Pas13}.

\begin{para}
Let $G=\prod_{i=1}^f \GL_2(\Q_p)$ and $Z(G)\simeq \prod_{i=1}^f\Q_p^\times$ be its centre. Fix a character $\zeta:Z(G)\to\cO^\times$. We may formulate categories $\mathrm{Mod}_{G,\zeta}^{\mathrm{sm}}(\cO),\mathrm{Mod}^{\mathrm{l\,fin}}_{G,\zeta}(\cO),\mathrm{Mod}^{\mathrm{l\, adm}}_{G,\zeta}(\cO),\kC_{G,\zeta}(\cO)$. See notations in the beginning of the paper. We note that the last two categories are anti-equivalent to each other. And in fact, the middle two are the same:

\end{para}
\begin{lem}
Any locally admissible representation in $\mathrm{Mod}_{G,\zeta}^{\mathrm{sm}}(\cO)$ is locally of finite length.
\end{lem}
\begin{proof}
This is Theorem 2.3.8 of \cite{Eme10a} when $f=1$ and Lemma B.7 of \cite{GN16} in general. In fact, Gee and Newton work with products of $\PGL_2(\Q_p)$, but their proof also works with products of $\GL_2(\Q_p)$ with fixed determinant.
\end{proof}

\begin{para}[\textbf{Blocks}] \label{Blocks}
Let $\mathrm{Irr}_{G,\zeta}$ be the set of irreducible representations in $\mathrm{Mod}_{G,\zeta}^{\mathrm{sm}}(\cO)$. We consider the following equivalence relation $\sim$ on $\mathrm{Irr}_{G,\zeta}$: $\pi\sim\tau$ if there exists a sequence of irreducible representations $\pi_1=\pi,\pi_2,\cdots,\pi_n=\tau$ such that $\Ext^{1}_G(\pi_i,\pi_{i+1})\neq 0$ or $\Ext^{1}_G(\pi_{i+1},\pi_{i})\neq 0$ or $\pi_i\cong\pi_{i+1}$. An equivalence class is called a \textit{block}.

There exists a natural decomposition of $\mathrm{Mod}^{\mathrm{l\, adm}}_{G,\zeta}(\cO)$ with respect to the blocks:
\begin{eqnarray}
\mathrm{Mod}^{\mathrm{l\, adm}}_{G,\zeta}(\cO)\cong \prod_{\mathfrak{B}\in\mathrm{Irr}_{G,\zeta}/\sim}\mathrm{Mod}^{\mathrm{l\, adm}}_{G,\zeta}(\cO)^\kB,
\end{eqnarray}
where $\mathrm{Mod}^{\mathrm{l\, adm}}_{G,\zeta}(\cO)^\kB$ is the full subcategory of $\mathrm{Mod}^{\mathrm{l\, adm}}_{G,\zeta}(\cO)$ consisting of representations with all irreducible subquotients in $\kB$. Taking Pontryagin dual, this gives:
\[\kC_{G,\zeta}(\cO)\cong\prod_{\mathfrak{B}\in\mathrm{Irr}_{G,\zeta}/\sim}\kC_{G,\zeta}(\cO)^\kB\]

For a block $\kB$, write $\pi_{\kB}=\bigoplus_{\pi\in\kB_i}\pi$, where $\kB_i$ is the set of isomorphism classes of elements of $\kB$. We will see that this is actually a finite set. Let $\pi_\kB\hookrightarrow J_{\kB}$ be an \textit{injective envelope} of $\pi_{\kB}$ in $\mathrm{Mod}^{\mathrm{l\, adm}}_{G,\zeta}(\cO)$. Its Pontryagin dual $P_{\kB}:=J_{\kB}^\vee$ is a \textit{projective envelope} of $\pi^\vee\cong\bigoplus_{\pi\in\kB_i}\pi^\vee$ in $\kC_{G,\zeta}(\cO)$. Let
\[E_{\kB}:=\End_{\kC_{G,\zeta}(\cO)}(P_{\kB})\cong \End_G(J_{\kB}).\] 
This is a pseudo-compact ring. The topology is given as follows: for any quotient map $q:P_\kB\twoheadrightarrow M$ to some $M\in\kC_{G,\zeta}(\cO)$ of finite length, we may define a right ideal:
\[\mathfrak{r}(M)=\{\phi\in E_{\kB},q\circ \phi=0\}.\]
Such $\{\mathfrak{r}(M)\}$ forms a basis of open neighborhood of $0$ in $E_{\kB}$. We note that $P_{\kB}$ is a natural left $E_{\kB}$-module.

Suppose $M\in \kC_{G,\zeta}(\cO)$, then $E_{\kB}$ acts naturally on $\Hom_{\kC_{G,\zeta}(\cO)}(P_{\kB},M)$ (on the right). By writing $M=\varprojlim M_i$ with $M_i$ of finite length, we can equip the projective topology on $\Hom(P_{\kB},M)=\varprojlim\Hom(P_\kB,M_i)$, which makes $\Hom(P_{\kB},M)$ into a pseudo-compact $E_{\kB}$-module. In fact, this functor 
\[M\mapsto \Hom_{\kC_{G,\zeta}(\cO)}(P_\kB,M)\] 
defines an \textit{anti-equivalence} of categories between $\kC_{G,\zeta}(\cO)^\kB$ and the category of right pseudo-compact $E_{\kB}$-modules. An inverse functor is given by 
\[\mathrm{m}\mapsto (\mm\hat{\otimes}_{E_{\kB}}P_{\kB}).\]

$P_\kB$ is called a \textit{projective generator} of $\kB$.
\end{para}

\begin{para} \label{pas}
In \cite{Pas13}\cite{Pas16}, Pa\v{s}k\={u}nas computes $E_{\kB}$ and its centre when $G=\GL_2(\Q_p)$ in almost all cases. We now recall his results. 

Suppose $G=\GL_2(\Q_p)$. In this case, blocks containing an absolutely irreducible representation are computed in \cite{Pas14} Cor 6.2. For our purpose, we only list the results when $p\geq 3$.
\begin{enumerate}
\item $\kB=\{\pi\}$, $\pi$ is supersingular.
\item $\kB=\{(\Ind^G_B\delta_1\otimes\delta_2\omega^{-1})_{\mathrm{sm}},(\Ind^G_B\delta_2\otimes\delta_1\omega^{-1})_{\mathrm{sm}}\}$ with $\delta_2\delta_1^{-1}\neq \omega^{\pm 1},\mathbf{1}$.
\item $\kB=\{(\Ind^G_B\delta\otimes\delta\omega^{-1})_{\mathrm{sm}}\}$.
\item If $p\geq 5$, $\kB=\{\mathbf{1},\mathrm{Sp},(\Ind^G_B\omega\otimes \omega^{-1})_{\mathrm{sm}}\}\otimes \delta\circ\det$.
\item If $p=3$, $\kB=\{\mathbf{1},\mathrm{Sp},\omega\circ\det,\mathrm{Sp}\otimes\omega\circ\det\} \otimes \delta\circ\det$.
\end{enumerate}
for some smooth characters $\delta,\delta_1,\delta_2:\Q_p^\times\to \F^\times$. Here $B$ denotes the upper triangular Borel subgroup of $B$, $(\Ind^G_B)$ denotes the smooth induction and $\mathrm{Sp}$ denotes the mod $p$ special representation. See \cite{Pas14} for the precise definitions.

To describe (the centre of) $E_{\kB}$, we need to attach a semi-simple $2$-dimensional representation $\bar{\rho}_{\kB}$ of $G_{\Q_p}$ over $\F$ to each block. This is given by the following list:
\begin{enumerate}
\item $\bar{\rho}_{\kB}=\mathbf{V}(\pi)$ if a supersingular $\pi\in\kB$ , where $\mathbf{V}$ is the Colmez's functor normalized in \S5.7 of \cite{Pas13} 
\item $\bar{\rho}_{\kB}=\delta_1\oplus\delta_2$ if $(\Ind^G_B\delta_1\otimes\delta_2\omega^{-1})_{\mathrm{sm}}\in\kB$ with $\delta_2\delta_1^{-1}\neq \omega^{\pm 1}$.
\item $\bar{\rho}_{\kB}=\delta\oplus\delta\omega$ if $\delta\circ\det\in\kB$.
\end{enumerate}
Under this correspondence, the determinant of $\bar{\rho}_\kB$ is $\zeta\varepsilon\mod\varpi$. Note that this actually defines a bijection between the set of blocks containing an absolutely irreducible representation and the isomorphism classes of two-dimensional representations of $G_{\Q_p}$ over $\F$ which are direct sums of absolutely irreducible representations.
\end{para}

\begin{thm}[Pa\v{s}k\={u}nas] \label{pa1}
Let $R^{\ps,\zeta\varepsilon}_{\kB}$ be the universal deformation ring which parametrizes all $2$-dimensional pseudo-representations of $G_{\Q_p}$ lifting $\tr\bar{\rho}_\kB$ with determinant $\zeta\varepsilon$. Then there exists a natural isomorphism between the centre of $E_{\kB}$ and $R^{\ps,\zeta\varepsilon}_{\kB}$ in case (1)(2)(3)(4).
\end{thm}
\begin{proof}
This is Theorem 1.5 of \cite{Pas13} and Theorem 1.3 of \cite{Pas16}.
\end{proof}

\begin{thm}[Pa\v{s}k\={u}nas] \label{pa2}
$E_\kB$ is a finitely generated module over its centre in case (1)(2)(3)(4).
\end{thm}
\begin{proof}
See Theorem 1.2 of \cite{Pas16} and Corollary 8.11, Corollary 9.25, Lemma 10.90 of \cite{Pas13}.
\end{proof}

\begin{rem}
Both theorems are also known in case (5) and all the cases when $p$=2 thanks to the recent work  \cite{PT21} of Pa\v{s}k\={u}nas-Tung.
\end{rem}

\begin{para} 
Now we generalize the above results to the case $G=\prod_{i=1}^f\GL_2(\Q_p)$. First we need to classify the blocks containing an absolutely irreducible representation in this case.
\end{para}

\begin{lem} \label{irrof}
\hspace{2em}
\begin{enumerate}
\item For any absolutely irreducible representation $\pi$ in $\mathrm{Mod}_{G,\zeta}^{\mathrm{sm}}(\cO)$, there is a finite extension $\F'/\F$ such that $\pi\otimes_{\F}\F'$ is isomorphic to some $\bigotimes_{i=1}^f\pi_i$, where $\pi_i$ are absolutely irreducible $\GL_2(\Q_p)$ representations over $\F'$.
\item Moreover assume that $\F'=\F$. Let $P_r\to\pi_r^\vee$ be a projective envelope of $\pi_r^\vee$ for $r=1,\cdots,f$. Then $P_1\hat{\otimes}\cdots\hat{\otimes}P_f\to\widehat{\bigotimes}_{i=1}^f\pi_i^\vee$ is a projective envelope of $\widehat{\bigotimes}_{i=1}^f\pi_i^\vee\cong(\bigotimes_{i=1}^f\pi_i)^\vee$ in $\kC_{G,\zeta}(\cO)$. 
\end{enumerate}
\end{lem}
\begin{proof}
This is Lemma B.7, B.8 of \cite{GN16}.
\end{proof}

\begin{lem} \label{seprod}
Write $G=G_1\times G_2$, where each $G_r=\prod_{i=1}^{f_r}\GL_2(\Q_p)$ with centre $Z_r,~r=1,2$. Suppose $M_r,N_r\in\kC_{G_r,\zeta|_{Z_r}}(\cO),~r=1,2$. Then there exists a natural isomorphism:
\[\Hom_{\kC_{G,\zeta}(\cO)}(M_1\hat{\otimes}M_2,N_1\hat{\otimes} N_2)\cong\Hom_{\kC_{G_1,\zeta|_{Z_1}}(\cO)}(M_1,N_1)\hat{\otimes}\Hom_{\kC_{G_2,\zeta|_{Z_2}}(\cO)}(M_2,N_2)\]
\end{lem}
\begin{proof}
The argument in proof of Lemma B.8 of \cite{GN16} works here without any changes.
\end{proof}

\begin{lem} \label{fblk}
Given $f$ blocks $\kB_1,\cdots,\kB_f$ of $\GL_2(\Q_p)$ that contain an absolutely irreducible representation, then
\[\kB_1\otimes\cdots\otimes\kB_f:=\{\pi_1\otimes\cdots\otimes\pi_f|\pi_r\in\kB_r,r=1,\cdots,f\}\]
is a block of $G$. 
\end{lem}

\begin{proof}
We will only prove the case $f=2$. The general case will follow directly by induction on $f$. Let $\pi=\pi_1\otimes\pi_2,\pi'=\pi'_1\otimes\pi'_2$ be two irreducible representations of $G$ such that $\pi_1,\pi'_1$ are in not the same block of $\GL_2(\Q_p)$. We claim that $\Ext^1_G(\pi,\pi')=\Ext^1_G(\pi',\pi)=0$.  To see this, let $\pi_1\hookrightarrow J_1,\pi_2\hookrightarrow J_2$ be injective envelopes of $\pi_1,\pi_2$.  Suppose we have a non-split extension: 
\[0\to\pi\to \pi''\to\pi'\to 0.\]
By lemma \ref{irrof}, $J:=(J_1^\vee\hat{\otimes}J_2^\vee)^\vee$ is an injective envelope of $\pi$. It is easy to see that all irreducible subquotients of $J$ must be a tensor product of irreducible subquotients of $J_1,J_2$. By universal property, we get a map $\pi''\to J$ which has to be injective. But $\pi'_1$ cannot appear as a subquotient of $J_1$ since $\pi_1,\pi'_1$ live in different blocks. Thus there is no such non-split extension. This proves that each block of $G$ must be contained in some $\kB_1\otimes\kB_2$.

It rests to prove that if $\Ext^1_{\GL_2(\Q_p)}(\pi'_1,\pi_1)\neq 0$ and $\pi_1\ncong\pi'_1$, then $\Ext^1_G(\pi'_1\otimes\pi_2,\pi_1\otimes\pi_2)\neq 0$ for irreducible representations $\pi_1,\pi'_1,\pi_2$ of $\GL_2(\Q_p)$. Choose a non-split extension $\pi_3$ of $\pi'_1$ by $\pi_1$. Then $\Hom_G(\pi_3\otimes\pi_2,\pi_1\otimes\pi_2)\cong\Hom_{\GL_2(\Q_p)}(\pi_3,\pi_1)\otimes \Hom_{\GL_2(\Q_p)}(\pi_2,\pi_2)=0$ by the previous lemma. Thus $\pi_3\otimes\pi_2$ is a non-split extension of $\pi'_1\otimes\pi_2$ by $\pi_1\otimes\pi_2$. 
\end{proof}

\begin{cor} \label{fpas}
Let $\kB=\kB_1\otimes\cdots\otimes\kB_f$ be a block of $G$ such that each $\kB_i$ contains an absolutely irreducible representation. If $p=3$, we assume no $\kB_r$ contains $\delta\circ\det$ (case (5)). Then
\begin{enumerate}
\item $E_\kB\cong \widehat{\bigotimes}_{r=1}^f E_{\kB_r}$
\item The natural inclusion $R^{\ps,\zeta\varepsilon}_{\kB_r}\to E_{\kB_r}$ in Theorem \ref{pa1} induces a natural finite map 
\[\widehat{\bigotimes}_{r=1}^f R^{\ps,\zeta\varepsilon}_{\kB_r}\to E_{\kB},\] 
which makes $E_{\kB}$ into a finitely generated module over $\widehat{\bigotimes}_{r=1}^f R^{\ps,\zeta\varepsilon}_{\kB_r}$.
\item The centre of $E_\kB$ is Noetherian and $E_{\kB}$ is a finitely generate module over its centre.
\end{enumerate}
\end{cor}
\begin{proof}
Clearly the third part follows directly from the second since the image of $\widehat{\bigotimes}_{r=1}^f R^{\ps,\zeta\varepsilon}_{\kB_r}$ is in the centre of $E_{\kB}$. By the previous lemmas, $P_{\kB}\cong\widehat{\bigotimes}P_{\kB_r}$ and $\End(P_\kB)\cong \widehat{\bigotimes} \End(P_{\kB_r})$. The rest of the corollary all follows from Theorem \ref{pa2}.
\end{proof}

\begin{rem}
It seems that $\widehat{\bigotimes}_{r=1}^f R^{\ps,\zeta\varepsilon}_{\kB_r}$ is exactly the centre of $E_{\kB}$. This is at least true if all $\kB_r$ are in case (1)(2)(3) since $E_{\kB_r}$ is a free module over its centre in these cases.
\end{rem}

\subsection{Local-global compatibility} \label{sblgc}
In this subsection, we use Pa\v{s}k\={u}nas' theory to define another action of `$R^{\ps}_p$' on the completed cohomology and prove that it is equal to the action defined from Galois side at the end of subsection \ref{haapr}. 

\begin{para} \label{defnBmv}
Back to the setting in subsection \ref{haapr}. Let $\km$ be a maximal ideal of $\T_\psi(U^p)$. We get a two-dimensional pseudo-representation with determinant $\psi\varepsilon^{-1}$:
\[T_\km:G_{F,S}\to \T_\psi(U^p)_\km.\]
Restrict it to $G_{F_v},v|p$, we can attach a two-dimensional semi-simple representation $\bar{\rho}_{\km,v}$ of $G_{F_v}$ over $\F$. Enlarging $\cO$ if necessary, we assume that $\bar{\rho}_{\km,v}$ is a direct sum of absolutely irreducible representations. Let $\bar{\rho}'_{\km,v}=\bar{\rho}_{\km,v}\otimes \varepsilon$. Using the recipe in \ref{pas}, we can define a block $\kB_{\km,v}$ of $\GL_2(F_v)\cong\GL_2(\Q_p)$ from $\bar{\rho}'_{\km,v}$. Let $\kB_\km=\otimes_{v|p}\kB_{\km,v}$ be the block of $D_p^\times=\prod_{v|p}\GL_2(F_v)$ defined in lemma \ref{fblk}. Note that it has central character $\psi$. As before, we denote its projective generator by $P_{\kB_{\km}}$.

One of the central objects in our study is
\end{para}

\begin{defn} \label{defnm}
$\mm:=\Hom_{\kC_{D_p^\times,\psi}(\cO)}(P_{\kB_\km},M_\psi(U^p)_\km)$.
\end{defn}

\begin{rem}
The twist of cyclotomic character in $\bar{\rho}'_{\km,v}$ comes from the normalization of Colmez's functor used in Pa\v{s}k\={u}nas' paper. For example, there would be no such twist if we are using the magical functor in \cite{Eme1}.
\end{rem}

\begin{para} \label{gaact}
There two actions of $R^{\ps,\psi\varepsilon^{-1}}_p:=\widehat{\bigotimes}_{v|p}R^{\ps,\psi\varepsilon^{-1}}_v$ on $\mm$: 
\begin{enumerate}
\item $\tau_{\Gal}$: defined in the last paragraph of subsection \ref{haapr}, which comes from the action of $\T_\psi(U^p)_\km$ on $M_\psi(U^p)_\km$.
\item $\tau_{\mathrm{Aut}}$: which comes from the action of $\widehat{\bigotimes}_{v|p}R^{\ps,\psi\varepsilon}_v\cong\widehat{\bigotimes}_{v|p}R^{\ps,\psi\varepsilon^{-1}}_v$ on $P_{\kB_\km}$ via $E_{\kB_\km}$ in corollary \ref{fpas}.  The natural isomorphisms between $R^{\ps,\psi\varepsilon}_v$ and $R^{\ps,\psi\varepsilon^{-1}}_v$ are given by twisting the inverse of cyclotomic character. We fix this isomorphism from now on.
\end{enumerate}

Now we can state our main result of this section.
\end{para}

\begin{thm}[Local-global compatibility] \label{lgc}
If $p=3$, we assume that $\delta\circ\det\notin\kB_{\km,v}$ for any $v|p$. Then
\begin{enumerate}
\item $M_\psi(U^p)_\km\in\kC_{D_p^\times,\psi}(\cO)^{\kB_{\km}}$ where $\kB_{\km}$ is the block defined above.
\item Both actions $\tau_{\Gal},\tau_{\mathrm{Aut}}$ of $R^{\ps,\psi\varepsilon^{-1}}_p$ on $\mm$ are the same.
\end{enumerate}
\end{thm}

We will follow the strategy of Emerton in \cite{Eme1}. The idea is that using the density result (proposition \ref{density}), it suffices to check the compatibility of both actions on classical \textit{crystalline} points. But this is a consequence of the results and Berger-Breuil \cite{BB10} and classical local-global compatibility. We also remark that the argument relies on the semi-simplicity of the Hecke actions at finite levels.

\begin{proof}
Let $v$ be a place above $p$. It suffices to prove 
\begin{enumerate} 
\item $M_\psi(U^p)_\km\in\kC_{\GL_2(F_v),\psi|_{F_v^\times}}(\cO)^{\kB_{\km,v}}$ where $\kB_{\km,v}$ is the block defined in \ref{defnBmv},
\item both actions $\tau_{\Gal},\tau_{\mathrm{Aut}}$ of $R^{\ps,\psi\varepsilon^{-1}}_v$ on $\Hom_{\kC_{\GL_2(F_v),\psi|_{F_v^\times}}(\cO)}(P_{\kB_{\km,v}},M_\psi(U^p)_\km)$ are the same.
\end{enumerate}

Since the formulation of the Theorem is compatible with twisting of characters, we may assume $\psi$ is crystalline at $v$ of Hodge-Tate weight $w_\psi$. Also it is clear that we may shrink $U^p$ so that $\psi|_{U^p\cap(\A_F^\infty)^\times}$ is trivial and $U^pK_p$ is sufficiently small. 
 
Consider the isomorphism in \ref{rlcc}: for any $(\vec{k},\vec{w})\in \Z_{>1}^{\Hom(F,\overbar{\Q_p})}\times  \Z^{\Hom(F,\overbar{\Q_p})}$ such that $k_\sigma+2w_\sigma=w_\psi+2$ independent of $\sigma$ and $U_p=\GL_2(O_{F_v})U^v$ with $U^v$ an open subgroup of  $\prod_{w\neq v, w|p}\GL_2(O_{F_w})$, we have
\begin{eqnarray*} \label{rlcc2}
S_{(\vec{k},\vec{w}),\psi}(U^pU_p,E)\cong \Hom_{E[U_p]}(W^{*}_{(\vec{k},\vec{w}),E},S_\psi(U^p)_E).
\end{eqnarray*}
It is clear that this isomorphism is Hecke-equivariant. We get a natural surjective map $\T_\psi(U^p)[\frac{1}{p}]\to\T_{(\vec{k},\vec{w}),\psi}(U^pU_p,E)$ sending $T_w$ to $T_w$ for $w\notin S$. Hence it follows from the theory of classical automorphic forms that the action of $\T_\psi(U^p)[\frac{1}{p}]$ on $S_{(\vec{k},\vec{w}),\psi}(U^pU_p,E)$ is \textit{semi-simple}. Let $\kp$ be a prime ideal of $\T_{(\vec{k},\vec{w}),\psi}(U^pU_p,E)\otimes_E\overbar{\Q_p}$. Then it corresponds to an automorphic representation $\pi_\kp=\pi_\kp^\infty\otimes(\pi_\kp)_\infty$ on $\DAi$. From the discussion in \ref{pcaf}, we know that
\[(S_{(\vec{k},\vec{w}),\psi}(U^pU_p,E)\otimes_{E,\iota_p} \bC)[\kp]\cong (\pi_\kp^\infty)^{U^pU_p}. \]
Fix $U^v$ and take the limit over all open compact subgroups $U_v$ of $\GL_2(O_{F_v})$. We get
\[\varinjlim_{U_v} (S_{(\vec{k},\vec{w}),\psi}(U^pU^vU_v,E)\otimes_{E,\iota_p} \bC)[\kp]\cong (\pi_\kp^\infty)^{U^pU^v}\cong (\pi_\kp)_v^{\oplus d(\kp)},\]
for some $d(\kp)>0$. Here $(\pi_\kp)_v$ is the local representation of $\pi_\kp$ at place $v$. Take a finite extension $E(\kp)$ of $E$ such that $\kp$ is defined over $E(\kp)$. There exists a model $\pi^{E(\kp)}_v$ over $E(\kp)$ of $(\pi_\kp)_v$. Then we have 
\[\varinjlim_{U_v} S_{(\vec{k},\vec{w}),\psi}(U^pU^vU_v,E(\kp))[\kp]\cong (\pi^{E(\kp)}_v)^{\oplus d(\kp)}.\]
Combining with the isomorphism in \ref{rlcc}, we get a map (by abuse of notation, $\kp$ is viewed as a maximal ideal of $\T_\psi(U^p)\otimes E(\kp)$)
\[\Phi_{\kp}:W^{*}_{(\vec{k},\vec{w}),E}\otimes_E (\pi^{E(\kp)}_v)^{\oplus d(\kp)}\to (S_\psi(U^p)\otimes_{\cO} E(\kp))[\kp].\]
Using remark \ref{pact}, it is easy to see that this map is actually $\GL_2(F_v)U^v$-equivariant. Moreover the image contains the $\GL_2(O_{F_v})U^v$-algebraic vectors of $(S_\psi(U^p)_E\otimes_E E(\kp))[\kp]$ (see \ref{density} for the precise definition here). This is essentially because 
\[\Hom_{E[U_p]}(W^{*}_{(\vec{k'},\vec{w'}),E},S_\psi(U^p)_E\otimes E(\kp)[\kp])=0\]
unless $\vec{k}=\vec{k'},\vec{w}=\vec{w'}$. We denote the \textit{closure} of the image of $\Phi_\kp$ in $S_{\psi}(U^p)\otimes_\cO E(\kp)$ by $\Pi(\kp)$. Recall that $S_{\psi}(U^p)\otimes_\cO E(\kp)$ is a Banach space with a unit ball $S_{\psi}(U^p)\otimes_\cO O_{E(\kp)}$.

Let $\Pi_{\kB_{\km,v}}:=\Hom^{\cont}_\cO(P_{\kB_{\km,v}},E)$. This is a Banach space with unit ball $\Hom^{\cont}_\cO(P_{\kB_{\km,v}},\cO)$. 

\begin{lem} \label{inccr}
The inclusion map $\Pi(\kp)\hookrightarrow S_{\psi}(U^p)\otimes_{\cO} E(\kp)$ induces a natural injective map:
\[\Hom^{\cont}_{E[\GL_2(F_v)]}(S_{\psi}(U^p)_\km\otimes_\cO E,\Pi_{\kB_{\km,v}})\hookrightarrow\prod_{U^v}\prod_{(\vec{k},\vec{w})}\prod_\kp \Hom^{\cont}_{E(\kp)[\GL_2(F_v)]}(\Pi(\kp),\Pi_{\kB_{\km,v}}\otimes E(\kp)),\]
where $U^v$ runs over all open subgroups of $\prod_{w\neq v,w|p} \GL_2(O_{F_w})$, the pair $(\vec{k},\vec{w})$ runs over all elements  in $\Z_{>1}^{\Hom(F,\overbar{\Q_p})}\times  \Z^{\Hom(F,\overbar{\Q_p})}$ such that $k_\sigma+2w_\sigma=w_\psi+2$  for any $\sigma$, and $\kp$ runs over all the maximal ideals of $\T_{(\vec{k},\vec{w}),\psi}(U^p\GL_2(O_{F_v})U^v)_\km\otimes \overbar{\Q_p}$.
\end{lem}

\begin{proof}
By proposition \ref{density}, any continuous map from $S_{\psi}(U^p)_\km\otimes E$ is determined by its value on $(S_{\psi}(U^p)_\km\otimes E)^{v-\mathrm{a},v'-\mathrm{la}}$. But the action of Hecke algebra on this space is semi-simple, hence it follows from our previous discussion that it is contained in the space generated by all $\Pi(\kp)\cap S_{\psi}(U^p)_\km\otimes E$. This clearly proves the lemma.
\end{proof}

Note that the left hand side of the previous lemma is nothing but:
\begin{eqnarray*}
\Hom^{\cont}_{E[\GL_2(F_v)]}(S_\psi(U^p)_\km\otimes E,\Pi_{\kB_{\km,v}})\cong \Hom_{\kC_{\GL_2(F_v),\psi|_{F_v^\times}}(\cO)} (P_{\kB_{\km,v}},M_\psi(U^p)_\km) \otimes E.
\end{eqnarray*}

Since $S_{\psi}(U^p)$ is $p$-torsion free, it follows from the previous lemma that it suffices to prove 
\[\tau_{\Gal}|_{R^{\ps,\psi\varepsilon^{-1}}_v}=\tau_{\mathrm{Aut}}|_{R^{\ps,\psi\varepsilon^{-1}}_v}~\mbox{ on }\Hom^{\cont}_{E(\kp)[\GL_2(F_v)]}(\Pi(\kp),\Pi_{\kB_{\km,v}}\otimes E(\kp))\] 
for any $\kp$. Fix such a $\kp$ and  suppose it comes from $\T_{(\vec{k},\vec{w}),\psi}(U^p\GL_2(O_{F_v})U^v)_\km$. Since all the formulations are compatible with base change, we may assume $E(\kp)=E$.

There are two possibilities for $\kp$ depending on whether the automorphic representation $\pi_\kp$ factors through the reduced norm map $N_{D/F}$ or not. First we assume that $\pi_\kp$ does \textit{not} factor through $N_{D/F}$. Then the under Jacquet-Langlands correspondence, $\pi_\kp$ corresponds to a regular algebraic cuspidal automorphic representation of $\GL_2(\A_F)$. See lemma 1.3. of \cite{Ta06} for more details.

\noindent \underline{\textbf{Galois side}} The action $\tau_{\Gal}$ of $R^{\ps,\psi\varepsilon^{-1}}_v$ on $\Hom^{\cont}_{E[\GL_2(F_v)]}(\Pi(\kp),\Pi_{\kB_{\km,v}}\otimes E)$ is clear by our knowledge on the classical local-global compatibility at primes above $p$. 

\begin{lem} \label{clgc}
Let $\kp_v=R^{\ps,\psi\varepsilon^{-1}}_v[\frac{1}{p}]\cap\kp$ and $\rho(\kp)_v:G_{F_v}\to\GL_2(E)$ be the semi-simple representation given by $\kp_v$. Then
\begin{enumerate}
\item $\rho(\kp)_v$ is de Rham of Hodge-Tate weights $(w_{\sigma_v},w_{\sigma_v}+k_{\sigma_v}-1)$, where $\sigma_v:F\to E$ is the embedding induced by $v$. More precisely, 
\[\gr^i(\rho(\kp)_v\otimes B_{\dR})^{G_{F_v}}=0\] 
unless $i=w_{\sigma_v},w_{\sigma_v}+k_{\sigma_v}-1$.
\item The semi-simple Weil-Deligne representation $\mathrm{WD}(\varepsilon\otimes\rho(\kp)_v)^{\mathrm{ss}}$ corresponds to $\pi_v^{E(\kp)}$ under the Hecke correspondence in the sense of \cite{De72}. In our case, $\mathrm{WD}(\varepsilon\otimes\rho(\kp)_v)$ is just $(\varepsilon\otimes\rho(\kp)_v\otimes B_\cris)^{G_{F_v}}$ with $\Frob_v$ acting via $\varphi$.
\end{enumerate}
\end{lem}
\begin{proof}
This is the main result of \cite{Ski09}. See also Theorem 1.1 of \cite{BLGHT11}.
\end{proof}

\begin{rem}
Under the Hecke correspondence, $\chi_1\oplus\chi_2$ will correspond to $\Ind_{B(F_v)}^{\GL_2(F_v)}\chi_1\otimes \chi_2|\cdot|^{-1}$ (generically). This is also the one used in \cite{CDP14}.
\end{rem}

\noindent \underline{\textbf{Automorphic side}} Now we need to determine $\Pi(\kp)$. Recall that this is the closure of 
\[W^{*}_{(\vec{k},\vec{w}),E}\otimes (\pi^{E(\kp)}_v)^{\oplus d(\kp)}=[(\Sym^{k_{\sigma_v-2}}(E^2)\otimes \det{}^{w_{\sigma_v}})^*\otimes \pi^{E(\kp)}_v]^{\oplus d(\kp)'}\]
in $S_\psi(U^p)_E$. Here $d(\kp)'$ is some multiple of $d(\kp)$.

Let $\Pi_v$ be the universal unitary completion of $(\Sym^{k_{\sigma_v-2}}(E^2)\otimes \det{}^{w_{\sigma_v}})^*\otimes \pi^{E(\kp)}_v$ as a $E$-representation of $\GL_2(F_v)$. We note that $\pi^{E(\kp)}_v$ is an irreducible principal series. Otherwise it is one-dimensional and $\pi_\kp$ has to factor through the reduced norm map by the approximation Theorem, which we assume not the case. By the main results of \cite{BB10}, \cite{Pas09} in the non-ordinary case and Proposition 2.2.1 of \cite{BE10} in the ordinary case, this is a topologically irreducible admissible unitary $\GL_2(F_v)$ representation. 

\begin{lem}
$\Pi(\kp)$ is a quotient of $\Pi_v^{\oplus d(\kp)'}$.
\end{lem}

\begin{proof}
By the universal property, we get a continuous map $\Pi_v^{\oplus d(\kp)'}\to\Pi(\kp)$ with dense image. Note that both $\Pi_v^{\oplus d(\kp)'}$ and $\Pi(\kp)$ are both admissible representations of $\GL_2(F_v)$. The surjectivity of this map follows from Proposition 3.1.3 of \cite{Eme1}.
\end{proof}

As a corollary, we get an injective map
\[\Hom^{\cont}_{E[\GL_2(F_v)]}(\Pi(\kp),\Pi_{\kB_{\km,v}})\hookrightarrow \Hom^{\cont}_{E[\GL_2(F_v)]}(\Pi_v^{\oplus d(\kp)'},\Pi_{\kB_{\km,v}}).\]
Let $\Pi_v^0$ be a $\GL_2(F_v)$-invariant bounded open ball of $\Pi_v$ and denote $\Hom_\cO(\Pi_v^0,\cO)$ by $M_v$. Then
\[\Hom^{\cont}_{E[\GL_2(F_v)]}(\Pi_v^{\oplus d(\kp)'},\Pi_{\kB_{\km,v}})\cong E^{\oplus d(\kp)'}\otimes \Hom_{\kC_{\GL_2(F_v),\psi|_{F_v^\times}}(\cO)}(P_{\kB_{\km,v}},M_v)\]
 
Now we only need prove that  the action of the centre $R^{\ps,\psi\varepsilon^{-1}}_v$ of $E_{\mathfrak{B}}$ on $\Hom(P_{\kB_{\km,v}},M_v)\otimes E$ also factors through $R^{\ps,\psi\varepsilon^{-1}}_v[\frac{1}{p}]/\kp_v$. Note that $\Pi_v$ is topologically irreducible. By corollary 1.9 of \cite{Pas13}, it suffices to show that $M_v$ appears as a subquotient of $P_{\kB_{\km,v}}/\kp_v P_{\kB_{\km,v}}$. Here we consider $\kp_v$ as a prime ideal of $R^{\ps,\psi\varepsilon}_v$ by the isomorphism in \ref{gaact}.

If $\rho(\kp)_v$ is absolutely \textit{irreducible}, then by Theorem 1.10 of \cite{Pas13}, up to isomorphism, there is only one irreducible Banach representation $\Pi'_v$ appeared in the subquotient of $\Hom^{\cont}_\cO(P_{\kB_{\km,v}}/\kp_v P_{\kB_{\km,v}},E)$, which is characterized by $\mathbf{V}(\Pi'_v)\cong\rho(\kp)_v\otimes\varepsilon$. Here $\mathbf{V}$ is Colmez's functor normalized as in \cite{Pas13}. By Theorem 1.3 of \cite{CDP14} (the convention for Hodge-Tate weight of $\varepsilon$ is $1$ there) and lemma \ref{clgc}, $\Pi_v'$ is a unitary completion of 
\[\Sym^{k_{\sigma_v}-2}(E^2)\otimes \det{}^{-(w_{\sigma_v}+k_{\sigma_v}-2)}\otimes \pi_v^{E(\kp)}.\]
But the result of Berger and Breuil says that such unitary completion is unique. Hence $\Pi'_v\cong\Pi_v$ and $M_v\otimes E$ is even a quotient of $(P_{\kB_{\km,v}}/\kp_v P_{\kB_{\km,v}})\otimes E$.

If $\rho(\kp)_v=\psi_1\oplus\psi_2$ is \textit{reducible}, we may assume Hodge-Tate weight of $\psi_1$ (resp. $\psi_2$) is $w_{\sigma_v}$ (resp. $w_{\sigma_v}+k_{\sigma_v}-1$). Then 
\[\pi^{E(\kp)}_v\cong(\Ind^{\GL_2(F_v)}_{B(F_v)}\psi_1\varepsilon^{w_{\sigma_v}}|\cdot|^{1-w_{\sigma_v}}\otimes\psi_2\varepsilon^{w_{\sigma_v}+k_{\sigma_v}-1}|\cdot|^{-w_{\sigma_v}-k_{\sigma_v}+1})_{\mathrm{sm}}\]
is irreducible by our assumption. Hence $\psi_1/\psi_2\neq \varepsilon^{\pm1}$. Proposition 2.2.1 of \cite{BE10} tells us that $\Pi_v$ is the unitary parabolic induction $(\Ind^{\GL_2(F_v)}_{B(F_v)}\psi_2\varepsilon\otimes \psi_1)_{\cont}$. Compared with Theorem 1.11 of \cite{Pas13}, we also conclude that $M_v$ appears in the subquotient of $P_{\kB_{\km,v}}/\kp_v P_{\kB_{\km,v}}$.

Finally we treat the case where $\pi_\kp$ factors through the reduced norm map. In this case, $\Pi(\kp)$ has the form $\eta\circ\det$ for some continuous character $\eta:F_v^\times\to\cO^\times$ and the corresponding pseudo-character of $G_{F_v}$ is $\eta+\eta\epsilon^{-1}$. Here as usual $\eta$ is also viewed as a character of $G_{F_v}$ by the class field theory. Our claim follows directly from proposition 10.107 of \cite{Pas13}.

This finishes the proof of the second statement of the Theorem. As for the first part, note that in lemma \ref{inccr}, we can replace $\kB_{\km,v}$ by any other block $\kB'$. Since we have already seen that $\Pi(\kp)$ belongs to the block $\kB_{\km,v}$, it is clear that $\Hom^{\cont}_{E[\GL_2(F_v)]}(S_{\psi}(U^p)_\km\otimes E,\Pi_{\kB'})=0$ unless $\kB'=\kB_{\km,v}$. 
\end{proof}

\begin{cor} \label{mclgc}
Under the same assumption as in the Theorem, 
\begin{enumerate}
\item $\mm$ (defined in \ref{defnm}) is a faithful, finitely generated $\T_\psi(U^p)_\km$-module. 
\item $\T_\psi(U^p)_\km$ is a finite $R^{\ps,\psi\varepsilon^{-1}}_p$-algebra.
\end{enumerate}
\end{cor}
\begin{proof}
The faithfulness follows from the first part of the Theorem. Note that $S_\psi(U^p,E/\cO)$ is an admissible representation of $D_p^\times$. By Proposition 4.17 of \cite{Pas13}, $\mm$ is a finitely generated $E_{\kB_{\km}}$-module. But $E_{\kB_{\km}}$ is a finite algebra over $R^{\ps,\psi\varepsilon^{-1}}_p$ (corollary \ref{fpas}), hence $\mm$ is a finite module over $R^{\ps,\psi\varepsilon^{-1}}_p$ via $\tau_{\mathrm{Aut}}$. On the other hand, $\tau_{\Gal}=\tau_{\Aut}$ factors through $\T_\psi(U^p)_\km$. This proves both finiteness assertions in the corollary.
\end{proof}

\begin{cor} \label{finitelength}
$S_\psi(U^p,\F)[\km]$ is a representation of $D_p^\times$ of finite length. 
\end{cor}
\begin{proof}
It follows from the first part of corollary \ref{mclgc} that $\mm/\km\mm$ is actually a finite dimensional $\F$-vector space hence a finite length $E_{\kB_{\km}}$-module.
\end{proof}

\begin{cor} \label{classicality}
For any maximal ideal $\kp$ of $\T_{\psi}(U^p)_\km[\frac{1}{p}]$ such that for any $v|p$,
\begin{itemize}
\item $\rho(\kp)|_{G_{F_v}}$ is absolutely irreducible and de Rham with distinct Hodge-Tate weights,
\end{itemize}
where $\rho(\kp):G_{F,S}\to\GL_2(k(\kp))$ is the semi-simple representation associated to $\kp$. Then $\kp$ is a pull-back of a maximal ideal of $\T_{(\vec{k},\vec{w}),\psi}(U^pU_p)[\frac{1}{p}]$ for some weight $(\vec{k},\vec{w})\in \Z_{>1}^{\Hom(F,\overbar{\Q_p})}\times  \Z^{\Hom(F,\overbar{\Q_p})}$ and open compact subgroup $U_p\subseteq K_p$. In other words, $\kp$ comes from a classical automorphic representation on $(D\otimes \A_F^\infty)^\times$. By Jacquet-Langlands correspondence, it also arises from a regular algebraic cuspidal automorphic representation of $\GL_2(\A_F)$.
\end{cor}

\begin{proof}
Let $\kp_v\in \Spec R^{\ps,\psi\varepsilon^{-1}}_v$ and $\tilde\kp\in\Spec\T_\psi(U^p)_\km$ be the pull-back of $\kp$. Since $\rho(\kp)|_{G_{F_v}}$ is absolutely irreducible, all the irreducible subquotients of $\Hom^{\cont}_\cO(P_{\kB_{\km,v}}/\kp_v P_{\kB_{\km,v}}, E)$ are isomorphic (Theorem 1.10 of \cite{Pas13}). We denote any such irreducible subquotient by $\Pi_v$. Enlarge $E$ if necessary, we may assume $\Pi_v$ is absolutely irreducible.

\begin{lem} \label{irred}
Under all the assumptions as in the proof,
\begin{enumerate}
\item The unitary representation $\Pi:=\widehat{\bigotimes}_{v|p} \Pi_v$ of $D_p^\times$ is topologically irreducible.
\item Let $\kp_p=\kp\cap R^{\ps,\psi\varepsilon^{-1}}_p$. Then $E_{\kB_\km}[\frac{1}{p}]/(\kp_p)$ is a central simple $E$-algebra.
\end{enumerate}
\end{lem}

\begin{proof}
We are going to use the relations between $E_{\kB_\km}[\frac{1}{p}]$-modules and Banach space representations. The reference here is \S 4 of \cite{Pas13}. 

To prove the first assertion, we will apply Theorem 4.34 of \cite{Pas13} with $G=D_p^\times$. The first condition in that Theorem is clearly satisfied. The second condition follows from Proposition 4.20 by corollary \ref{fpas}. Thus it suffices to prove that 
\[\mm(\Pi):=\Hom_{\kC_{D_p^\times,\psi}(\cO)}(P_{\kB_\km},\Pi_0^d)\otimes E\]
is a simple $E_{\kB_\km}[\frac{1}{p}]$-module, where $\Pi_0$ is an open $D_p^\times$-invariant unit ball of $\Pi$ and $\Pi_0^d:=\Hom_{\cO}(\Pi_0,\cO)$. Using lemma \ref{seprod}, it is easy to see that 
\[\mm(\Pi)\cong\widehat{\bigotimes}_{v|p}\mm(\Pi_v),\]
where $\mm(\Pi_v):=\Hom_{\kC_{\GL_2(F_v),\psi|_{F_v^\times}}(\cO)}(P_{\kB_{\km,v}},\Hom_{\cO}(\Pi_{v,0},\cO))\otimes E$ and $\Pi_{v,0}$ is an open $\GL_2(F_v)$-invariant unit ball of $\Pi_v$. Let $E_v$ be the image of $E_{\kB_{\km,v}}[\frac{1}{p}]$ in $\End(\Pi_v)$. Since we assume $\Pi_v$ is absolutely irreducible, by Lemma 4.1 and Proposition 4.19 of \cite{Pas13}, we know that $E_v$ is a central simple algebra over $E$ and $\mm(\Pi_v)$ is a simple $E_v$-module. Hence $\widehat{\bigotimes}_{v|p}\mm(\Pi_v)\cong\bigotimes_{v|p} \mm(\Pi_v)$ is a simple module of $\bigotimes E_v$. This is exactly what we need to show.

For the second claim, since we assume $\rho(\kp_v)$ is absolutely irreducible, it follows from Theorem 1.10 of  \cite{Pas13} that $E_v$ is in fact $E_{\kB_{\km,v}}[\frac{1}{p}]/(\kp_v)$. Hence our claim is clear by the discussion in the previous paragraph.
\end{proof}

By Pontryagin duality, we can write
\[\Hom^{\cont}_\cO(M_\psi(U^p)_\km/\tilde\kp M_\psi(U^p)_\km ,E)\cong S_\psi(U^p)_\km[\tilde\kp]\otimes E\cong (S_\psi(U^p)_\km\otimes E)[\kp].\]

As a consequence of our local-global compatibility result, $\mm$ is a faithful, finitely generated $\T_\psi(U^p)_\km$-module. Therefore $\mm/\tilde\kp\mm\otimes E$ is a non-zero $E_{\kB_\km}[\frac{1}{p}]/(\kp_p)$-module. Hence $(S_\psi(U^p)_\km\otimes E)[\kp]\neq 0$ and lemma \ref{irred} even implies that 
\[(S_\psi(U^p)_\km\otimes E)[\kp]\cong(\widehat{\bigotimes}_{v|p} \Pi_v)^d\] 
for some positive integer $d$. Note that this is the analogue of Emerton's local-global compatibility conjecture (Conjecture 1.1.1 of \cite{Eme10a}) in this situation. By Theorem 1.3 of \cite{CDP14}, for each $v|p$, $\Pi_v$ has non-zero locally algebraic vectors of $\GL_2(F_v)$. Thus $(S_\psi(U^p)_\km\otimes E)[\kp]$ contains non-zero locally algebraic vectors of $D_p^\times$. In view of the discussion in \ref{rlcc}, this implies that $\kp$ comes from some $\T_{(\vec{k},\vec{w}),\psi}(U^pU_p)$.
\end{proof}

\subsection{A lower bound on the dimension of Hecke algebra} \label{edha}
As another application of our local-global compatibility result, we prove
\begin{thm} \label{dh}
Same assumption as in Theorem \ref{lgc}. Then each irreducible component of $\T_\psi(U^p)_\km$ is of characteristic zero and of dimension at least $1+2[F:\Q]$. 
\end{thm}

\begin{para}
Since $\mm$ is $p$-torsion free and is a faithful, finitely generated $\T_\psi(U^p)_\km$-module, it is clear that each irreducible component of $\T_\psi(U^p)_\km$ is of characteristic zero. 

We will establish a formula which relates the Gelfand-Kirillov dimension of $M_\psi(U^p)_\km/\varpi$ and $ M_\psi(U^p)_\km/\km M_\psi(U^p)_\km$ as $\F[[K_1]]$-modules (see definition below) with the usual dimension of $\mm/\varpi\mm$ as a $R^{\ps,\psi\varepsilon^{-1}}_p$-module. First we recall the definition and some basic properties of Gelfand-Kirillov dimension.
\end{para}

\begin{para} \label{pkn}
Let $K_n=\prod_{v|p}(1+p^n\mathrm{M}_2(O_{F_v}))$ for some $n>0$ large enough such that $U^pK_n$ is sufficiently small and let $Z_n$ be the center of $K_n$. Denote $K_n/Z_n$ by $PK_n$. This is a $p$-adic Lie group of dimension $3[F:\Q]$. It is clear that $PK_n$ is torsion-free and $[PK_n,PK_n]\subseteq (PK_n)^p$, the subgroup generated by $g^p,g\in PK_n$. Hence by Theorem 4.5 of \cite{DdSMS99}, $PK_n$ is uniform (see Definition 4.1 ibid.). Let $\Lambda=\F[[PK_n]]$ be the completed group ring of $PK_n$ over $\F$ with maximal ideal $J_1$. It follows from Theorem 7.24 ibid. that
\end{para}

\begin{lem}
$\gr(\Lambda):=\bigoplus_{k\ge 0}J_1^k/J_1^{k+1}$ is isomorphic to $\F[x_1,\cdots,x_{3[F:\Q]}]$ as a graded ring.
\end{lem}

\begin{cor}
$\Lambda$ is left and right Noetherian and has no zero-divisors.
\end{cor}
\begin{proof}
This is Corollary 7.25 of \cite{DdSMS99}.
\end{proof}

\begin{defn}
Let $\km_p$ be the maximal ideal of $R^{\ps,\psi\varepsilon}_p$. Define $R\Lambda$ to be the completed tensor product of $R^{\ps,\psi\varepsilon}_p$ and $\Lambda$ over $\F$ with respect to $\km_p$-adic and $J_1$-adic topology. This is a local ring. We denote its maximal ideal by $J_2$.
\end{defn}

\begin{para}
By definition, $\gr(R\Lambda)=\bigoplus_{k\ge 0}J_2^k/J_2^{k+1}\cong \gr(R^{\ps,\psi\varepsilon}_p/\varpi)\otimes \F[x_1,\cdots,x_{3[F:\Q]}]$ is Noetherian.  Hence $R\Lambda$ is left and right Noetherian by Chapter II.1.2 Proposition 3 of \cite{LvO96}. We note that $R\Lambda$ acts on $P_{\kB_{\km}}$ naturally (see the proof of Lemma 2.7 of \cite{Pas13}) and makes it into a finitely generated $R\Lambda$-module since $(P_{\kB_{\km}}/\km_pP_{\kB_{\km}})^\vee$ is an admissible representation of $D_p^\times$ (of finite length). 

Let $R$ be either $R\Lambda$ or $\Lambda$ with maximal ideal $J$. Let $M$ be a finitely generated left $R$-module. Consider $\gr(M):=\bigoplus_{k\ge0}J^kM/J^{k+1}M$. This is a finitely generated graded $\gr(R)$-module. Its \textit{Hilbert polynomial} $\varphi_M(t)$ (see \cite{Mat1} \S13) is defined to be unique polynomial satisfying $\varphi_M(k)=\dim_\F(J^kM/J^{k+1}M)$ for $k$ large enough. The \textit{Gelfand-Kirillov dimension} $\dim_R(M)$ of $M$ over $R$ is defined to be the degree of the polynomial $t\varphi_M(t)$. Equivalently $\dim_R(M)=\limsup_k\log_k(\dim_\F M/J^kM)$. For example, $\dim_\Lambda(\Lambda)=3[F:\Q]$. The dimension is independent of the choice of the open compact subgroup $K_n$.

There is a natural map $\Lambda\to R\Lambda$ such that $J_1J_2=J_2J_1$. The next result is well-known in commutative algebra and the proof is the same.
\end{para}

\begin{lem} \label{dtr}
Let $M$ be a finitely generated left $R\Lambda$-module. Assume it is also finitely generated as a $\Lambda$-module. Then $\dim_\Lambda(M)=\dim_{R\Lambda}(M)$.
\end{lem}
\begin{proof}
Since $M/J_1M$ is a finite $\F$-vector space, there exists an integer $r$ such that $J_2^rM\subseteq J_1M$. Hence $J_1^{rk}M\subseteq J_2^{rk}M\subseteq J_1^kM$. The desired result follows from the definition.
\end{proof}

Another ingredient we need is Artin-Rees property.
\begin{lem}
Let $R$ be either $\Lambda$ or $R\Lambda$ with maximal ideal $J$. Let $M$ be a finitely generated left $R$-module and $N\subseteq M$ be a submodule. Then there exists $c\in \Z_{>0}$ such that for any $k$,
\[J^{k+c}M\cap N\subseteq J^k N.\]
\end{lem}
\begin{proof}
This follows from Chapter II 1.1 Proposition 3 of \cite{LvO96}. 
\end{proof}

A lot of results in classical commutative algebra are also true in this setting.
\begin{lem} \label{seqd}
Let $0\to M'\to M \to M''\to 0$ be a short exact sequence of finitely generated left $R$-modules. Then $\dim_R(M)=\max (\dim_R (M'),\dim_R (M''))$.
\end{lem}
\begin{proof}
Apply the previous lemma to $M'\subseteq M$ and get an integer $c$ as in the lemma. Then
\[\dim_\F(M/J^kM)\leq \dim_\F(M'/J^kM')+\dim_\F(M''/J^kM''),\]
\[\dim_\F(M/J^kM)\geq \dim_\F(M'/J^{k-c}M')+\dim_\F(M''/J^kM'').\]
The desired result is clear.
\end{proof}

\begin{lem} \label{xtf1}
Let $M$ be a finitely generated left $R$-module and $x\in J$ be an element in the centre of $R$. Suppose $M$ has no $x$-torsion, then $\dim_R(M/xM)=\dim_R(M)-1$.
\end{lem}
\begin{proof} Note that
\[\dim_\F(M/(J^kM+xM))= \dim_\F(M/J^k M)-\dim_\F(xM/(J^kM\cap xM))\]
By Artin-Rees lemma, there exists an integer $c$ such that for $k\geq c$, $J^kM\cap xM\subseteq J^{k-c}xM$. 
\begin{eqnarray*}
\dim_\F(M/(J^kM+xM)) &\leq& \dim_\F(M/J^kM)-\dim_\F (xM/xJ^{k-c}M)\\
&=& \dim_\F(M/J^kM)-\dim_\F (M/J^{k-c}M)
\end{eqnarray*}
Thus $\dim_R(M/xM)\leq\dim_R(M)-1$. On the other hand, $J^kM\cap xM\supseteq xJ^{k-1}M$. We have
\[\dim_\F(M/(J^kM+xM))\geq \dim_\F(M/J^k M)-\dim_\F(M/J^{k-1}M).\]
This implies that $\dim_R(M/xM)\geq\dim_R(M)-1$ and hence the equality.
\end{proof}

Now we can state our main result. Assuming it, we can give a proof of Theorem \ref{dh}.

\begin{prop}[Dimension formula] \label{df}
Let $N$ be a finitely generated right $E_{\kB_{\km}}/(\varpi)$-module with the induced topology. Then 
\[\dim_{R\Lambda} (N{\otimes} _{E_{\kB_\km}}P_{\kB_\km})\leq \dim_{R^{\ps,\psi\epsilon}_p}N+[F:\Q],\]
where $\km_p$ is the maximal ideal of $R^{\ps,\psi\epsilon}_p$. Note that $ \dim_{R^{\ps,\psi\epsilon}_p}N$ makes sense since $N$ is also a finitely generated $R^{\ps,\psi\epsilon}_p$-module.
\end{prop}

\begin{rem}
This formula roughly says `dimension of total space $\leq$ dimension of the base space + dimension of the special fibre'. This is exactly the heuristics in page 19 of \cite{CE12}.
 
Another remark is that it will be clear in the proof that if no $\kB_{\km,v}$ belong to the last two blocks in \ref{pas}, then the inequality in the proposition is in fact an equality. The problem of the last two blocks is that $\delta\circ\det$ has Gelfand-Kirillov dimension $0$ rather than $1$. 
\end{rem}

\begin{proof}[Proof of Theorem \ref{dh}]
Note that $\psi|_{1+pO_{F,p}}$ is trivial modulo $\varpi$. The same proof of Proposition \ref{chproj} shows that $M_\psi(U^p)_\km/\varpi$ is a projective, hence free module over $\Lambda=\F[[PK_n]]$ (defined in \ref{pkn}).  Let 
\[\bar{\mm}=\mm/\varpi\mm=\Hom_{\kC_{D_p^\times,\psi}(\cO)}(P_{\kB_\km},M_\psi(U^p)_\km/\varpi).\] 
Then $\bar\mm$ has full support on $\T_\psi(U^p)/(\varpi)$ and it follows from the discussion in \ref{Blocks} that $\bar{\mm}\otimes_{E_{\kB_\km}}P_{\kB_\km}=M_\psi(U^p)_\km/\varpi\cong \Lambda^{\oplus d}$ for some $d$.

Let $\kp$ be a minimal prime ideal of $\T_\psi(U^p)_\km$ and $\mm[\kp]$ be the set of elements of $\mm$ killed by $\kp$. Denote $\mm[\kp]/\varpi\mm[\kp]$ by $N$ and the image of $N\to\bar{\mm}$ by $N'$. Note that $\mm$ is torsion-free and has full support on $\T_\psi(U^p)_\km$, hence $N'\neq 0$ and $0\neq N'\otimes_{E_{\kB_\km}}P_{\kB_\km}\subseteq \Lambda^{\oplus d}$. Since $\Lambda$ does not have zero-divisors, $N'\otimes_{E_{\kB_\km}}P_{\kB_\km}$ has at least a copy of $\Lambda$ inside. Hence
\[\dim_\Lambda(N\otimes_{E_{\kB_\km}}P_{\kB_\km})\geq\dim_\Lambda(N'\otimes_{E_{\kB_\km}}P_{\kB_\km})\geq \dim_\Lambda \Lambda=3[F:\Q]\]
by lemma \ref{seqd}. On the other hand, the other direction of the inequality is also true since $N\otimes_{E_{\kB_\km}}P_{\kB_\km}$ is a finitely generated $\Lambda$-module. Thus 
\[\dim_\Lambda(N\otimes_{E_{\kB_\km}}P_{\kB_\km})=3[F:\Q].\]

Note that this is also $\dim_{R\Lambda}(N\otimes_{E_{\kB_\km}}P_{\kB_\km})$ by lemma \ref{dtr}. Apply proposition \ref{df} with $N=N$. We get 
\[\dim_{R^{\ps,\psi\epsilon}_p} N\geq 2[F:\Q].\] 
Since the action of $R^{\ps,\psi\epsilon}_p$ on $N$ factors through $\T_\psi(U^p)_\km/(\kp,\varpi)$, we have
\[\dim \T_\psi(U^p)_\km/\kp-1=\dim \T_\psi(U^p)_\km/(\kp,\varpi)=\dim_{\T_\psi(U^p)_\km/(\kp,\varpi)} N=\dim_{R^{\ps,\psi\epsilon}_p} N\geq 2[F:\Q].\]
The first equality follows from the fact that $\kp$ has characteristic zero. This finishes the proof.
\end{proof}

\begin{proof}[Proof of Proposition \ref{df}]
Write $d(N)=\dim_{R\Lambda}(N\otimes_{E_{\kB_\km}}P_{\kB_\km})$. The proof is by induction on the dimension of $N$ over $R^{\ps,\psi\varepsilon}_p$. If $\dim_{R^{\ps,\psi\varepsilon}_p}N=0$, then $(N\otimes_{E_{\kB_\km}}P_{\kB_\km})^\vee$ is a smooth representation of $D_p^\times$ of finite length. Since each irreducible constituent has the form $\otimes_{v|p}\pi_v$, where $\pi_v$ is an irreducible representation of $\GL_2(F_v)$, the result follows from

\begin{lem}
Let $K$ be a pro-$p$ open subgroup of $\GL_2(F_v)$ and $\pi$ be any smooth irreducible representation of $\GL_2(\Q_p)$ over $\F$ with central character $\psi$. Then $\dim_{\F[[K]]}\pi^\vee=1$ unless $\pi$ is one-dimensional.
\end{lem}

\begin{proof}
See the proof of Corollary 7.5 of \cite{SS16}. Or one can prove this directly: the Gelfand-Kirillov dimension of principal series and special series can be computed by hand; the case of supersingular representations can be computed using Theorem 1.2 of \cite{Pas10}.
\end{proof}

Suppose we have proved for all $N$ of dimension at most $r$ over $R^{\ps,\psi\varepsilon}_p$. Let $N$ be a finitely generated right $E_{\kB_\km}$-module of dimension $r+1$ over  $R^{\ps,\psi\varepsilon}_p$. Choose $x\in R^{\ps,\psi\varepsilon}_p$ such that $\dim_{R^{\ps,\psi\varepsilon}_p} N/xN=r$, i.e. $x$ does not vanish on each irreducible component of the support of $N$ of dimension $r+1$. Hence the support of $N[x]$ has dimension at most $r$. After replacing $x$ by its power, we may assume $N[x]=N[x^\infty]$. Let $N'=N/N[x]$. Then $N'$ has no $x$-torsion. $d(N')=d(N'/xN')+1$ by lemma \ref{xtf1}.
 
Since $\dim_{R^{\ps,\psi\varepsilon}_p}N'/xN' \leq \dim_{R^{\ps,\psi\varepsilon}_p} N/xN = r$, by induction hypothesis, 
\[d(N[x])\leq \dim_{R^{\ps,\psi\varepsilon}_p} N[x]+[F:\Q]\leq r+[F:\Q],\]
\[d(N')=d(N'/xN')+1\leq r+[F:\Q]+1.\]
Hence $d(N)\leq r+1+[F:\Q]$ by lemma \ref{seqd}.
\end{proof}

\subsection{A variant of completed homology} \label{varch}
In order to apply Taylor's Ihara avoidance in \cite{Ta08}, we need to introduce a variant of completed homology considered before. Most of arguments in previous subsections still work here. We keep the notations as in the beginning of this section.

\begin{para}
Let $U=\prod_{v} U_v$, where $U_v$ is an open compact subgroup $(D\otimes F_v)^\times$. Write $U^p=\prod_{v\nmid p} U_v$ and $U_p=\prod_{v|p}U_v$.  Let $\xi:U^p\to\cO^\times$ be a continuous smooth character and $\psi:(\A^\infty_F)^\times/F^\times_{>>0}\to\cO^\times$ be a continuous character such that $\psi|_{\prod_{v\nmid p} (O^\times_{F_v}\cap U_v)}=\xi|_{\prod_{v\nmid p} (O^\times_{F_v}\cap U_v)}$. 

Given a topological $\Z_p$-algebra $A$ and a continuous representation $\tau:\prod_{v|p}U_v\to\Aut(W\tau)$, we may define $S_{\tau,\psi,\xi}(U,A)$ as the space of continuous functions:
\[f:D^\times\setminus \DAi\to W_\tau,\]
such that for any $g\in\DAi,z\in \AFi,u=u^pu_p\in U$, we have 
\[f(guz)=\psi(z)\xi(u^p)\tau(u_p^{-1})(f(g)).\]

If $\psi|_{U_p\cap O_{F,p}^\times}=\tau^{-1}|_{U_p\cap O_{F,p}^\times}$, then as in \eqref{dcp}, we have
\[S_{\tau,\psi,\xi}(U,A)\simeq \bigoplus_{i\in I}W_{\tau}^{(t_i^{-1}D^\times t_i\cap U\AFi)/F^{\times}},\]
where $I=D^\times\setminus \DAi/U\AFi$ and $\{t_i\}_{i\in I}$ is a set of representatives. If $U$ is sufficiently small, corollary \ref{sfsmf} is still valid. We will simply write $S_{\psi,\xi}(U,A)$ if $\tau$ is the trivial action on $A$.
\end{para}

\begin{para} 
We can introduce completed homology $M_{\psi,\xi}(U^p)$ and cohomology $S_{\psi,\xi}(U^p)$ similarly as in subsection \ref{chac} and Hecke algebra $\T_{\psi,\xi}(U^p)$ as in subsection \ref{haapr}. Let $U'^p$ be the kernel of $\xi$.Then we have a natural Hecke equivariant inclusion $S_{\psi,\xi}(U^p)\hookrightarrow S_{\psi}(U'^p)$. From this, we can deduce our local-global compatibility (Theorem \ref{lgc}) for $M_{\psi,\xi}(U^p)$ directly from the case we have proved before. The same estimate of dimension of $\T_{\psi,\xi}(U^p)$ can be obtained exactly the same as in subsection \ref{edha}. We leave the details to interested readers.
\end{para}

\section{Patching at a one-dimensional prime} \label{Paao-dp}
We are going to prove a `$R_\kq=\T_\kq$' result, where $\kq$ is a one-dimensional prime ideal rather than the maximal ideal.

\subsection{Setup and the statement of the main result} \label{satsotmr}
\begin{para} \label{defrps1}
In this section, $F$ denotes a totally real field of even degree over $\Q$ in which $p$ splits completely. Write $\Sigma_p$ as the set of places of $F$ above $p$. Let $S$ be a finite set of finite places containing $\Sigma_p$ such that $p|N(v)-1$ for all $v\in S\setminus \Sigma_p$, and let $\chi:G_{F,S}\to\cO^\times$ be a continuous character such that 
\begin{itemize}
\item $\chi$ is unramified at places outside of $\Sigma_p$.
\item $\chi(\Frob_v)\equiv 1\mod \varpi$ for $v\in S\setminus \Sigma_p$.
\item $\chi(c)=-1$ for any complex conjugation $c\in G_{F,S}$.
\end{itemize}
Denote by $\bar{\chi}$ the reduction of $\chi$ modulo $\varpi$. Let $\xi_v:k(v)^\times\to\cO^\times$ be characters of $p$-power order for $v\in S\setminus \Sigma_p$. We will view $\xi_v$ as characters of $I_{F_v}$ by the local class field theory. 

Consider the universal deformation ring $R^{\ps,\{\xi_v\}}$ which pro-represents the functor from $\cOf$ to the category of sets sending $R$ to the set of two-dimensional pseudo-representations $T$ of $G_{F,S}$ over $R$  such that $T$ is a lifting of $1+\bar\chi$ with determinant $\chi$ and
\[T|_{I_{F_v}}=\xi_v+\xi_v^{-1} \]
for any $v\in S\setminus \Sigma_p$. If $\xi_v$ are all trivial, we will simply write $R^{\ps,1}$.

We fix a complex conjugation $\sigma^*\in G_{F,S}$ so that we can associate a two-dimensional semi-simple representation $\rho(\kp):G_{F,S}\to\GL_2(k(\kp))$ with trace $T^{univ} \mod\kp$ for any $\kp\in\Spec R^{\ps,\{\xi_v\}}$ as in \ref{tar}. Here, $T^{univ}:G_{F,S}\to R^{\ps,\{\xi_v\}}$ is the pseudo-representation given by the universal property.
\end{para}

\begin{para} \label{autlev}
On the automorphic side, let $D$ be a quaternion algebra over $F$ ramified exactly at all infinite places. Fix an isomorphism between $\DAi$ and $\GL_2(\A_F^\infty)$. By global class field theory, we may view $\psi=\chi\varepsilon$ as a character of $\AFi/F^\times_{>>0}$. We also define a tame level $U^p=\prod_{v\nmid p}U_v$ as follows: $U_v=\GL_2(O_{F_v})$ if $v\notin S$ and 
\[U_v=\mathrm{Iw}_v:=\{g\in\GL_2(O_{F_v}),g\equiv \begin{pmatrix}*&*\\0&*\end{pmatrix}\mod \varpi_v\}\]
otherwise. For any $v\in S\setminus \Sigma_p$, the map $\begin{pmatrix}a&b\\c&d\end{pmatrix}\mapsto \xi_v(\frac{a}{d}\mod \varpi_v)$ defines a character of $U_v$. The product of $\xi_v$ can be viewed as a character $\xi$ of $U^p$ by projecting to $\prod_{v\in S\setminus \Sigma_p}U_v$. As in the previous section, we denote $K_p=\prod_{v|p}\GL_2(O_{F_v}),D_p^\times=\prod_{v|p}\GL_2(F_v)$.

Using this, we can define a Hecke algebra $\T:=\T_{\psi,\xi}(U^p)$ as in section \ref{varch}. We also make the following assumption in this section (which defines a maximal ideal $\km$ of $\T$):

\noindent \textbf{\underline{Assumption}}: $T_v-(1+\chi(\Frob_v)),v\notin S$ and $\varpi$ generate a maximal ideal $\km$ of $\T$.
\end{para}

\begin{para}
By the discussion in \ref{hagal}, there is a natural pseudo-representation $T_\km:G_{F,S}\to\T_\km$ with determinant $\chi$ sending $\Frob_v$ to $T_v$ for $v\notin S$. Let $R^{\ps}$ be the universal object in $C_{\cO}$ which pro-represents the functor from $\cOf$ to the category of sets sending $R$ to the set of two-dimensional pseudo-representations $T$ of $G_{F,S}$ over $R$ lifting $1+\bar\chi$ with determinant $\chi$. By the universal property, there is a natural map $R^{\ps}\to \T_\km$, which is surjective since $\T_\km$ is topologically generated by $T_v,v\notin S$.

We claim this map factors through $R^{\ps,\{\xi_v\}}$, i.e. $T_\km|_{I_{F_v}}=\xi_v+\xi_v^{-1}$ for $v\in S\setminus \Sigma_p$. If $\psi$ is of finite order,  then this is a direct consequence of classical local-global compatibility at such $v$ at finite levels. In general, we can reduce to the previous case by twisting everything with a certain character (see for example the argument in \ref{exgal}).

Using the construction for $R^{\ps,\{\xi_v\}}$, we can define a two-dimensional semi-simple representation $\rho(\kp):G_{F,S}\to\GL_2(k(\kp))$ with trace $T_\km \mod\kp$ for any $\kp\in\Spec \T_\km$.

We will say a prime $\kq\in\Spec R^{\ps,\{\xi_v\}}$ is \textit{pro-modular} if it comes from a prime of $\T_\km$. 
\end{para}

\begin{defn} \label{nice}
Let $\kq$ be a prime ideal of $\T_\km$ and $A$ be the normal closure of $\T_\km/\kq$ in $k(\kq)$. We say $\kq$ is \textit{nice} if $\kq$ contains $p$ and $\dim \T_\km/\kq =1$ and there exists a two-dimensional representation
\[\rho(\kq)^o:G_{F,S}\to \GL_2(A)\]
satisfying the following properties:
\begin{enumerate}
\item $\rho(\kq)^o \otimes k(\kq)\cong  \rho(\kq)$ is irreducible. In other words, $\rho(\kq)^o$ is a lattice of $\rho(\kq)$.
\item The mod $\km_A$ reduction $\bar{\rho}_b$ of $\rho(\kq)^o$ is a non-split extension and has the form
\[\bar{\rho}_b(g)=\begin{pmatrix}*&*\\0&*\end{pmatrix},~g\in G_{F,S}.\]
Here $\km_A$ is the maximal ideal of $A$.
\item If $\rho(\kq)$ is dihedral, namely isomorphic to $\Ind_{G_L}^{G_F}\theta$ for some quadratic extension $L$ of $F$ and continuous character $\theta:G_L\to k(\kq)^\times$, then $L\cap F(\zeta_p)=F$, where $\zeta_p\in \overbar{F}$ is a primitive $p$-th root of unity. 
\item $\rho(\kq)^o|_{G_{F_v}}=\bar{\rho}_b|_{G_{F_v}}$ for any $v\in S\setminus \Sigma_p$, i.e. $\rho(\kq)^o|_{G_{F_v}}$ is the trivial lift. Here we view $\GL_2(\F)$ as a subgroup of $\GL_2(A)$ by the canonical embedding $\F\to A$.
\end{enumerate}
By abuse of notation, we say a prime $\kq^{\ps}\in\Spec R^{\ps,\{\xi_v\}}$ is \textit{nice} if it comes from a nice prime $\kq$ of $\T_\km$ in the above sense.
\end{defn}

\begin{rem}
This is different from the definition in the beginning of \S 6 of \cite{SW99}. See the proof of lemma \ref{ktwl} below for an explanation.
\end{rem}

\begin{rem}
The last assumption will be used in lemma \ref{cctp} and lemma \ref{ccal}. The idea is that we need to understand the ``completion at $\rho(\kq)^o|_{G_{F_v}}$'' of some local deformation rings at $v$. If $\rho(\kq)^o|_{G_{F_v}}$ is the trivial lifting, these are just usual local deformation rings. Without this condition, it seems a bit difficult to control such completions.
\end{rem}

Now we can state the main result of this section:
\begin{thm} \label{thmA}
Under the assumptions for $F,\chi$ as in this subsection, let $\kq\in \Spec \T_\km$ be a nice prime and $\kq^{\ps}=\kq\cap R^{\ps,\{\xi_v\}}$. If $p=3$, we further assume $\bar{\chi}|_{G_{F_v}}\neq \omega^{\pm 1}$ for any $v|p$. Then the natural surjective map 
\[(R^{\ps,\{\xi_v\}})_{\kq^{\ps}}\to\T_\kq\]
has nilpotent kernel.
\end{thm}

\begin{cor} \label{corA}
Under the same assumptions as in the previous Theorem, let $\kp$ be a maximal ideal of $R^{\ps,\{\xi_v\}}[\frac{1}{p}]$. Assume that 
\begin{itemize}
\item For any $v|p$, $\rho(\kp)|_{G_{F_v}}$ is irreducible and de Rham with distinct Hodge-Tate weights.
\item There exists an irreducible component of $R^{\ps,\{\xi_v\}}$ containing both $\kp$ and a nice prime $\kq$.
\end{itemize}
Then $\rho(\kp)$ arises from a regular algebraic cuspidal automorphic representation of $\GL_2(\A_F)$.
\end{cor}
\begin{proof}
This follows directly from the previous Theorem and corollary \ref{classicality}.
\end{proof}

\begin{para} \label{kqrho}
The rest of the section is to prove this Theorem. Fix a nice $\kq$ and a choice of $\rho(\kq)^0$ as in the definition. We may assume $A$ has residue field $\F$. If not, let $\cO'$ be an unramified extension of $\cO$ with the same residue field as $A$ and choose a prime ideal $\kq'\in \Spec\T_{\km}\otimes_{\cO}\cO'$ above $\kq$. Then the normal closure of $(T_{\km}\otimes_{\cO}\cO')/\kq'$ has residue field $\F'$. Denote $\kq'\cap (R^{\ps,\{\xi_v\}}\otimes_\cO \cO')$ by $\kq'^{\ps}$. Hence  $(R^{\ps,\{\xi_v\}}\otimes_\cO \cO')_{\kq'^{\ps}}\to(\T\otimes_\cO\cO')_{\kq'}$ has nilpotent kernel by the Theorem. Since the natural map $(R^{\ps,\{\xi_v\}})_{\kq^{\ps}}\to (R^{\ps,\{\xi_v\}}\otimes_\cO \cO')_{\kq'^{\ps}}$ is faithfully flat, this implies that $(R^{\ps,\{\xi_v\}})_{\kq^{\ps}}\to\T_\kq$ has nilpotent kernel as well.

From now on, we fix an isomorphism $A\cong \F[[T]]$.

The following lemma gives some sufficient conditions for the third condition in \ref{nice}.
\end{para}

\begin{lem} \label{nirred}
Let $\kq$ be a prime ideal of $\Spec R^{\ps,\{\xi_v\}}$ containing $p$ such that $R^{\ps,\{\xi_v\}}/\kq$ is one-dimensional.  Suppose $\rho(\kq)$ is irreducible. Then the third condition in \ref{nice} holds for $\kq$ if one of the following conditions holds:
\begin{enumerate}
\item $\bar{\chi}$ is not quadratic. This includes the cases where $\bar{\chi}|_{G_{F_v}}=\omega$ or $\omega^{-1}$ for some $v|p$.
\item $\bar{\chi}|_{G_{F_v}}=\mathbf{1}$ for some $v|p$.
\item There exists a place $v|p$ such that $\bar{\chi}|_{G_{F_v}}\neq\mathbf{1}$ and $\rho(\kq)^o|_{G_{F_v}}\cong \begin{pmatrix}\chi_{v,1}& *\\ 0 & \chi_{v,2}\end{pmatrix}$ is reducible. Moreover $\chi_{v,1}/\chi_{v,2}$ is of infinite order, which is equivalent with saying $\chi_{v,1}$ is of infinite order as $\chi_{v,1}\chi_{v,2}=\bar{\chi}$ is of finite order.
\end{enumerate}
\end{lem}

\begin{proof}
Since the semi-simplification of the reduction of $\rho(\kq)$ is $1+\bar{\chi}$, it is clear that if $\rho(\kq)$ is isomorphic to $\Ind_{G_L}^{G_F}\theta$, then $\bar{\chi}$ is quadratic and $L=F(\bar{\chi})=\overbar{F}^{\ker(\bar{\chi})}$. If $\bar{\chi}|_{G_{F_v}}=\mathbf{1}$ for some $v|p$, then $F(\bar{\chi})\cap F(\zeta_p)=F$ as we assume $p$ completely splits in $F$. This proves the first two parts of the lemma.

As for the third part, suppose $\rho(\kq)\cong \Ind_{G_{F(\bar{\chi})}}^{G_F}\theta$. Then by our assumption, $v$ is inert or ramified in $F(\bar{\chi})$. Hence $\rho(\kq)|_{G_{F_v}}\cong \Ind_{G_{L_w}}^{G_{F_v}}\theta|_{G_{L_w}}$, where $w$ is the place above $v$ in $L$. Since $\rho(\kq)|_{G_{F_v}}$ is reducible, $\rho(\kq)|_{G_{L_w}}\cong \theta\oplus\theta$. This contradicts our assumption that $\chi_{v,1}/\chi_{v,2}$ is of infinite order.
\end{proof}

\begin{rem} \label{nicerem}
It is clear from the proof  that lemma \ref{nirred} holds if we assume that $[F_v(\zeta_p):F_v]=p-1$ for any $v|p$ instead of that $p$ completely splits in $F$ as in \ref{defrps1}.
\end{rem}

\subsection{Some local and global (framed) deformation rings}
We introduce several universal lifting rings and recall some of their basic properties. 

\begin{defn} \label{ldfr}
Let $v$ be a finite place of $F$. 
\begin{itemize} 
\item If $\bar{\rho}_b|_{G_{F_v}}$ is unramified, we define $R^{\square,ur}_{v}$ to be the universal object in $C_\cO$ that pro-represents the functor from $\cOf$ to the category of sets sending $R$ to the set of unramified liftings $\rho_R:G_{F_v}/I_{F_v}\to\GL_2(R)$ of $\bar{\rho}_b|_{G_{F_v}}$ to $R$ with determinant $\chi$.
\item If $v\in S\setminus \Sigma_p$, we define $R^{\square,\xi_v}_{v}$ to be the universal object in $C_\cO$ that pro-represents the functor from $\cOf$ to the category of sets sending $R$ to the set of liftings $\rho_R:G_{F_v}\to\GL_2(R)$ of $\bar{\rho}_b|_{G_{F_v}}$ to $R$ with determinant $\chi$ such that 
\[\tr(\rho_R)|_{I_{F_v}}=\xi_v+\xi_v^{-1}.\]
If $\xi_v$ is trivial, we write $R^{\square,\xi_v}_{v}$ as $R^{\square,1}_{v}$.
\item We also define $R^{\square}_{v}$ to be the unrestricted (i.e. no condition on the liftings) universal lifting ring of $\bar{\rho}_b|_{G_{F_v}}$ with determinant $\chi$. 
\end{itemize}
\end{defn}

\begin{para} \label{defnP}
Recall that in the previous subsection, we defined $\rho(\kq)^o:G_{F,S}\to\GL_2(A)$. Here $A=\F[[T]]$ equipped with $T$-adic topology. Let $B$ be the topological closure of the $\F$-algebra generated by all the entries of $\rho(\kq)^o(G_{F,S})$. By Chebotarev's density Theorem, we may find a finite set of primes $T'$ disjoint with $S$ such that the entries of $\rho(\kq)^o(\Frob_v),v\in T',$ topologically generate $B$. In fact, let $w\notin S$ be a place such that $\rho(\kq)^o(\Frob_w)\notin \GL_2(\F)$. We denote by $A'$ the topological closure of the $\F$-subalgebra in $A$ generated by the entries of $\rho(\kq)^o(\Frob_w)$. Then $B$ is a finite $A'$-module and by the ascending chain condition we can find such a finite set $T'$. We fix such a $T'$ from now on and denote 
\[P=T'\cup S.\]

For any Noetherian local $\cO$-algebras $R_1,R_2$, we define their completed tensor product $R_1\widehat{\otimes}_\cO R_2$ to be $\varprojlim_{n} (R_1/\km_1^n\otimes_\cO R_2/\km_2^n)$, where $\km_i$ is the maximal ideal of $R_i$. It follows from this definition that $R_1\widehat{\otimes}_\cO R_2\cong \widehat{R_1}\widehat{\otimes}_\cO R_2$ and there is a natural map $\widehat{R_1}\to R_1\widehat{\otimes}_\cO R_2$, where $\widehat{R_1}$ is the $\km_1$-adic completion of $R_1$. If $R_1$ has residue field $\F$, then $R_1\widehat{\otimes}_\cO R_2$ is a complete Noetherian local $\cO$-algebra as well (Lemma 1.3 of \cite{Tho15}). 
\end{para}

\begin{defn} \label{rxivloc}
We define $R^{\{\xi_v\}}_{\loc}$ to be
\[(\widehat{\bigotimes}_{v\in \Sigma_p}R^{\square}_v)\widehat{\otimes}(\widehat{\bigotimes}_{v\in S\setminus \Sigma_p}R^{\square,\xi_v}_v)\widehat{\otimes}(\widehat{\bigotimes}_{v\in T'}R^{\square,ur}_v),\]
where all the completed tensor products are taken over $\cO$. By the universal property, $\rho(\kq)^o$ gives rise to a one-dimensional prime $\kq^{\{\xi_v\}}_{\loc}$ of $R^{\{\xi_v\}}_{\loc}$. Note that for any $v\in S\setminus \Sigma_p$, the pull-back of $\kq^{\{\xi_v\}}_{\loc}$ to $R^{\square,\xi_v}_v$ is the maximal ideal of $R^{\square,\xi_v}_v$ by our  last assumption in \ref{nice}.
\end{defn}

The main result of this subsection is
\begin{prop} \label{rlocp}
The $\kq^{\{\xi_v\}}_{\loc}$-adic completion $\widehat{(R^{\{\xi_v\}}_{\loc})_{\kq^{\{\xi_v\}}_{\loc}}}$ of $(R^{\{\xi_v\}}_{\loc})_{\kq^{\{\xi_v\}}_{\loc}}$ is equidimensional of dimension $3[F:\Q]+3|P|$. The generic point of each irreducible component has characteristic zero. Moreover,
\begin{enumerate}
\item If all $\xi_v$ are non-trivial, then $\widehat{(R^{\{\xi_v\}}_{\loc})_{\kq^{\{\xi_v\}}_{\loc}}}$ is integral.
\item In general, each minimal prime of $\widehat{(R^{\{\xi_v\}}_{\loc})_{\kq^{\{\xi_v\}}_{\loc}}}/(\varpi)$  contains a unique minimal prime of $\widehat{(R^{\{\xi_v\}}_{\loc})_{\kq^{\{\xi_v\}}_{\loc}}}$.
\end{enumerate}
\end{prop}

We first collect some useful results in commutative algebra (see also \S1 of \cite{Tho15}).

\begin{lem} \label{scca}
Let $R,S$ be complete Noetherian local $\cO$-algebras with maximal ideals $\km_R,\km_S$ respectively. Suppose that $R/\km_R=\F$ and $S$ is flat over $\cO$. Then
\begin{enumerate}
\item The natural map $R\to R\widehat{\otimes}_\cO S$ is faithfully flat.
\item For any finitely generated $R$-module $M$, the $(\km_R,\km_S)$-adic completion $M\widehat{\otimes}_\cO S$ of $M\otimes_\cO S$ as an $R\otimes S$-module is canonically isomorphic to its $\km_R$-adic completion. Moreover there is an natural isomorphism:
\[M\otimes_R(R\widehat{\otimes}_\cO S)\cong M\widehat{\otimes}_\cO S.\]
In particular, for any ideal $I$ of $R$, we have
\[I(R\widehat{\otimes}_\cO S)\cong I\otimes_R (R\widehat{\otimes}_\cO S)\cong I\widehat{\otimes}_\cO S. \]
\item For any ideal $\kp$ of $S$ such that $S/\kp$ is $\cO$-flat, we have
\[(R\widehat{\otimes}_\cO S)/(\kp)\cong R\widehat{\otimes}_\cO (S/\kp).\]
\end{enumerate}
\end{lem}

\begin{proof}
The flatness of $R\widehat{\otimes}_\cO S$ over $R$ is Lemma 1.3 of \cite{Tho15}. Also it is clear that $R\to R\widehat{\otimes}_\cO S$ is a local homomorphism, hence faithfully flat. 

To prove the second part of the lemma, we note that $M/\km_R^nM$ is of finite length as a $\cO$-module for any $n>0$. Thus $M/\km_R^nM\otimes_\cO S\cong \varprojlim_k M/\km_R^nM\otimes_\cO S/\km_S^k$. This proves the first assertion in the second part of the lemma. The second assertion is trivially true if $M$ is a free $R$-module. The general case follows by writing $M$ as a finite presentation $R^{\oplus r}\to R^{\oplus s}\to M\to 0$. 

For the last part,  by our assumption there is an exact sequence for any $n$:
\[0\to R/\km_R^n\otimes_\cO\kp\to R/\km_R^n\otimes_\cO S \to R/\km_R^n\otimes_\cO S/\kp \to 0.\]
The inverse limit over $n$ remains exact. Now it suffices to show that 
\[\kp(R\widehat{\otimes}_\cO S)\to \varprojlim_n (R/\km_R^n\otimes_\cO\kp)\] 
is surjective. But this is clear as we can write $\kp$ as a quotient of $S^{\oplus r}$ for some $r$ and apply Theorem 8.1 of \cite{Mat1}.
\end{proof}

\begin{lem} \label{cctp}
Let $R,S\in C_\cO$ with maximal ideals $\km_R,\km_S$ respectively. Let $\kp\in \Spec R$ containing $p$ and $\kp'=(\kp,\km_S)\in\Spec R\widehat{\otimes}_\cO S$. Then there is a canonical isomorphism
\[\widehat{(R\widehat{\otimes}_\cO S)_{\kp'}}\cong \widehat{R_\kp}\widehat{\otimes}_\cO S.\]
\end{lem}
\begin{proof}
This is Lemma 1.5 of \cite{Tho15}.
\end{proof}

Later on, $R$ will be $(\widehat{\bigotimes}_{v\in \Sigma_p}R^{\square}_v)\widehat{\otimes}(\widehat{\bigotimes}_{v\in T'}R^{\square,ur}_v)$ and $S$ will be $(\widehat{\bigotimes}_{v\in S\setminus \Sigma_p}R^{\square,\xi_v}_v)[[x_1,\cdots,x_g]]$ for some $g$. 

\begin{lem} \label{ccal}
Let $S$ be a complete local Noetherian flat $\cO$-algebra of dimension $e$ and $R$ be a finitely generated $\cO$-algebra. Suppose that $S$ and $R$ satisfy the following conditions:
\begin{itemize}
\item $S\otimes_\cO O_L$ is integral for any finite extension $L/E$ with ring of integers $O_L$ (i.e. $S$ is geometrically integral over $\cO$). Moreover, we assume $S/pS\neq 0$.
\item For each minimal prime $\kp$ of $R$, $R/\kp$ is $\cO$-flat of dimension $d+1$ and $R/(\kp,\varpi)$ is generically reduced.
\item Each minimal prime of $R/(\varpi)$ contains a unique minimal prime of $R$.
\end{itemize}
Let $\km$ be a maximal ideal of $R$ with residue field $\F$. Then
\begin{enumerate}
\item If $R$ is normal, then $S\widehat{\otimes}_\cO R_\km$ is integral of dimension $d+e$.
\item For each minimal prime $\kq$ of $S\widehat{\otimes}_\cO R_\km$, $S\widehat{\otimes}_\cO R_\km/\kq$ is $\cO$-flat of dimension $d+e$.
\item Each minimal prime of $(S\widehat{\otimes}_\cO R_\km)/(\varpi)$ contains a unique minimal prime of $S\widehat{\otimes}_\cO R_\km$.
\end{enumerate}
\end{lem}

\begin{proof}
Before giving the proof, we note the following useful fact: Let $D$ be an excellent local ring and $\widehat{D}$ be its completion. Then $D\to \widehat{D}$ is regular and faithfully flat. Hence $D$ is regular (normal, reduced) if and only if $\widehat{D}$ is  regular (normal, reduced). See \S32 of \cite{Mat1} for basic properties of excellent rings. In practice, all the rings below in this proof are excellent.

First we assume $R$ is normal. Let $\km'\in \Spec S\otimes_\cO R_\km$ be the maximal ideal above $\km$. Then $(S\otimes_\cO R_\km)_{\km'}$ is flat over $R_\km$. Hence 
\[\dim (S\otimes_\cO R_\km)_{\km'}=\dim S\otimes\F +\dim R_\km=d+e.\]
We conclude that $\dim S\widehat{\otimes}_\cO R_\km=d+e$ since taking completion preserves the dimension.

Note that $S\otimes_\cO R_\km$ is reduced (which is clear after inverting $p$) and excellent. $S\widehat{\otimes}_\cO R_\km$ is also reduced. It rests to prove that $S\widehat{\otimes}_\cO R_\km$ is irreducible.

For simplicity, we denote $\Spec S\widehat{\otimes}_\cO R_\km$ by $X$. By Corollary 10.5.8 of \cite{EGA}, the intersection of all maximal ideals of $\widehat{R_\km}[\frac{1}{p}]$ is trivial. Let $I$ be the set of pull-back of these maximal ideals to $\widehat{R_\km}$. Then $\bigcap_{\kp\in I} \kp=\{0\}$. Since $\widehat{R_\km}$ is compact, for any $n>0$, we may find  a finite subset $I_n\subseteq I$ such that $\bigcap_{\kp\in I_n} \kp \subseteq \km^n\widehat{R_\km}$. Let $J=\bigcup_n I_n$. Then 
\[\bigcap_{\kp\in J}\kp(S\widehat{\otimes}_\cO \widehat{R_\km})\subseteq \bigcap_n \km^n(S\widehat{\otimes}_\cO \widehat{R_\km})=\{0\}.\]
In other words, $J$ is dense in $\Spec S\widehat{\otimes}_\cO \widehat{R_\km}$.

Now given two non-empty open sets $U_1,U_2$ of $X$, by the above equality, we can find $\kp_i\in J$ for $i=1,2$ such that $U_i\cap \Spec (S\widehat{\otimes}_\cO \widehat{R_\km})/(\kp_i)$ is non-empty. Note that 
\[(S\widehat{\otimes}_\cO \widehat{R_\km})/(\kp_i)\cong S\widehat{\otimes}_\cO (\widehat{R_\km}/\kp_i) \cong S\otimes_\cO (\widehat{R_\km}/\kp_i)\]
by Lemma \ref{scca} and the fact that $\widehat{R_\km}/\kp_i$ is a finite $\cO$-domain. Hence $X_{\kp_i}\stackrel{\mathrm{def}}{=}\Spec S\otimes_\cO k(\kp_i)$, the fibre of $X$ over $\kp_i\in \Spec \widehat{R_\km}$, is dense in $\Spec S\otimes_\cO (\widehat{R_\km}/\kp_i)$. In particular, $X_{\kp_i}\cap U_i$ is a non-empty open set of $X_{\kp_i}$. From this and Corollary 10.5.8 of \cite{EGA} again, it is easy to see that we may find $\kq\in\Spec S$ which is the pull-back of a maximal ideal of $S[\frac{1}{p}]$, such that $X_{\kq}\cap X_{\kp_i}\cap U_i$ is non-empty for $i=1,2$. Therefore, we only need to show $X_\kq$ is irreducible.

By the last part of Lemma \ref{scca}, $X_\kq\cong\Spec (S/\kq\widehat{\otimes}_{\cO}\widehat{R_\km})[\frac{1}{p}]$. Let $\tilde{S}$ be the normal closure of $S/\kq$ in its fractional field. This is a complete discrete valuation ring with $\varpi$ contained in the maximal ideal. Fix a uniformizer $\lambda$ of $\tilde{S}$. Note that $\tilde{S}$ is a finite $S/\kq$-algebra since $S$ is excellent. It is easy to see that $(S/\kq\widehat{\otimes}_{\cO}\widehat{R_\km})[\frac{1}{p}]\cong (\tilde{S}\widehat{\otimes}_{\cO}\widehat{R_\km})[\frac{1}{p}]\cong (\tilde{S}\widehat{\otimes}_{\cO}{R_\km})[\frac{1}{p}]$. We claim that $\tilde{S}\widehat{\otimes}_{\cO}{R_\km}$ is normal. Consequently, $X_\kq$ will be irreducible. 

Let $\km'$ be the unique maximal ideal of $\tilde{S}{\otimes}_{\cO}{R_\km}$ containing $\km$. Note that $\tilde{S}\widehat{\otimes}_{\cO}{R_\km}$ is the $\km'$-adic completion of $(\tilde{S}\otimes_{\cO}{R_\km})_{\km'}$. Hence to prove $\tilde{S}\widehat{\otimes}_{\cO}{R_\km}$ is normal, it suffices to show $(\tilde{S}\otimes_{\cO}{R_\km})_{\km'}$ is normal since $(\tilde{S}\otimes_{\cO}{R_\km})_{\km'}$ is excellent. We apply Serre's criterion for normality: ($S_2$) is true as $\tilde{S}$ is flat over $\cO$. For ($R_1$), note that the map $R_\km[\frac{1}{p}]\to \tilde{S}{\otimes}_{\cO}{R_\km}[\frac{1}{p}]$, as a base change of $E\to \tilde{S}[\frac{1}{p}]$,  is regular. Hence $\tilde{S}{\otimes}_{\cO}{R_\km}[\frac{1}{p}]$ is normal and any height one prime of it is regular. For height one prime of $(\tilde{S}\otimes_{\cO}{R_\km})_{\km'}$ containing $p$, since $R/(\varpi)$ is generically reduced, so is $(\tilde{S}\otimes_{\cO}{R_\km})_{\km'}/(\lambda)=(\tilde{S}/(\lambda)\otimes_\F R_\km/(\varpi))_{\km'}$ as $\F$ is perfect. Thus any height one prime containing $\varpi$ is also regular. This proves that $(\tilde{S}\otimes_{\cO}{R_\km})_{\km'}$ is normal and finishes the proof of the first part of the lemma.

For the last two parts of the lemma, we follow the argument of the proof of lemma 2.7 in \cite{Ta08} (see also lemma 1.6 of \cite{Tho15}). First we note that for any finite extension $L/E$ with ring of integers $\cO'$, $S'=S\otimes_{\cO}\cO'$ and $R'=R\otimes_{\cO}\cO'$ also satisfy the assumptions in the lemma, with $\cO$ replaced by $\cO'$. Let $\km'\in\Spec R'$ be the unique maximal ideal above $\km$. Using the isomorphism $(S\widehat{\otimes}_{\cO}R_\km)\otimes_{\cO}\cO'\cong S'\widehat{\otimes}_{\cO'}R'_{\km'}$, it is easy to see that it suffices to show the same assertions for $S',R',\cO'$. Therefore, we may replace $\cO$ by its extension if necessary.

We may assume $R$ is reduced. Denote the minimal primes of $R_\km$ by $\kp_1,\cdots,\kp_r$ and $R_i=R_{\km}/\kp_i$. Let $\tilde{R_i}$ be the normalization of $R_i$ and $\tilde{R}=\prod_i^r \tilde{R_i}$ be the normalization of $R_\km$. Hence $\tilde{R}$ is a finite $R_\km$-algebra. Let $Q$ be $R_\km$-module $\tilde{R}/R_{\km}$. It follows from our assumptions that $(R_{\km})_{\kp_{i,j}}$ is a discrete valuation ring with maximal ideal $(\varpi)$ for any minimal prime $\kp_{i,j}$ of $R_{\km}/(\varpi, \kp_i)$. Therefore $Q_{\kp_{i,j}}=0$. Consider the exact sequence:
\[0\to R_\km\to \tilde{R} \to Q\to 0.\]
and tensor it with $\widehat{R_\km}$ over $R_\km$ (which is an exact functor):
\[0\to \widehat{R_\km}\to \prod_{i,k}\widehat{(\tilde{R_i})_{\km_{i,k}}} \to \widehat{Q}\to 0,\]
where the product in the middle term is taken over all maximal ideals $\km_{i,k}$ of $\tilde{R_i}$ and $\widehat{Q}$ is the $\km$-adic completion of $Q$. Here we use the fact that the $\km$-adic completion of $\tilde{R_i}$ is isomorphic to $\prod_k \widehat{(\tilde{R_i})_{\km_{i,k}}}$. Replace $E$ by some unramified extension if necessary. We may assume $\tilde{R_i}/\km_{i,k}=\F$ for any $i,k$. Note that each term in the above short exact sequence is a finite $\widehat{R_\km}$-module. We can tensor it with $S\widehat{\otimes}_\cO \widehat{R_\km}$ (over $\widehat{R_\km}$) and apply lemma \ref{scca}:
\[0\to S\widehat{\otimes}_\cO \widehat{R_\km} \to \prod_{i,k} S\widehat{\otimes}_\cO \widehat{(\tilde{R_i})_{\km_{i,k}}} \to S\widehat{\otimes}_\cO \widehat{Q}\to 0.\]
Note that $(\tilde{R_i})_{\km_{i,k}}$ is a normal local flat $\cO$-algebra and $(\tilde{R_i})_{\km_{i,k}}/(\varpi)$ is generically reduced as so is $R/(\varpi)$. We may apply the first part of the lemma to conclude that $S\widehat{\otimes}_\cO \widehat{(\tilde{R_i})_{\km_{i,k}}}$ is an integral domain with dimension $d+e$.

Let $\kq$ be a minimal prime ideal of $S\widehat{\otimes}_\cO \widehat{R_\km}/(\varpi)$. Then by going-down Theorem, the pull-back of $\kq$ to $R_{\km}/(\varpi)$ is also a minimal prime $\kp'$. Hence $(S\widehat{\otimes}_\cO \widehat{Q})_{\kq}=0$ since $Q_{\kp'}=0$. In particular, we have 
\[(S\widehat{\otimes}_\cO \widehat{R_\km})_\kq \cong \prod_{i,k} (S\widehat{\otimes}_\cO \widehat{(\tilde{R_i})_{\km_{i,k}}})_\kq.\]
The left hand side is a local ring. Hence $(S\widehat{\otimes}_\cO \widehat{(\tilde{R_i})_{\km_{i,k}}})_\kq\neq 0$ for a unique pair $(i,k)$. This implies that $(S\widehat{\otimes}_\cO \widehat{R_\km})_\kq$ is an integral domain. In other words, $\kq$ contains a unique minimal prime of $S\widehat{\otimes}_\cO \widehat{R_\km}$. This proves the third part of the lemma.

Now let $\kp$ be a minimal prime of $S\widehat{\otimes}_\cO \widehat{R_\km}$. Again by going-down Theorem, its pull-back to $R_{\km}$ defines a minimal prime and the same argument as in the previous paragraph shows that $(S\widehat{\otimes}_\cO \widehat{R_\km})_{\kp}\cong (S\widehat{\otimes}_\cO \widehat{(\tilde{R_i})_{\km_{i,k}}})_\kp$ for a unique pair $(i,k)$. Hence $(S\widehat{\otimes}_\cO \widehat{R_\km})/{\kp}$ maps injectively into $S\widehat{\otimes}_\cO \widehat{(\tilde{R_i})_{\km_{i,k}}}$. Note that $S\widehat{\otimes}_\cO \widehat{(\tilde{R_i})_{\km_{i,k}}}$ is a finite $S\widehat{\otimes}_\cO \widehat{R_\km}$-algebra. Thus 
\[\dim (S\widehat{\otimes}_\cO \widehat{R_\km})/\kp=\dim S\widehat{\otimes}_\cO \widehat{(\tilde{R_i})_{\km_{i,k}}}=d+e.\]
The flatness of $(S\widehat{\otimes}_\cO \widehat{R_\km})/\kp$ over $\cO$ is also clear. This proves the second part of the lemma.
\end{proof}

\begin{para}
To prove proposition \ref{rlocp}, we also need some basic properties of the local framed deformation rings defined in the beginning of this subsection here.
\end{para}

\begin{lem} \label{varlocdim}
\hspace{2em}
\begin{enumerate}
\item If $v\notin S$, then $\bar{\rho}_b|_{G_{F_v}}$ is unramified and $R^{\square,ur}_{v}\cong \cO[[x_1,x_2,x_3]]$.
\item If $v|p$, then $R^{\square}_v$ is isomorphic to $\cO[[x_1,\cdots,x_6]]$ or $\cO[[c_0,c_1,d_0,d_1,x_1,x_2,x_3]]/(c_0d_1-c_1d_0)$. Hence $R^{\square}_v$ is a normal domain of dimension $1+6$.
\end{enumerate}
\end{lem}

\begin{proof}
The first part is clear since we only need to lift $\Frob_v$. The second part is clear in the case $\bar{\rho}_b|_{G_{F_v}}\ncong \begin{pmatrix} 1&*\\0&\omega\end{pmatrix}\otimes \eta$ for some character $\eta$ as $H^2(G_{F_v},\ad^0\bar{\rho}_b)=0$ and $R^{\square}_v$ is smooth in this case. The computation of dimension is standard. If $\bar{\rho}_b|_{G_{F_v}}$ is a non-split extension of $\omega$ by $\mathbf{1}$ up to some character, then it follows from Corollary B.5 of \cite{Pas13} that $R^{\square}_v\cong \cO[[c_0,c_1,d_0,d_1,x_1,x_2,x_3]]/(c_0d_1+c_1d_0+pc_0)$ since $R^{\square}_v$ is formally smooth of relative dimension $3$ over the  universal deformation ring of $\bar{\rho}_b|_{G_{F_v}}$ with determinant $\chi$. Note that we assume $p>3$ in this case so that we can apply the quoted result. If $\bar{\rho}_b|_{G_{F_v}}\cong \eta\oplus\eta\omega$ for some $\eta$, then its versal deformation ring $R^{ver}$ is isomorphic to $\cO[[c_0,c_1,d_0,d_1,b]]/(c_0d_1-c_1d_0)$ by Corollary 3.6, 3.7 of \cite{HT15}. Note that in this case $R^{\square}_v$ is formally smooth of relative dimension $2$ over $R^{ver}$ by Proposition 2.1 of \cite{KW09b}. Hence we have exhausted all the cases of the second part of the lemma.
\end{proof}

\begin{lem} \label{lrca}
For $v\in S\setminus \Sigma_p$, there exists a finite type $\cO$-algebra $A^{\xi_v}$ and an $\F$-valued point $x\in \Spec A^{\xi_v}$ such that $\widehat{(A^{\xi_v})_x}\cong R^{\square,\xi_v}_{v}$. Moreover,
\begin{enumerate}
\item If $\xi_v$ is trivial, then for each minimal prime $\kp$ of $A^{\xi_v}$, $A^{\xi_v}/\kp$ is $\cO$-flat of dimension $3+1$ and $A^{\xi_v}/(\kp,\varpi)$ is generically reduced. Each minimal prime of $A^{\xi_v}/(\varpi)$ contains a unique minimal prime of $A^{\xi_v}$.
\item If $\xi_v$ is non-trivial, then 
\begin{itemize}
\item $R^{\square,\xi_v}_{v}$ is irreducible of dimension $1+3$ and flat over $\cO$.
\item $R^{\square,\xi_v}_v/(\varpi)\cong R^{\square,1}_v/(\varpi)$.
\item $R^{\square,\xi_v}_{v}[\frac{1}{p}]$ and $(A^{\xi_v})_x[\frac{1}{p}]$ are regular.
\item Let $S$ be a complete local Noetherian flat $\cO$-algebra such that $S/pS\neq 0$ and $S[\frac{1}{p}]$ is geometrically connected and geometrically normal, i.e. $S[\frac{1}{p}]\otimes_E L$ is connected and normal for any finite extension $L/E$. Then $(S\widehat{\otimes}_\cO R^{\square,\xi_v}_{v})[\frac{1}{p}]$ is geometrically connected and geometrically normal as well. The Krull dimension of $S\widehat{\otimes}_\cO R^{\square,\xi_v}_{v}$ is $\dim S+3$.
\end{itemize} 
\end{enumerate}
\end{lem}

\begin{proof}
For any $R\in C_\cO$, since $\bar{\rho}_b|_{G_{F_v}}$ has unipotent images, any lifting $\rho_R:G_{F_v}\to\GL_2(R)$ factors through the pro-$p$ quotient $I_{F_v}(p)$ of $I_{F_v}$ when restricted to the inertia. Choose a topological generator $t$ of $I_{F_v}(p)$. Then $\rho_R$ is determined by the pair of matrices $(\rho_R(\Frob_v),\rho_R(t))$. Hence $R^{\square,\xi_v}_{v}$ represents the functor sending $R\in \cOf$ to the pair of $2\times 2$ matrices $(\Phi,\Sigma)$ that lifts $(\bar{\rho}_b(\Frob_v),\bar{\rho}_b(t))$ satisfying 
\begin{itemize}
\item $\Phi\Sigma\Phi^{-1}=\Sigma^{N(v)}$
\item $\Sigma$ has characteristic polynomial $(X-\xi_v(t))(X-\xi_v^{-1}(t))$.
\item $\det \Phi=\chi(\Frob_v)$.
\end{itemize}
See \S3 of \cite{Ta08} for more details. Then we can take $\Spec A^{\xi_v}$ to be the moduli space (over $\cO$) of pair of matrices $(\Phi,\Sigma)$ satisfying the above conditions and $x$ to be the maximal ideal given by $(\bar{\rho}_b(\Frob_v),\bar{\rho}_b(t))$. 

If $\xi_v$ is trivial, then all the assertions follow from the first part of Proposition 3.1 of \cite{Ta08}. Note that $\mathcal{M}((X-1)^2,N(v))$ defined there is isomorphic to the spectrum of polynomial ring of one variable over $A^{\xi_v}$ since we require the determinant of $\Phi$ to be $\chi(\Frob_v)$.

Now we assume $\xi_v$ is non-trivial. The isomorphism between $R^{\square,\xi_v}_v/(\varpi)$ and $R^{\square,1}_v/(\varpi)$ is clear since they represent the same deformation problem. For other assertions, note that the natural map $(A^{\xi_v})_x[\frac{1}{p}]\to R^{\square,\xi_v}_{v}[\frac{1}{p}]$ is regular. Hence $R^{\square,\xi_v}_{v}[\frac{1}{p}]$ is regular if and only so is $(A^{\xi_v})_x[\frac{1}{p}]$. Moreover, if this holds, all the maps in $S[\frac{1}{p}]\to S\otimes_E (A^{\xi_v})_x[\frac{1}{p}] \to (S\widehat{\otimes}_\cO R^{\square,\xi_v}_{v})[\frac{1}{p}]$ would be regular. Thus $R^{\square,\xi_v}_{v}[\frac{1}{p}]$ being regular implies that $(S\widehat{\otimes}_\cO R^{\square,\xi_v}_{v})[\frac{1}{p}]$ is normal. Similar result holds for base change to finite extension of $E$.

For any finite extension $L/E$ with ring of integers $O_L$, we may replace everything by its tensor with $O_L$ over $\cO$. Hence the connectedness of $(S\widehat{\otimes}_\cO R^{\square,\xi_v}_{v})[\frac{1}{p}]$ will imply that it is geometrically connected. If $\bar{\rho}_b|_{G_{F_v}}$ is trivial, in this case all the claims follow from proposition 3.15, 3.16 of \cite{Tho15}. If $\bar{\rho}_b|_{I_{F_v}}$ is non-trivial, then $R^{\square,\xi_v}_{v}\cong\cO[[x_1,x_2,x_3]]$ by proposition 5.8 (1) of \cite{Sho16} and all the assertions are clear. If $\bar{\rho}_b|_{I_{F_v}}$ is trivial but $\bar{\rho}_b|_{G_{F_v}}$ is non-trivial, then $R^{\square,\xi_v}_{v}\cong\cO[[x_1,x_2,x_3,x_4]]/(x_1^2x_2-c^2)$ with $c=\xi_v(t)-\xi_v(t)^{-1}$ (proposition 5.8 (2) of \cite{Sho16}). Clearly we are left to show $(S\widehat{\otimes}_\cO R^{\square,\xi_v}_{v})[\frac{1}{p}]$ is connected. Consider the homomorphism
\[\varphi:S[[x_1,x_2,x_3,x_4]]/(x_1^2x_2-c^2)\longrightarrow S[[x_1,y_2,x_3,x_4]]/(x_1y_2-c)\] 
sending $x_2$ to $y_2^2$. It is easy to see that this map is surjective after inverting $p$. Moreover if $f(x_1,x_2)\in S[[x_3,x_4]][[x_1,x_2]]$ is in the kernel of the map $\varphi$ i.e. $f(x_1,y_2^2)=g(x_1,y_2)(x_1y_2-c)$ for some $g(x_1,y_2)\in S[[x_3,x_4]][[x_1,y_2]]$, then $f(x_1,y_2^2)^2=-g(x_1,y_2)g(x_1,-y_2)(x_1^2y_2^2-c^2)$. But $g(x_1,y_2)g(x_1,-y_2)=h(x_1,y_2^2)$ for some $h(x_1,y_2)\in S[[x_3,x_4]][[x_1,y_2]]$. Hence $f(x_1,x_2)^2=0$ in $S[[x_1,x_2,x_3,x_4]]/(x_1^2x_2-c^2)$. Thus $\varphi$ has nilpotent kernel after inverting $p$ and it suffices to show that $\Spec (S[[x_1,y_2,x_3,x_4]]/(x_1y_2-c))[\frac{1}{p}]$ is connected. This can be proved by the same argument as in the first part of the proof of lemma \ref{ccal} since $\cO[[x_1,y_2,x_3,x_4]]/(\varpi,x_1y_2-c)$ is generically reduced.

The dimension of $S\widehat{\otimes}_\cO R^{\square,\xi_v}_{v}$ is easy to compute as it is flat over $R^{\square,\xi_v}_{v}$. We remark that the local deformation problems considered in all these references have no condition on the determinants of the liftings. Hence their deformation rings have one more formal variable than our deformation rings.
\end{proof}

Now we can prove proposition \ref{rlocp}.
\begin{proof}[Proof of proposition \ref{rlocp}]
We denote $(\widehat{\bigotimes}_{v\in\Sigma_p}R^{\square}_v)\widehat{\otimes}(\widehat{\bigotimes}_{v\in T'}R^{\square,ur}_v)$ by $R_1$ and the pull-back of $\kq^{\{\xi_v\}}_{\loc}$ to $R_1$ by $\kq_1$. It follows from the explicit descriptions in lemma \ref{varlocdim} that $R_1$ is normal, hence $\widehat{(R_1)_{\kq_1}}$, as a completion of $(R_1)_{\kq_1}$, is also normal. Moreover, using our assumption in the beginning of \ref{kqrho}, we see that for any finite extension $L/E$ with ring of integers $O_L$, $\widehat{(R_1)_{\kq_1}}\otimes_\cO O_L$ is a normal local ring. In particular, $\widehat{(R_1)_{\kq_1}}[\frac{1}{p}]$ is geometrically normal and geometrically connected.

Let $\Sigma_1$ (resp. $\Sigma_2$) be the set of $v\in S\setminus\Sigma_p$ such that $\xi_v$ is non-trivial (resp. trivial). We may apply the second part of lemma \ref{lrca} repeatedly and conclude that $(\widehat{(R_1)_{\kq_1}}\widehat{\otimes} \widehat{\bigotimes}_{v\in \Sigma_1}R^{\square,\xi_v}_v)[\frac{1}{p}]$ is geometrically normal and geometrically connected. Clearly $S_1=\widehat{(R_1)_{\kq_1}}\widehat{\otimes}_\cO \widehat{\bigotimes}_{v\in \Sigma_1}R^{\square,\xi_v}_v$ is flat over $\cO$ as $R^{\square,\xi_v}_v$ and $R_1$ are all flat $\cO$-algebras. Hence $S_1$ is geometrically integral. This proves the first part of the proposition. The dimension of $S_1$ is $3[F:\Q]+3|P|-3|\Sigma_2|$.

For each $v\in \Sigma_2$, we let $X_v^1$ be the finite type $\cO$-algebra $A^{\xi_v}$ defined in lemma \ref{lrca} and $x_v\in\Spec X_v^1$ be the $x$ there. Then in order to apply \ref{ccal} with $R=X,S=S_1$ and $\km$ given by the product of all $x_v$, it suffices to show that $X=\bigotimes_{v\in\Sigma_2}X_v^1$ (over $\cO$) satisfies:
\begin{itemize}
\item For each minimal prime $\kp$ of $X$, $X/\kp$ is $\cO$-flat of dimension $1+3|\Sigma_2|$ and $X/(\varpi,\kp)$ is generically reduced.
\item Each minimal prime of $X/(\varpi)$ contains a unique a minimal prime of $X$.
\end{itemize}
To see this, we may assume $X^1_v$ is reduced. Then $X$ is also reduced and $\cO$-flat and we are left to show that
\begin{itemize}
\item $X$ is equidimensional of dimension $1+3|\Sigma_2|$.
\item For each height one prime $\kp$ of $X$ with $\varpi\in \kp$, $X_\kp$ is a DVR with maximal ideal $(\varpi)$.
\end{itemize}
Let $\kp$ be a height one prime ideal containing $\varpi$ and $\kp_v$ be its pull-back to $X^1_v$. Then $\kp_v$ is a minimal prime of $X^1_v/(\varpi)$. In particular, $X_{\kp}$ is a localization of $\bigotimes_{v\in\Sigma_2}(X_v^1)_{\kp_v}$. By our assumption, $(X_v^1)_{\kp_v}$ is a DVR with maximal ideal $(\varpi)$. Hence the homomorphism $\cO\to (X_v^1)_{\kp_v}$ is regular and an easy induction argument shows that $\bigotimes_{v\in\Sigma_2}(X_v^1)_{\kp_v}$ is a regular ring. Therefore $X_{\kp}$ is regular. Note that $\bigotimes_{v\in\Sigma_2}(X_v^1)_{\kp_v}/(\varpi)$ is reduced as it is the tensor product of $k(\kp_v)$ over $\F$. This implies that $X_{\kp}/(\varpi)$ is reduced. Hence $X_\kp$ is a DVR with maximal ideal $(\varpi)$.

Since $\kp$ can be viewed as a minimal prime of $\bigotimes_{v\in\Sigma_2} X_v^1/\kp_v$, it is easy to see that $X/\kp$ has dimension $\sum_{v\in\Sigma_2}\dim X_v^1/\kp_v=3|\Sigma_2|$. Hence $X$ is equidimensional of dimension $1+3|\Sigma_2|$ as $X$ is $\cO$-flat. This finishes the proof of proposition \ref{rlocp}.
\end{proof}

\begin{para} \label{grdr}
We also need to define some global framed deformation rings with certain local conditions. 
\end{para}

\begin{defn} \label{globaldefr}
Suppose $M,Q$ are finite sets of primes of $F$. 
\begin{enumerate}
\item We define $R^{\square_M,\{\xi_v\}}_{\bar{\rho}_b,Q}$ to be the universal object in $C_{\cO}$ pro-representing the functor $\mathrm{Def}^{\square_M,\{\xi_v\}}_{\bar{\rho}_b,Q}$ from $\cOf$ to the category of sets sending $R$ to the set of tuples $(\rho_R;\alpha_v)_{v\in M}$ modulo the equivalence relation $\sim_M$ where
\begin{itemize}
\item $\rho_R:G_{F,S\cup Q}\to\GL_2(R)$ is a lifting of $\bar{\rho}_b$ to $R$ with determinant $\chi$ such that $\tr (\rho_R)|_{I_{F_v}}=\xi_v+\xi_v^{-1}$ for any $v\in S\setminus \Sigma_p$.
\item $\alpha_v\in 1+M_2(\km_R),v\in M$. Here $\km_R$ is the maximal ideal of $R$.
\item $(\rho_R;\alpha_v)_{v\in M}\sim_M (\rho'_R;\alpha'_v)_{v\in M}$ if there exists an element $\beta\in 1+M_2(\km_R)$ with $\rho'_R=\beta\rho_R\beta^{-1},\alpha'_v=\beta\alpha_v$ for any $v\in M$.
\end{itemize}
We define $R^{\square_M}_{\bar{\rho}_b,Q}$ in the same way without the local conditions $\tr (\rho_R)|_{I_{F_v}}=\xi_v+\xi_v^{-1}$ for any $v\in S\setminus \Sigma_p$. If $M$ is empty, we will drop the superscript $\square_M$ in $R^{\square_M,\{\xi_v\}}_{\bar{\rho}_b,Q}$ and $R^{\square_M}_{\bar{\rho}_b,Q}$. If all $\xi_v$ are trivial, we will simply write $R^{\square_M,1}_{\bar{\rho}_b,Q}$ for $R^{\square_M,\{\xi_v\}}_{\bar{\rho}_b,Q}$.
\item We define $R^{\ps,\{\xi_v\}}_{Q}$ to be the universal deformation ring which pro-represents the functor from $\cOf$ to the category of sets sending $R$ to the set of two-dimensional pseudo-representations $T$ of $G_{F,S\cup Q}$ over $R$  such that $T$ is a lifting of $1+\bar\chi$ with determinant $\chi$ and
\[T|_{I_{F_v}}=\xi_v+\xi_v^{-1} \]
for any $v\in S\setminus \Sigma_p$. If $\xi_v$ are all trivial, we will simply write $R^{\ps,1}_Q$.

We also define $R^{\ps}_Q$ in the same way without any condition for $v\in S\setminus \Sigma_p$.
\end{enumerate}
\end{defn}

\begin{para} \label{pscl}
Suppose $M\supseteq S$. Given a tuple $(\rho_R;\alpha_v)_{v\in M}$ as in the above definition, then for $v\in M$, $\alpha_v^{-1}\rho_R\alpha_v|_{G_{F_v}}$ is a well-defined lifting of $\bar{\rho}_b|_{G_{F_v}}$. This induces natural maps $R^{\square,\xi_v}_v\to R^{\square_M,\{\xi_v\}}_{\bar{\rho}_b,Q}$ and $R^{\square}_v\to R^{\square_M}_{\bar{\rho}_b,Q}$. It is easy to see that 
\[R^{\square_M,\{\xi_v\}}_{\bar{\rho}_b,Q}\cong R^{\square_M}_{\bar{\rho}_b,Q}\otimes_{(\bigotimes_{v\in S\setminus\Sigma_p} R^{\square}_v)}(\bigotimes_{v\in S\setminus\Sigma_p} R^{\square,\xi_v}_v).\]

Note that our local deformation constraints are all defined via the traces. So we can rewrite the above isomorphism as:
\[R^{\square_M,\{\xi_v\}}_{\bar{\rho}_b,Q}\cong R^{\square_M}_{\bar{\rho}_b,Q}\otimes_{R^{\ps}_Q}R^{\ps,\{\xi_v\}}_Q,\]
where the map $R^{\ps}_Q\to R^{\square_M}_{\bar{\rho}_b,Q}$ is given by evaluating the trace of the universal lifting.
\end{para}

\subsection{Existence of Taylor-Wiles primes} \label{eotwp}
\begin{para}
We will freely use the notation introduced in \ref{defnP}. In \ref{rxivloc}, we defined $R^{\{\xi_v\}}_{\loc}$ to be
\[(\widehat{\bigotimes}_{v\in \Sigma_p}R^{\square}_v)\widehat{\otimes}(\widehat{\bigotimes}_{v\in S\setminus \Sigma_p}R^{\square,\xi_v}_v)\widehat{\otimes}(\widehat{\bigotimes}_{v\in T'}R^{\square,ur}_v),\]
where all the completed tensor products are taken over $\cO$. Then there is a natural map 
\[R^{\{\xi_v\}}_{\loc}\to R^{\square_P,\{\xi_v\}}_{\bar{\rho}_b,Q}\]
for any finite set of primes $Q$. Recall that we fix a nice prime $\kq$ in the beginning of this section \ref{kqrho}. By the universal property, $(\rho(\kq)^o;1)_{v\in P}$ gives rise to a prime $\kq_{b,Q}$ of $R^{\square_P,\{\xi_v\}}_{\bar{\rho}_b,Q}$, whose pull-back to $R^{\{\xi_v\}}_{\loc}$ is $\kq^{\{\xi_v\}}_{\loc}$. Note that by our choice of $T'$,  we have $B=R^{\square_P,\{\xi_v\}}_{\bar{\rho}_b,Q}/\kq_{b,Q}=R^{\{\xi_v\}}_{\loc}/\kq^{\{\xi_v\}}_{\loc}$.

For any $A$-module $M$ of finite length, we denote its length by $\ell(M)$. Now we can state the main result of this subsection. See also Proposition 6.10 of \cite{SW99}.
\end{para}

\begin{prop} \label{exttwp}
Let $r=\dim_{k(\kq)}H^1(G_{F,P},\ad^0\rho(\kq)(1))$, where $\ad^0\rho(\kq)$ denotes the trace $0$ subspace of the adjoint representation $\ad \rho(\kq)$ of $\rho(\kq)$ and $(1)$ denotes the Tate twist. Then there exists an integer $C$ such that for any positive integer $N$, we can find a finite set of primes $Q_N$ disjoint with $P$ such that
\begin{enumerate}
\item $|Q_N|=r$.
\item $N(v)\equiv 1\mod p^N$ for any $v\in Q_N$.
\item $\rho(\kq)(\Frob_v)$ has distinct eigenvalues $\alpha_v,\beta_v$ with $\ell(A/(\alpha_v-\beta_v)^2)<C$ for any $v\in Q_N$.
\item There exists an $A$-module $M_N$ with $\ell(M_N)<C$ such that 
\[\kq_{b,Q_N}/(\kq_{b,Q_N}^2,\kq^{\{\xi_v\}}_{\loc})\otimes_B A \cong A^{\oplus g} \oplus M_N\] 
as $A$-modules, where $g=r+|P|-[F:\Q]-1$.
\item There exists a map $R^{\{\xi_v\}}_{\loc}[[x_1,\cdots,x_g]]\to R^{\square_P,\{\xi_v\}}_{\bar{\rho}_b,Q_N}$ such that the images of $x_i$ are in $\kq_{b,Q_N}$ and $\kq_{b,Q_N}/(\kq_{b,Q_N}^2,\kq^{\{\xi_v\}}_{\loc},x_1,\cdots,x_g)$ is killed by some element $f\in B$ with $\ell(A/(f))<C$. Note that this implies such $f$ can be chosen independently of $N$.
\end{enumerate}
\end{prop}

\begin{para}
Note that the last part is a direct consequence of the previous one. Indeed, denote $\kq_{b,Q_N}/(\kq_{b,Q_N}^2,\kq^{\{\xi_v\}}_{\loc})$ by $M$. Since $A$ and $B$ have the same residue fields, we can find $x_1,\cdots,x_g\in M$ such that $\left(M/(Bx_1+\cdots+Bx_g)\right)\otimes_B A\cong M_N$.  By our construction, we may assume  $T^{C}A\subseteq B$ after possibly increasing $C$. Now we claim that $M/(Bx_1+\cdots+Bx_g)$ is killed by some element $f\in B$ with $\ell(A/(f))\leq2C$. Let $M'=M/(Bx_1+\cdots+Bx_g)$. It suffices to show that $T^{C}$ annihilates $T^CM'$. Consider the map $M_N\cong M'\otimes_B A\to M'$ sending $m\otimes a\in M'\otimes_B A$ to $(T^Ca)m$. Clearly this map is $B$-linear and the image contains $T^CM'$. Hence our claim follows from $\ell(M_N)<C$. Therefore it suffices to prove the existence of $Q_N$ satisfying the first four properties. 

The following result is standard.
\end{para}

\begin{lem} \label{stdcodt}
Let $Q$ be a finite set of primes disjoint with $P$ such that $N(v)\equiv 1\mod p$ and $\rho(\kq)^o(\Frob_v)$ has distinct eigenvalues $\alpha_v,\beta_v$ for any $v\in Q$. Denote $\kq_{b,Q}/(\kq_{b,Q}^2,\kq^{\{\xi_v\}}_{\loc})\otimes_B A/(T^n)$ by $V_n$. Then
\[\ell(V_n)=\ell(H^1_{\cL^{\perp}_{P\cup Q}}(G_{F,P\cup Q},W_n^*))-\ell(H^0(G_{F},W_n^*))+gn+\sum_{v\in Q}\ell(A/(T^{2n},(\alpha_v-\beta_v)^2)),\]
where \begin{itemize}
\item $W_n=\ad^0 \rho(\kq)^o\otimes A/(T^n)$ and $W_n^*=\Hom_{\F}(W_n,\F)(1)$.
\item $H^1_{\cL^{\perp}_{P\cup Q}}(G_{F,P\cup Q},W_n^*)=\ker (H^1(G_{F,P\cup Q},W_n^*)\to \bigoplus_{v\in Q}H^1(G_{F_v},W_n^*))$.
\item $g=r+|P|-[F:\Q]-1$ as defined in the proposition.
\end{itemize}
\end{lem}

\begin{proof}
We give a sketch of proof here. By Pontryagin duality, $\ell(V_n)$ is equal to the length of
\[\Hom_A(V_n,A/(T^n))=\Hom_B(\kq_{b,Q}/(\kq_{b,Q}^2,\kq^{\{\xi_v\}}_{\loc}),A/(T^n)).\]
Since $R^{\{\xi_v\}}_{\loc}+\kq_{b,Q}=R^{\square_P,\{\xi_v\}}_{\bar{\rho}_b,Q}$, any element $h_0$ in $\Hom_B(\kq_{b,Q}/(\kq_{b,Q}^2,\kq^{\{\xi_v\}}_{\loc}),A/(T^n))$ can be uniquely extended to a $R^{\{\xi_v\}}_{\loc}$-algebra homomorphism 
\[h: R^{\square_P,\{\xi_v\}}_{\bar{\rho}_b,Q}\to B\oplus (A/T^n)\epsilon\] 
sending $x\in \kq_{b,Q}$ to $h_0(x)\epsilon$. Here the ring structure of $B\oplus (A/T^n)\epsilon$ is given by $(b+a\epsilon)(b'+a'\epsilon)=bb'+(ab'+a'b)\epsilon$ and we view it as a $R^{\{\xi_v\}}_{\loc}$-algebra by $R^{\{\xi_v\}}_{\loc}\to R^{\{\xi_v\}}_{\loc}/\kq^{\{\xi_v\}}_{\loc}=B\hookrightarrow B\oplus (A/T^n)\epsilon$. Note that by construction, $h(\kq_{b,Q})\subseteq (A/T^n)\epsilon$. In fact, this induces an isomorphism between $\Hom_B(\kq_{b,Q}/(\kq_{b,Q}^2,\kq^{\{\xi_v\}}_{\loc}),A/(T^n))$ and 
\[\{h\in \Hom_{R^{\{\xi_v\}}_{\loc}-alg}(R^{\square_P,\{\xi_v\}}_{\bar{\rho}_b,Q},B\oplus (A/T^n)\epsilon), h(\kq_{b,Q})\subseteq (A/T^n)\epsilon\}.\]
By the universal property, this is the subset $\{[(\rho;\alpha_v)_{v\in P}]\}$ of $\mathrm{Def}^{\square_P,\{\xi_v\}}_{\bar{\rho}_b,Q}(B\oplus (A/T^n)\epsilon)$ with 
\begin{itemize}
\item $\alpha_v\in\ker(\GL_2(B\oplus (A/T^n)\epsilon)\to \GL_2(B)),v\in P$.
\item $\rho\mod (A/T^n)\epsilon = \rho(\kq)^o$ and $\alpha_v^{-1}\rho|_{G_{F_v}}\alpha_v=\rho(\kq)^o|_{G_{F_v}}$ for $v\in P$.
\end{itemize}
Write $\rho(\sigma)=(1+\phi(\sigma)\epsilon)\rho(\kq)^o(\sigma), \alpha_v=1+m_v\epsilon$. It is easy to check that $\phi:G_{F,P\cup Q}\to \ad^0\rho(\kq)^o\otimes A/T^n$ defines a $1$-cocycle in $Z^1(G_{F,P\cup Q},\ad^0\rho(\kq)^o\otimes A/T^n)$ with $\phi|_{G_{F_v}}=\partial m_v$ for $v\in P$. Conversely, any $(\rho;\alpha_v)_{v\in P}$ arises from such a $1$-cocycle and $m_v$. Using this, it is not hard to see that 
\[\ell(V_n)=\sum_{v\in P}\ell(H^0(G_{F_v},W_n))+\ell(H^1_{\cL_{P\cup Q}}(G_{F},W_n))-\ell(H^0(G_{F,P\cup Q},W_n))+(|P|-1)n,\]
where $H^1_{\cL_{P\cup Q}}(G_{F,P\cup Q},W_n)=\ker (H^1(G_{F,P\cup Q},W_n)\stackrel{\mathrm{res}}{\longrightarrow} \bigoplus_{v\in P\cup Q}H^1(G_{F_v},W_n)/\cL_v)$ and $\cL_v=0$ if $v\in P$ and $\cL_v=H^1(G_{F_v},W_n)$ otherwise. By Poitou-Tate long exact sequence and global Euler characteristic formula (see Theorem 2.19 of \cite{DDT95}), we may compute
\begin{eqnarray*}
\ell(H^1_{\cL_{P\cup Q}}(G_{F,P\cup Q},W_n))=&\ell(H^0(G_{F},W_n))-\ell(H^0(G_{F},W_n^*))+\ell(H^1_{\cL^{\perp}_{P\cup Q}}(G_{F,P\cup Q},W_n^*))\\
&+ \sum_{v\in P\cup  Q}(\ell(\cL_v)-\ell(H^0(G_{F_v},W_n)))-\sum_{v|\infty}\ell(H^0(G_{F_v},W_n)),
\end{eqnarray*}
where $H^1_{\cL^{\perp}_{P\cup Q}}(G_{F,P\cup Q},W_n^*)$ and $W_n^*$ were defined in the lemma. By local Euler characteristic formula and local Tate duality, for $v\in Q$, it follows from our assumption on $v$ that
\[\ell(\cL_v)-\ell(H^0(G_{F_v},W_n))=\ell(H^0(G_{F_v},W_n^*))=\ell(H^0(G_{F_v},W_n))=n+\ell(A/(T^{2n},(\alpha_v-\beta_v)^2)).\]

For any $v|\infty$, by our oddness assumption,
\[\ell(H^0(G_{F_v},W_n))=n.\]

The lemma is a direct consequence of all these formulae.
\end{proof}

\begin{para}
Note that the perfect pairing $W_n\times W_n\to A/T^n, (X,Y)\mapsto \tr(XY)$ induces an isomorphism $W_n(1)\cong W_n^*$ by sending $\sum_{i=0}^{n-1}a_iT^{i}\in A/T^n, a_i\in \F$ to $a_{n-1}\in \F$. It follows from the third condition in the definition of nice prime (\ref{nice}) that $\ell(H^0(G_{F},W_n(1)))$ is bounded independent of $n$. Thus in view of the previous lemma, proposition \ref{exttwp} follows from
\end{para}

\begin{lem} \label{lem1}
There exists an integer $C$ such that for any $N>0$, we can find a finite set of primes $Q_N$ disjoint with $P$ such that the first three parts of proposition \ref{exttwp} hold and 
\[\ell(H^1_{\cL^{\perp}_{P\cup Q_N}}(G_{F,P\cup Q_N},W_n^*))<C\]
for any $n$.
\end{lem}

\begin{para}
Let $Q$ be a finite set of primes as in lemma \ref{stdcodt}. Combine the following two exact sequences:
\[0\to H^1_{\cL^{\perp}_{P\cup Q}}(G_{F,P\cup Q},W_n^*)\to H^1(G_{F,P\cup Q},W_n^*)\to\bigoplus_{v\in Q}H^1(G_{F_v},W_n^*),\]
\[0\to H^1(G_{F,P},W_n^*)\to H^1(G_{F,P\cup Q},W_n^*)\to\bigoplus_{v\in Q}H^1(G_{F_v},W_n^*)/H^1(G_{k(v)},W_n^*).\]
We get that 
\[0\to H^1_{\cL^{\perp}_{P\cup Q}}(G_{F,P\cup Q},W_n^*)\to H^1(G_{F,P},W_n^*)\to\bigoplus_{v\in Q}H^1(G_{k(v)},W_n^*),\]
where the last map is the restriction map. We note that $H^1(G_{k(v)},W_n^*)\cong W_n^*/(\Frob_v-1)W_n^*$.
\end{para}

\begin{lem} \label{lem2}
Let $F(\zeta_{p^\infty})=\bigcup_n F(\zeta_{p^n})$, where $\zeta_{p^n}$ are primitive $p^n$-th roots of unity. Denote its absolute Galois group by $G_{F(\zeta_{p^\infty})}\subset G_{F}$. Then there exists $\sigma_1,\cdots,\sigma_r\in G_{F(\zeta_{p^\infty})}$ such that
\begin{itemize}
\item $\rho(\kq)(\sigma_i)$ has distinct eigenvalues for each $i$. 
\item The kernel and cokernel of the map $H^1(G_{F,P},W^*)\to \bigoplus_{i=1}^rW^*/(\sigma_i-1)W^*$ have finite lengths, where $W^*$ denotes $\Hom_A(\ad^0\rho(\kq)^o,A(1))\cong \ad^0\rho(\kq)^o(1)$, equipped with $T$-adic topology.
\end{itemize}
\end{lem}

\begin{para}
Before giving the proof of this lemma, we first show that this implies lemma \ref{lem1}. Fix $\sigma_1,\cdots,\sigma_r\in G_{F(\zeta_{p^\infty})}$ as in the lemma. Then the image of
\[H^1(G_{F,P},W^*)\to \bigoplus_{i=1}^rW^*/(\sigma_i-1)W^*\]
contains $T^M\left(\bigoplus_{i=1}^rW^*/(\sigma_i-1)W^*\right)$ for some integer $M\geq0$.
Now given a positive integer $N$, we may choose places $v_1,\cdots,v_r$ of $F$ with $\Frob_{v_i}\in G_{F(\zeta_{p^N})}$ close enough to $\sigma_i$ so that for $i=1,\cdots,r$, 
\begin{itemize}
\item $\rho(\mathfrak{q})(\Frob_{v_i})$ has distinct eigenvalues $\alpha_i,\beta_i$ and
$A/(\alpha_i-\beta_i)^2=A/(\alpha'_i-\beta'_i)^2$ where $\alpha'_i$ and $\beta'_i$ denote the eigenvalues of $\rho(\kq)(\sigma_i)$ ;
\item $(\sigma_i-1)W_{M+1}^*=(\Frob_{v_i}-1)W_{M+1}^*$; 
\item $\phi(\sigma_i)-\phi(\Frob_{v_i})\in T^{M+1}W^*$ for any $[\phi]\in H^1(G_{F,P},W^*)$. This is possible as $H^1(G_{F,P},W^*)$ is a finitely generated $A$-module.
\end{itemize} 
The last two assumptions imply that the map 
\[H^1(G_{F,P},W^*)/T^{M+1}\to \bigoplus_{i=1}^rW^*/\left((\Frob_{v_i}-1)W^*+T^{M+1}W^*\right)=\bigoplus_{i=1}^rW_{M+1}^*/(\Frob_{v_i}-1)W_{M+1}^*\]
agrees with $H^1(G_{F,P},W^*)/T^{M+1}\to \bigoplus_{i=1}^rW_{M+1}^*/(\sigma_{i}-1)W_{M+1}^*$, hence its image
contains $T^M\left( \bigoplus_{i=1}^rW_{M+1}^*/(\Frob_{v_i}-1)W_{M+1}^*\right)$. Since both $H^1(G_{F,P},W^*)$ and $\bigoplus_{i=1}^rW^*/(\Frob_{v_i}-1)W^*$  are $T$-adically complete, it follows that the image of the map
\[H^1(G_{F,P},W^*)\to\bigoplus_{i=1}^rW^*/(\Frob_{v_i}-1)W^*\] 
contains $T^M\left(\bigoplus_{i=1}^rW^*/(\Frob_{v_i}-1)W^*\right)$. Thus  this map becomes an isomorphism after inverting $T$ by considering the dimensions over $k(\kq)$.  Hence there is a uniform upper bound for the length of the kernel of 
\[H^1(G_{F,P},W^*)/T^n\to  \bigoplus_{i=1}^r H^1(G_{k(v_i)},W^*)/T^n\]
for any $n\geq 0$ and any choice of $v_1,\cdots,v_r$.

 Let $Q_N=\{v_1,\cdots,v_r\}$. Consider the following diagram:
\[\begin{tikzcd}
& & H^1(G_{F,P},W^*)/T^n \arrow[r] \arrow[d, hook] & \bigoplus_{i=1}^r H^1(G_{k(v_i)},W^*)/T^n \arrow[d, equal] \\
0 \arrow[r] &  H^1_{\cL^{\perp}_{P\cup Q_N}}(G_{F,P\cup Q_N},W_n^*) \arrow[r] & H^1(G_{F,P},W_n^*) \arrow[r] \arrow[d] & \bigoplus_{i=1}^r H^1(G_{k(v_i)},W_n^*) \\
& & H^2(G_{F,P},W^*)^{tor} &.
\end{tikzcd}\]
Here $H^2(G_{F,P},W^*)^{tor}$ denotes the $T^\infty$-torsion part of $H^2(G_{F,P},W^*)$. From this, it is easy to see that $H^1_{\cL^{\perp}_{P\cup Q_N}}(G_{F,P\cup Q_N},W_n^*)$ are bounded independent of $n$ and $Q_N$.

Now it suffices to prove the lemma below, which implies lemma \ref{lem2} by induction on $\sigma_i$.
\end{para}

\begin{lem} \label{lem3}
Fix an inclusion $A^r\hookrightarrow H^1(G_{F,P},W^*)$. For any $[\phi]\in A^r\setminus TA^r$ (viewed as an element in $H^1(G_{F,P},W^*)$), there exists an element $\sigma\in G_{F(\zeta_{p^\infty})}$ such that $\rho(\kq)(\sigma)$ has distinct eigenvalues and 
\[\ell(W^*/((\sigma-1)W^*+A\phi(\sigma)))<+\infty.\]
\end{lem}

\begin{proof}
Let $G_{\tilde{F},P}=\ker(\varepsilon)\cap \ker(\rho(\kq))\subseteq G_{F,P}$ and $\Gamma=G_{F,P}/G_{\tilde{F},P}$. Consider the following short exact sequence given by Hochschild-Serre spectral sequence:
\[0\to H^1(\Gamma,W^*)\to H^1(G_{F,P},W^*)\to H^1(G_{\tilde{F},P},W^*)^\Gamma(=\Hom_\Gamma(G_{\tilde{F},P},W^*)).\]
We quote the following result which is mainly due to Skinner-Wiles:

\begin{lem}\label{ktwl}
$\ell(H^1(\Gamma,W^*))<\infty$.
\end{lem}

\begin{proof}
This is lemma 6.9 of \cite{SW99} if $\rho(\kq)$ is not dihedral. Note that their paper requires $\rho(\kq)|_{G_{F_v}}\cong \begin{pmatrix}\chi_{v,1}& *\\ 0 & \chi_{v,2}\end{pmatrix}$ with $\chi_{v,1}/\chi_{v,2}$ of infinite order for \textit{each} $v|p$. However this assumption is only used to establish their lemma 6.1, which holds here as we assume $\rho(\kq)$ is irreducible and not dihedral. 

If $\rho(\kq)$ is isomorphic to $\Ind_{G_L}^{G_F}\theta$, then from the proof of lemma \ref{nirred} we know that $\bar{\chi}$ is quadratic and $L=F(\bar{\chi})$. Hence $\rho(\kq)|_{G_L}\cong \theta\oplus \theta^{-1}$. Therefore the only possible non-trivial abelian quotient of $\im(\rho(\kq))$ is $\bar{\chi}$, where $\im(\rho(\kq))$ denotes the image of $G_F$ in $\GL_2(k(\kq))$ under $\rho(\kq)$. Since $L\cap F(\zeta_p)=F$, we have $\overbar{F}^{\ker(\rho(\kq))}\cap F(\zeta_{p^\infty})=F$ and $\Gamma\cong \im(\rho(\kq))\times \Gal(F(\zeta_{p^\infty})/F)$. It is clear that $H^1(\Gamma,W^*)=0$ as $\Gal(F(\zeta_{p})/F)$ acts non-trivially on $W^*\cong \ad^0 \rho(\kq)^o(1)$.
\end{proof}

Back to the proof of lemma \ref{lem3}, we first deal with the case that $\rho(\kq)$ is \textit{not dihedral}. A direct consequence of the previous lemma is that $\phi|_{G_{\tilde{F},P}}\in \Hom_\Gamma(G_{\tilde{F},P},W^*)$ is non-zero. Since $\rho(\kq)$ is irreducible and not dihedral, $W^*\otimes k(\kq)$ is an irreducible representation of $\Gamma$ and the $A$-module generated by $\phi(G_{\tilde{F},P})$ contains $T^{C}W^*$ for some $C>0$. 

Choose $\sigma_0\in G_{F(\zeta_{p^\infty})}$ such that $\rho(\kq)(\sigma_0)$ has distinct eigenvalues. This is possible, otherwise $\rho(\kq)(G_{F(\zeta_{p^\infty})})$ would be a pro-$p$ group and $\rho(\kq)|_{G_{F(\zeta_{p^\infty})}}$ would be reducible, which contradicts our third condition in \ref{nice}. Choose $\tau\in G_{\tilde{F},P}$ so that 
\[\ell(W^*/(A(\phi(\tau)+\phi(\sigma_0))+(\sigma_0-1)W^*))<+\infty.\]
It is clear that $\sigma=\sigma_0\tau$ satisfies all the requirements in the lemma. This proves lemma \ref{lem3} and hence proposition \ref{exttwp} in the non-dihedral case.

Now assume that we are in the dihedral case, i.e. $\rho(\kq)\cong\Ind_{G_L}^{G_F}\theta$. In view of the proof in the non-dihedral case, it is enough to show that for any non-zero irreducible subrepresentation $V$ of $W^*\otimes k(\kq)$ (viewed as a representation of $\Gamma$), we can find an element $\sigma_0\in G_{F(\zeta_{p^\infty})}$ such that $\rho(\kq)(\sigma_0)$ has distinct eigenvalues and $V$ contains the $\sigma_0$-invariant subspace $(W^*)^{\sigma_0=1}$.

By our discussion in the second paragraph of the proof of lemma \ref{ktwl}, $L=F(\bar{\chi})$ and $\overbar{F}^{\ker(\rho(\kq))}\cap F(\zeta_{p^\infty})=F$. It is clear that $W^*\cong \bar{\chi}(1)\oplus (\Ind_{G_L}^{G_F}\theta^2)(1)$. For $\sigma_0 \in G_{F(\zeta_{p^\infty})}$ such that $\theta(\sigma_0)\neq1$, we have $(W^*)^{\sigma_0=1}=\bar{\chi}(1)$. For $\sigma_0 \in G_{F(\zeta_{p^\infty})}\setminus G_{L(\zeta_{p^\infty})}$, we have $(W^*)^{\sigma_0=1}$ contained in $(\Ind_{G_L}^{G_F}\theta^2)(1)$. Hence this finishes the proof of lemma \ref{lem3}.
\end{proof}

\subsection{Completed homology with auxiliary levels}
\begin{para}
In this subsection, we study the completed homology with certain auxiliary levels. For any positive integer $N$, we fix a set of primes $Q_N$ as in proposition \ref{exttwp}. Denote the unique quotient of $k(v)^\times$ of order $p^N$ by $\Delta_v$ and $\bigoplus_{v\in Q_N}\Delta_v$ by $\Delta_N$.

Recall that $U^p=\prod_{v\nmid p}U_v$ where $U_v=\GL_2(O_{F_v})$ if $v\notin S$ and $U_v=\mathrm{Iw}_v$ otherwise. Define tame levels $U^p_{Q_N,0}=\prod_{v\nmid p}U_{Q_N,0,v} ,U^p_{Q_N}=\prod_{v\nmid p}U_{Q_N,v}$ as follows: 
\begin{itemize}
\item $U_{Q_N,0,v}=U_{Q_N,v}=U_v$ if $v\notin Q_N$.
\item $U_{Q_N,0,v}=\mathrm{Iw}_v$ if $v\in Q_N$.
\item $U_{Q_N,v}=\ker(\mathrm{Iw}_v\to\Delta_v)$ if $v\in Q_N$, where the map is the composite of $\mathrm{Iw}_v\to k(v)^\times: \begin{pmatrix}a&b\\c&d\end{pmatrix}\mapsto \frac{a}{d}\mod \varpi_v$ and the natural quotient map $k(v)^\times\to \Delta_v$.
\end{itemize}

We can define Hecke algebra $\T_{\psi,\xi}(U^p_{Q_N}),\T_{\psi,\xi}(U^p_{Q_N,0})$ as before. By abuse of notation, we also view $\km$ as maximal ideals of these Hecke algebras.

It is clear that $\Delta_v=\mathrm{Iw}_v/U_{Q_N,v}$ acts naturally on the completed homology  $M_{\psi,\xi}(U^p_{Q_N})$ and $S_{\psi,\xi}(U^p_{Q_N}U_p,\cO/\varpi^n)$ via the right translation of $\mathrm{Iw}_w$ for any open compact subgroup $U_p\subseteq K_p$ and positive integer $n$. Hence all these spaces are $\cO[\Delta_N]$-modules. Let $\mathfrak{a}_{Q_N}$ be the augmentation ideal of $\cO[\Delta_N]$.
\end{para}

\begin{lem} \label{flmf}
Suppose $U_p\subseteq K_p$ is an open compact subgroup such that $\psi|_{U_p\cap O_{F,p}^\times}$ is trivial modulo $\varpi^n$ for some $n$ and $U^pU_p$ is sufficiently small. Then 
\begin{enumerate}
\item $S_{\psi,\xi}(U^p_{Q_N}U_p,\cO/\varpi^n)$ and $S_{\psi,\xi}(U^p_{Q_N}U_p,\cO/\varpi^n)^\vee$ are finite flat $\cO/\varpi^n[\Delta_N]$-modules.
\item The natural map $S_{\psi,\xi}(U^p_{Q_N}U_p,\cO/\varpi^n)^\vee\to S_{\psi,\xi}(U^p_{Q_N,0}U_p,\cO/\varpi^n)^\vee$ induces a natural isomorphism:
\[S_{\psi,\xi}(U^p_{Q_N}U_p,\cO/\varpi^n)^\vee/\ka_{Q_N}S_{\psi,\xi}(U^p_{Q_N}U_p,\cO/\varpi^n)^\vee\cong S_{\psi,\xi}(U^p_{Q_N,0}U_p,\cO/\varpi^n)^\vee.\]
\item The natural map $S_{\psi,\xi}(U^p_{Q_N}U_p,\cO/\varpi^{n+1})^\vee\to S_{\psi,\xi}(U^p_{Q_N}U_p,\cO/\varpi^n)^\vee$ is surjective with kernel $\varpi^n S_{\psi,\xi}(U^p_{Q_N}U_p,\cO/\varpi^{n+1})^\vee$.
\item $U'_p\subseteq U_p$ is another open subgroup. Then the natural map $S_{\psi,\xi}(U^p_{Q_N}U'_p,\cO/\varpi^n)^\vee\to S_{\psi,\xi}(U^p_{Q_N}U_p,\cO/\varpi^n)^\vee$ is surjective.
\end{enumerate}
The same results hold with everything localized at $\km$.
\end{lem}
\begin{proof}
All these claims follow from the discussion in \ref{Quaformssm}. The first two parts are the Pontryagin dual of Lemma 2.1.4 of \cite{Kis09a}.
\end{proof}

\begin{cor}
$M_{\psi,\xi}(U^p_{Q_N})$ is a flat $\cO[\Delta_N]$-module and 
\[M_{\psi,\xi}(U^p_{Q_N})/\ka_{Q_N} M_{\psi,\xi}(U^p_{Q_N})\cong M_{\psi,\xi}(U^p_{Q_N,0}).\]
\end{cor}
\begin{proof}
By writing $M_{\psi,\xi}(U^p_{Q_N})=\varprojlim_{n}\varprojlim_{U_p} S_{\psi,\xi}(U^p_{Q_N}U_p,\cO/\varpi^n)^\vee$, this follows from the lemma below.
\end{proof}

\begin{lem} \label{fllem}
Let $R$ be a local complete Noetherian ring.
\begin{enumerate}
\item Suppose $R$ is Artinian. Let $\{M_i\}_{i\in \mathbf{N}}$ be a projective system of flat $R$-modules with surjective transition maps. Then $M=\varprojlim M_i$ is also flat over $R$ and $M/JM\cong \varprojlim M_i/JM_i$ for any ideal $J$ of $R$.
\item Let $I$ be an ideal of $R$ and $\{N_i\}_{i\in \mathbf{N}}$ be a projective system of $R$-modules such that $N_i$ is a flat $R/I^i$-module and the natural transition maps induce isomorphisms $N_i\cong N_{i+1}/I^i N_{i+1}$. Then $N=\varprojlim N_i$ is a flat $R$-module and $N_i\cong N/I^iN$. Moreover for any ideal $J$ of $R$, $N/JN\cong \varprojlim N_i/JN_i$.
\end{enumerate}
\end{lem}

\begin{proof}
Both are special cases of Lemma 15.27.4. of The Stacks project \cite[\href{https://stacks.math.columbia.edu/tag/0912}{Tag 0912}]{stacks-project}. I would like to thank one referee for providing such a reference.

\end{proof}

\begin{para} \label{tdm}
Let $P(Q_N)$ be the power set of $Q_N$. Now we consider the following map:
\begin{eqnarray*}
\eta'_{N}:S_{\psi,\xi}(U^pU_p,\cO/\varpi^n)^{P(Q_N)}&\to& S_{\psi,\xi}(U^p_{Q_N,0}U_p,\cO/\varpi^n)\\
	(f_X)_{X\in P(Q_N)}&\mapsto&  \sum_{X}\prod_{v\in X}\begin{pmatrix} 1 & 0\\0 &\varpi_v\end{pmatrix}\cdot f_X,
\end{eqnarray*}
where $\begin{pmatrix} 1 & 0\\0 &\varpi_v\end{pmatrix}\cdot f_X$ denotes the right translation of $f_X$ by $\begin{pmatrix} 1 & 0\\0 &\varpi_v\end{pmatrix}$. If $X$ is the empty set, we set $\prod_{v\in X}\begin{pmatrix} 1 & 0\\0 &\varpi_v\end{pmatrix}\cdot f_X=f_X$. It is clear that this map commutes with the action of the Hecke operators away from $Q_N$. Hence we may localize this map at $\km$ and denote it by $\eta_N$.

Recall that there is a map $R^{\ps}_{Q_N}\to \T_{\psi,\xi}(U^p_{Q_N,0})_\km$ sending $T^{univ}(\Frob_v)$ to $T_v$ for $v\notin S\cup Q_N$, which factors through $R^{\ps,\{\xi_v\}}_{Q_N}$. Here $T^{univ}:G_{F,S\cup Q_N}\to R^{\ps}_{Q_N}$ is the universal trace.
\end{para}

\begin{prop} \label{twkcc}
There exists a constant $C$ and elements $\tilde{f}_N\in R^{\ps,\{\xi_v\}}_{Q_N}$ that satisfy the following properties: for any $N$ and any open pro-$p$ subgroup $U_p\subseteq K_p$ such that $\psi|_{N_{D/F}(U_p)}$ is trivial modulo $\varpi^n$ and $U^pU_p$ is sufficiently small, we have
\begin{enumerate}
\item $\tilde{f}_N$ kills the kernel and cokernel of $\eta_N$.
\item The image $\bar{f}_N$ of $\tilde{f}_N$ in $R^{\ps,\{\xi_v\}}_{Q_N}\to\T_{\psi,\xi}(U^p_{Q_N,0})_\km\to\T_{\psi,\xi}(U^p)_\km\stackrel{\mod\kq}{\longrightarrow} A$ is non-zero and the length $\ell(A/(\bar{f}_N))<C$.
\end{enumerate}
\end{prop}

\begin{proof}
Define $\tilde{f}_N$ to be $\theta_{Q_N}^{3r}$, where $r$ is defined in proposition \ref{exttwp}, $\theta_{Q_N}=\prod_{w\in Q_N}\theta_w$ and
\[\theta_w:=(1+N(w))^2T^{univ}(\Frob_w^2)-(1+N(w)^2)T^{univ}(\Frob_w)^2.\]
It rests to check all the desired properties of $\tilde{f}_N$.

\noindent \underline{\textbf{Kernel of $\eta_N$:}} We first identify the kernel of $\eta_N$, following the proof of Lemma 2 of \cite{DT94}. Let $(f_X)_{X\in P(Q_N)}\in\ker(\eta_N)$ and $w$ be a place in $Q_N$. Define
\begin{eqnarray*}
f_1=\sum_{X\ni w}\prod_{v\in X}\begin{pmatrix} 1 & 0\\0 &\varpi_v\end{pmatrix}\cdot f_X,\\
f_2=\sum_{X\not\owns w}\prod_{v\in X}\begin{pmatrix} 1 & 0\\0 &\varpi_v\end{pmatrix}\cdot f_X.
\end{eqnarray*}
Since $f_X$ is right invariant by $U_w=\GL_2(O_{F_w})$, it is clear that $f_1=-f_2$ is right invariant by $\GL_2(O_{F_w})$ and $\begin{pmatrix} 1 & 0\\0 &\pi_w\end{pmatrix}^{-1}\GL_2(O_{F_w})\begin{pmatrix} 1 & 0\\0 &\varpi_w\end{pmatrix}$, which generate a subgroup of $\GL_2(F_w)$ containing $\SL_2(F_w)$. Hence by strong approximation Theorem, $f_1$ and $f_2$ are both right invariant by $(D\otimes \A_F^\infty)^{\times, \det=1}$. Repeating the same argument, we can show that each $f_X$ is right invariant by $(D\otimes \A_F^\infty)^{\times, \det=1}$. Hence $f_X$ factors through the reduced norm map.
Using this description, it is easy to check that any $\theta_w$ kills the kernel.

\noindent \underline{\textbf{Cokernel of $\eta_N$:}} We will only show that $\tilde{f}_N$ kills the cokernel of $\eta_N$ in the case $|Q_N|=1$. The general case follows by induction on the number of primes in $Q_N$. Write $Q_N=\{w\}$. We want to show that the cokernel of 
\begin{eqnarray*}
S_{\psi,\xi}(U^pU_p,\cO/\varpi^n)_\km^{\oplus 2} \to S_{\psi,\xi}(U^p_{\{w\},0}U_p,\cO/\varpi^n)_\km\\
(f_1,f_2)\mapsto f_1+\begin{pmatrix}1 & 0\\ 0&\varpi_w\end{pmatrix}\cdot f_2
\end{eqnarray*}
is killed by $\theta_w^3$. 

Let $\tilde\psi$ be the Teichm\"uller lifting of $\psi\mod \varpi$ and write $\tilde\psi=\psi\theta^2$ for some continuous character $\theta:(\A^\infty_F)^\times/F^\times_{>>0}\to \cO^\times$ which is trivial on the kernel of $\psi$. By twisting with $\theta$ and arguing as in \ref{exgal}, it suffices to prove the case $\psi=\tilde\psi$ (note that the action of $\theta_w$ only differs by a unit in $\cO^\times$). 

Assume $\psi=\tilde\psi$ from now on. Then $\psi|_{U_p\cap O_{F,p}^\times}$ is trivial. Hence the natural map 
\[S_{\psi,\xi}(U^pU_p,\cO)_\km\to S_{\psi,\xi}(U^pU_p,\cO/\varpi^n)_\km\]
and the similar map for $U^p_{\{w\},0}$ are surjective. Thus it is enough to shows that the cokernel of
\begin{eqnarray*}
j:S_{\psi,\xi}(U^pU_p,\cO)_\km^{\oplus 2} \to S_{\psi,\xi}(U^p_{\{w\},0}U_p,\cO)_\km\\
(f_1,f_2)\mapsto f_1+\begin{pmatrix}1 & 0\\ 0&\varpi_w\end{pmatrix}\cdot f_2
\end{eqnarray*}
is killed by $\theta_w^3$. Let $U$ be either $U^pU_p$ or $U^p_{\{w\},0}U_p$. We may define a perfect pairing: $\langle\cdot,\cdot\rangle_U:S_{\psi,\xi}(U,\cO)_\km\times S_{\psi,\xi}(U,\cO)_\km \to \cO$ by 
\[\langle\varphi_1,\varphi_2\rangle_U:=\sum_{g\in D^\times\setminus \DAi/U(\A_F^\infty)^\times}\varphi_1(g)\varphi_2(g)\psi(N_{D/F}(g))^{-1}\]
for $\varphi_1,\varphi_2\in S_{\psi,\xi}(U,\cO)_\km$. It is easy to check that $\langle T_v\cdot\varphi_1,\varphi_2\rangle_U=\langle \varphi_1,T_v\cdot\varphi_2\rangle_U$ for $v\notin S\cup Q$.

For simplicity, we write $S_0,S_1$ for $S_{\psi,\xi}(U^pU_p,\cO)_\km,S_{\psi,\xi}(U^p_{\{w\},0}U_p,\cO)_\km$. Elements in these spaces can be viewed as automorphic forms on $D^\times$. Let $j^+:S_1\to S_0^{\oplus 2}$ be the adjoint map of $j$ and $\tilde{j^+}$ be the composite of $S_1\stackrel{j^+}{\to} S_0^{\oplus 2}\to S_0^{\oplus 2}/\ker j$. Consider the following diagram:
\[\begin{tikzcd}
0 \arrow[r]  & S_0^{\oplus 2}/\ker j \arrow[d,"\tilde{j^+}\circ j"] \arrow[r,"j"] & S_1 \arrow[r] \arrow[d,"\tilde{j^+}"] & \coker{j} \arrow[r] \arrow[d] & 0\\
0 \arrow[r]  & S_0^{\oplus 2}/\ker j  \arrow[r,equal] & S_0^{\oplus 2}/\ker j   \arrow[r] & 0&.
\end{tikzcd}\]
Then the snake lemma gives  us a short exact sequence:
\[\ker\tilde{j^+}\to \coker j \to \coker(\tilde{j^+}\circ j).\]

Note that the image of $j$ contains exactly the automorphic forms whose corresponding automorphic representation is unramified at $w$. Since the pairings are Hecke-equivariant, the automorphic representations generated by elements in $\ker j^+$ are either Steinberg at $w$ or factor through the reduced norm map. In either case, the associated Galois representation at $w$ is of the form $\begin{pmatrix} \theta\varepsilon & *\\ 0 & \theta \end{pmatrix}$ for some unramified character $\theta$. A direct computation shows that $\theta_w$ kills $\ker j^+$. Hence $\theta_w^2$ kills $\ker \tilde{j^+}$.

We claim that $\theta_w$ kills $\coker(\tilde{j^+}\circ j)$. Combined with the previous paragraph, this will imply $\theta_w^3$ kills $\coker j$. Since $\coker(\tilde{j^+}\circ j)$ is a quotient of $\coker(j^+\circ j)$, it suffices to prove that $\theta_w$ kills $\coker(j^+\circ j)$. A direct computation shows that $j^+\circ j:S_0^{\oplus 2}\to S_0^{\oplus 2}$ is 
\[\begin{pmatrix}
N(w)+1 & T_w\\
T_w & (N(w)+1)\psi(\varpi_w)
\end{pmatrix}.\]
Thus $\coker (j^+\circ j)$ is killed by the determinant of this matrix. Using $T^{univ}(\Frob_w)=T_w$ and $T^{univ}(\Frob_w^2)=T_w^2-2\psi(\pi_w)N(w)$ on $S_0$, it is easy to see that this determinant is simply $(-2N(w))^{-1}\theta_w$. Hence $\theta_w$ kills $\coker (j^+\circ j)$. This finishes the proof of the first part of the proposition. 

We still need to understand the image of $\tilde{f}_N$ in $A$. Let $\pi_N$ be the map $\T_{\psi,\xi}(U^p_{Q_N,0})_\km\to\T_{\psi,\xi}(U^p)_\km\stackrel{\mod\kq}{\longrightarrow} A$ and $\alpha_w,\beta_w$ be the eigenvalues of $\rho(\kq)(\Frob_w)$ with $w\in Q_N$. Then $\pi_N(T^{univ}(\Frob_w^i))=\alpha_w^i+\beta_w^i$. Hence using the assumption $N(w)\equiv 1\mod p$, we have
\[\pi_N(\theta_w)=(1+N(w))^2(\alpha_w^2+\beta_w^2)-(1+N(w)^2)(\alpha_w+\beta_w)^2=2(\alpha_w-\beta_w)^2.\]
Therefore it follows from our choice of $Q_N$ that $\ell(A/(\pi_N(\theta_w)))$ is uniformly bounded for any $N,w$. This proves the second part of the proposition.
\end{proof}

\begin{para}
The usual patching argument requires a Galois-theoretic interpretation of the action of $\cO[\Delta_{Q_N}]$ on $M_{\psi,\xi}(U^p_{Q_N})_{\km}$. We will only do it for a subring of $\cO[\Delta_{Q_N}]$. More precisely, let $\cO[\Delta_{Q_N}]'\subseteq \cO[\Delta_{Q_N}]$ be the $\cO$-subalgebra generated by elements of the form $g+g^{-1}$ with $g\in\Delta_{v},v\in Q_N$.  We may think $\cO[\Delta_{Q_N}]'$ as a subring of $\End (M_{\psi,\xi}(U^p_{Q_N})_{\km})$.
\end{para}

\begin{prop} \label{sinfty'}
$\cO[\Delta_{Q_N}]'$ is contained in the image of 
\[R^{\ps,\{\xi_v\}}_{Q_N}\to \T_{\psi,\xi}(U^p_{Q_N})\to \End (M_{\psi,\xi}(U^p_{Q_N})_{\km}).\]
Moreover let $\ka_{Q_N}'=\ka_{Q_N}\cap \cO[\Delta_{Q_N}]'$. Then $\ka_{Q_N}'$ is contained in the image of $\ker(R^{\ps,\{\xi_v\}}_{Q_N}\to R^{\ps,\{\xi_v\}})$ in $\End (M_{\psi,\xi}(U^p_{Q_N})_{\km})$.
\end{prop}

\begin{proof}
Consider the natural map $I_{F_v}\to k(v)^\times\to\Delta_v$ given by the local class field theory. For any $\sigma_0\in \Delta_v$, we choose a lifting $\sigma\in I_{F_v}$. It suffices to show that $T^{univ}(\sigma)$ acts as $\sigma_0+\sigma_0^{-1}$ on $M_{\psi,\xi}(U^p_{Q_N})_{\km}$. By definition, we only need to check this for the action of $T^{univ}(\sigma)$ on $S_{\psi,\xi}(U^p_{Q_N}U_p,\cO/\varpi^n)_\km$ where $U_p$ is an open subgroup of $K_p$ such that $\psi|_{U_p\cap O_{F,p}^\times}\mod \varpi^n$ is trivial and $U^p_{Q_N}U_p$ is sufficiently small.

Since the formulation commutes with twisting with a character unramified at $v$, we may argue as in the proof of proposition \ref{twkcc} and assume $\psi$ is equal to the Teichm\"uller lifting of its mod $\varpi$ reduction. Now it suffices to show that the action of $T^{univ}(\sigma)$ on $S_{\psi,\xi}(U^p_{Q_N}U_p,\cO)_\km$ is equal to $\sigma_0+\sigma_0^{-1}$. But this is a consequence of the local-global compatibility result at $v$: if $f\in S_{\psi,\xi}(U^p_{Q_N}U_p,\cO)_\km$ generates an automorphic representation, then the local representation at $v$ is either Steinberg or a principal series since it has $U^p_{Q_N,v}$-fixed vectors (see for example Proposition 14.3 of \cite{BH06}). In both cases, the desired result is clear.
\end{proof}

\subsection{Patching I: patched completed homology} \label{p1pch}
\begin{para}
We summarize what we have done so far. For any $N$, we have 
\begin{itemize}
\item a finite set of primes $Q_N$ with cardinality $r$ given by proposition \ref{exttwp}.
\item $R^{\ps,\{\xi_v\}}_{Q_N}\twoheadrightarrow \T_{\psi,\xi}(U^p_{Q_N})_{\km} \to \End_{\kC_{D_p^\times,\psi}(\cO)}(M_{\psi,\xi}(U^p_{Q_N})_{\km})$, where $D_p^\times=\prod_{v|p}\GL_2(F_v)$.
\item $\Delta_{Q_N}=\prod_{v\in Q_N}\Delta_v$ and $\Delta_v$ is the unique quotient of $k(v)^\times$ of order $p^N$.
\item $\cO[\Delta_{Q_N}]\hookrightarrow \End_{\kC_{D_p^\times,\psi}(\cO)}(M_{\psi,\xi}(U^p_{Q_N})_{\km})$ that makes $M_{\psi,\xi}(U^p_{Q_N})_{\km}$ into a flat $\cO[\Delta_{Q_N}]$-module. Also the image of $\cO[\Delta_{Q_N}']$ is contained in the image of  $R^{\ps,\{\xi_v\}}_{Q_N}$.
\item Elements $\tilde{f}_N\in R^{\ps,\{\xi_v\}}_{Q_N}$ and natural maps $\eta_N$ (see proposition \ref{twkcc}).
\end{itemize}

We also set: 
\begin{itemize}
\item $\Delta_\infty:=\Z_p^{\oplus r}$. Fix surjective maps $\Z_p\to \Delta_v$ for all $v\in Q_N$ and thus surjective maps $\Delta_\infty\to \Delta_{Q_N}$. 
\item $\cO_\infty=\cO[[y_1,\cdots,y_{4|P|-1}]]$ with maximal ideal $\kb$ and prime ideal $\kb_1=(y_1,\cdots,y_{4|P|-1})$.
\item $S_\infty=\cO_\infty[[\Delta_\infty]]\cong \cO_{\infty}[[s_1,\cdots,s_r]]$. This is a local $\cO_\infty$-algebra with maximal ideal $\ka$. 
\item $\ka_0=\ker(S_\infty\to\cO_\infty)=(s_1,\cdots,s_r)$ the augmentation ideal, and $\ka_1=(\ka_0,\kb_1)=(y_1,\cdots,y_{4|P|-1},s_1,\cdots,s_r)$. Hence $\ka=(\kb)+\ka_0$ and $\ka_1=(\kb_1)+\ka_0$.
\item $S_\infty'\subseteq\cO_\infty[[\Delta_\infty]]$ is the closure (under the profinite topology) of the $\cO_\infty$-subalgebra generated by all elements of the form $g+g^{-1}$ with $g=(0,\cdots,0,a,0,\cdots,0)\in\Delta_\infty$ for some $a\in\Z_p$. This is a regular local $\cO_\infty$-algebra and $S_\infty$ is a finite free $S_\infty'$-algebra.
\item $\ka_0'=\ka_0\cap S_\infty',\ka_1'=\ka_1\cap S_\infty'$. We may find $r$ elements $s_1',\cdots,s_r'$ that generate $\ka_0'$ and $S_\infty'\cong\cO_\infty[[s_1',\cdots,s_r']]$.
\end{itemize} 
\end{para}

\begin{para}
Following \cite{Sch15}, we will use the language of ultrafilters to define patched completed homology (see also \cite{GN16}). This language seems to be essential here. For example, I don't know how to rewrite the arguments below in the classical language (as in \cite{SW99}).

Let $\cI$ be the set of positive integers  and $\fR=\prod_\cI \cO$. From now on we fix a non-principal ultrafilter $\kF$ on $\cI$. Then $\kF$ gives rise to a multiplicative set $S_\kF\subseteq \fR$ consisting of all idempotents $e_{I}$ with $I\in\kF$ where $e_I(i)=1$ if $i\in I$, $e_I(i)=0$ otherwise. 

We define $\fR_\kF=S_\kF^{-1}\fR$. This is a quotient of $\fR$ as all the elements in $S_\kF$ are idempotents. Taking tensor product with $\fR_\kF$ over $\fR$ is an exact functor. Since $\kF$ is non-principal, for any finite set $T\subseteq \cI$, we have $\fR_T\otimes_\fR \fR_\kF\cong \fR_\kF$, where $\fR_T$ is the quotient of $\fR$ by elements of the form $(a_i)_{i\in \cI}$ with $a_i=0$ for $i\notin T$.

The following lemma is easy but extremely useful.
\end{para}

\begin{lem} \label{uflem}
Suppose for any $i\in\cI$, $M_i$ is an $\cO$-module with decreasing filtrations of $\cO$-modules $M_i\supseteq M_{i,1}\supseteq M_{i,2}\supseteq\cdots$. Then the natural map
\[\prod_{i\in\cI} M_i\to \varprojlim_n((\prod_{i\in\cI}M_i/M_{i,n})\otimes_\fR \fR_\kF)\]
is surjective. The kernel contains all elements of the form $(a_i)_{i\in \cI}$ such that for any $n$, there exists $I_n\in\kF$ with $a_i\in M_{i,n}$ for any $i\in I_n$.
\end{lem}

\begin{proof}
The description of the kernel follows from the definition of $R_\kF$. To prove the surjectivity, let $\varprojlim_n [(a_{i,n})_i]$ with $a_{i,n}\in M_i/M_{i,n}$ be an element in the projective limit. This means that for any $n$, there exists $I_n\in\kF$ such that $a_{i,n}\equiv a_{i,n+1}\mod M_{i,n}$ for any $i\in I_n$. We may assume $I_n$ contains $I_{n+1}$ for any $n$ and the intersection of all $I_n$ is empty. Now for any $i\in I_n\setminus I_{n+1}$, let $a_i\in M_i$ be a lifting of $a_{i,n+1}$.  Set $a_i=0$ for $i\notin I_1$. It is easy to see that $(a_i)_{i\in\cI}$ maps to $\varprojlim_n [(a_{i,n})_i]$. This finishes the proof of the lemma.
\end{proof}

Note that since $\cO/(\varpi^n)$ has finite cardinality, 
\begin{eqnarray} \label{ufi}
\fR_\kF/(\varpi^n)=(\prod_\cI \cO/\varpi^n)\otimes_\fR \fR_\kF\cong \cO/(\varpi^n).
\end{eqnarray}
As a special case of the previous lemma, we have
\begin{cor} \label{speufi}
The inverse limit of \eqref{ufi} gives a surjective map $\prod_{\cI} \cO\to \cO$. The kernel contains all $(a_i)_{i\in \cI}$ such that for any $n$, there exists $I_n\in\kF$ with $a_i\in (\varpi^n)$ for any $i\in I_n$.
\end{cor}

\begin{para}
We now define patched completed homology. Recall that $S_{\psi,\xi}(U^p_{Q_N}U_p,\cO/\varpi^n)_\km$ is an $\cO[\Delta_{Q_N}]$-module. Hence $S_{\psi,\xi}(U^p_{Q_N}U_p,\cO/\varpi^n)_\km^\vee\otimes_\cO \cO_\infty$ is a natural $S_\infty$-module. For simplicity, we denote it by $M(U_p,N,n)$. Then $\prod_{N\in\cI}M(U_p,N,n)$ has a natural $\fR$-module structure. 
\end{para}

\begin{defn}[Patched completed homology] \label{Pch}
For any integer $n>0$, we define
\begin{itemize}
\item $M^{\patch,\{\xi_v\}}_n:=\varprojlim_{U_p}(\prod_{N\in\cI} M(U_p,N,n)/\ka^n M(U_p,N,n)\otimes_\fR \fR_\kF)$.
\item $M^{\patch,\{\xi_v\}}_{n,0}:=\varprojlim_{U_p}(\prod_{N\in\cI} (S_{\psi,\xi}(U^p_{Q_N,0}U_p,\cO/\varpi^n)_\km^\vee\otimes_\cO \cO_\infty/\kb^n) \otimes_\fR \fR_\kF)$.
\item $M^{\{\xi_v\}}_{n}:=\varprojlim_{U_p}(\prod_{\cI} (S_{\psi,\xi}(U^pU_p,\cO/\varpi^n)_\km^\vee\otimes_\cO \cO_\infty/\kb^n) \otimes_\fR \fR_\kF)$.
\end{itemize}
Here $U_p$ runs through all open compact subgroups of $D_p^\times=\prod_{v|p}\GL_2(F_v)$ in all projective limits above. We note that $M^{\{\xi_v\}}_{n}$ is nothing but 
\[\varprojlim_{U_p}(S_{\psi,\xi}(U^pU_p,\cO/\varpi^n)_\km^\vee\otimes_\cO \cO_\infty/\kb^n)=M_{\psi,\xi}(U^p)_\km \otimes_\cO \cO_\infty/\kb^n\]
since $S_{\psi,\xi}(U^pU_p,\cO/\varpi^n)_\km^\vee\otimes_\cO \cO_\infty/\kb^n$ has finite cardinality.
\end{defn}

\begin{para}
In the definition of $M^{\patch,\{\xi_v\}}_n,M^{\patch,\{\xi_v\}}_{n,0}$, we can replace the limits taken over all open compact subgroups of $D_p^\times$ by over all open subgroups of $K_p$. From this description, it is clear that both patched completed homology are natural $\cO[[K_p]]$-modules. The action of $K_p$ can be extended to $D_p^\times=\prod_{v|p}\GL_2(F_v)$ in the usual way: the action of $g\in D_p^\times$ induces isomorphisms $M(U_p,N,n)\stackrel{\sim}{\longrightarrow}M(g^{-1}U_pg,N,n)$.

We prove some simple properties of $M^{\patch,\{\xi_v\}}_n$. Note that the diagonal action of  $S_\infty$ on $\prod_{N\in\cI} M(U_p,N,n)/\ka^n M(U_p,N,n)$ defines a natural $S_\infty$-module structure on $M^{\patch,\{\xi_v\}}_n$.
\end{para}

\begin{prop} \label{mnflat}
$M^{\patch,\{\xi_v\}}_n$ is a flat $S_\infty/\ka^n$-module. Moreover the natural maps $M^{\patch,\{\xi_v\}}_n$ to $M^{\patch,\{\xi_v\}}_{n-1}$ and $M^{\patch,\{\xi_v\}}_{n,0}$ induce isomorphisms
\begin{eqnarray*}
M^{\patch,\{\xi_v\}}_n/\ka^{n-1}M^{\patch,\{\xi_v\}}_n\cong M^{\patch,\{\xi_v\}}_{n-1}.\\
M^{\patch,\{\xi_v\}}_n/\ka_0 M^{\patch,\{\xi_v\}}_n\cong M^{\patch,\{\xi_v\}}_{n,0}.
\end{eqnarray*}
\end{prop}
\begin{proof}
For $U_p$ small enough and $N>n$,  $M(U_p,N,n)/\ka^n M(U_p,N,n)$ is a flat $S_\infty/\ka^n$-module by lemma \ref{flmf}. Hence $\prod_{N>n} M(U_p,N,n)/\ka^n M(U_p,N,n)$ is a flat $S_\infty/\ka^n$-module. Since $\otimes \fR_\kF$ is an exact functor and $S_\infty/\ka^n$ is Artinian, applying lemma \ref{fllem}, we see that $M^{\patch,\{\xi_v\}}_n$ is a flat $S_\infty/\ka^n$-module. 

Note that by lemma \ref{flmf} again,  we have
\begin{eqnarray*}
M(U_p,N,n)/\ka^{n-1} M(U_p,N,n)\cong M(U_p,N,n-1)/\ka^{n-1} M(U_p,N,n-1).\\
M(U_p,N,n)/(\ka^n+\ka_0) M(U_p,N,n)\cong S_{\psi,\xi}(U^p_{Q_N,0}U_p,\cO/\varpi^n)_\km^\vee\otimes_\cO \cO_\infty/\kb^n.
\end{eqnarray*}
Both isomorphisms in the proposition now follow from the first part of lemma \ref{fllem} and the following easy lemma which we omit the proof here.
\end{proof}

\begin{lem}
Let $R$ be a Noetherian ring and $M=\prod_{i\in \cI} M_i$ be a product of $R$-modules $M_i$. Then $M/IM\cong \prod_{i\in \cI}M_i/IM_i$ for any ideal $I$ of $R$.  
\end{lem}

\begin{para} \label{etanpa}
Fix isomorphisms between $P(Q_N)$ and $\{1,\cdots,2^r\}$ for any $N$ from now on. Then the map $\eta_N$ defined in \ref{tdm} induces a map:
\[\eta^{\patch}_n: M^{\patch,\{\xi_v\}}_{n,0}\to (M^{\{\xi_v\}}_n)^{\oplus 2^r},\]
which commutes with the action of $D_p^\times$ and $\cO[[K_p]]$. 

It is clear that $\prod_N R^{\ps,\{\xi_v\}}_{Q_N}$ acts naturally on $M^{\patch,\{\xi_v\}}_{n}$ and $M^{\patch,\{\xi_v\}}_{n,0}$. Let $\tilde{f}=\prod_N\tilde{f}_N\in\prod_N R^{\ps,\{\xi_v\}}_{Q_N}$ where $\tilde{f}_N$ are defined in proposition \ref{twkcc}. Here all the products are taken over all $N\in \cI$ and we will keep this convention from now on in this section. The following result is a direct consequence of proposition \ref{twkcc} and lemma \ref{limf} below.
\end{para}

\begin{prop} \label{fkkc}
$\tilde{f}^2$ kills the kernel and cokernel of $\eta^{\patch}_n$.
\end{prop}

\begin{lem}\label{limf}
Let $R$ be a Noetherian ring and $\{M_i\}_{i\in \Z_{>0}}\to \{N_i\}_{i\in \Z_{>0}}$ be a map between two projective systems of $R$-modules. Suppose there exists an element $f$ that kills the kernel and cokernel of $M_i\to N_i$ for any $i$. Then $f^2$ kills the kernel and cokernel of $\varprojlim M_i\to \varprojlim N_i$.
\end{lem}
\begin{proof}
Suppose $(a_i)_i\in \varprojlim N_i$. There exists $b_i\in M_i$ mapping to $fa_i\in N_i$ for any $i$. Then $(fb_i)_i$ defines an element in $\varprojlim M_i$ and maps to $f^2(a_i)_i$. The claim for the kernel is clear.
\end{proof}

\begin{cor} \label{finpath}
Let $\tilde{M}^{\patch,\{\xi_v\}}_n$ be the image of $M^{\patch,\{\xi_v\}}_n$ into $(M^{\patch,\{\xi_v\}}_n)_{\tilde{f}}$, the localization of $M^{\patch,\{\xi_v\}}_n$ by the powers of $\tilde{f}$. Then $\tilde{M}^{\patch,\{\xi_v\}}_n$ is a finitely generated module over $\cO[[K_p]]$. In other words, the Pontryagin dual of $\tilde{M}^{\patch,\{\xi_v\}}_n$ is an admissible representation of $D_p^\times$. In particular, $\tilde{M}^{\patch,\{\xi_v\}}_n$ is an object in $\kC_{D_p^\times,\psi}(\cO)$.
\end{cor}

\begin{proof}
For simplicity, we write $X$ for $M^{\patch,\{\xi_v\}}_n$. We define a decreasing filtration on $X$ with 
\begin{eqnarray*}
\Fil^k X &=& X, ~k\leq 0\\
\Fil^k X &=& \{x\in X,\tilde{f}^{2k}x\in \ka^k X\}\cap \Fil^{k-1}X,~k>0 .
\end{eqnarray*}
Since $\ka^n X=0$, we have $\Fil^k X\subseteq \Fil^n X\subseteq X[\tilde{f}^{2n}],k\ge n$. Hence $\tilde{f}^{2n}\Fil^k X=0,k\geq n$. We claim that  $\tilde{f}^{2n}\Fil^k X/\tilde{f}^{2n}\Fil^{k+1}X$ is a finite $\cO[[K_p]]$-module for any $k$. Note that this will imply $\tilde{f}^{2n}X$ is a finite $\cO[[K_p]]$-module and hence prove the corollary. In fact, we are going to prove a stronger result:
\[\tilde{f}^{2k+2}\Fil^k X/\tilde{f}^{2k+2}\Fil^{k+1} X\]
is finite over $\cO[[K_p]]$ for any $k\le n-1$.

Consider the map:
\begin{eqnarray*}
\varphi_k:\tilde{f}^{2k+2}\Fil^k X&\to& \tilde{f}^{2}(\ka^k X/\ka^{k+1}X)\subseteq \ka^k X/\ka^{k+1}X \\
\tilde{f}^{2k+2}x &\mapsto & \tilde{f}^2 (\tilde{f}^{2k}x).
\end{eqnarray*}
It is easy to see that $\ker (\varphi_k)=\tilde{f}^{2k+2}\Fil^{k+1}X$. Hence we get an injective map:
\[\tilde{f}^{2k+2}\Fil^k X/\tilde{f}^{2k+2}\Fil^{k+1} X\hookrightarrow \tilde{f}^2(\ka^k X/\ka^{k+1}X).\]
But $X$ is flat over $S_\infty/\ka^n$. Thus 
\[\ka^k X/\ka^{k+1}X\cong \ka^k/\ka^{k+1}\otimes X/\ka X\cong \ka^k/\ka^{k+1}\otimes M^{\patch,\{\xi_v\}}_{1,0}.\]
By the previous proposition, the kernel of $\eta_1^{\patch}:M^{\patch,\{\xi_v\}}_{1,0}\to M^{\{\xi_v\}}_1$ is killed by $\tilde{f}^2$. The result follows from the fact that $M^{\{\xi_v\}}_1\cong M_{\psi,\xi}(U^p)_\km/\varpi M_{\psi,\xi}(U^p)_\km$ is finite over $\cO[[K_p]]$.
\end{proof}

We note that the same argument shows that 
\begin{lem} \label{f2nmNn}
For any integer $n>0$ and open subgroup $U_p$ of $D_p^\times$ small enough so that $U^pU_p$ is sufficiently small, the cardinality of 
\[\tilde{f_N}^{2n}(M(U_p,N,n)/\ka^n M(U_p,N,n))\]
has a uniform upper bound which is independent of $N$.
\end{lem}

\subsection{Patching II: patched deformation rings}
\begin{para}
In this subsection, we are going to define a patched (local) deformation ring acting on our patched completed homology. First we consider the action of $R^{\ps,\{\xi_v\}}_{Q_N}$. Denote $R^{\ps,\{\xi_v\}}_{Q_N}\hat{\otimes}_\cO \cO_{\infty}$ by $R^{\ps}_N$ and its maximal ideal by $\km^{\ps}_{N}$. Note that $R^{\ps}_N$ acts on $M(U_p,N,n)$ $\cO_{\infty}$-linearly. We will freely use the notation defined in the previous subsection.
\end{para}

\begin{lem}
For any integer $n>0$ and open subgroup $U_p\subseteq K_p$ small enough, there exists a constant $C=C_{n,U_p}$ (independent of $N$) such that $(\km^{\ps}_{N}\tilde{f}_N)^C$ annihilates 
\[M(U_p,N,n)/\ka^n M(U_p,N,n)\]
for any $N$.
\end{lem}

\begin{proof}
This is an immediate consequence of lemma \ref{f2nmNn}.
\end{proof}

\begin{defn} \label{defnrpsp}
\[R^{\ps,\patch}:=\varprojlim_n(\prod_{N\in\cI}(R^{\ps}_N/(\km^{\ps}_N\tilde{f}_N)^n)\otimes_{\fR}\fR_{\kF}).\]
It follows from the previous lemma that $R^{\ps,\patch}$ acts on $M^{\patch,\{\xi_v\}}_n$.
\end{defn}

\begin{para} \label{rpsps}
There is a natural surjective map $R^{\ps,\patch}\to R^{\ps,\square}:=R^{\ps,\{\xi_v\}}\hat{\otimes} \cO_\infty$ by taking the limits of
\[\prod_{N\in\cI}(R^{\ps}_N/(\km^{\ps}_N\tilde{f}_N)^n)\otimes_{\fR}\fR_{\kF}\to \prod_{N\in\cI}(R^{\ps,\square}/(\km^{\ps,\square})^n)\otimes_{\fR}\fR_{\kF}=R^{\ps,\square}/(\km^{\ps,\square})^n.\]
Here $\km^{\ps,\square}$ denotes the maximal ideal of $R^{\ps,\square}$. Denote the prime ideal $(\kq^{\ps},\kb_1)$ of $R^{\ps,\square}$ by $\kq^{\ps,\square}$ and its pull-back to $R^{\ps}_N$ by $\kq^{\ps}_N$. We have
\[B'=R^{\ps,\{\xi_v\}}/\kq^{\ps}=R^{\ps,\square}/\kq^{\ps,\square}=R^{\ps}_N/\kq^{\ps}_N.\]
Recall that the integral closure of $B'$ in its fraction field is $A=\F[[T]]$ (see \ref{kqrho}).

Consider the natural map $\prod_{N\in\cI}R^{\ps}_N\to R^{\ps,\patch}$, which is \textit{surjective} by lemma \ref{uflem}.
\end{para}

\begin{lem} [Definition of $\kq^{\ps,\patch}$] \label{defqps}
The image of $\prod \kq^{\ps}_N\subseteq \prod R^{\ps}_N$ in $R^{\ps,\patch}$ defines a prime ideal $\kq^{\ps,\patch}$, which is the pull-back of $\kq^{\ps,\square}$ via the natural map $R^{\ps,\patch}\to R^{\ps,\square}$.
\end{lem}

\begin{proof}
Let $I^\ps$ be the kernel of $\prod_{N\in\cI}R^{\ps}_N\to R^{\ps,\patch}$. It suffices to prove the composite $\prod R^{\ps}_N\to R^{\ps,\patch}\to R^{\ps,\square}$ induces an isomorphism:
\[(\prod R^{\ps}_N) /(I^{\ps},\prod \kq^{\ps}_N)\stackrel{\sim}{\longrightarrow} R^{\ps,\square}/\kq^{\ps,\square}=R^{\ps,\{\xi_v\}}/\kq^{\ps}=B'.\]

By lemma \ref{uflem}, $I^{\ps}$ is the set of elements $(x_N)_N\in \prod R^{\ps}_N$ such that for any $n>0$, there exists an element $I_n\in \kF$ with $x_N\in (\km^{\ps}_N\tilde{f}_N)^n$ for any $N\in I_n$. Note that the image of $\tilde{f}_N$ in $R^{\ps}_N/\kq^{\ps}_N=B'\subseteq A$ has a uniform bound on its $T$-adic valuation. It is easy to see that the image $I_{B'}$ of $I^{\ps}$ in $R^{\ps}_N/\prod \kq^{\ps}_N=\prod_N B'$ consists of elements $\{b_N\}_N$ such that for any $n>0$, there exists an element $I_n\in \kF$ with $b_N\in (\km_{B'})^n$ for any $N\in I_n$, where $\km_{B'}$ is the maximal ideal of $B'$. By lemma \ref{uflem} again (see also corollary \ref{speufi}), $(\prod_N B')/I_{B'}$ is nothing but $B'$. This finishes the proof.
\end{proof}
 
\begin{para}
In \ref{grdr}, we defined several global universal lifting rings. See the notations there and also in \ref{eotwp}. For each $N$, we fix isomorphisms
\[R^{\square_P,\{\xi_v\}}_{\bar{\rho}_b,Q_N}\cong R^{\{\xi_v\}}_{\bar{\rho}_b,Q_N}\hat{\otimes}_\cO \cO_\infty\]
such that $\kb_1\subseteq \cO_\infty$ is contained in $\kq_{b,Q_N}$ for each $i$. The natural map $R^{\ps,\{\xi_v\}}_{Q_N}\to R^{\{\xi_v\}}_{\bar{\rho}_b,Q_N}$ given by evaluating the universal trace can be extended naturally to an $\cO_\infty$-algebra homomorphism $R^{\ps}_N=R^{\ps,\{\xi_v\}}_{Q_N}\hat{\otimes}_\cO \cO_{\infty} \to R^{\square_P,\{\xi_v\}}_{\bar{\rho}_b,Q_N}$. For simplicity, we write $R_{b,N}$ for $R^{\square_P,\{\xi_v\}}_{\bar{\rho}_b,Q_N}$.
\end{para}

\begin{defn}
\[R^{\patch}_b:=(\prod_{N\in \cI}R_{b,N})\otimes_{\prod_{N\in \cI}R^{\ps}_N} R^{\ps,\patch}.\]
This is a quotient of $\prod_{N\in \cI}R_{b,N}$.
\end{defn}

\begin{lem}[Definition of $\kq^{\patch}_b$] \label{defqc}
The image of $\prod_N \kq_{b,Q_N}\subseteq \prod_{N\in \cI}R_{b,N}$ in $R^{\patch}_b$ defines a prime ideal $\kq^{\patch}_b$ of $R^{\patch}_b$. Its pull-back to $R^{\ps,\patch}$ is $\kq^{\ps,\patch}$. Moreover $R^{\patch}_b/\kq^{\patch}_b$ is naturally isomorphic to $B=R_{b,N}/\kq_{b,Q_N}$ (defined in the beginning of subsection \ref{eotwp}).
\end{lem}

\begin{proof}
We use the notations in the proof of the previous lemma. Then
\[R^{\patch}_b/\kq^{\patch}_b=(\prod_{N\in \cI}R_{b,N})/(I^{\ps},\prod_N \kq_{b,Q_N})=(\prod_N B)/(I_{B'})\cong B.\]
The last isomorphism follows from the explicit description of $I_{B'}$ and the fact $B$ is a finite $B'$-algebra.
\end{proof}

\begin{defn}
For any positive integer $n$, we let $\kq^{\ps,[n]}$ be the image of $\prod_N (\kq^{\ps}_N)^n\subseteq \prod_N R^{\ps}_N$ in $R^{\ps,\patch}$. Similarly, we define $\kq^{[n]}_b$ as the image of $\prod_N \kq_{b,Q_N}^n\subseteq \prod_N R_{b,N}$ in $R^{\patch}_b$. 
\end{defn}

\begin{rem}
By definition, we have 
\begin{itemize}
\item $\kq^{\ps,[1]}=\kq^{\ps,\patch}$, $\kq^{[1]}_b=\kq^{\patch}_b$.
\item $(\kq^{\ps,\patch})^n\subseteq \kq^{\ps,[n]}$. 
\end{itemize}
It is likely true that the inclusion above is an equality \footnote{This is actually proved in a slightly different context in a recent paper \cite[Lemma 4.21]{NT20} by Newton-Thorne. I believe their argument  works in our setting.}. However we don't quite need this. For our purpose, the following lemma is enough.
\end{rem}

\begin{lem} \label{q[n]qn}
For any positive integers $n,j$ and open subgroup $U_p$ of $D_p^\times$, the images of $\prod_N (\kq^{\ps}_N)^j$ and $(\prod_N \kq^{\ps}_N)^j$ in the endomorphism ring of 
\[\prod_{N\in\cI} \tilde{f_N}^{2n}(M(U_p,N,n)/\ka^n M(U_p,N,n))\otimes_\fR \fR_\kF=\tilde{f}^{2n}(\prod_{N\in\cI} M(U_p,N,n)/\ka^n M(U_p,N,n)\otimes_\fR \fR_\kF)\]
are the same.
\end{lem}
\begin{proof}
By lemma \ref{f2nmNn}, there is a uniform upper bound $C$ on the cardinality of
\[\tilde{f_N}^{2n}(M(U_p,N,n)/\ka^n M(U_p,N,n))\] 
for all $N$. The lemma is clear as there are only finitely many isomorphism classes of triples $(M,R,I)$ where $M$ is a finite group of cardinality at most $C$, $R$ is a commutative subring of the endomorphism ring of $M$ and $I\subset R$ is an ideal.
\end{proof}

\begin{para}
Recall that in proposition \ref{rpsrc} and corollary \ref{pscrem}, we proved that there exist elements $c_N\in R^{\ps}_N$ for $N>0$ such that
\begin{itemize}
\item For any $n$, we may find a constant $k_n$ (which is $N$ in proposition \ref{rpsrc}) independent of $N$  such that $c_N^{k_n}$ kills the kernel and cokernel of the map $R^{\ps}_N/(\kq^{\ps}_N)^n\to R_{b,N}/(\kq_{b,Q_N})^n$. 
\item The image $c'$ of $c_N$ in $R^{\ps}_N/\kq^{\ps}_N=B'\hookrightarrow A$ is non-zero and independent of $N$.
\end{itemize}
To be more precise, we take $\Gamma=G_{F,S\cup Q_N},\Gamma_0=G_{F,S},\rho_0=\rho(\kq)^o$ in the setup of proposition \ref{rpsrc} and $I$ be the kernel of $R^{\ps}\to R^{\ps,\{\xi_v\}}$ in corollary \ref{pscrem}.

We define $\tilde{c}$ to be image of $(c_N)_N\in \prod_N R^{\ps}_N$ in $R^{\ps,\patch}$. It is easy to see that $\tilde{c}\mod \kq^{\ps,\patch}$ is equal to $c'\neq 0$. A direct corollary of the above discussion is 
\end{para}

\begin{cor} \label{pscpatch}
\hspace{2em}
\begin{enumerate}
\item $\tilde{c}\notin \kq^{\ps,\patch}$.
\item For any $n$, there exists an integer $k_n$ such that $\tilde{c}^{k_n}$ kills the kernel and cokernel of the natural map $R^{\ps,\patch}/\kq^{\ps,[n]} \to R^{\patch}_{b}/\kq^{[n]}_b$.
\item Let $\widehat{(R^{\ps,\patch})_{[\kq^{\ps,\patch}]}}$ (resp. $\widehat{(R^{\patch}_{b})_{[\kq^{\patch}_b]}}$) denote the completion of the localization $(R^{\ps,\patch})_{\kq^{\ps,\patch}}$ (resp. $(R^{\patch}_{b})_{\kq^{\patch}_b}$) with respect to the topology with a system of open neighborhoods of zero given by $(\kq^{\ps,[n]})_{n>0}$ (resp. $(\kq^{[n]}_b)_{n>0}$). Then
\[\widehat{(R^{\ps,\patch})_{[\kq^{\ps,\patch}]}}\cong \widehat{(R^{\patch}_{b})_{[\kq^{\patch}_b]}}.\] 
\end{enumerate}
\end{cor}

\begin{proof}
The first part is clear and the last part follows directly from the first two parts. As for the second part, consider the natural map
\[\prod_N R^{\ps}_N/(\kq^{\ps}_N)^n\to \prod_N R_{b,N}/(\kq_{b,Q_N})^n,\]
whose kernel and cokernel are killed by the $k_n$-th power of $(c_N)_N$. We obtain the desired results by taking tensor product of this map with $R^{\ps,\patch}$ over $\prod_N R^{\ps}_N$.
\end{proof}

\begin{para}
We still need to relate these patched (global) deformation rings with some local deformation ring. By the construction in proposition \ref{exttwp}, for any $N$, we have a map
\[R^{\{\xi_v\}}_{\loc}[[x_1,\cdots,x_g]]\to R^{\square_P,\{\xi_v\}}_{\bar{\rho}_b,Q_N}=R_{b,N}\]
sending $x_i$ to $\kq_{b,Q_N}$ and $\kq_{b,Q_N}/(\kq_{b,Q_N}^2,\kq^{\{\xi_v\}}_{\loc},x_1,\cdots,x_g)$ is killed by some element $f'\in B$ with $\ell(A/(f'))<C$. It is clear that we can take one $f'\neq 0$ which works for all $N$. It is easy to see that $(f')^n$ kills $\kq_{b,Q_N}^n/(\kq_{b,Q_N}^{n+1},(\kq^{\{\xi_v\}}_{\loc},x_1,\cdots,x_g)^n)$.
\end{para}

\begin{defn}
We define 
\[ R^{\{\xi_v\}}_{\infty}:= R^{\{\xi_v\}}_{\loc}[[x_1,\cdots,x_g]]\]
and $\kq^{\{\xi_v\}}_{\infty}$ to be its prime ideal generated by $\kq^{\{\xi_v\}}_{\loc},x_1,\cdots,x_g$. 
\end{defn}

\begin{prop}
For any $n>0$, the diagonal map $R^{\{\xi_v\}}_\infty \to \prod_{N\in\cI} R_{b,N}$ induces a natural map 
\[R^{\{\xi_v\}}_\infty \to R^{\patch}_b\]
which sends $(\kq^{\{\xi_v\}}_{\infty})^n$ into $\kq^{[n]}_b$. Moreover, the $B$(=$R^{\patch}_b/\kq^{\patch}_b$)-module $\kq^{[n]}_b/(\kq^{[n+1]}_b,(\kq^{\{\xi_v\}}_{\infty})^n)$ is killed by $(f')^n\neq 0$.
\end{prop}

\begin{proof}
The first claim is clear. We will only show that $\kq^{[n]}_b/(\kq^{[n+1]}_b,(\kq^{\{\xi_v\}}_{\infty})^n)$ is killed by $(f')^n$ when $n=1$. The general case is the same. Since the cokernel of the composite map
\[\prod_N (\kq^{\{\xi_v\}}_{\infty}/(\kq^{\{\xi_v\}}_{\infty})^2)\to \prod_N (\kq_{b,Q_N}/(\kq_{b,Q_N})^2) \to \kq^{\patch}_b/\kq^{[2]}_b\]
is killed by $f'$, it suffices to show that the image of $\kq^{\{\xi_v\}}_{\infty}/(\kq^{\{\xi_v\}}_{\infty})^2\to \kq^{\patch}_b/\kq^{[2]}_b$ is equal to the image of $\prod_N (\kq^{\{\xi_v\}}_{\infty}/(\kq^{\{\xi_v\}}_{\infty})^2)\to\kq^{\patch}_b/\kq^{[2]}_b$. Note this latter map is a homomorphism between $\prod_N R^{\{\xi_v\}}_\infty/\kq^{\{\xi_v\}}_{\infty}=\prod_N B$-modules, and the action of $\prod_N B$ on the target factors through $R^{\patch}_b/\kq^{\patch}_b=B$. Thus it is enough to show that the composite map
\begin{eqnarray} \label{diasur}
\kq^{\{\xi_v\}}_{\infty}/(\kq^{\{\xi_v\}}_{\infty})^2\to \prod_N (\kq^{\{\xi_v\}}_{\infty}/(\kq^{\{\xi_v\}}_{\infty})^2) \to \prod_N (\kq^{\{\xi_v\}}_{\infty}/(\kq^{\{\xi_v\}}_{\infty})^2)\otimes_{\prod_N B} B
\end{eqnarray}
is surjective. Here the first map is given by the diagonal embedding and  $B$ is viewed as a $\prod B$-algebra by identifying them with $R^{\patch}_b/\kq^{\patch}_b$ and $\prod_N R^{\{\xi_v\}}_\infty/\kq^{\{\xi_v\}}_{\infty}$. Explicitly, the map $\prod_N B \to B$ is the one considered in lemma \ref{uflem} with $M_i=B$ and $M_{i,n}=\km_B^n$ where $\km_B$ is the maximal ideal of $B$. Note that the composite map \eqref{diasur} is clearly an isomorphism if $\kq^{\{\xi_v\}}_{\infty}/(\kq^{\{\xi_v\}}_{\infty})^2$ is a finite free $B$-module. In general, we can get the surjectivity by writing $\kq^{\{\xi_v\}}_{\infty}/(\kq^{\{\xi_v\}}_{\infty})^2$ as a quotient of a finite free $B$-module. 
\end{proof}

\begin{cor} \label{patdefr}
Let $\widehat{(R^{\{\xi_v\}}_\infty)_{\kq^{\{\xi_v\}}_{\infty}}}$ be the $\kq^{\{\xi_v\}}_{\infty}$-adic completion of $(R^{\{\xi_v\}}_\infty)_{\kq^{\{\xi_v\}}_{\infty}}$. Then we have a surjective local homomorphism:
\[\widehat{(R^{\{\xi_v\}}_\infty)_{\kq^{\{\xi_v\}}_{\infty}}} \twoheadrightarrow \widehat{(R^{\patch}_{b})_{[\kq^{\patch}_b]}}~(\cong \widehat{(R^{\ps,\patch})_{[\kq^{\ps,\patch}]})}.\] 
\end{cor}

\begin{proof}
The only non-trivial part is the surjectivity. We have the following commutative diagram:
\[\begin{tikzcd}
  0\to ((\kq^{\{\xi_v\}}_{\infty})^n/(\kq^{\{\xi_v\}}_{\infty})^{n+1})_{\kq^{\{\xi_v\}}_{\infty}} \arrow[d] \arrow[r] & (R^{\{\xi_v\}}_\infty/(\kq^{\{\xi_v\}}_{\infty})^{n+1})_{\kq^{\{\xi_v\}}_{\infty}} \arrow[d,"\pi_{n+1}"]\arrow[r] &(R^{\{\xi_v\}}_\infty/(\kq^{\{\xi_v\}}_{\infty})^n)_{\kq^{\{\xi_v\}}_{\infty}}  \arrow[d,"\pi_n"] \to 0\\
0\longrightarrow  (\kq^{[n]}_b/\kq^{[n+1]}_b)_{\kq^{\patch}_b} \arrow[r] & (R^{\patch}_{b}/\kq^{[n+1]}_b)_{\kq^{\patch}_b} \arrow[r] &(R^{\patch}_{b}/\kq^{[n]}_b)_{\kq^{\patch}_b}   \longrightarrow 0.
\end{tikzcd}\]
The first vertical arrow is surjective by the previous proposition. So an induction on $n$ shows that $\pi_n$ is surjective for any $n$ and kernels of $\pi_n$ form a projective system with surjective transition maps. This clearly implies the surjectivity of the limit of $\pi_n$.
\end{proof}

\subsection{Patching III: An application of Pa\v{s}k\={u}nas' theory and the local-global compatibility result} \label{p3apop}

\begin{para}
In subsection \ref{sblgc}, we attach a block $\kB_\km$ of $D_p^\times$ to the maximal ideal $\km$ of Hecke algebra. By Theorem \ref{lgc}, $M_{\psi,\xi}(U^p)_\km\in \kC_{D_p^\times,\psi}(\cO)^{\kB_{\km}}$. From this, it is clear from the proof of corollary \ref{finpath} that $\tilde{M}^{\patch,\{\xi_v\}}_n\in  \kC_{D_p^\times,\psi}(\cO)^{\kB_{\km}}$. Recall that $\tilde{M}^{\patch,\{\xi_v\}}_n$ is defined to be the image of $M^{\patch,\{\xi_v\}}_n$ into $(M^{\patch,\{\xi_v\}}_n)_{\tilde{f}}$. Similarly, we denote the image of $M^{\patch,\{\xi_v\}}_{n,0}$ into $(M^{\patch,\{\xi_v\}}_{n,0})_{\tilde{f}}$ by $\tilde{M}^{\patch,\{\xi_v\}}_{n,0}$. Let $P_{\kB_\km}$ be the projective generator of $\kB_\km$. See \ref{sblgc} for the notations here.
\end{para}

\begin{defn}
For any positive integer $n$, we define
\begin{itemize}
\item $\tilde{\mm}^{\patch,\{\xi_v\}}_n:=\Hom_{\kC_{D_p^\times,\psi}(\cO)}(P_{\kB_\km},\tilde{M}^{\patch,\{\xi_v\}}_n)$.
\item $\tilde{\mm}^{\patch,\{\xi_v\}}_{n,0}:=\Hom_{\kC_{D_p^\times,\psi}(\cO)}(P_{\kB_\km},\tilde{M}^{\patch,\{\xi_v\}}_{n,0})$.
\item $\mm^{\{\xi_v\}}:=\Hom_{\kC_{D_p^\times,\psi}(\cO)}(P_{\kB_\km},M_\psi(U^p)_\km)$.
\item $\mm^{\{\xi_v\}}_{n}:=\Hom_{\kC_{D_p^\times,\psi}(\cO)}(P_{\kB_\km},M^{\{\xi_v\}}_{n})\cong \mm^{\{\xi_v\}}\otimes_\cO \cO_\infty/\kb^n$.
\end{itemize}
\end{defn}

\begin{para}
By our local-global compatibility result (corollary \ref{mclgc}), $\mm^{\{\xi_v\}}$ is a faithful, finitely generated $\T_{\psi,\xi}(U^p)_\km$-module. We note that $\Hom_{\kC_{D_p^\times,\psi}(\cO)}(P_{\kB_\km},\cdot)$ is an exact functor from $\kC_{D_p^\times,\psi}(\cO)^{\kB_{\km}}$ to the category of right pseudo-compact $E_{\kB_\km}$-modules (see \ref{Blocks}).
\end{para}

\begin{lem}
In the following statements, for any $R^{\ps,\{\xi_v\}}_{Q_N}$-module $M$, we use $M_{\tilde{f}}$ to denote its localization by powers of $\tilde{f}$.
\begin{enumerate}
\item $(\tilde{\mm}^{\patch,\{\xi_v\}}_n)_{\tilde{f}}$ is a flat $S_\infty/\ka^n$-module and there are natural isomorphisms 
\[(\tilde{\mm}^{\patch,\{\xi_v\}}_{n+1})_{\tilde{f}}/\ka^n (\tilde{\mm}^{\patch,\{\xi_v\}}_{n+1})_{\tilde{f}}\cong (\tilde{\mm}^{\patch,\{\xi_v\}}_n)_{\tilde{f}}.\]
\item The isomorphism in proposition \ref{mnflat} and $\eta^{\patch}_n$ in \ref{etanpa} induce isomorphisms:
\[(\tilde{\mm}^{\patch,\{\xi_v\}}_n)_{\tilde{f}}/\ka_0(\tilde{\mm}^{\patch,\{\xi_v\}}_n)_{\tilde{f}}\cong (\tilde{\mm}^{\patch,\{\xi_v\}}_{n,0})_{\tilde{f}} \cong (\mm^{\{\xi_v\}}_{n})_{\tilde{f}}^{\oplus 2^r}.\]
\end{enumerate}
\end{lem}

\begin{proof}
For any finitely generated $S_\infty/\ka^n$-module $N$, it is easy to see that 
\[\tilde{\mm}^{\patch,\{\xi_v\}}_n\otimes_{S_\infty/\ka^n} N\cong \Hom_{\kC_{D_p^\times,\psi}(\cO)}(P_{\kB_\km},\tilde{M}^{\patch,\{\xi_v\}}_n\otimes_{S_\infty/\ka^n}N).\]
The isomorphism in the first part now follows from proposition \ref{mnflat}. Again by proposition \ref{mnflat}, $M^{\patch,\{\xi_v\}}_n$ is flat over $S_\infty/\ka^n$. Hence for any short exact sequence of finitely generated $S_\infty/\ka^n$-modules $0\to N'\to N\to N''\to 0$, 
\[0\to \tilde{M}^{\patch,\{\xi_v\}}_n\otimes_{S_\infty/\ka^n}N'\to \tilde{M}^{\patch,\{\xi_v\}}_n\otimes_{S_\infty/\ka^n}N\to \tilde{M}^{\patch,\{\xi_v\}}_n\otimes_{S_\infty/\ka^n}N''\to 0\]
is exact up to $\tilde{f}$-torsion. Applying the exact functor $\Hom_{\kC_{D_p^\times,\psi}(\cO)}(P_{\kB_\km},\cdot)$ and inverting $\tilde{f}$, we deduce easily that $(\tilde{\mm}^{\patch,\{\xi_v\}}_n)_{\tilde{f}}$ is a flat $S_\infty/\ka^n$-module.

By definition, the kernel of $\tilde{\mm}^{\patch,\{\xi_v\}}_n/\ka_0 \tilde{\mm}^{\patch,\{\xi_v\}}_n \to \tilde{\mm}^{\patch,\{\xi_v\}}_{n,0}$ is killed by powers of $\tilde{f}$. This proves the first isomorphism in the second part. The second isomorphism follows directly from proposition \ref{fkkc}.
\end{proof}

\begin{cor} \label{patmloc}
In definition \ref{defnrpsp}, we defined a $R^{\ps,\patch}$-module structure on $M^{\patch,\{\xi_v\}}_n$ hence also on $\tilde{M}^{\patch,\{\xi_v\}}_n,\tilde{\mm}^{\patch,\{\xi_v\}}_n$. Then the localization $(\tilde{\mm}^{\patch,\{\xi_v\}}_n)_{\kq^{\ps,\patch}}$ of $\tilde{\mm}^{\patch,\{\xi_v\}}_n$ at $\kq^{\ps,\patch}$ is flat over $S_\infty/\ka^n$ and we have isomorphisms:
\begin{eqnarray*}
(\tilde{\mm}^{\patch,\{\xi_v\}}_{n+1})_{\kq^{\ps,\patch}}/\ka^{n} (\tilde{\mm}^{\patch,\{\xi_v\}}_{n+1})_{\kq^{\ps,\patch}}&\cong& (\tilde{\mm}^{\patch,\{\xi_v\}}_n)_{\kq^{\ps,\patch}};\\
(\tilde{\mm}^{\patch,\{\xi_v\}}_n)_{\kq^{\ps,\patch}}/\ka_0 (\tilde{\mm}^{\patch,\{\xi_v\}}_n)_{\kq^{\ps,\patch}}&\cong &(\tilde{\mm}^{\patch,\{\xi_v\}}_{n,0})_{\kq^{\ps,\patch}}\cong (\mm^{\{\xi_v\}}_n)^{\oplus 2^r}_{\kq^{\ps,\patch}}.
\end{eqnarray*}
\end{cor}

\begin{proof}
In view of the previous lemma, it suffices to prove that the image of $\tilde{f}=\prod \tilde{f}_N$ in $R^{\ps,\patch}/\kq^{\ps,\patch}=B'$ is non-zero. Let $\prod_N\tilde{f}_N$ be the image of $\tilde{f}$ in $\prod_{N\in\cI}(R^{\ps}_N/\kq^{\ps}_N)=\prod_N B'$. By proposition \ref{twkcc}, $A/(\bar{f}_N)$ has a uniform bounded length for all $N$. Note that the natural map $\prod_N B'\to R^{\ps,\patch}/\kq^{\ps,\patch}=B'$ is the one considered in the proof of lemma \ref{defqps}. The image of $\prod_N\tilde{f}_N$ in $B'$ is non-zero by the explicit description of the kernel of this map there.
\end{proof}

\begin{defn}
For any $n>0$, we define
\begin{enumerate}
\item $\displaystyle \mm^{\patch,\{\xi_v\}}_n:=\varprojlim_k (\tilde{\mm}^{\patch,\{\xi_v\}}_n)_{\kq^{\ps,\patch}}/\kq^{\ps,[k]}(\tilde{\mm}^{\patch,\{\xi_v\}}_n)_{\kq^{\ps,\patch}}$, the completion of $(\tilde{\mm}^{\patch,\{\xi_v\}}_n)_{\kq^{\ps,\patch}}$ with respect to the system of ideals $(\kq^{\ps,[k]})_{k>0}$. See lemma \ref{comexa} below for an equivalent definiton.
\item  $\displaystyle \mm^{\patch,\{\xi_v\}}_{n,0}:=\varprojlim_k (\tilde{\mm}^{\patch,\{\xi_v\}}_{n,0})_{\kq^{\ps,\patch}}/\kq^{\ps,[k]}(\tilde{\mm}^{\patch,\{\xi_v\}}_{n,0})_{\kq^{\ps,\patch}}$, the completion of $(\tilde{\mm}^{\patch,\{\xi_v\}}_{n,0})_{\kq^{\ps,\patch}}$ with respect to the system of ideals $(\kq^{\ps,[k]})_{k>0}$.
\item $\displaystyle\mm^{\{\xi_v\}}_{\infty}:=\varprojlim_k \mm^{\patch,\{\xi_v\}}_k$.
\item $\mm_0^{\{\xi_v\}}:=$ the $\kq$-adic completion of $(\mm^{\{\xi_v\}})_\kq$ as a $\T_{\psi,\xi}(U^p)_\kq$-module. Recall that $\kq$ is a prime ideal of $\T_{\psi,\xi}(U^p)_\km$ defined in \ref{kqrho}.
\end{enumerate}
Clearly $\widehat{(R^{\ps,\patch})_{[\kq^{\ps,\patch}]}}$ hence $\widehat{(R^{\{\xi_v\}}_\infty)_{\kq^{\{\xi_v\}}_{\infty}}}$ (by corollary \ref{patdefr}) act on these spaces.
\end{defn}

The main result of this subsection is

\begin{prop}
\begin{enumerate}
\item $\mm^{\{\xi_v\}}_\infty$ is a finitely generated $\widehat{(R^{\{\xi_v\}}_\infty)_{\kq^{\{\xi_v\}}_{\infty}}}$-module.
\item $\mm^{\{\xi_v\}}_\infty$ is a flat $S_\infty$-module. Moreover, the isomorphisms in corollary \ref{patmloc} induce
\[\mm^{\{\xi_v\}}_\infty/\ka_1\mm^{\{\xi_v\}}_\infty\cong (\mm_0^{\{\xi_v\}})^{\oplus 2^r}.\]
\end{enumerate}
\end{prop}

\begin{proof}
In view of the second part of lemma \ref{fllem} and Theorem 8.4 of \cite{Mat1}, the proposition is a direct consequence of the following lemma. 
\end{proof}

\begin{lem} \label{comexa}
\hspace{2em}
\begin{enumerate}
\item $\mm^{\patch,\{\xi_v\}}_n$ is isomorphic to the $\kq^{\ps,\patch}$-adic completion of $(\tilde{\mm}^{\patch,\{\xi_v\}}_n)_{\kq^{\ps,\patch}}$.
\item $\mm^{\patch,\{\xi_v\}}_1$ is a finitely generated $\widehat{(R^{\ps,\patch})_{[\kq^{\ps,\patch}]}}$-module.
\item For each $n>0$,  $\mm^{\patch,\{\xi_v\}}_n$ is a flat $S_\infty/\ka^n$-module and there are natural isomorphisms
\begin{eqnarray*}
\mm^{\patch,\{\xi_v\}}_{n+1}/\ka^n\mm^{\patch,\{\xi_v\}}_{n+1}&\cong &\mm^{\patch,\{\xi_v\}}_n.\\
\mm^{\patch,\{\xi_v\}}_n/\ka_1 \mm^{\patch,\{\xi_v\}}_n &\cong& (\mm_0^{\patch,\{\xi_v\}} /\varpi^n\mm_0^{\patch,\{\xi_v\}} )^{\oplus 2^r}.
\end{eqnarray*}
\end{enumerate}
\end{lem}

\begin{para}
To prove this lemma, we need some more notations. For $v|p$, let $R^{\ps,\chi}_v$ be the universal deformation ring which parametrizes all two-dimensional pseudo-representations of $G_{F_v}$ which lifts $1+\bar{\chi}|_{G_{F_v}}$ with determinant $\chi$. Put $R^{\ps}_p=\widehat{\bigotimes}_{v|p}R^{\ps,\chi}_v$. By corollary \ref{mclgc}, $\mm^{\{\xi_v\}}$ is a finitely generated $R^{\ps}_p$-module. 

Note that there are two actions of $R^{\ps}_p$ on $\tilde{\mm}^{\patch,\{\xi_v\}}_n$: one comes from the center of $E_{\kB_\km}$ (see \ref{Blocks} for notations here), one comes from mapping $R^{\ps}_p$ diagonally into $\prod_N R^{\ps}_N$. We will always view $\tilde{\mm}^{\patch,\{\xi_v\}}_n$ as a $R^{\ps}_p$-module by the \textit{second} action. It follows from Theorem \ref{lgc} and the proof of corollary \ref{finpath} that $\tilde{\mm}^{\patch,\{\xi_v\}}_n$ is finitely generated over $R^{\ps}_p$.
\end{para}

\begin{defn}
We define $R_n$ to be the image of $R^{\ps,\patch}$ in $\End_{R^{\ps}_p}(\tilde{\mm}^{\patch,\{\xi_v\}}_n)$. This is a natural $R^{\ps}_p$-algebra.
\end{defn}

\begin{para}
It is clear that $R_n$ is a \textit{finite} $R^{\ps}_p$-algebra hence a noetherian ring. We claim that the kernel of $R^{\ps,\patch}\to\End_{R^{\ps}_p}(\tilde{\mm}^{\patch,\{\xi_v\}}_n)$ is contained in $\kq^{\ps,\patch}$, so that $\kq^{\ps,\patch}$ defines a prime ideal $\kq_n$ of $R_n$. It suffices to show that 
\[(\tilde{\mm}^{\patch,\{\xi_v\}}_n)_{\kq^{\ps,\patch}}\neq 0.\]
By corollary \ref{patmloc}, we only need to check $(\mm^{\{\xi_v\}}_n)_{\kq^{\ps,\patch}}\neq 0$. Note that $\mm^{\{\xi_v\}}_{n}/\kb \mm^{\{\xi_v\}}_{n}\cong \mm^{\{\xi_v\}}/\varpi\mm^{\{\xi_v\}}$ and the action of $R^{\ps,\patch}$ on it factors through $\T_{\psi,\xi}(U^p)_\km$. Also it follows from lemma \ref{defqps} that $\kq^{\ps,\patch}$ is the pull-back of $\kq\in\Spec\T_{\psi,\xi}(U^p)_\km$. Thus
\[(\tilde\mm^{\{\xi_v\}}_n/\kb \tilde\mm^{\{\xi_v\}}_{n})_{\kq^{\ps,\patch}}=(\tilde\mm^{\{\xi_v\}}/\varpi\tilde\mm^{\{\xi_v\}})_\kq\neq 0\]
because $\kq$ contains $\varpi$ and $\mm$ is a faithful finitely generated $\T_{\psi,\xi}(U^p)_\km$-module.
\end{para}

\begin{proof}[Proof of lemma \ref{comexa}]
For any positive integer $j$, let $\kq_n^{[j]}$ be the image of $\kq^{\ps,[j]}$ in $R_n$. Recall that this is also the image of $\prod_N  (\kq^{\ps}_N)^j$. We claim that $\kq^{[j]}_n=\kq_n^j$. This will imply our first assertion in the lemma.

Let $\km_p$ be the maximal ideal of $R^{\ps}_p$. Both $\kq_n^{[j]}$ and $\kq_n^j$ are finite $R^{\ps}_p$-modules, hence naturally profinite groups by the $\km_p$-adic topology. In particular, they are compact.

In the proof of corollary  \ref{finpath}, we showed that $\tilde{f}^{2n} M^{\patch,\{\xi_v\}}_n$ is a finite $\cO[[K_p]]$-module. Therefore we can find an integer $C\geq2n$ such that $\tilde{f}^{C} M^{\patch,\{\xi_v\}}_n$ has no $\tilde{f}$-torsion. The multiplication by $\tilde{f}^C$ induces  a natural isomorphism between $\tilde{M}^{\patch,\{\xi_v\}}_n$ and $\tilde{f}^{C} M^{\patch,\{\xi_v\}}_n$. We also note that the $\cO[[K_p]]$-module structure makes $\tilde{f}^{C} M^{\patch,\{\xi_v\}}_n$ into a natural profinite group.

Recall that in \ref{Pch}, we defined $M^{\patch,\{\xi_v\}}_n$ as $\varprojlim_{U_p}(\prod_{N\in\cI} M(U_p,N,n)/\ka^n M(U_p,N,n)\otimes_\fR \fR_\kF)$. Let $M(U_p,C)$ be the image of $\tilde{f}^{C} M^{\patch,\{\xi_v\}}_n$ in $\prod_{N\in\cI} M(U_p,N,n)/\ka^n M(U_p,N,n)\otimes_\fR \fR_\kF$. Clearly $M(U_p,C)$ is contained in $\tilde{f}^C(\prod_{N\in\cI} M(U_p,N,n)/\ka^n M(U_p,N,n)\otimes_\fR \fR_\kF)$, which is of finite cardinality by lemma \ref{f2nmNn}. Since $\tilde{f}^C M^{\patch,\{\xi_v\}}_n$ is compact, we have
\[\tilde{M}^{\patch,\{\xi_v\}}_n\cong \tilde{f}^C M^{\patch,\{\xi_v\}}_n\cong \varprojlim_{U_p} M(U_p,C).\]
It follows from lemma \ref{q[n]qn} that $\kq_n^{[j]}$ and $\kq_n^j$ have the same images in the endomorphism rings of $M(U_p,C)$. Taking the limits over $U_p$ and using the compactness of $\kq_n^{[j]}$ and $\kq_n^j$, we conclude that they are equal.
 
As a consequence, ${\mm}^{\patch,\{\xi_v\}}_n$ is isomorphic to the $\kq_n$-adic completion of $(\tilde{\mm}^{\patch,\{\xi_v\}}_n)_{\kq_n}$ as a $(R_n)_{\kq_n}$-module. Since $\tilde{\mm}^{\patch,\{\xi_v\}}_n$ is a finitely generated module over $R^{\ps}_p$ hence also finitely generated over the noetherian ring $R_n$, we have (Theorem 8.7 \cite{Mat1})
\[{\mm}^{\patch,\{\xi_v\}}_n\cong \tilde{\mm}^{\patch,\{\xi_v\}}_n\otimes _{R_n} \widehat{(R_n)_{\kq_n}}.\]
Here $\widehat{(R_n)_{\kq_n}}$ is the $\kq_n$-adic completion of $(R_n)_{\kq_n}$, which is flat over $R_n$ (Theorem 8.8 ibid.). The third part of the lemma now follows directly from corollary \ref{patmloc}. The second part holds as $\widehat{(R_n)_{\kq_n}}$ is a quotient of $\widehat{(R^{\ps,\patch})_{[\kq^{\ps,\patch}]}}$ (Theorem 8.1 ibid.) and $\mm_1^{\patch,\{\xi_v\}}$ is a finite $\widehat{(R_1)_{\kq_1}}$-module.
\end{proof}

\begin{para}
Recall that in the beginning of subsection \ref{p1pch}, we define $S_\infty'\subseteq\cO_\infty[[\Delta_\infty]]$ to be the closure (under the profinite topology) of the $\cO_\infty$-subalgebra generated by all elements of the form $g+g^{-1},g=(0,\cdots,0,a,0\cdots,0)\in\Delta_\infty$ and ideal $\ka_1'=\ka_1\cap S_\infty'$.
\end{para}

\begin{lem} \label{sinf'}
$S_\infty'$ is contained in the image of $\widehat{(R^{\{\xi_v\}}_\infty)_{\kq^{\{\xi_v\}}_{\infty}}}\to \End (\mm^{\{\xi_v\}}_\infty)$. Moreover $\ka_1'$ is contained in the image of $\ker (\widehat{(R^{\{\xi_v\}}_\infty)_{\kq^{\{\xi_v\}}}}\to \widehat{(R^{\ps,\{\xi_v\}})_{\kq^{\ps}}})$ induced by the map $R^{\ps,\patch}\to R^{\ps,\{\xi_v\}}$ (see \ref{rpsps}), where $\widehat{(R^{\ps,\{\xi_v\}})_{\kq^{\ps}}}$ denotes the $\kq^{\ps}$-adic completion of $(R^{\ps,\{\xi_v\}})_{\kq^{\ps}}$.
\end{lem}
\begin{proof}
This follows from proposition \ref{sinfty'}.
\end{proof}
In other words, since $S_\infty'\cong\cO[[y_1,\cdots,y_{4|P|-1},s_1',\cdots,s_r']]$, we may choose a lifting $S_\infty'\to \widehat{(R^{\{\xi_v\}}_\infty)_{\kq^{\{\xi_v\}}}}$ with $\ka_1'$ mapping into $\ker (\widehat{(R^{\{\xi_v\}}_\infty)_{\kq^{\{\xi_v\}}}}\to \widehat{(R^{\ps,\{\xi_v\}})_{\kq^{\ps}}})$.

\subsection{Proof of Theorem \ref{thmA}}
\begin{para}
We first summarize what we have done so far in the following commutative diagram:
\[\begin{tikzcd}
& S_\infty' \arrow[r,hook] \arrow[d,dashed] & S_\infty \arrow[d,hook] \\
& \widehat{(R^{\{\xi_v\}}_\infty)_{\kq^{\{\xi_v\}}_{\infty}}} \arrow[ld, two heads,"\pi_R"]  \arrow [d, two heads] \arrow[r] & \End_{S_\infty} (\mm^{\{\xi_v\}}_\infty) \arrow[d] \\
\widehat{(R^{\ps,\{\xi_v\}})_{\kq^{\ps}}} \arrow[r,two heads]  &\widehat{\T_{\kq}} \arrow[r,hook] & \End_{\cO} (\mm_0^{\{\xi_v\}}),
\end{tikzcd}\]
where the dashed arrow exists by lemma \ref{sinf'} and the image of $\ka_1'$ is contained in $\ker(\pi_R)$. In addition, we have
\begin{itemize}
\item $\mm^{\{\xi_v\}}_\infty$ is a finitely generated $\widehat{(R^{\{\xi_v\}}_\infty)_{\kq^{\{\xi_v\}}_{\infty}}}$-module.
\item $\mm^{\{\xi_v\}}_\infty$ is a flat $S_\infty$-module and $\mm^{\{\xi_v\}}_\infty/\ka_1 \mm^{\{\xi_v\}}_\infty\cong (\mm_0^{\{\xi_v\}})^{\oplus 2^r}$.
\item $\mm_0^{\{\xi_v\}}=\widehat{(\mm^{\{\xi_v\}})_\kq}\cong (\mm^{\{\xi_v\}})_\kq\otimes_{\T_\kq} \widehat{\T_\kq}$ is a finitely generated faithful $\widehat{\T_\kq}$-module. By Theorem \ref{dh}, each irreducible component of $\widehat{\T_\kq}$ has dimension at least $2[F:\Q]$.
\end{itemize}
Let $(R^{\{\xi_v\}})'$ be $\widehat{(R^{\{\xi_v\}}_\infty)_{\kq^{\{\xi_v\}}_{\infty}}}\otimes_{S_\infty'} S_\infty$. This is also a local ring with $\widehat{\T_\kq}$ as a natural quotient by mapping $S_\infty$ to $S_\infty/\ka_1=\cO$. Hence $\dim_{(R^{\{\xi_v\}})'} (\mm^{\{\xi_v\}}_0)= \dim_{\widehat{\T_\kq}} (\mm^{\{\xi_v\}}_0)$ is at least $2[F:\Q]$. Note that $y_1,\cdots,y_{4|P|-1},s_1,\cdots,s_r\in\ka_1$ form a regular sequence of $\mm^{\{\xi_v\}}_\infty$. We see immediately from these results that
\end{para}

\begin{lem} \label{dimsupp}
$\dim_{(R^{\{\xi_v\}})'} (\mm^{\{\xi_v\}}_\infty)\geq 4|P|-1+r+2[F:\Q]$.
\end{lem}

On the other hand, by proposition \ref{rlocp} and lemma \ref{cctp}, we have

\begin{lem}
$\widehat{(R^{\{\xi_v\}}_\infty)_{\kq^{\{\xi_v\}}_{\infty}}}$ is equidimensional of dimension $4|P|-1+r+2[F:\Q]$. \end{lem}

As a corollary $(R^{\{\xi_v\}})'$ is also equidimensional of dimension $4|P|-1+r+2[F:\Q]$ since it is finite free over $\widehat{(R^{\{\xi_v\}}_\infty)_{\kq^{\{\xi_v\}}_{\infty}}}$. Combining this with lemma \ref{dimsupp}, we deduce that

\begin{cor} \label{supirrc}
The support $\Supp_{\widehat{(R^{\{\xi_v\}}_\infty)_{\kq^{\{\xi_v\}}_{\infty}}}}(\mm^{\{\xi_v\}}_\infty)$ contains at least one irreducible component of $\widehat{(R^{\{\xi_v\}}_\infty)_{\kq^{\{\xi_v\}}_{\infty}}}$ and any minimal prime of $\Supp_{\widehat{(R^{\{\xi_v\}}_\infty)_{\kq^{\{\xi_v\}}_{\infty}}}}(\mm^{\{\xi_v\}}_\infty)$ has characteristic zero.
\end{cor}
\begin{proof}
The first claim is clear as $(R^{\{\xi_v\}})'$ is finite over $\widehat{(R^{\{\xi_v\}}_\infty)_{\kq^{\{\xi_v\}}_{\infty}}}$, hence 
\[\dim_{\widehat{(R^{\{\xi_v\}}_\infty)_{\kq^{\{\xi_v\}}_{\infty}}}} (\mm^{\{\xi_v\}}_\infty)=\dim _{(R^{\{\xi_v\}})'} (\mm^{\{\xi_v\}}_\infty)=4|P|-1+r+2[F:\Q]=\dim \widehat{(R^{\{\xi_v\}}_\infty)_{\kq^{\{\xi_v\}}_{\infty}}}.\]
The second claim comes from the fact that $\mm^{\{\xi_v\}}_\infty$ is flat over $\cO$.
\end{proof}

\begin{lem} \label{fulsupthmA}
Theorem \ref{thmA} holds if $\mm^{\{\xi_v\}}_\infty$ has full support on $\widehat{(R^{\{\xi_v\}}_\infty)_{\kq^{\{\xi_v\}}_{\infty}}}$.
\end{lem}

\begin{proof}
As we mentioned before, the natural map $\widehat{(R^{\{\xi_v\}}_\infty)_{\kq^{\{\xi_v\}}_{\infty}}}\to \widehat{(R^{\ps,\{\xi_v\}})_{\kq^{\ps}}}$ extends to a map $(R^{\{\xi_v\}})'\to \widehat{(R^{\ps,\{\xi_v\}})_{\kq^{\ps}}}$ by sending $S_\infty$ to $S_\infty/\ka_1=\cO$. Let $\kp_0$ be a prime of $\widehat{(R^{\ps,\{\xi_v\}})_{\kq^{\ps}}}$ and denote its pull-back to $\widehat{(R^{\{\xi_v\}}_\infty)_{\kq^{\{\xi_v\}}_{\infty}}}$ by $\kp_1$. It is easy to see that there is a unique prime $\kp'_1$ of $(R^{\{\xi_v\}})'$ above $\kp_1$, which is nothing but $\kp_0\cap (R^{\{\xi_v\}})'$. Since $\kp_1$ is in the support of $\mm^{\{\xi_v\}}_\infty$ over $\widehat{(R^{\{\xi_v\}}_\infty)_{\kq^{\{\xi_v\}}_{\infty}}}$, we see that $\kp'_1\in \Supp_{(R^{\{\xi_v\}})'} (\mm^{\{\xi_v\}}_\infty)$. Notice that $\ka_1\subseteq \kp'_1$. Thus
\[\kp'_1\in \Supp_{(R^{\{\xi_v\}})'} (\mm^{\{\xi_v\}}_\infty/\ka_1\mm^{\{\xi_v\}}_\infty)=\Supp_{(R^{\{\xi_v\}})'} (\mm^{\{\xi_v\}}_0).\]
But the action of $(R^{\{\xi_v\}})'$ on $\mm^{\{\xi_v\}}_0$ factors through $\widehat{(R^{\ps,\{\xi_v\}})_{\kq^{\ps}}}$. Hence 
\[\kp_0\in \Supp_{\widehat{(R^{\ps,\{\xi_v\}})_{\kq^{\ps}}}} (\mm^{\{\xi_v\}}_0).\]
Since this works for all primes $\kp_0\in\Spec \widehat{(R^{\ps,\{\xi_v\}})_{\kq^{\ps}}}$, we conclude that $\mm^{\{\xi_v\}}_0$ has full support on $\widehat{(R^{\ps,\{\xi_v\}})_{\kq^{\ps}}}$. Thus $\widehat{(R^{\ps,\{\xi_v\}})_{\kq^{\ps}}}\to \widehat{\T_{\kq}}$ has nilpotent kernel. Note that this map is the $\kq^{\ps}$-adic completion of $(R^{\ps,\{\xi_v\}})_{\kq^{\ps}}\to \T_{\kq}$. Hence $(R^{\ps,\{\xi_v\}})_{\kq^{\ps}}\to \T_{\kq}$ also has nilpotent kernel. 
\end{proof}

\begin{para}
Now we can prove Theorem \ref{thmA} when \underline{$\xi_v$ \textbf{are all non-trivial}}. In this case, it follows from proposition \ref{rlocp} and lemma \ref{cctp} that $\widehat{(R^{\{\xi_v\}}_\infty)_{\kq^{\{\xi_v\}}_{\infty}}}$ is irreducible. Hence $\mm^{\{\xi_v\}}_\infty$ has full support on $\widehat{(R^{\{\xi_v\}}_\infty)_{\kq^{\{\xi_v\}}_{\infty}}}$. The previous lemma implies Theorem \ref{thmA} directly.
\end{para}

\begin{para}
In general, we use Taylor's trick in \cite{Ta08}. Let $\xi'_v:k(v)\to \cO^\times$ be non-trivial characters of $p$-power order for $v\notin S\setminus \Sigma_p$. Then the product of $\xi'_v$ can be viewed as a character $\xi'$ of $U^p$ (defined in \ref{autlev}) and we can define completed cohomology $S_{\psi,\xi'}(U^p)$, $S_{\psi,\xi'}(U^p,E/\cO)$ and Hecke algebra $\T':=\T_{\psi,\xi'}(U^p)$ as in subsection \ref{varch}.
\end{para}

\begin{lem} \label{SWbc}
$T_v-(1+\chi(\Frob_v)),v\notin S$ and $\varpi$ generate a maximal ideal $\km'$ of $\T'$.
\end{lem}
\begin{proof}
This is equivalent with saying that $S_{\psi,\xi'}(U^p,\F)[\km']$ is non-zero. Since $\xi_v\equiv \xi'_v\mod\varpi$, we have $S_{\psi,\xi'}(U^p,\F)[\km']=S_{\psi,\xi}(U^p,\F)[\km]$, which is non-zero by our assumption.
\end{proof}

\begin{para}
Therefore we get a non-zero surjective map $R^{\ps,\{\xi'_v\}}\to\T'_{\km'}$. Note that $R^{\ps,\{\xi'_v\}}/(\varpi)\cong R^{\ps,\{\xi_v\}}/(\varpi)$ as both rings represent the same universal problem. Under this isomorphism, $\kq^{\ps}$ can be viewed as a prime ideal $\kq'{}^{\ps}$ of $R^{\ps,\{\xi'_v\}}$ as $\varpi\in\kq^{\ps}$. The same argument of the previous lemma shows that $\kq'{}^{\ps}$ comes from a prime ideal $\kq'$ of $\T'$. 

It is clear from the proof of proposition \ref{exttwp} that the same set of Taylor-Wiles primes $Q_N$ will also satisfy proposition \ref{exttwp} with $\xi_v$ replaced by $\xi'_v$, and we can choose the map $R^{\{\xi'_v\}}_{\loc}[[x_1,\cdots,x_g]]\to R^{\square_P,\{\xi'_v\}}_{\bar{\rho}_b,Q_N}$ in proposition \ref{exttwp} to be the same as $R^{\{\xi_v\}}_{\loc}[[x_1,\cdots,x_g]]\to R^{\square_P,\{\xi_v\}}_{\bar{\rho}_b,Q_N}$ after reducing mod $\varpi$, under the isomorphism $R^{\{\xi'_v\}}_{\loc}/(\varpi)\cong R^{\{\xi_v\}}_{\loc}/(\varpi)$ and the one similar for $R^{\square_P,\{\xi'_v\}}_{\bar{\rho}_b,Q_N}/(\varpi)$. Thus we can use the same primes to patch our completed homology and get a diagram as in the beginning of this subsection:
\[\begin{tikzcd}
& \widehat{(R^{\{\xi'_v\}}_\infty)_{\kq^{\{\xi'_v\}}_{\infty}}} \arrow[ld, two heads,"\pi'_R"]  \arrow [d, two heads] \arrow[r] & \End_{S_\infty} (\mm^{\{\xi'_v\}}_\infty) \arrow[d] \\
\widehat{(R^{\ps,\{\xi'_v\}})_{\kq'{}^{\ps}}} \arrow[r,two heads]  &\widehat{\T'_{\kq'}} \arrow[r,hook] & \End_{\cO} (\mm_0^{\{\xi'_v\}}).
\end{tikzcd}\]
Moreover, we have the following commutative diagram
\[\begin{tikzcd}
\widehat{(R^{\{\xi'_v\}}_\infty)_{\kq^{\{\xi'_v\}}_{\infty}}}/(\varpi)   \arrow [d,"\cong"] \arrow[r] & \End_{S_\infty} (\mm^{\{\xi'_v\}}_\infty/\varpi \mm^{\{\xi'_v\}}_\infty) \arrow[d,"\cong"] \\
\widehat{(R^{\{\xi_v\}}_\infty)_{\kq^{\{\xi_v\}}_{\infty}}}/(\varpi)  \arrow[r] & \End_{S_\infty} (\mm^{\{\xi_v\}}_\infty/\varpi \mm^{\{\xi_v\}}_\infty) \\
\end{tikzcd}\]
under the natural isomorphisms $\widehat{(R^{\{\xi'_v\}}_\infty)_{\kq^{\{\xi'_v\}}_{\infty}}}/(\varpi)\cong \widehat{(R^{\{\xi_v\}}_\infty)_{\kq^{\{\xi_v\}}_{\infty}}}/(\varpi)$ and $\mm^{\{\xi'_v\}}_\infty/\varpi \mm^{\{\xi'_v\}}_\infty\cong \mm^{\{\xi_v\}}_\infty/\varpi \mm^{\{\xi_v\}}_\infty$.
Since $\xi'_v$ are all non-trivial, $\mm^{\{\xi'_v\}}_\infty$ has full support on $\widehat{(R^{\{\xi'_v\}}_\infty)_{\kq^{\{\xi'_v\}}_{\infty}}}$ by our previous result. Hence $\mm^{\{\xi_v\}}_\infty/\varpi \mm^{\{\xi_v\}}_\infty$ also has full support on $\widehat{(R^{\{\xi_v\}}_\infty)_{\kq^{\{\xi_v\}}_{\infty}}}/(\varpi)$ by the above diagram. By  corollary \ref{supirrc}, any minimal prime of the support of $\mm^{\{\xi_v\}}_\infty$ on$\widehat{(R^{\{\xi_v\}}_\infty)_{\kq^{\{\xi_v\}}_{\infty}}}$ has characteristic zero. It follows from proposition \ref{rlocp} that all minimal primes of $\widehat{(R^{\{\xi_v\}}_\infty)_{\kq^{\{\xi_v\}}_{\infty}}}$ are in the support. Theorem \ref{thmA} now follows from lemma \ref{fulsupthmA}.
\end{para}

\section{A generalization of results of Skinner-Wiles I} \label{AgoroS-W1}
In this section, we prove the modularity of some ordinary representations and some finiteness results, which partially generalize the work of Skinner-Wiles in \cite{SW99}. We follow the beautiful method of \cite{SW99} by establishing some ``$R=\T$" results for ordinary representations. One main difference in the proof is that we adopt Taylor's Ihara avoidance trick \cite{Ta08} rather than the original level raising argument in \cite{SW99}. Our main result (Theorem \ref{thmB} below) completely removes the assumption in \cite{SW99} that the reduction of $\psi_{v,1}$ modulo $\varpi$ is $\mathbf{1}$ for \textit{all} $v|p$ (see the statement of Theorem \ref{thmB} for the notation here). This requires some new results \ref{EoEmi} on the existence of certain Eisenstein maximal ideal of the ordinary Hecke algebra. 

The main result of this section will be a key ingredient of our proof of the modularity in the \textit{non-ordinary} case (in section \ref{Tmt}). In fact, we will use these ordinary points to find enough pro-modular points to apply the result in the previous section. See section \ref{Tmt} for more details here. We note that the method in this subsection does not work (at least so far) in the non-ordinary case. This is because $p$-adic local Langlands correspondence is only established for $\GL_2(\Q_p)$. However one main step \ref{Washington} in the proof is to bound the $p$-part of the class group by taking suitable field extension, in which $p$ might not be split. On the other hand, Hida theory works well for all finite extensions of $\Q_p$ hence is more flexible with base change. This is why the ordinary case can be handled directly. 

The case we exclude here ($\bar{\chi}|_{G_{F_v}}= \mathbf{1},v|p$) will be treated in section \ref{AgoroS-W2}.

\subsection{Statement of the main results} \label{sasotmr}
\begin{para} \label{sotmr1}
In this subsection, $F$ denotes an \textit{abelian} totally real extension of $\Q$ in which $p$ is \textit{unramified}. Let $S$ be a finite set of finite places containing all places above $p$. Let $\chi:G_{F,S}\to \cO^\times$ be a continuous character such that
\begin{itemize}
\item $\chi(c)=-1$ for any complex conjugation $c\in G_{F,S}$.
\item $\bar{\chi}$, the reduction of $\chi$ modulo $\varpi$, can be extended to a character of $G_\Q$.
\item $\bar{\chi}|_{G_{F_v}}\neq \mathbf{1}$ for any $v|p$.
\item $\chi|_{G_{F_v}}$ is de Rham for any $v|p$. In other words, $\chi=\varepsilon^k\psi_0$ with $k$ an integer and $\psi_0$ a character of finite order. 
\end{itemize}

Consider the universal deformation ring $R^{\ps,\ord}$ which pro-represents the functor from $\cOf$ to the category of sets sending $R$ to the set of two-dimensional pseudo-representations $T$ of $G_{F,S}$ over $R$  such that $T$ is a lifting of $1+\bar\chi$ with determinant $\chi$ and $T|_{G_{F_v}}$ is \textit{reducible} for any $v|p$, i.e. $x(\sigma,\tau)=0$ for any $\sigma,\tau\in G_{F_v},v|p$, cf. \ref{tar}. Denote the universal pseudo-representation by $T^{univ}:G_{F,S}\to R^{\ps,\ord}$. 

Since $\bar{\chi}|_{G_{F_v}}\neq\mathbf{1}$, we have $T^{univ}|_{G_{F_v}}=\psi^{univ}_{v,1}+\psi^{univ}_{v,2}$ for some characters $\psi^{univ}_{v,1},\psi^{univ}_{v,2}:G_{F_v}\to (R^{\ps,\ord})^\times$ which are liftings of $\mathbf{1},\bar{\chi}|_{G_{F_v}}$ respectively. (Using the notation in \ref{tar}, $\psi^{univ}_{v,1}=a|_{G_{F_v}}$ and $\psi^{univ}_{v,2}=d|_{G_{F_v}}$.) By the class field theory, $\psi^{univ}_{v,1}|_{I_{F_v}}$ induces a homomorphism $\cO[[O_{F_v}^\times(p)]]\to R^{\ps,\ord}$ for any $v|p$. Here $O_{F_v}^\times(p)$ denotes the $p$-adic completion of $O_{F_v}^\times$. Taking the completed tensor product over $\cO$ for all $v|p$, we get a map:
\[\Lambda_F:=\widehat{\bigotimes}_{v|p}\cO[[O_{F_v}^\times(p)]]\to R^{\ps,\ord}.\]
\end{para}

Now we can state the main results of this section:
\begin{thm} \label{thmB}
Under the assumptions for $F,\chi$ as above, we have
\begin{enumerate}
\item $R^{\ps,\ord}$ is a finite $\Lambda_F$-algebra.
\item For any maximal ideal $\kp$ of $R^{\ps,\ord}[\frac{1}{p}]$, we denote the associated semi-simple representation $G_{F,S}\to\GL_2(k(\kp))$ by $\rho(\kp)$ (see \ref{tar}). Assume 
\begin{itemize}
\item $\rho(\kp)$ is irreducible.
\item For any $v|p$, $\rho(\kp)|_{G_{F_v}}\cong\begin{pmatrix}\psi_{v,1} & *\\ 0 & \psi_{v,2}\end{pmatrix}$ such that $\psi_{v,1}$ is de Rham and has strictly less Hodge-Tate number than $\psi_{v,2}$ for any embedding $F_v\hookrightarrow \overbar{\Q_p}$.
\end{itemize}
Then $\rho(\kp)$ comes from a twist of a Hilbert modular form.
\end{enumerate}
\end{thm}

\begin{rem}
The condition that $p$ is unramified in $F$ can be weakened. However we decide to impose this condition here as this can simplify some arguments.
\end{rem}

\subsection{Hida families}
As we remarked before, roughly speaking, the main results are proved by identifying $R^{\ps,\ord}$ with some ordinary Hecke algebra. We first collect some basic results for Hida families.
\begin{para} \label{hidatheory}
In this subsection, let $F$ be a totally real field of even degree over $\Q$ in which $p$ is unramified and $D$ be a totally definite quaternion algebra over $F$ which splits at all finite places. We fix isomorphisms $D\otimes_F F_v\cong M_2(F_v)$ for any finite place $v$. 

Let $S$ be a finite set of places of $F$ that contains all places above $p$ and $U^p=\prod_{v\nmid p}U_v$ be an open compact subgroup of $\prod_{v\nmid p}\GL_2(O_{F_v})$ such that $U_v=\GL_2(O_{F_v})$ for $v\notin S$. For any positive integer $c$ and $v|p$, we denote
\[ \mathrm{Iw}_1(v^c)=\{g\in\GL_2(O_{F_v}),g\equiv \begin{pmatrix} 1& *\\ 0 & 1\end{pmatrix} \mod \varpi_v^c\},\]
and $U^p(c)=U^p\prod_{v|p}\mathrm{Iw}_1(v^c)$, an open compact subgroup of $\GL_2(\A_F^\infty)$.

Fix a continuous character $\psi:\AFi/F_{>>0}^\times(\AFi\cap U^p)\to\cO^\times$ such that $\psi(a_p)=N_{F/\Q}(a_p)^{-w}$ for some integer $w$ and all $a_p$ in some open subgroup of $O_{F,p}^\times$. Here $N_{F/\Q}:O_{F,p}^\times\to \Z_p^\times\to E^\times$ is the usual norm map. Recall that in \ref{pcaf}, for $(\vec{k},\vec{w})\in\Z_{>1}^{\Hom(F,\overbar{\Q_p})}\times \Z^{\Hom(F,\overbar{\Q_p})}$ such that $k_\sigma+2w_\sigma=w+2$ independent of $\sigma\in \Hom(F,\overbar{\Q_p})$, we defined an algebraic representation $\tau_{(\vec{k},\vec{w})}$ of $D_p^\times=(D\otimes \Q_p)^\times$ on
\[W_{(\vec{k},\vec{w}),E}=\bigotimes_{\sigma:F\to E}(\Sym^{k_\sigma-2}(E^2)\otimes \det{}^{w_\sigma}).\]
By abuse of notation, for any topological $\cO$-algebra $A$, we use $\tau_{(\vec{k},\vec{w})}$ to denote the representation on $W_{(\vec{k},\vec{w}),A}=\bigotimes_{\sigma:F\to E}(\Sym^{k_\sigma-2}(A^2)\otimes \det{}^{w_\sigma})$. Then we can define $S_{(\vec{k},\vec{w}),\psi}(U^p(c),A)$ as  in \ref{quaform}. Note that as we discussed in \ref{pcaf}, $S_{(\vec{k},\vec{w}),\psi}(U^p(c),E)$ can be considered as a space of automorphic forms.

For any $\gamma\in O_{F,p}\cap (F\otimes\Q_p)^\times$, we define $\langle\gamma\rangle\in\End_A(S_{(\vec{k},\vec{w}),\psi}(U^p(c),A))$ to be the double coset action
\[\langle\gamma\rangle=(\prod_{\sigma\in \Hom(F,\overbar{\Q_p})}\sigma(\gamma)^{-w_{\sigma}})U^p(c)\begin{pmatrix}\gamma & 0 \\ 0 & 1 \end{pmatrix} U^p(c).\]
Explicitly, since $U^p(c)\begin{pmatrix}\gamma & 0 \\ 0 & 1 \end{pmatrix} U^p(c)=\bigsqcup_{\alpha\in O_{F,p}/(\gamma)}\begin{pmatrix}\gamma & \alpha \\ 0 & 1 \end{pmatrix} U^p(c)$,
\[(\langle\gamma\rangle\cdot f)(g)=\prod_{\sigma\in \Hom(F,\overbar{\Q_p})}\sigma(\gamma)^{-w_{\sigma}}\sum_{\alpha\in O_{F,p}/(\gamma)}\tau_{(\vec{k},\vec{w})}(\begin{pmatrix}\gamma & \alpha \\ 0 & 1 \end{pmatrix} )\cdot f(g\begin{pmatrix}\gamma & \alpha \\ 0 & 1 \end{pmatrix} ).\]
This action is independent of $c$ and defines a morphism of monoids $O_{F,p}\cap (F\otimes\Q_p)^\times\to \End_A(S_{(\vec{k},\vec{w}),\psi}(U^p(c),A))$. Hida's idempotent $\mathbf{e}$ is defined to be
\[\mathbf{e}:=\varinjlim_{n\to+\infty}\langle p \rangle^{n!}\in \End_A(S_{(\vec{k},\vec{w}),\psi}(U^p(c),A)).\]

For any topological $\cO$-algebra $A$ and $k=(\vec{k},\vec{w})$ as above, the space of ordinary forms is defined to be 
\[S_{k,\psi}^{\ord}(U^p(c),A):=\mathbf{e}S_{k,\psi}(U^p(c),A).\]
Note that $\langle p \rangle$ acts by a unit on this space. Hence the morphism of monoids $O_{F,p}\cap (F\otimes\Q_p)^\times\to \End_A(S_{k,\psi}(U^p(c),A))$ extends to a homomorphism
\[\langle\cdot \rangle:(F\otimes\Q_p)^\times\to \End_A(S_{k,\psi}^{\ord}(U^p(c),A)).\]

We define the ordinary Hecke algebra $\T^{\ord}_{k,\psi}(U^p(c),A)$ to be the $A$-subalgebra generated by $T_v$, $v\notin S$ (see \ref{haapr}) and $\langle \gamma \rangle$, $\gamma\in O_{F,p}\cap (F\otimes\Q_p)^\times$. Hence $\langle\cdot \rangle$ induces a smooth character $F_v^\times\to \T^{\ord}_{k,\psi}(U^p(c),A)^\times$. by the class field theory, this defines a character:
\[\psi_{v,1}:G_{F_v}\to\T^{\ord}_{k,\psi}(U^p(c),A)^\times.\]

If $A=\cO$, there is a two-dimensional pseudo-representation with determinant $\psi\varepsilon^{-1}$
\[T_c:G_{F,S}\to \T^{\ord}_{k,\psi}(U^p(c),\cO)\]
sending $\Frob_v$ to $T_v$. 
\end{para}

\begin{para}
Now we can introduce Hida families. We define
\begin{itemize}
\item $S^{\ord}_{\psi}(U^p,E/\cO):=\varinjlim_{c,n>0} S_{k,\psi}^{\ord}(U^p(c),\varpi^{-n}\cO/\cO)$.
\item $M_{\psi}^{\ord}(U^p):=S^{\ord}_{\psi}(U^p,E/\cO)^\vee$.
\item $\T_\psi^{\ord}(U^p):=\varprojlim_c  \T^{\ord}_{k,\psi}(U^p(c),\cO)$ (Hecke algebra).
\end{itemize}
It is well-known (cf. Theorem 2.3 of \cite{Hida89}) that these are all independent of the choice of $k=(\vec{k},\vec{w})$. By abuse of notation, we also denote by $\psi_{v,1}:G_{F_v}\to \T^{\ord}_\psi(U^p)$ the character induced by $\langle\cdot\rangle$. Moreover there is a two-dimensional pseudo-representation with determinant $\psi\varepsilon^{-1}$
\[T^{\ord}:G_{F,S}\to \T^{\ord}_{\psi}(U^p)\]
sending $\Frob_v$ to $T_v$, such that for any $\kp\in\Spec T^{\ord}_{\psi}(U^p)$, its associated two-dimensional semi-simple representation $\rho(\kp)$ has the property 
\[\rho(\kp)|_{G_{F_v}}\cong \begin{pmatrix} \psi_{v,1} & *\\ 0 & *\end{pmatrix}, v|p.\]
See proposition 2.3 of \cite{Hida88}, though our convention for Hecke algebra is slightly different.

Let $\Lambda_F=\widehat{\bigotimes}_{v|p}\cO[[O_{F_v}^\times(p)]]$ defined in the previous subsection. Then there is a natural map
\[\Lambda_F\to \T^\ord_\psi(U^p(c))\]
induced by $\psi_{v,1}$. Hence $M_{\psi}^{\ord}(U^p)$ has a $\Lambda_F$-module structure. It is well-known (cf. Theorem 3.8 of \cite{Hida89}, though the determinant is not fixed there) that $M_{\psi}^{\ord}(U^p)$ is in fact finite free over $\Lambda_F$. Hence $\T^{\ord}_\psi(U^p)$ is a finite torsion-free $\Lambda_F$-algebra with same dimensions ($[F:\Q]+1$).
\end{para}

\begin{para} \label{finlevhecke}
The ordinary forms at finite level $\T^{\ord}_{k,\psi}(U^p(c),\cO)$ can be recovered in the following way (cf.  \cite[Theorem 2.4]{Hida89} and the proof of \cite[Proposition 2.20]{Ger10}): for $k=(\vec{k},\vec{w})\in\Z_{>1}^{\Hom(F,\overbar{\Q_p})}\times \Z^{\Hom(F,\overbar{\Q_p})}$ such that $k_\sigma+2w_\sigma=w+2$ independent of $\sigma\in \Hom(F,\overbar{\Q_p})$, let $\ka_{k,c}$ be the ideal of $\Lambda_F$ generated by $\langle \gamma \rangle-\prod_{\sigma}\sigma(\gamma)^{-w_\sigma}$, $\gamma\in 1+p^cO_{F,p}$. Suppose $U^p(c)$ is sufficiently small. Then
\[M_{\psi}^{\ord}(U^p)/\ka_{k,c}M_{\psi}^{\ord}(U^p)\cong S^{\ord}_{k,\psi}(U^p(c),E/\cO)^\vee.\]
Roughly speaking, finite level ordinary forms correspond to the locus where $\psi_{v,1}$ are locally algebraic characters with certain weights.
\end{para}

\begin{para}
By the same argument as in the proof of proposition \ref{semiloc}, we know that $\T^{\ord}_\psi(U^p)$ is a complete semi-local ring. Let $\km_1,\cdots,\km_s$ be its maximal ideals. Then $\T^{\ord}_\psi(U^p)=\T^{\ord}_\psi(U^p)_{\km_1}\times\cdots\times \T^{\ord}_\psi(U^p)_{\km_s}$ and for any maximal ideal $\km$, $M^{\ord}_{\psi}(U^p)_{\km}$ is a direct summand of $M_{\psi}^{\ord}(U^p)$ and hence finite free over $\Lambda_F$. Since $\Lambda_F$ is regular, we have
\[\depth_{\T^{\ord}_{\psi}(U^p)_\km} (M^{\ord}_{\psi}(U^p)_{\km})\geq \depth_{\Lambda_F} (M^{\ord}_{\psi}(U^p)_{\km})=\dim{\Lambda_F} = \dim \T^{\ord}_{\psi}(U^p)_\km.\]
Therefore $M^{\ord}_{\psi}(U^p)_{\km}$ is a Cohen-Macaulay $\T^{\ord}_{\psi}(U^p)_\km$-module. Note that $M^{\ord}_{\psi}(U^p)_{\km}$ is also a faithful $\T^{\ord}_{\psi}(U^p)_\km$-module. By Theorem 17.3 of \cite{Mat1}, for any minimal prime $\kp$ of $\T^{\ord}_{\psi}(U^p)_\km$
\[\dim \T^{\ord}_{\psi}(U^p)_\km/\kp=\dim \T^{\ord}_{\psi}(U^p)_\km=\dim \Lambda_F=[F:\Q]+1.\]
$\kp$ must be of characteristic zero as $\varpi\in\T^{\ord}_{\psi}(U^p)_\km$ is a regular element. Thus we get (compare this with Theorem \ref{dh})
\end{para}
\begin{cor} \label{orddim}
$\T^{\ord}_{\psi}(U^p)_\km$ is equidimensional of dimension $[F:\Q]+1$. Moreover, any minimal prime ideal has characteristic zero.
\end{cor} 

\begin{para}
 As in subsection \ref{varch}, we also need a variant of the space of ordinary forms. Let $\xi_v:U_v\to\cO^\times$ be smooth characters for $v\in S\setminus \Sigma_p$. Then product of $\xi_v$ defines a character $\xi$ of $U^p$. 
We define $S^{\ord}_{k,\psi,\xi}(U^p(c),A)$ to be $\mathbf{e}S_{k,\psi,\xi}(U^p(c),A)$. Similarly we can define $S^{\ord}_{\psi,\xi}(U^p,E/\cO), M_{\psi,\xi}^{\ord}(U^p),\T_{\psi,\xi}^{\ord}(U^p)$ as above and have $\dim \T_{\psi,\xi}^{\ord}(U^p)_{\km}=[F:\Q]+1$ for any maximal ideal $\km$ of $\T_{\psi,\xi}^{\ord}(U^p)$.
\end{para}

\subsection{Patching at a one-dimensional prime: ordinary case I} \label{Paao-dproc1}
In this subsection, we are going to prove a result similar to Theorem \ref{thmA} in the ordinary setting, which basically says an irreducible component of $R^{\ps,\ord,\{\xi_v\}}_{\Sigma^o}$ (see below for the definition) is pro-modular if it contains a nice prime. 

\begin{para} \label{setupord1}
\noindent \underline{\textbf{Setup}} (Compare with \ref{satsotmr}) Let $F,D$ be as in the previous subsection and $S$ be a finite set of places of $F$ containing all places above $p$ such that for any $v\in S\setminus \Sigma_p$,
\[N(v)\equiv 1\mod p.\]
Let $\xi_v:k(v)^\times\to\cO^\times$ be a character of $p$-power order for each $v\in S\setminus \Sigma_p$. We will view $\xi_v$ as characters of $I_{F_v}$ by the class field theory. We also fix a complex conjugation $\sigma^*\in G_{F}$. 

Fix a continuous character $\chi:G_{F,S}\to\cO^\times$ such that
\begin{itemize}
\item $\chi$ is unramified at places outside of $\Sigma_p$.
\item $\chi(\Frob_v)\equiv1\mod \varpi$ for $v\in\Sigma_p\setminus S$.
\item $\chi(c)=-1$ for any complex conjugation $c\in G_F$.
\item For any $v|p$, $\bar{\chi}|_{G_{F_v}}\neq \mathbf{1}$ and $\chi|_{G_{F_v}}$ is Hodge-Tate. Here $\bar{\chi}$ denotes the reduction of $\chi$ modulo $\varpi$.
\item $[F_v:\Q_p]\geq2$ for any $v|p$.
\end{itemize}

Let $\Sigma^o$ be a subset of $\Sigma_p$. We define $\bar{\psi}_{v,1}:G_{F_v}\to \F^\times$ to be $\mathbf{1}$ if $v\in\Sigma^o$ and $\bar{\chi}|_{G_{F_v}}$ if $v\notin\Sigma^o$. 

For a pseudo-representation, we have the following definition for it being ordinary.
\end{para}

\begin{defn} \label{psiord} \footnote{I learned this definition from Professor Richard Taylor's lectures at Princeton University in 2016 Spring. As a referee pointed out, it also appears in \cite[Definition 2.5]{CS19}.}
Let $T:G_{F}\to R$ be a two-dimensional pseudo-representation over some ring $R$ such that for some place $v$, $T|_{G_{F_v}}=\psi_1+\psi_2$ is a sum of two characters. We say $T$ is $\psi_1$-\textit{ordinary} if for any $\sigma,\tau\in G_{F_v},\eta\in G_F$,
\[T(\sigma\tau\eta)-\psi_1(\sigma)T(\tau\eta)-\psi_2(\tau)T(\sigma\eta)+\psi_1(\sigma)\psi_2(\tau)T(\eta)=0.\] 
\end{defn}

\begin{lem} \label{ordpseudo}
Let $\rho:G_F\to\GL_2(R)$ be a two-dimensional irreducible representation over some ring $R$ with trace $T$ such that $T|_{G_{F_v}}=\psi_1+\psi_2$, a sum of two characters. Suppose $\rho|_{G_{F_v}}$ has the form $\begin{pmatrix}\psi_1&*\\0& \psi_2\end{pmatrix}$. Then $T$ is $\psi_1$-ordinary. Conversely, if $R$ is a field and $T$ is $\psi_1$-ordinary, then after possibly enlarging $R$,
\[\rho|_{G_{F_v}}\cong\begin{pmatrix}\psi_1&*\\0& \psi_2\end{pmatrix}.\]
\end{lem}

\begin{proof}
The first claim follows from some simple computation. For the second claim, we may assume $\psi_1\neq \psi_2$. Enlarge $R$ if possible, then under suitable basis, we may assume $\rho|_{G_{F_v}}=\begin{pmatrix}\psi_1&*\\0& \psi_2\end{pmatrix}$ or $\rho|_{G_{F_v}}=\begin{pmatrix}\psi_2&b\\0& \psi_1\end{pmatrix}$ for some $b:G_{F_v}\to R$. Suppose we are in the second case, the ordinary condition is the same as
\[\tr[(\rho(\sigma)-\psi_1(\sigma))(\rho(\tau)-\psi_2(\tau))\rho(\eta)]=0\]
for any $\sigma,\tau\in G_{F_v},\eta\in G_F$. Equivalently, if we write $\rho(\eta)=\begin{pmatrix} *&*\\ c& *\end{pmatrix}$ and $\phi=\psi_1-\psi_2$, then
\[(b(\sigma)\phi(\tau)-b(\tau)\phi(\sigma))c=0.\]
Since $\rho$ is irreducible, we can always choose $\eta$ so that $c\neq 0$, hence $b(\sigma)\phi(\tau)-b(\tau)\phi(\sigma)=0$. Choose $\sigma\in G_{F_v}$ such that $\phi(\sigma)\neq0$, then for any $\tau\in G_{F_v}$,
\[\begin{pmatrix} 0&1\\ 1 & \frac{b(\sigma)}{\phi(\sigma)} \end{pmatrix}\rho(\tau) \begin{pmatrix}-\frac{b(\sigma)}{\phi(\sigma)} &1\\ 1 & 0 \end{pmatrix}=\begin{pmatrix}\psi_1(\tau)&0\\0& \psi_2(\tau)\end{pmatrix}.\]
This is exactly what we want.
\end{proof}

\begin{para} \label{rpsordxiv}
Now consider the universal deformation ring $R^{\ps,\ord,\{\xi_v\}}_{\Sigma^o}$ which pro-represents the functor from $\cOf$ to the category of sets sending $R$ to the set of tuples $(T,\{\psi_{v,1}\}_{v\in\Sigma_p})$ where
\begin{itemize}
\item $T$ is a two-dimensional pseudo-representation of $G_{F,S}$ over $R$ lifting $1+\bar\chi$ with determinant $\chi$ and for any $v\in S\setminus \Sigma_p$,
\[T|_{I_{F_v}}=\xi_v+\xi_v^{-1} .\]
\item For any $v|p$, $\psi_{v,1}:G_{F_v}\to R^\times$ is a character which lifts $\bar{\psi}_{v,1}$ and satisfies
\[T|_{G_{F_v}}=\psi_{v,1}+\psi_{v,1}^{-1}\chi.\]
Moreover $T$ is $\psi_{v,1}$-ordinary.
\end{itemize}
We remark that $R^{\ps,\ord,\{\xi_v\}}_{\Sigma^o}$ is a quotient of $R^{\ps,\ord}$ (defined in \ref{sasotmr}) as all $\psi_{v,1}$ already take values over $R^{\ps,\ord}$. There is a natural map $\Lambda_F\to R^{\ps,\ord,\{\xi_v\}}_{\Sigma^o}$ coming from the universal lifting of $\bar{\psi}_{v,1}$ when restricted to the inertia. Note that this might be \textit{different} from the one considered in \ref{sasotmr} unless $\Sigma^o$ is empty. Since we have fixed a complex conjugation $\sigma^*\in G_{F,S}$, we can attach a semi-simple representation $\rho(\kp):G_{F,S}\to \GL_2(k(\kp))$ for any $\kp\in\Spec R^{\ps,\ord,\{\xi_v\}}_{\Sigma^o}$ with trace $T^{univ}\mod\kp$ as in \ref{tar}. Here $T^{univ}:G_{F,S}\to R^{\ps,\ord,\{\xi_v\}}_{\Sigma^o}$ is the universal trace.
\end{para}

\begin{para} \label{ass1}
On the automorphic side, we put tame level $U^p=\prod_{v\notin\Sigma_p}U_v$ with $U_v=\GL_2(O_{F_v})$ if $v\notin S$ and $U_v=\mathrm{Iw}_v$ if $v\in S\setminus \Sigma_p$ as in \ref{satsotmr}. For any $v\in S\setminus \Sigma_p$, the map 
\[\begin{pmatrix}a&b\\c&d\end{pmatrix}\mapsto \xi_v(\frac{a}{d}\mod \varpi_v)\] 
defines a character of $U_v$, which we also denote by $\xi_v$ by abuse of notation. The product of $\xi_v$ can be viewed as a character $\xi$ of $U^p$ by projecting to $\prod_{v\notin S\setminus \Sigma_p}U_v$. By global class field theory, we may view $\psi=\chi\varepsilon$ as a character of $\AFi/F^\times_{>>0}$. Then we can consider a space of ordinary forms $M_{\psi,\xi}^{\ord}(U^p)$ and ordinary Hecke algebra $\T^{\ord}:=\T^{\ord}_{\psi,\xi}(U^p)$ as in the previous subsection. Recall that $\langle\,\rangle$ induces characters $\psi_{v,1}:G_{F_v}\to(\T^{\ord})^\times$ for $v|p$ and homomorphism $\Lambda_F\to \T^\ord$.

We make the following assumption here. We will see later that this assumption holds after making suitable field extension of $F$.
\begin{itemize}
\item \textbf{Assumption}: $T_v-(1+\chi(\Frob_v)),v\notin S$ and $\psi_{v,1}(\gamma)-\tilde{\psi}_{v,1}(\gamma),\gamma\in F_v^\times,v|p$ and $\varpi$ generate a maximal ideal $\km$ of $\T^\ord$. Here $\tilde{\psi}_{v,1}(\gamma)\in\cO$ is any lifting of $\bar{\psi}_{v,1}(\gamma)\in\F$.
\end{itemize}

As we recalled before, there is a two-dimensional pseudo-representation $T^{\ord}_\km:G_{F,S}\to \T^\ord_\km$ with determinant $\chi$ sending $\Frob_v$ to $T_v$. Moreover $T^\ord_\km$ is $\psi_{v,1}$-ordinary and $T_v|_{I_{F_v}}=\xi_v+\xi_v^{-1}$. These results can be checked on the finite level. Therefore we get  a map
\[R^{\ps,\ord,\{\xi_v\}}_{\Sigma^o}\to\T^{\ord}_{\km}\]
which is necessarily surjective since the topological generators $T_v,v\notin S$ and $\psi_{v,1}(F_v^\times),v|p$ are in the image. This map is also a $\Lambda_F$-algebra homomorphism.
\end{para}

\begin{para} \label{ordnice}
We say a prime $\kp\in\Spec R^{\ps,\ord,\{\xi_v\}}_{\Sigma^o}$ is \textit{pro-modular} if it comes from a prime of $\T^{\ord}_\km$. Recall that in \ref{nice}, we call a prime $\kq\in\Spec\T^{\ord}_\km$ \textit{nice} if $\T^\ord/\kq$ is one-dimensional of characteristic $p$ and there exists a two-dimensional representation
\[\rho(\kq)^o:G_{F,S}\to \GL_2(A)\]
satisfying the following properties:
\begin{enumerate}
\item $\rho(\kq)^o \otimes k(\kq)\cong  \rho(\kq)$ is irreducible. In other words, $\rho(\kq)^o$ is a lattice of $\rho(\kq)$.
\item The mod $\km_A$ reduction $\bar{\rho}_b$ of $\rho(\kq)^o$ is a non-split extension and has the form
\[\bar{\rho}_b(g)=\begin{pmatrix}*&*\\0&*\end{pmatrix},~g\in G_{F,S}.\]
Here $\km_A$ is the maximal ideal of $A$.
\item $\rho(\kq)^o|_{G_{F_v}}=\bar{\rho}_b|_{G_{F_v}}$ for any $v\in S\setminus \Sigma_p$ under the canonical inclusion map $\GL_2(\F)\to \GL_2(k(\kq))$.
\end{enumerate}

Note that this is slightly different from \ref{nice}: we drop the third condition in the original definition. This is because that condition is automatically true here:  by definition, the image of $\kq$ under the finite map $\Spec\T^{\ord}_\km \to \Spec \Lambda_F$ is not the maximal ideal. Hence $\psi_{v}\mod \kq$ is of infinite order for at least one $v|p$. It then follows from the third part of lemma \ref{nirred} that $\rho(\kq)$ is not dihedral.

The main technical result we are going to prove in this subsection is:
\end{para}

\begin{prop} \label{propA}
Under all the assumptions in this subsection, if $\kq\in\Spec\T^{\ord}_\km$ is a nice prime, then
\[(R^{\ps,\ord,\{\xi_v\}}_{\Sigma^o})_{\kq^\ps}\to \T^{\ord}_\kq\]
has nilpotent kernel. Here $\kq^{\ps}=\kq\cap R^{\ps,\ord,\{\xi_v\}}_{\Sigma^o}$.
\end{prop}

\begin{para}
The proof of proposition \ref{propA} is almost the same as the proof of Theorem \ref{thmA} with completed cohomology replaced by Hida family. We decide to omit the details here. Rather, we would like to point out some technical differences between these two cases.

\begin{enumerate}\label{rdxbq}
\item In the definition \ref{rxivloc} of $R^{\{\xi_v\}}_{\loc}$, we need to replace the unrestricted deformation rings $R^{\square}_v$, $v|p$ by \textit{ordinary deformation rings} $R^{\Delta}_v$ which pro-represents the functor from $\cOf$ to the category of sets sending $R$ to the set of pairs $(\rho_R,\psi_R)$ such that
\begin{itemize}
\item $\rho_R:G_{F_v}\to\GL_2(R)$ is a lifting of $\bar{\rho}_b|_{G_{F_v}}$ with determinant $\chi|_{G_{F_v}}$.
\item $\psi_R:G_{F_v}\to R^\times$ is a lifting of $\bar{\psi}_{v,1}$ such that $\rho_R$ has a (necessarily unique) $G_{F_v}$-stable rank one $R$-submodule, which is a direct summand of $R^{\oplus 2}$ as a $R$-module, with $G_{F_v}$ acting via $\psi_R$.
\end{itemize}
The second part of lemma \ref{varlocdim} is now replaced by:
\begin{itemize}
\item $R^{\Delta}_v$ is a \textit{normal} domain of dimension $1+2[F_v:\Q_p]+3$ and flat over $\cO$.
\end{itemize}
The proof will be given below (\ref{orddefprop1}). Once this is proved, proposition \ref{rlocp} in this case can be established in exactly the same way.

\item As for global deformation rings considered in \ref{globaldefr}, we also need to take quotients of them by ordinary conditions. More precisely, instead of using $R^{\square_P,\{\xi_v\}}_{\bar{\rho}_b,Q}$, we should consider
\[R^{\Delta_P,\{\xi_v\}}_{\bar{\rho}_b,Q}:=R^{\square_P,\{\xi_v\}}_{\bar{\rho}_b,Q}\otimes_{(\bigotimes_{v|p}R^{\square}_v)}(\bigotimes_{v|p}R^{\Delta}_v)\]
and unframed deformation rings
\[R^{\Delta,\{\xi_v\}}_{\bar{\rho}_b,Q}:=R^{\{\xi_v\}}_{\bar{\rho}_b,Q}\otimes_{(\bigotimes_{v|p}R^{\square}_v)}(\bigotimes_{v|p}R^{\Delta}_v).\]
We will omit $Q$ in the subscript if $Q$ is empty. Note that such ordinary conditions are not defined via traces, hence the natural surjective map
\[R^{\{\xi_v\}}_{\bar{\rho}_b}\otimes_{R^{\ps,\{\xi_v\}}}R_{\Sigma^o}^{\ps,\ord,\{\xi_v\}}\to R^{\Delta,\{\xi_v\}}_{\bar{\rho}_b}\]
might not be an isomorphism. Thus in the proof of corollary \ref{pscpatch}, we cannot apply corollary \ref{pscrem}, which is a consequence of proposition \ref{rpsrc}, directly. However, it follows from lemma \ref{ordpseudo} that the kernel is nilpotent and an easy computation shows that the kernel of the above map is killed by the element $c$ in proposition \ref{rpsrc}. Therefore corollary \ref{pscpatch} still can be proved in the same way.

\item In the patching argument, since $M^\ord_{\psi,\xi}(U^p)_\km$ is already finite over the Hecke algebra, we can work with the patched Hida family directly. In fact, the whole point of using Pa\v{s}k\={u}nas' theory in \ref{p3apop} is to make completed homology into some module which is finitely generated over the Hecke algebra. The role of $\Lambda_F$ in Hida's theory is replaced by $R^{\ps,\psi\varepsilon^{-1}}_p$ in the completed cohomology side (see Theorem \ref{lgc}).
\end{enumerate}
\end{para}

\begin{lem} \label{orddefprop1}
$R^{\Delta}_v$ is a \textit{normal} domain of dimension $1+2[F_v:\Q_p]+3$ and flat over $\cO$ if $[F_v:\Q_p]\geq2$. Recall that we made this assumption in our setup \ref{setupord1}.
\end{lem}

\begin{proof}
We will only deal with the case $\bar{\psi}_{v,1}= \mathbf{1}$ (the other case $\bar{\psi}_{v,1}=\bar{\chi}$ is similar). Since we assume $\bar{\chi}|_{G_{F_v}}\neq \mathbf{1}$, for any  lifting ($\rho_R,\psi_R)$ to some $R\in\cOf$ as in the definition of $R^{\Delta}_v$, we may find a unique matrix $U=\begin{pmatrix}1 &0 \\ n & 1 \end{pmatrix}$ such that $U^{-1}\rho_R U$ is of the form $\begin{pmatrix}\psi_R & * \\ 0 & * \end{pmatrix}$. Hence $R^{\Delta}_v\cong R^B_v[[n]]$, where $R^B_v$ parametrizes all liftings of $\bar{\rho}_b|_{G_{F_v}}$ into the Borel subgroup $B$ of upper triangular matrices with determinant $\chi|_{G_{F_v}}$. 

It is well known that we can write $R^{B}_v\cong \cO[[x_1,\cdots,x_g]]/(f_1,\cdots,f_r)$, where $g-r=2[F_v,\Q_p]+2$ and $r=\dim H^2(G_{F_v},\ad^0\bar{\rho}_b^B)$. Here $\bar{\rho}_b^B$ is $\bar{\rho}_b|_{G_{F_v}}$ viewed as a representation of $G_{F_v}$ into the upper triangular subgroup $B(\F)$ and  $\ad^0\bar{\rho}_b^B$ denotes the trace zero subspace of its adjoint representation. There is a canonical exact sequence as $G_{F_v}$-representations:
\[0\to \bar{\chi}^{-1}\to \ad^0\bar{\rho}_b^B \to \mathbf{1} \to 0.\]
Hence $r\leq 1$. The lemma is clear if $\bar{\chi}\neq \omega^{-1}$ as $r=0$ in this case.

Now we assume $\bar{\chi}= \omega^{-1}$. Note that from the expression $R^{B}_v\cong \cO[[x_1,\cdots,x_g]]/(f_1,\cdots,f_r)$, it is easy to see that each irreducible component of $R^B_v$ has dimension at least $2[F_v:\Q_p]+3$. We claim that $R_v^B/(\varpi)$ is normal of dimension $2[F_v:\Q_p]+2$. Assuming this, it is clear from comparing the dimensions that $\varpi$ is a regular element in $R_v^B$. This implies that $R_v^B$ is an $\cO$-flat normal algebra of dimension $2[F_v:\Q_p]+3$ by Serre's criterion for normality. Hence we are left to show our claim for $R_v^B/(\varpi)$. Since $R_v^B/(\varpi)\cong \F[[x_1,\cdots,x_g]]/(f_1,\cdots,f_r)$ with $r\leq 1$, we see that $R_v^B/(\varpi)$ is a complete intersection ring. It suffices to show that the non-regular locus has codimension at least $2$ and $\dim R^B_v/(\varpi)=2[F:\Q_p]+2$.

Write $R_1=R_v^B/(\varpi)$. Suppose $\kq\in\Spec R_1$ is a one-dimensional prime. We use $\rho(\kq)^B:G_{F_v}\to B(k(\kq))$ to denote the push-forward of the universal lifting over $R_v^B$ to $k(\kq)$. We need the following result concerning the smoothness of $(R_1)_{\kq}$.

\begin{lem} \label{1dimmoduli}
$(R_1)_{\kq}$ is regular of dimension $2[F_v:\Q_p]+1$ if $\rho(\kq)^B$ is not isomorphic to the form $\begin{pmatrix}1 &* \\ 0 & \omega^{-1} \end{pmatrix}$.
\end{lem}

\begin{proof}
We will relate the $\kq$-adic completion of $(R_1)_{\kq}$ with some other universal lifting ring. First by enlarging $\F$ if necessary, we may assume the normal closure of $R_1/\kq$ is isomorphic to $A=\F[[T]]$ (as $\F$-algebras). Fix such an isomorphism and hence an embedding $A\hookrightarrow k(\kq)$.

Clearly $R_2:=R_1\hat{\otimes}_\F A\cong R_1[[T_1]]$ pro-represents the functor assigning each Artinian local $A$-algebra $R$ with residue field $\F$ to the set of continuous homomorphisms of $G_{F_v}$ to $B(R)$ with determinant $\chi$ that lift $\bar\rho_b^B$. Viewed as taking values in $B(A)$, the representation $\rho(\kq)^B$ induces a prime ideal $\kp$ of $R_2$. We claim that the $\kp$-adic completion $\widehat{(R_2)_\kp}$ is the universal lifting ring for $\rho(\kq)^B:G_{F_v}\to\GL_2(k(\kq))$ (with determinant $\chi$), i.e. if $R$ is an Artinian local $k(\kq)$-algebra with residue field $k(\kq)$ and if $\rho^B:G_{F_v}\to B(R)$ is a continuous lifting of $\rho(\kq)^B$ with determinant $\chi$, then there exists a unique map of $k(\kq)$-algebras $\widehat{(R_2)_\kp}\to R$ such that the push forward of the universal lifting on $R_2$ is $\rho^B$. The proof is standard: let $R^0$ be the $A$-subalgebra of $R$ generated by the matrix entries of $\rho^B$. This is a finite local $A$-algebra with residue field $\F$ and $\rho^B$ can be viewed as taking values in $B(R^0)$. Hence we get a natural map $R_2\to R^0$ and our claim follows easily. In particular, under our assumption on $\rho(\kq)^B$, the deformation ring $\widehat{(R_2)_\kp}$ is formally smooth of relative dimension $2[F_v:\Q_p]+2$ over $k(\kq)$.

Let $a\in (R_1)_{\kq}$ be a lifting of $T\in A\hookrightarrow (R_1)_{\kq}/\kq$. Using the natural map $(R_1)_\kq \to (R_2)_\kp$, we also view $a$ as an element of $\widehat{(R_2)_\kp}$ by abuse of notation. We claim that
\begin{enumerate}
\item The map sending $T_1$ to $a$ induces an isomorphism $\widehat{(R_2)_\kp}/(T_1-a)\stackrel{\sim}{\longrightarrow} \widehat{(R_1)_\kq}$.
\item $T_1-a\notin (\kp(R_2)_\kp)^2$.
\end{enumerate}
Note that these two claims imply lemma \ref{1dimmoduli}. For the second assertion, we have
\[(R_2/\kq R_2)\otimes_{R_1} k(\kq)\cong (R_1/\kq)[[T_1]]\otimes_{R_1} k(\kq)\cong \F[[T,T_1]][\frac{1}{T}].\]
Under the above isomorphism, $T_1-a$ is identified with $T_1-T$ in $\F[[T,T_1]][\frac{1}{T}]$ and the image of $\kp$ is generated by $T_1-T\neq 0$. From this explicit description, the assertion is clear.

For the first claim, we use $B_\kq$ to denote $B\otimes_{R_1} (R_1)_{\kq}$ for any $R_1$-algebra $B$. It follows from the previous paragraph that $\kp(R_1[[T_1]])_\kq=(\kq,T_1-a)$ is a maximal ideal in $(R_1[[T_1]])_\kq$. Hence it suffices to prove that for any integer $n>0$,
\[(R_1[[T_1]])_\kp/(\kq^n,T_1-a)=(R_1[[T_1]])_\kq/(\kq^n,T_1-a)\stackrel{\sim}{\longrightarrow} (R_1/\kq^n)_\kq.\]
It is easy to see that $T^k\in R_1/\kq$ for some integer $k$ as $T\in k(\kq)$. Hence there exists some element $b$ in the maximal ideal of $R_1$ such that $a^k-b\in \kq (R_1)_{\kq}$. Thus $(T_1^k-b)^n=0$ in $(R_1[[T_1]])_\kq/(\kq^n,T_1-a)$. Using this, we get
\[(R_1[[T_1]])_\kq/(\kq^n,T_1-a)\cong (R_1[T_1])_\kq/(\kq^n,T_1-a)\cong (R_1/\kq^n)_\kq.\]
\end{proof}

Back to the proof of lemma \ref{orddefprop1}. The locus where $\rho(\kq)^B$ is of the form $\begin{pmatrix}1 &* \\ 0 & \omega^{-1} \end{pmatrix}$ has dimension $\dim_\F Z^1(G_{F_v},\omega)=[F_v:\Q_p]+2\leq 2[F_v:\Q_p]$. Since each irreducible component of $R_1$ has dimension at least $2[F_v:\Q_p]+2$, we conclude from lemma \ref{1dimmoduli} that the non-regular locus has codimension at least $2$. This finishes the proof of lemma \ref{orddefprop1}.
\end{proof}

\subsection{\texorpdfstring{$R^\ord_b$}{Lg} is pro-modular for one \texorpdfstring{$b$}{Lg}} \label{Rordbimf1b}
It follows from the main result of the previous subsection that if we could show any irreducible component of $R^{\ps,\ord,\{\xi_v\}}$ has a nice prime, then any prime of $R^{\ps,\ord,\{\xi_v\}}$ would be pro-modular. However, the geometry of $\Spec R^{\ps,\ord,\{\xi_v\}}$ seems a bit hard to control. On the other hand, the ring-theoretic property of the usual global deformation ring $R^{\ord}_{b}$ (see the precise definition below) can be described by the Galois cohomology. The goal of this subsection is to show that under certain assumptions, the image of $\Spec R^{\ord}_{b}$ in $\Spec R^{\ps,\ord,\{\xi_v\}}$ is pro-modular for one extension class $b$.

\begin{para}
Keep all the notations $F,D,\chi,S,\Sigma^o,\sigma^*$ as in the previous subsection. We will define some global deformation ring $R^{\ord}_{b}$ and discuss its connectedness dimension and the dimension of its reducible locus. We first introduce the following Selmer group. Recall that $\Sigma^o$ is a subset of $\Sigma_p$.
\end{para}

\begin{defn} \label{selmext}
\[H^1_{\Sigma^o}(F):=\ker(H^1(G_{F,S},\F(\bar{\chi}^{-1}))\stackrel{\mathrm{res}}{\longrightarrow}\bigoplus_{v\in \Sigma_p\setminus \Sigma^o}H^1(G_{F_v},\F(\bar{\chi}^{-1}))).\]
\end{defn}

It is easy to that there is a bijection between $H^1_{\Sigma^o}(F)$ and the group of extensions 
\[0\to \mathbf{1}\to \bar{\rho} \to \bar{\chi} \to 0\]
as $\F[G_{F,S}]$-modules such that $\bar{\rho}$ is $\bar{\psi}_{v,1}$-ordinary for $v|p$. The bijection is given as follows: any cohomology class $b=[\phi_b]\in H^1_{\Sigma^o}(F)$ defines a two-dimensional representation
\[\bar{\rho}_b:G_{F,S}\to \GL_2(\F),~g\mapsto \begin{pmatrix} 1 & \phi_b(g)\bar{\chi}(g)\\0& \bar{\chi}(g)\end{pmatrix},\]
which can be viewed as an extension of $\bar\chi$ by $\mathbf{1}$. Since $\bar{\chi}(\sigma^*)=-1$, we can always choose $\phi_b$ so that $\phi_b(\sigma^*)=0$. We will keep this assumption for $\bar{\rho}_b$ from now on.

\begin{para} \label{Rbpsi}
For any non-zero $b\in H^1_{\Sigma^o}(F)$, we will write $R^{\ord}_{b}$ for $R^{\Delta,\{\xi_v\}}_{\bar{\rho}_b}$ introduced in the second part of \ref{rdxbq}. More precisely, it pro-represents the functor from $\cOf$ to the category of sets sending $R$ to the set of tuples $(\rho_R,\{\psi_{v,1}\}_{\Sigma_p})$ where
\begin{itemize}
\item $\rho_R:G_{F,S}\to\GL_2(R)$ is a lifting of $\bar{\rho}_b$ with determinant $\chi$.
\item For $v|p$, $\psi_{v,1}:G_{F_v}\to R^\times$ is a lifting of $\bar{\psi}_{v,1}$ such that $\rho_R$ has a (necessarily unique) $G_{F_v}$-stable rank one $R$-submodule, which is a direct summand of $\rho_R$ as a $R$-module, with $G_{F_v}$ acting via $\psi_R$.
\end{itemize}
Recall that in $\ref{rdxbq}$, we also introduce the framed ordinary deformation ring $R^{\Delta_S,\{\xi_v\}}_{\bar{\rho}_b}$ with framing at $S$. Then there is a non-canonical isomorphism by choosing a universal framing
\[R^{\Delta_S,\{\xi_v\}}_{\bar{\rho}_b}\cong R^{\ord}_{b}[[y_1,\cdots,y_{4|S|-1}]].\]

Let $R^{\ord}_{\loc}$ be the following completed tensor product over $\cO$:
\[(\widehat{\bigotimes}_{v\in \Sigma_p}R^{\Delta}_v)\widehat{\otimes}(\widehat{\bigotimes}_{v\in S\setminus \Sigma_p}R^{\square,\xi_v}_v),\]
where $R^{\Delta}_v$ is defined in the first part of \ref{rdxbq} and $R^{\square,\xi_v}_v$ is defined in \ref{ldfr}. It is well-known (for example see corollary 2.3.5 of \cite{CHT08}) that $R^{\Delta_S,\{\xi_v\}}_{\bar{\rho}_b}$ can be written as the form
\[R^{\Delta_S,\{\xi_v\}}_{\bar{\rho}_b}\cong R^{\ord}_{\loc}[[x_1,\cdots,x_{g_1}]]/(h_1,\cdots,h_{r_1})\]
where $g_1=\dim_\F H^1(G_{F,S},\ad^0\bar{\rho}_b(1))+|S|-1-[F:\Q]-\dim_\F H^0(G_{F,S},\ad^0\bar{\rho}_b(1))$ and $r_1=\dim_\F H^1(G_{F,S},\ad^0\bar{\rho}_b(1))$.
\end{para}

\begin{lem} \label{grrcord}
Each irreducible component of $R^{\ord}_b$ has dimension at least $[F:\Q]+1$ if $H^0(G_{F,S},\ad^0\bar{\rho}_b(1))=0$ and $[F:\Q]$ in general. Moreover, $R^{\ord}_b$ can be written as the form $R_p[[x_1,\cdots,x_{g_2}]]/(f_1,\cdots,f_{r_2})$ with 
\[g_2-r_2\geq -|S|-2|\Sigma_p|-[F:\Q]-1\]
and $R_p$ a domain of dimension $1+2[F:\Q]+3|\Sigma_p|$.
\end{lem}
\begin{proof}
By the first part of \ref{rdxbq}, lemma \ref{lrca} and lemma 3.3 of \cite{BLGHT11}, we know that $R^\ord_\loc$ is equidimensional of dimension $2[F:\Q]+3|S|+1$. Since $\dim H^0(G_{F,S},\ad^0\bar{\rho}_b(1))\leq 1$. The first claim now follows from a simple computation.

For the second part, let $R_p=\widehat{\otimes}_{v|p}R^{\Delta}_v$. This is a domain of dimension $1+2[F:\Q]+3|\Sigma_p|$ by lemma \ref{orddefprop1}. Note that for $v\in S\setminus \Sigma_p$, $R^{\square,\xi_v}_v$ is a quotient of the unrestricted deformation ring $R^{\square}_v$ by one equation $\tr(t)=\xi(t)+\xi^{-1}(t)$, where $t$ is a topological generator of the pro-$p$ quotient of $I_{F_v}$. On the other hand, it is well-known that $R^{\square}_v$ has the form $\cO[[x_1,\cdots,x_{g_v}]]/(f_1,\cdots,f_{r_v})$. Here $g_v=\dim_\F H^1(G_{F_v},\ad^0\bar{\rho}_b)+3-\dim_\F H^0(G_{F_v},\ad^0\bar{\rho}_b)$ and $r_v=\dim_\F H^2(G_{F_v},\ad^0\bar{\rho}_b)$. It follows from local Euler characteristic formula that $g_v-r_v=3$. From this, it is easy to deduce the second claim in the lemma.
\end{proof}

\begin{para}[\underline{\textbf{Connectedness dimension}}] \label{connectednessdim}
In order to show that there are enough `nice' primes, we need the notion of \textit{connectedness dimension}. Reference here is \cite{BR86}.
\end{para}
\begin{defn}
Let $R$ be a complete noetherian local ring. The connectedness dimension of $R$ is defined to be 
\[c(R)=\min_{Z_1,Z_2}\{\dim (Z_1\cap Z_2)\}\] 
where $Z_1,Z_2$ are non-empty unions of irreducible components of $\Spec R$ such that $Z_1\cup Z_2=\Spec R$.
\end{defn}

Suppose $I$ is an ideal of a complete noetherian local ring $R$. The arithmetic rank $r(I)$ is defined as the minimal number of elements in $I$ which span an ideal with the same radical as $I$. The following result is Theorem 2.4 of \cite{BR86}.
\begin{prop} \label{cndimformula}
Let $I,R$ be as above. Then $c(R/I)\geq c(R)-r(I)-1$.
\end{prop}

Combining this with lemma \ref{grrcord}, we get
\begin{cor} \label{crord}
$c(R^{\ord}_b/\varpi)\geq [F:\Q]-|S|+|\Sigma_p|-2$.
\end{cor}

\begin{para}[\textbf{Reducible locus of }$R^{\ord}_b$] \label{redlocrc}
Let $c\in H^1_{\Sigma^o}(F)\setminus\{0\}$. We choose a universal deformation 
\[\rho^{univ}:G_{F,S}\to\GL_2(R^{\ord}_b)~g\mapsto \begin{pmatrix} \tilde{a}(g) & \tilde b(g) \\ \tilde c(g) & \tilde d(g)\end{pmatrix}\]
so that $\rho^{univ}(\sigma^*)=\begin{pmatrix}1&0\\0&-1\end{pmatrix}$.

It is clear that the set of $\kp\in\Spec R^{\ord}_b$ such that $\rho^{univ}\otimes k(\kp)$ is reducible is closed in $\Spec R^{\ord}_b$. In fact, it is generated by the set of $x(\sigma,\tau)$ (see the discussion and notation in \ref{tar}). We denote its reduced subscheme by $\Spec R^{\mathrm{red}}_b$ and call it the reducible locus of $\Spec R^{\ord}_b$.

The dimension of reducible locus is crucial in our later analysis of $\Spec R^{\ord}_b$. An upper bound is given in lemma 2.7 of \cite{SW99}. We recall their argument here.

Let $F(\bar{\chi})$ be the splitting field of $\bar{\chi}$ and $\tilde S$ be the set of places of $F(\bar{\chi})$ above $S$. Then
\[H^1(G_{F,S},\F(\bar{\chi}^{-1}))=\Hom_{\Gal(F(\bar{\chi})/F)}(G_{F(\bar{\chi}),\tilde S}, \F(\bar{\chi}^{-1})).\]
In other words, we get a pairing $H^1(G_{F,S},\F(\bar{\chi}^{-1}))\times G_{F(\bar{\chi}),\tilde S}\to\F$. 
Denote $\dim_\F H^1_{\Sigma^o}(F)$ by $r_s$. Then we can choose elements $\sigma_1,\cdots,\sigma_{r_s}\in G_{F(\bar{\chi}),\tilde S}$ such that they form a basis of $H^1_{\Sigma^o}(F)^\vee$ under this pairing.

On the other hand, let $G_{F,S}^{\mathrm{ab}}(p)$ be the maximal pro-$p$ abelian quotient of $G_{F,S}$. We denote its $\Z_p$-rank by $\delta_F+1$ and choose a set of elements $\tau_0,\cdots,\tau_{\delta_F}\in G_{F(\bar{\chi}),\tilde S}$ whose images in $G_{F,S}^{\mathrm{ab}}(p)\otimes_{\Z_p} \Q_p$ form a basis. Since $F\neq \Q$, by the class field theory we have a trivial bound $\delta_F\leq [F:\Q]-2$.
\end{para}

\begin{prop} \label{redloc}
Let $I$ be the ideal of $R^{\mathrm{red}}_b$ generated by $\varpi$ and the image of elements 
\begin{itemize}
\item $\tilde a(\tau_i)+\tilde d(\tau_i)-2$, $i=0,\cdots,\delta_F$,
\item $\tilde{b}(\sigma_i)-\tilde{b_i}$, $i=1,\cdots, r_s$, where $\tilde{b_i}\in\cO$ is a lifting of the reduction of $\tilde{b}(\sigma_i)$ modulo the maximal ideal of $R^{\ord}_b$.
\end{itemize}
Then $R^{\mathrm{red}}_b/I$ has finite length. In particular, $\dim R^{\mathrm{red}}_b/(\varpi)\leq \delta_F+\dim_\F H^1_{\Sigma^o}(F)+1$.
\end{prop}

\begin{proof}
Let $\kp$ be a prime of $R^{\mathrm{red}}_b/I$. It suffices to prove $\rho(\kp)=\rho^{univ}\mod \kp$ is a trivial deformation. First note that the semi-simplification of $\rho(\kp)$ is a sum of two characters $\chi_1,\chi_2$ of $G_{F,S}$. By the definition of $I$, $\chi_1(\tau_i)+\chi_2(\tau_i)=2$ and $\chi_1(\tau_i)\chi_2(\tau_i)=\bar{\chi}(\tau_i)=1$. Hence $\chi_1(\tau_i)=\chi_2(\tau_i)=1$. Therefore $\chi_1,\chi_2$ are of finite orders and have to be $\mathbf{1},\bar{\chi}$.

Thus $\rho(\kp)|_{G_{F(\bar{\chi}),\tilde S}}$ is unipotent. Since $\rho^{univ}(\sigma^*)=\begin{pmatrix}1&0\\0&-1\end{pmatrix}$, we must have
\[\rho(\kp)|_{G_{F(\bar{\chi}),\tilde S}}=\begin{pmatrix} 1 & *\\ 0 & 1\end{pmatrix}.\]
Note that $*$ can be viewed as an element in 
\[H^1_{\Sigma^o}(F)\otimes k(\kp)\subseteq \Hom_{\Gal(F(\bar{\chi})/F)}(G_{F(\bar{\chi}),\tilde S}, k(\kp)(\bar{\chi}^{-1})).\] 
It follows from our construction of $I$ that $\rho(\kp)=\bar{\rho}_b\otimes k(\kp)$.
\end{proof}

\begin{para}[\textbf{Reducible locus of }$R^{\ps,\ord,\{\xi_v\}}_{\Sigma^o}$] \label{redlrps}
Similarly the set of $\kp\in\Spec R_{\Sigma^o}^{\ps,\ord,\{\xi_v\}}$ such that $\rho(\kp)$ is reducible is closed in $\Spec R^{\ps,\ord,\{\xi_v\}}_{\Sigma^o}$. We denote its reduced subscheme by $\Spec R^{\ps,\mathrm{red}}$. Note that the universal pseudo-character becomes a sum of two characters when restricted on $\Spec R^{\ps,\mathrm{red}}$. One is a lifting of $\bar{\chi}$ and the other one is a lifting of $\mathbf{1}$. The latter one induces a surjective map $\cO[[G_{F,S}^\mathrm{ab}(p)]]\to R^{\ps,\mathrm{red}}$. Hence $\dim R^{\ps,\mathrm{red}}\leq\delta_F+2$.
\end{para}

\begin{para}
Now we can state and prove the main result of this subsection. For any non-zero $b\in H^1_{\Sigma^o}(F)$, we say a prime $\kp\in\Spec R^{\ord}_b$ is \textit{pro-modular} if its image in $\Spec R^{\ps,\ord,\{\xi_v\}}_{\Sigma^o}$ is pro-modular. We say a subset $Z\subseteq \Spec R^{\ord}_b$ is pro-modular if any prime of $Z$ is pro-modular. Note that the set of pro-modular primes is closed in $\Spec R^{\ord}_b$.

Similarly, we say a prime $\kp\in\Spec R^{\ord}_b$ is \textit{nice} if its image in $\Spec R^{\ps,\ord,\{\xi_v\}}_{\Sigma^o}$ is nice (defined in \ref{ordnice}). Then proposition \ref{propA} implies that an irreducible component of $\Spec R^{\ord}_b$ is pro-modular if it contains a nice prime.
\end{para}

\begin{prop}  \label{rcmodc}
Let $b\in H^1_{\Sigma^o}(F)\setminus\{0\}$. Assume
\begin{itemize}
\item There exists a pro-modular prime $\kp\in \Spec R^{\ord}_b$ such that $R^{\ord}_b/\kp$ is one-dimensional and $\rho^{univ}\otimes k(\kp)$ is irreducible. (This prime ideal $\kp$ could be characteristic $0$ or $p$.) In particular, the assumption in \ref{ass1} holds.
\item $[F:\Q]-4|S|+4|\Sigma_p|-3>\delta_F+\dim_\F H^1_{\Sigma^o}(F)$.
\end{itemize}
Then $R^{\ord}_b$ is pro-modular, i.e. any prime of $R^{\ord}_b$ is pro-modular.
\end{prop}

\begin{proof}
Let $\kp'\in\Spec R^{\ps,\ord,\{\xi_v\}}_{\Sigma^o}$ be the image of $\kp$. This is still one-dimensional as $\rho^{univ}\otimes k(\kp)$ is irreducible. By corollary \ref{orddim}, the pro-modular locus in $\Spec R^{\ps,\ord,\{\xi_v\}}_{\Sigma^o}$ is equidimensional of dimension $[F:\Q]+1$. In particular, we can choose a pro-modular prime $\mathfrak{P}$ contained in $\kp'$ with $\dim (R^{\ps,\ord,\{\xi_v\}}_{\Sigma^o}/\mathfrak{P})_{\kp'}=[F:\Q]$. Notice that
\[\dim (R^{\ps,\ord,\{\xi_v\}}_{\Sigma^o}/\mathfrak{P})_{\kp'}=\dim \widehat{(R^{\ps,\ord,\{\xi_v\}}_{\Sigma^o})_{\kp'}}/(\kP)=\dim \widehat{(R^{\ord}_b)_{\kp}}/(\kP)= \dim (R^{\ord}_b/(\kP))_{\kp} \]
where the equality in the middle follows from corollary \ref{pscrem} and the discussion in the second part of \ref{rdxbq}. Thus we may find a pro-modular prime $\kQ$ of $R^{\ord}_b$ such that $\dim R^{\ord}_b/\kQ=1
+[F:\Q]$. I claim we can find \textit{one} nice prime containing $\kQ$. Hence at least \textit{one} irreducible component of $R^{\ord}_b$ is pro-modular.

For $v\in S\setminus \Sigma_p$, we know that $R^{\square,\xi_v}_v/(\varpi)$ is $3$-dimensional by lemma \ref{lrca}. Hence we may find $f_{v,1},f_{v,2},f_{v,3}$ in $R^{\square,\xi_v}_v$ that form a system of parameters of $R^{\square,\xi_v}_v/(\varpi)$. Consider the quotient $R'$ of $R^{\ord}_b/\kQ$ by $\varpi$ and all such $f_{v,1},f_{v,2},f_{v,3}$, $v\in S\setminus \Sigma_p$. It has dimension at least
\[[F:\Q]-3(|S|-|\Sigma_p|).\]

On the other hand, it follows from proposition \ref{redloc}, that the reducible locus in the special fibre $R^{\mathrm{red}}_b/(\varpi)$ has dimension at most $\delta_F+\dim_\F H^1_{\Sigma^o}(F)+1$, which is less than $[F:\Q]-3(|S|-|\Sigma_p|)$ by our assumption. Hence there must exist a one-dimensional prime $\kq\in\Spec R'$ such that $\rho^{univ}\otimes k(\kq)$ is irreducible. It is easy to see that $\kq$ is nice in view of the definition of `nice' above proposition \ref{propA}.

Now let $Z_1$ be the union of pro-modular irreducible components of the special fibre $R^{\ord}_b/(\varpi)$ and $Z_2$ be the union of other irreducible components of the special fibre. It suffices to show $Z_2$ is empty. Suppose not, we have already seen $Z_1$ is non-empty, hence by corollary \ref{crord}, 
\[\dim Z_1\cap Z_2\geq [F:\Q]-|S|+|\Sigma_p|-2.\]
Note that by our assumption, this is larger than $3(|S|-|\Sigma_p|)+\dim R^{\mathrm{red}}_b/(\varpi)$. Therefore arguing as above, we can find a nice prime in $Z_1\cap Z_2$. This implies that some irreducible component in $Z_2$ is also pro-modular. Contradiction. Thus $R^{\ord}_b$ is pro-modular.
\end{proof}

\begin{cor} \label{rcmod1c}
Assume 
\begin{itemize}
\item (Assumption in \ref{ass1}) $T_v-(1+\chi(\Frob_v)),v\notin S$ and $\psi_{v,1}(\gamma)-\tilde{\psi}_{v,1}(\gamma),\gamma\in F_v^\times,v|p$ and $\varpi$ generate a maximal ideal $\km$ of $\T^\ord$. Here $\tilde{\psi}_{v,1}(\gamma)\in\cO$ is any lifting of $\bar{\psi}_{v,1}(\gamma)\in\F$.
\item $[F:\Q]-4|S|+4|\Sigma_p|-3>\delta_F+\dim_\F H^1_{\Sigma^o}(F)$.
\end{itemize}
Enlarge $\cO$ if necessary, then there exists a non-zero $b\in H^1_{\Sigma^o}(F)$ such that $R^{\ord}_b$ is pro-modular.
\end{cor}
\begin{proof}
It is clear that we only need to find an element $b\in H^1_{\Sigma^o}(F)$ that satisfies the conditions in the previous proposition. By our first assumption and corollary \ref{orddim}, the pro-modular locus in $\Spec R^{\ps,\ord,\{\xi_v\}}_{\Sigma^o}$ has dimension $[F:\Q]+1$. Since the reducible locus has dimension at most $\delta_F+2\leq [F:\Q]$, we may choose a one-dimensional pro-modular prime $\kp'$ such that $\rho(\kp)$ is irreducible. Then after enlarging $\cO$ if necessary, we can find a representation $G_{F,S}\to\GL_2(A)$, where $A$ is  normalization of $R^{\ps,\ord,\{\xi_v\}}_{\Sigma^o}/\kp'$, such that its residual representation gives rise to some non-zero $b\in H^1_{\Sigma^o}(F)$. Clearly it has the properties we want.
 \end{proof}

\subsection{\texorpdfstring{$R^{\ord}_{b}$}{Lg} is pro-modular for any \texorpdfstring{$b$}{Lg}} \label{Rordbmodab}
The goal of this subsection is to extend the result in the previous subsection to any non-zero extension class $b\in H^1_{\Sigma^o}(F)$ (under certain assumptions).

\begin{prop} \label{propB}
Keep all the notations as in the previous subsection and all the assumptions in corollary \ref{rcmod1c}. Then after replacing $E$ by some finite extension, $R^{\ord}_b$ is pro-modular for any non-zero $b\in H^1_{\Sigma^o}(F)$.
\end{prop}

\begin{proof}
First by corollary \ref{rcmod1c}, we may assume there exists a non-zero class $b_1\in H^1_{\Sigma^o}(F)$ such that $R^{\ord}_{b_1}$ is pro-modular (after possibly enlarging $\cO$). Now let $b_2\in H^1_{\Sigma^o}(F)$ be another class in $H^1_{\Sigma^o}(F)\setminus \{\F b_1\}$. In view of proposition \ref{rcmodc}, it suffices to find a pro-modular one-dimensional prime of $R^{\ps,\ord,\{\xi_v\}}_{\Sigma^o}$ whose associated representation $\rho(\kp)$ is irreducible and admits a lattice $\rho(\kp)^o$ with residue representation corresponding to the extension class $b_2\in H^1_{\Sigma^o}(F)$.

By our assumption, we may extend $b_1,b_2$ to a basis $b_1,\cdots,b_{r_s}$ of $H^1_{\Sigma^o}(F)$. Write $b_i=[\phi_{b_i}]$, $\phi_{b_i}:G_{F,S}\to \F(\bar\chi^{-1})$. Also we can choose elements $\sigma_1,\cdots,\sigma_{r_s}\in G_{F(\bar{\chi}),\tilde S}$ such that 
\[\phi_{b_i}(\sigma_j)=\delta_{ij}\]
under the pairing before proposition \ref{redloc}.

Consider the following deformation of $\bar{\rho}_{b_1}$:
\[\rho_{12}:G_{F,S}\to \GL_2(\F[[T]]),~\sigma\mapsto \begin{pmatrix}1 & \bar{\chi}(\sigma)(\phi_{b_1}(\sigma)+\phi_{b_2}(\sigma)T)\\ 0 & \bar{\chi}(\sigma) \end{pmatrix}.\]
One way to think about this deformation is that it becomes $\bar\rho_{b_2}$ when $T$ ``goes to $\infty$''. This gives rise to a one-dimensional prime $\kq_{12}$ of $R^{\ord}_{b_1}$. As in \ref{redlocrc}, we may choose a universal deformation 
\[\rho^{univ}:G_{F,S}\to\GL_2(R^{\ord}_{b_1})~g\mapsto \begin{pmatrix} \tilde{a}(g) & \tilde b(g) \\ \tilde c(g) & \tilde d(g)\end{pmatrix}\]
so that $\tilde b(\sigma_1)=1$, $\rho^{univ}(\sigma^*)=\begin{pmatrix}1&0\\0&-1\end{pmatrix}$ and the reduction of $\rho^{univ}$ modulo $\kq_{12}$ is $\rho_{12}$. It is clear that $\tilde{b}(\sigma_i)\in\kq_{12},i=3,\cdots,r_s$. Let $\kQ$ be a minimal prime of $R^{\ord}_{b_1}/(\tilde{b}(\sigma_3),\cdots,\tilde{b}(\sigma_{r_s}))$ contained in $\kq_{12}$. Then by lemma \ref{grrcord}, 
\[\dim R^{\ord}_{b_1}/\kQ\geq [F:\Q]-(r_s-2)=[F:\Q]+2-\dim_{\F} H^1_{\Sigma^o}(F).\]
We claim that $\rho^{univ}\otimes k(\kQ)$ is irreducible. If not, since $\tilde{b}(\sigma_i),i=3,\cdots,r_s$ are already contained in $\kQ$, it follows from proposition \ref{redloc} that $R^{\ord}_{b_1}/\kQ$ has dimension at most $1+\delta_F+3=\delta_F+4$. But by our second assumption 
\[[F:\Q]+2-\dim_{\F} H^1_{\Sigma^o}(F)>4|S|-4|\Sigma_p|+\delta_F+5\geq \delta_F+5.\]
In other words, this upper bound for $\dim R^{\ord}_{b_1}/\kQ$ contradicts the previous lower bound. Therefore $\kQ$ is not in the reducible locus. Hence let $\kQ'=\kQ\cap R^{\ps,\ord,\{\xi_v\}}_{\Sigma^o}$. We may apply the second part of corollary \ref{psccomp} and conclude
 \[\dim R^{\ps,\ord,\{\xi_v\}}_{\Sigma^o}/\kQ'\geq \dim R^{\ord}_{b_1}/\kQ\geq [F:\Q]+2-\dim_{\F} H^1_{\Sigma^o}(F).\]

Consider the ideal $I_{b_1}$ of $R^{\ps,\ord,\{\xi_v\}}_{\Sigma^o}/\kQ'$ generated by elements $x(\sigma_1,\tau),\tau\in G_{F,S}$ defined by the universal pseudo-character $T^{univ}:G_{F,S}\to R^{\ps,\ord,\{\xi_v\}}_{\Sigma^o}$. See \ref{adx} and the beginning of the proof of proposition \ref{rpsrc} for the precise definition. These elements have the property that $x(\sigma_1,\tau)$ maps to $\tilde{b}(\sigma_1)\tilde{c}(\tau)=\tilde{c}(\tau)$ in $R^{\ord}_{b_1}$. It follows from proposition \ref{ht1def} that $\mathrm{ht}(I_{b_1})\leq 1$, hence
\[\dim R^{\ps,\ord,\{\xi_v\}}_{\Sigma^o}/(\varpi,\kQ',I_{b_1})\geq [F:\Q]-\dim_{\F} H^1_{\Sigma^o}(F).\]
Note that we may apply proposition \ref{ht1def} here as $\kQ$ is contained in $\kq_{12}$ which is a one-dimensional prime ideal of $R^{\ord}_{b_1}$ mapping to the maximal ideal of $R^{\ps,\ord,\{\xi_v\}}_{\Sigma^o}$. Since
\[[F:\Q]-\dim_\F H^1_{\Sigma^o}(F)> \delta_F+2\geq \dim R^{\ps,\mathrm{red}}\]
by our assumption, we can find a one-dimensional prime $\kp$ containing $\varpi,\kQ',I_{b_1}$ such that  $\rho(\kp)$ is irreducible. It has to be pro-modular as $\kQ'$, the image of $\kQ$, is pro-modular. We claim that there is a lattice of $\rho(\kp)$ such that its residue representation belongs to the extension class $b_2\in H^1_{\Sigma^o}(F)$.

Let $A=\tilde{\F}[[T]]$ be the normalization of $R^{\ps,\ord,\{\xi_v\}}_{\Sigma^o}/\kp$, where $\tilde{\F}$ is a finite extension of $\F$. Then we may find a lattice of $\rho(\kp)$: $\rho^o:G_{F,S}\to\GL_2(A)$ such that the reduction of $\rho^o\mod T$ has the form
\[\begin{pmatrix}1 & * \\ 0 & \bar{\chi} \end{pmatrix},~*\neq 0.\]
and $\rho^o(\sigma^*)=\begin{pmatrix}1 & 0\\ 0 & -1 \end{pmatrix}$ \footnote{As pointed out by one referee, this is usually called Ribet's lemma, cf. Proposition 2.1. of \cite{Rib76}. We will freely use this result here.}. Write $\rho^o(g)=\begin{pmatrix} a'(g) & b'(g) \\ c'(g) & d'(g) \end{pmatrix}$. Since $\rho^o$ is irreducible, we may find some $\tau'\in G_{F,S}$ such that $c'(\tau)\neq0$. Now for $i=3,\cdots,r_s$, the image of $x(\sigma_i,\tau')$ in $R^{\ord}_{b_1}$ is
\[\tilde{b}(\sigma_i)\tilde{c}(\tau')\in \kQ,\]
hence $x(\sigma_i,\tau')\in\kQ'$ and $b'(\sigma_i)c'(\tau')=x(\sigma_i,\tau')=0$ in $A$. Therefore $b'(\sigma_i)=0,i=3,\cdots,r_s$.

Similarly $b'(\sigma_1)c'(\tau)=0$ in $A$ as $x(\sigma_1,\tau')\in I_{b_1}$ by our construction. Thus $b'(\sigma_1)=0$.

Hence $b'(\sigma_i)=0$ unless $i\neq 2$. Since we assume the reduction $\bar{\rho}$ of $\rho^o$ is non-split, it follows from the discussion below definition \ref{selmext} that after possibly conjugating $\rho^o$ by element of the form $\begin{pmatrix} n & 0\\ 0 & 1\end{pmatrix},n\in A^\times$, the reduction $\bar{\rho}$ has to belong to the extension class $b_2\in H^1_{\Sigma^o}(F)$. This is exactly what we want.
\end{proof}

\begin{cor} \label{corB}
Keep the same assumptions as in the previous proposition, then
\begin{enumerate}
\item A prime $\kp\in\Spec R^{\ps,\ord,\{\xi_v\}}_{\Sigma^o}$ is pro-modular if $\rho(\kp)$ is irreducible.
\item $R^{\ps,\ord,\{\xi_v\}}_{\Sigma^o}$ is a finite $\Lambda_F$-algebra.
\item If $\kp$ is a maximal ideal of $R^{\ps,\ord,\{\xi_v\}}_{\Sigma^o}[\frac{1}{p}]$ such that 
\begin{itemize}
\item $\rho(\kp)$ is irreducible.
\item Write $\rho(\kp)|_{G_{F_v}}\cong \begin{pmatrix} \psi_{v,1} & * \\ 0 & \psi_{v,2}\end{pmatrix}$. We assume $\psi_{v,1}$ is Hodge-Tate and has strictly less Hodge-Tate number than $\psi_{v,2}$ for any $v|p$ and any embedding $F_v\hookrightarrow \overbar{\Q_p}$.
\end{itemize}
Then $\rho(\kp)$ comes from a twist of a Hilbert modular form.
\end{enumerate}
\end{cor}

\begin{proof}
If $\kp$ in the first claim is one-dimensional, then enlarging $E$ if necessary, we can find a non-zero extension class $b\in H^1_{\Sigma^o}(F)$ such that $\kp$ is in the image of $\Spec R^{\ord}_b$. And the assertion follows from the previous proposition directly. In general, we may find a one-dimensional prime $\kp'$ containing $\kp$ such that $\rho(\kp')$ is irreducible. Suppose $\kp'$ is in the image of $\Spec R^{\ord}_b$, then by the first part of corollary \ref{psccomp}, $\kp$ is also in the image hence pro-modular as well. This proves the first assertion.

A direct consequence of this is that the natural $\Lambda_F$-equivariant map
\[R^{\ps,\ord,\{\xi_v\}}_{\Sigma^o}\to \T^{\ord}_\km \times R^{\ps,\mathrm{red}}\]
has nilpotent kernel. Hence it suffices to prove that the image is a finite $\Lambda_F$-algebra. This is clearly true for $\T^{\ord}_\km$. As for $R^{\ps,\mathrm{red}}$, as we discussed in \ref{redlrps}, $R^{\ps,\mathrm{red}}$ is a quotient of $\cO[[G_{F,S}^{\mathrm{ab}}(p)]]$. Its finiteness over $\Lambda_F$ follows from the global class field theory.

The last assertion follows from the first one and the discussion in \ref{finlevhecke}.
\end{proof}

\subsection{Existence of Eisenstein maximal ideal} \label{EoEmi}
In this subsection, we give sufficient conditions for the existence of Eisenstein maximal ideal in assumption \ref{ass1}. We keep the same notations as in the previous subsection. In particular, $\psi=\chi\varepsilon$ and $F$ is a totally real field of even degree over $\Q$ in which $p$ is unramified.

\begin{prop} \label{exoemi}
Assume 
\begin{itemize}
\item $\mathrm{ord}_{\varpi} \mathrm{L}_p(F,-1,\widetilde{\bar{\chi}\omega})>0$ if $\Sigma^o$ is $\Sigma_p$ or empty, where $\widetilde{\bar{\chi}\omega}$ is the Teichm\"uller lifting of $\bar{\chi}\omega$ and $\mathrm{L}_p(F,s,\widetilde{\bar{\chi}\omega})$ is the $p$-adic L-function associated to the character $\widetilde{\bar{\chi}\omega}$.
\end{itemize}
Then enlarging $\cO$ if necessary, there exists an open subgroup $U_e^p=\prod_{v\nmid p}U_{e,v}\subseteq \prod_{v\nmid p}\GL_2(O_{F_v})$ such that $\varpi$ and 
\begin{itemize}
\item $T_v-(1+\chi(\Frob_v))$ for $v\nmid p$ such that $U_{e,v}=\GL_2(O_{F_v})$;
\item $\psi_{v,1}(\gamma)-\tilde{\psi}_{v,1}(\gamma),\gamma\in F_v^\times,v|p$, where $\tilde{\psi}_{v,1}(\gamma)\in\cO$ is any lifting of $\bar{\psi}_{v,1}(\gamma)\in\F$;
\end{itemize}
generate a maximal ideal of $\T^\ord_{\psi}(U^p_e)$.
\end{prop}

\begin{rem} \label{efcase}
If $\Sigma^o$ is $\Sigma_p$, this is proposition 3.18 of \cite{SW99}. If $\Sigma^o$ is empty, we can twist everything by $\chi^{-1}$ and reduce to the previous case. 
\end{rem}

As in \cite{SW99}, we will use Eisenstein series and congruences to find a cuspidal ordinary eigenform with desired Hecke eigenvalues. To do this, we first review some basics of the theory of Hilbert modular forms (of parallel weight). References here are \cite{Shi78} and \S2 of \cite{Hida06}, though some conventions here are different.

\begin{para}
Let $\delta$ be the different of $F$. Write ${O_F}\otimes\hat{\Z}$ as $\widehat{O_F}$. For an ideal $\mathfrak{n}$ of $O_F$, we denote by $U^1(\mathfrak{n})$ the following congruence subgroup: 
\[\{\begin{pmatrix}a&b\\c&d\end{pmatrix}\in\GL_2(\A_F^\infty):a\in\widehat{O_F},b\in \delta^{-1}\widehat{O_F},c\in\delta\mathfrak{n}\widehat{O_F},d-1\in\mathfrak{n}\widehat{O_F},ad-bc\in(\widehat{O_F})^\times\}.\]

Let $\Sigma_\infty$ be the set of infinite places of $F$ and $\mathfrak{H}$ be the usual upper half plane. There is a natural action of $\GL_2^{+}(F\otimes_\Q \R)\cong \GL_2^{+}(\R)^{\Sigma_\infty}$ on $\mathfrak{H}^{\Sigma_\infty}$:
\[u_\infty(z)=(\frac{a_\tau z_\tau+b_\tau}{c_\tau z_\tau+d_\tau})_{\tau\in\Sigma_\infty},~u_\infty=\begin{pmatrix} a_\tau & b_\tau \\ c_\tau & d_\tau \end{pmatrix}\in\GL_2^{+}(\R)^{\Sigma_\infty}.\]
Here $\GL_2^{+}(\R)\subseteq\GL_2(\R)$ denotes the subgroup with positive determinant. Define $j:\GL_2(F\otimes_\Q \R)\times \mathfrak{H}^{\Sigma_\infty}\to \bC^{\Sigma_\infty}$
\[j(u_\infty,z)=(c_\tau z_\tau+d_\tau)_{\tau},~u_\infty=\begin{pmatrix} a_\tau & b_\tau \\ c_\tau & d_\tau \end{pmatrix}.\]

Denote by $z_0$ the point $(i,\cdots,i)\in \mathfrak{H}^{\Sigma_\infty}$. It is fixed by 
\[C_\infty:=(\mathrm{SO}_2(\R)\R^\times)^{\Sigma_\infty}\subseteq \GL_2^{+}(\R)^{\Sigma_\infty}\subseteq \GL_2(\A_F).\]
\end{para}

\begin{para}[\textbf{Adelic holomorphic modular forms on $\GL_2/F$ of parallel weight $k$}] \label{ahmf}
Let $k$ be a positive integer. For any $x\in\bC^{\Sigma_\infty}$, we write $x^k$ for the product $\prod_{\tau\in\Sigma_\infty} x_{\tau}^k$. If a function $f:\GL_2(\A_F)\to\bC$ satisfies $f|_k u_\infty=f$ for any $u_\infty\in C_\infty$, where
\[(f|_k u_\infty)(g)=j(u_\infty,z_0)^{-k}\det(u_\infty)^{1} f(gu_\infty^{-1}),\]
then we can define a function $f_x: \mathfrak{H}^{\Sigma_\infty}\to \bC$ for any $x\in\GL_2(\A_F^\infty)$ by
\[ f_x(z)=j(u_\infty,z_0)^k\det(u_\infty)^{-{1}}f(xu_\infty),~z=u_\infty(z_0)\]
with $u_\infty\in\GL_2^+(F\otimes \R)$.

Let $\mathfrak{n}$ be an ideal of $O_F$. A function $f: \GL_2(\A_F)\to \bC$ is called a modular form of weight $k$ and level $U^1(\mathfrak{n})$ if 
\begin{itemize}
\item $f(axu^\infty)=f(x)$ for any $a\in\GL_2(F),u^\infty\in U^1(\mathfrak{n})$,
\item $f|_k u_\infty=f$ for any $u_\infty\in C_\infty$,
\item $f_x$ is homomorphic on $\mathfrak{H}^{\Sigma_\infty}$ for any $x\in\GL_2(\A_F^\infty)$.
\end{itemize}
Such a function $f$ is called a cusp form if the constant form of the Fourier expansion of $f_x$  (see \ref{FeaDs} below)  vanishes for all $x$. We denote the space of modular forms and cusp forms by $M_k(\mathfrak{n})$ and $S_k(\mathfrak{n})$. Let $\theta:\A_F^\times/F^\times\to\bC^\times$ be a Hecke character such that $\theta(x)=N_{F/\Q}(x)^{-k+2},x\in(F\otimes\R)^\times$. We denote by $M_k(\mathfrak{n},\theta)\subseteq M_k(\mathfrak{n})$, $S_k(\mathfrak{n},\theta)\subseteq S_k(\mathfrak{n})$ the subspace of functions with the centre $\A_F^\times$ acting via $\theta$. 

We remark that the relation between functions considered here and those in \cite{Shi78}  is as follows: given $f\in M_k(\mathfrak{n},\theta)$, then we can define another function $f_0:\GL_2(\A_F)\to\bC$ by
\[f_0(g)=f(g)|\det(g)|^{\frac{k}{2}-1}_{\A_F}.\]
This is an element of $M_{k}(\mathfrak{n},\theta|\cdot|_{\A_F}^{k-2})$ defined in \S 2 of \cite{Shi78}. Here $|\cdot|_{\A_F}:\A_F^\times \to \R^\times$ denotes the adelic norm. 
\end{para}

\begin{para}[\textbf{Relation with classical Hilbert modular forms}]
Let $t^{(i)}\in\A_F^\times,i=1,\cdots,h$ be a set of representatives of the narrow class group $F^\times\setminus\A_F^\times/\widehat{O_F}^\times(F\otimes\R)^\times_{>>0}$ such that $t^{(i)}_w=1$ for $w|\mathfrak{n}p\infty$. Fix such a choice from now on. Write
\[\Gamma_i=\GL_2(F)\cap t^{(i)}U^1(\mathfrak{n})\GL_2^{+}(F\otimes\R)(t^{(i)})^{-1}.\]
Let $f\in M_k(\mathfrak{n})$ and write $x_i=\begin{pmatrix} t^{(i)}&0\\0& 1\end{pmatrix}$. Then $f_{i}=f_{x_i}:\mathfrak{H}^{\Sigma_\infty}\to \bC$ satisfies
\[f_{i}(z)=j(\gamma,z)^{-k}\det(\gamma)^{1}f_i(\gamma z),\gamma\in\Gamma_i.\]
Note that $\det(\gamma)^{1}=1$ here as $\gamma\in\Gamma_i$. Hence $f_i$ is a Hilbert modular form of level $\Gamma_i$ and parallel weight $k$ in the usual sense. We denote by $M_k(\Gamma_i)$ the space of holomorphic functions on $\mathfrak{H}^{\Sigma_\infty}$ satisfying the above condition. Then the map $f\mapsto (f_i)_{i=1,\cdots,h}$ induces an isomorphism $M_k(\mathfrak{n})\cong\bigoplus_{i=1}^h M_k(\Gamma_i)$.
\end{para}

\begin{para}[\textbf{Fourier expansion and Dirichlet series}] \label{FeaDs}
Each $f_i(z)$ has a Fourier expansion at $\infty$:
\[f_i(z)=a_i(0,f)+\sum_{\mu\in(t^{(i)})^{-1}_{>>0}}a_i(\mu,f)e(\mu\cdot z)\]
where 
\begin{itemize}
\item $(t^{(i)})$ is the fractional ideal associated to $t^{(i)}$ and $(t^{(i)})^{-1}_{>>0}\subseteq (t^{(i)})^{-1}$ is the subset of totally positive elements.
\item $e(\mu\cdot z)=e^{2\pi \sqrt{-1}\sum_{\tau\in\Sigma_\infty}\mu_\tau z_\tau}$.
\end{itemize}

For a fractional ideal $\ka$, we put $C(\ka,f)=a_i(\mu,f)|t^{(i)}|^{-1}_{\A_F}$ if $\ka$ is of the form $(\mu (t^{(i)})), \mu\in (t^{(i)})^{-1}_{>>0}$ and zero otherwise. We attach a Dirichlet series to $f$ by 
\[D(f,s)=\sum_{\ka\subseteq O_F}\frac{C(\ka,f)}{N(\ka)^s},\]
where $\ka$ runs over all non-zero ideals of $O_F$ and $N(\ka)$ denotes the norm of $\ka$.
\end{para}

\begin{rem}
Any modular form $f\in M_{k}(\mathfrak{n})$ has an adelic Fourier expansion (see proposition 2.26 of \cite{Hida06}, note that $[\kappa_1]=[0]$ here):
\[f(\begin{pmatrix}y & x \\ 0 & 1 \end{pmatrix}) =|y|_{\A_F}(a_0(y,f)+\sum_{\xi\in F_{>>0}} a_\infty(\xi y,f)e(\sqrt{-1}\xi \cdot y_\infty)e_F(\xi x)),x\in\A_F,y\in\A_F^\times\]
where $e_F:\A_F/F\to \bC^\times$ denotes the standard additive character determined by the condition $e_F(x_\infty)=e^{2\pi\sqrt{-1}\sum_\tau x_\tau}$ for $x_\infty=(x_\tau)\in F\otimes \R$. The relation between $a_\infty$ and $a_i$ is given by
\[a_\infty(\mu(t^{(i)}),f)=a_i(\mu,f)|t^{(i)}|^{-1}_{\A_F}=C(\mu(t^{(i)}),f).\]
One can check that $C(\ka,f)$ is the same as $C(\ka,f_0)$ considered in (2.24) of \cite{Shi78}, where $f_0$ is defined at the end of \ref{ahmf}. Hence $D(f,s)=D(s,f_0)$ defined in (2.25) of \cite{Shi78}.
\end{rem}

\begin{para}[\textbf{Hecke operators}]
For place $v\nmid p\mathfrak{n}$, we define Hecke operators $T_v\in\End (M_{k}(\mathfrak{n}))$ in the usual way:
\[(T_v f)(g)=\sum_{\gamma_i}f(g\gamma_i),~U^1(\mathfrak{n})\begin{pmatrix}\varpi_v&0\\0&1\end{pmatrix}U^1(\mathfrak{n})=\bigsqcup_{\gamma_i} \gamma_i U^1(\mathfrak{n}).\]
For $\gamma\in O_{F,p}\cap (F\otimes\Q_p)^\times$, we can define Hecke operator $\langle \gamma\rangle$ as in \ref{hidatheory} (with $\vec w=\vec 0$ here)
\[(T_v f)(g)=\sum_{\gamma_i}f(g\gamma_i),~U^1(\mathfrak{n})\begin{pmatrix}\gamma &0\\0&1\end{pmatrix}U^1(\mathfrak{n})=\bigsqcup_{\gamma_i} \gamma_i U^1(\mathfrak{n}).\]

In terms of the Fourier expansion, let $f\in M_{k}(\mathfrak{n},\theta)$, then a simple computation shows that
\begin{itemize}
\item $C(\ka,T_v f)=C(\ka(\varpi_v),f)+N(\varpi_v)\theta(\varpi_v)C(\frac{\ka}{(\varpi_v)},f)$ if $v\nmid p\mathfrak{n}$.
\item $C(\ka,\langle \varpi_v \rangle f)=C(\ka (\varpi_v),f)$ if $v|p$.
\item $C(\ka,\langle \gamma \rangle f)=C(\ka,f)$ for $\gamma \in O_{F_v}^\times$.
\end{itemize}

For any subring $R\subset \bC$, let $M_k(\mathfrak{n},R)=\{f\in M_k(\mathfrak{n}),a_i(\mu,f)\in R$ for any $i,\mu\}$. We view $\cO$ as a subring of $\bC$ via the isomorphism $\bC\cong \overbar{\Q_p}$. Suppose $\cO$ is large enough. It is clear that all the Hecke operators above preserve $M_k(\mathfrak{n},\theta,\cO):=M_k(\mathfrak{n},\cO)\cap M_k(\mathfrak{n},\theta)$. In fact, $M_k(\mathfrak{n},\theta,\cO)$ defines an integral structure of $M_k(\mathfrak{n},\theta)$ (2.3.18 of \cite{Hida06}):
\[ M_k(\mathfrak{n},\theta,\cO)\otimes_{\cO}\bC\cong M_k(\mathfrak{n},\theta).\]
The same result holds with space of cusp forms: let $S_k(\mathfrak{n},\theta,\cO):=M_k(\mathfrak{n},\cO)\cap S_k(\mathfrak{n},\theta)$, then
\[ S_k(\mathfrak{n},\theta,\cO)\otimes_{\cO}\bC\cong S_k(\mathfrak{n},\theta).\]

We define Hecke algebras $\tilde{\T}(\mathfrak{n},\theta)\subseteq \End_{\cO}(M_k(\mathfrak{n},\theta,\cO))$ and  $\T(\mathfrak{n},\theta)\subseteq \End_{\cO}(S_k(\mathfrak{n},\theta,\cO))$ as the $\cO$-subalgebras generated by $T_v, v\nmid p\mathfrak{n}$, and $\langle\gamma\rangle,\gamma\in O_{F,p}\cap (F\otimes \Q_p)^\times$. We can define Hida's idempotent $\mathbf{e}$ in $\End_{\cO}(M_k(\mathfrak{n},\theta,\cO))$ and $\End_{\cO}(S_k(\mathfrak{n},\theta,\cO))$ as in \ref{hidatheory} and the ordinary Hecke algebra $\T^{\ord}(\mathfrak{n},\theta):=\mathbf{e}\T(\mathfrak{n},\theta)$ similarly for cusp forms.
\end{para}

\begin{proof}[Proof of Proposition \ref{exoemi}]
As we remarked in \ref{efcase}, we may assume $\Sigma^o$ is not empty or $\Sigma_p$.

By global class field theory (see Theorem 5 of chapter 10 of \cite{AT68}), there exists a continuous character $\bar{\theta}:G_F\to \F^\times$ such that
\begin{itemize}
\item For $v|p$, $\bar{\theta}|_{G_{F_v}}=(\bar{\psi}_{v,1})^{-1}$ (defined before \ref{psiord}).
\item $\bar{\theta}$ is ramified at somewhere outside of $\Sigma_p$.
\end{itemize}
Since we assume that $\bar{\chi}|_{G_{F_v}}\neq\mathbf{1}$ for any $v|p$, it follows from $\Sigma^o\neq \Sigma_p$ that $\bar{\theta}|_{G_{F_v}}\neq \mathbf{1}$ for some $v|p$. Similarly, $\bar{\theta}\bar{\chi}|_{G_{F_v}}\neq \mathbf{1}$ for some $v|p$. Let ${\theta}_1$ (resp. $\theta_2$) be the Teichm\"uller lifting of $\bar{\theta}$ (resp. $\bar{\theta}\bar{\chi}$) and $\mathfrak{n}_1$ (resp. $\mathfrak{n}_2$) be its conductor. Hence $\theta_1|_{G_{F_v}}$ (resp. $\theta_2|_{G_{F_v}}$) is trivial for $v\in \Sigma^o$ (resp. $v\in\Sigma_p\setminus \Sigma^o$). We remark that for any $i\in\{1,2\}$, there always exists a place $v|p$ such that $\theta_i|_{G_{F_v}}\neq\mathbf{1}$. Put $\theta=\theta_1\theta_2|\cdot|_{\A_F}$ and $\mathfrak{n}=\mathfrak{n}_1\mathfrak{n}_2$. Consider the \textit{weight one Eisenstein series} $E_1=E_1(\theta_1,\theta_2) \in M_1(\mathfrak{n}p,\theta)$ with Dirichlet series (Proposition 3.4 of \cite{Shi78})
\[L^{(\Sigma_p\setminus \Sigma^o)}(F,s,\theta_1)L^{(\Sigma^o)}(F,s,\theta_2)\]
where for a finite set of places $S$, $L^{(S)}(F,s,\theta_i)$ denotes the usual $L$-function associated to $\theta_i$ of conductor $\mathfrak{n}_i$ with Euler factors at $v\in S$ removed. A simple calculation shows that the Hecke action on $E_1$ satisfies the following
\begin{itemize}
\item $D(T_v\cdot E_1,s)=D\left((\theta_1(\Frob_v)+\theta_2(\Frob_v))E_1,s\right)$ for $v\nmid p\mathfrak{n}$;
\item $D(\langle \gamma \rangle \cdot E_1,s)= D(E_1,s)$ for $\gamma\in O_{F,p}\cap (F\otimes\Q_p)^\times$.
\end{itemize}
Since $M_1(\mathfrak{n},\theta)$ is finite-dimensional over $\mathbb{C}$, we may assume that $E_1$ is an eigenform, i.e. 
\begin{itemize}
\item $T_v\cdot E_1=(\theta_1(\Frob_v)+\theta_2(\Frob_v))E_1$ for $v\nmid p\mathfrak{n}$;
\item $\langle \gamma \rangle \cdot E_1= E_1$ for $\gamma\in O_{F,p}\cap (F\otimes\Q_p)^\times$.
\end{itemize}
Enlarging $\cO$ if necessary, we may assume further that $E_1\in M_1(\mathfrak{n}p,\theta,\cO)\setminus \varpi M_1(\mathfrak{n}p,\theta,\cO)$.

Next we want to find a modular form of higher weight congruent to $E_1$. I thank Professor Richard Taylor for showing me the following lemma.
\begin{lem}
There exists a Hilbert modular form $\Theta \in M_{(p-1)2^j}(p^{c-1}, |\cdot|_{\A_F}^{-(p-1)2^j+2},\cO)$ for some positive integers $j$ and $c$, such that
\begin{itemize}
\item $a_i(0,\Theta)=1$ for any $i$;
\item $a_i(\mu,\Theta)\equiv0\mod p$ for any $\mu\neq 0$ and $i$.
\end{itemize}
\end{lem}
\begin{proof}
This follows from lemma 1.4.2. of \cite{Wi88}.
\end{proof}

Now consider the function $\mathbf{E}:\GL_2(\A_F)\to\bC$,
\[\mathbf{E}(u)=E_1(u)\Theta(u)|\det(u)|_{\A_F}^{-1}\]
It is easy to see that $\mathbf{E}\in M_{k_0}(\mathfrak{n}p^{c},\tilde\theta,\cO)$ with $k_0=(p-1)2^j+1,\tilde\theta=\theta|\cdot|_{\A_F}^{-(p-1)2^j}$  and
\[C_i(\mu,\mathbf{E})\equiv C_i(\mu,E_1)\mod \varpi.\]
Hence $\mathbf{E}\notin \varpi M_{k_0}(\mathfrak{n}p^{c},\tilde\theta,\cO)$ and
\begin{itemize}
\item $T_v\cdot \mathbf{E}\equiv (\theta_1(\Frob_v)+\theta_2(\Frob_v))\mathbf{E}\mod(\varpi)$ for $v\nmid p\mathfrak{n}$.
\item $\langle \gamma \rangle \cdot \mathbf{E}\equiv \mathbf{E}\mod (\varpi)$ for $\gamma\in O_{F,p}\cap (F\otimes\Q_p)^\times$.
\end{itemize}
Therefore $\varpi, T_v-(\theta_1(\Frob_v)+\theta_2(\Frob_v)),v\nmid p\mathfrak{n}$ and $\langle \gamma \rangle-1,\gamma\in O_{F,p}\cap (F\otimes\Q_p)^\times$ generate a maximal ideal $\km_1$ of $\tilde{\T}(\mathfrak{n}p^{c},\tilde\theta)$. Note that $\tilde{\T}(\mathfrak{n}p^{c},\tilde\theta)_{\km_1}$ is a direct summand of $\tilde{\T}(\mathfrak{n}p^{c},\tilde\theta)$. We denote the idempotent associated to $\km_1$ by $\epsilon_1$. Let $\epsilon_0$ be the composite of $\epsilon_1$ with Hida's idempotent $\mathbf{e}$. We claim that $\mathbf{E}^\ord:=\epsilon_0(\mathbf{E})$ is in fact a \textit{cusp form}.

By proposition 1.5 of \cite{Wi86}, a complement of $S_{k_0}(\mathfrak{n}p^{c},\tilde\theta)$ in $M_{k_0}(\mathfrak{n}p^{c},\tilde\theta)$ is spanned by some automorphic forms inside the automorphic representations $\pi(k_0,\psi_1,\psi_2)$ generated by the Eisenstein series $E_{k_0}(\psi_1,\psi_2)$ (see the discussion before proposition 1.5 \cite{Wi86} for the precise statement) with Dirichlet series $L(F,s,\psi_1)L(F,s-k_0+1,\psi_2)$, where $\psi_1,\psi_2$ are Hecke characters of finite orders such that 
\begin{itemize}
\item $\psi_1\psi_2|\cdot|_{\A_F}^{-k_0+2}=\tilde\theta$.
\item The product of conductors of $\psi_1$ and $\psi_2$ divides $p^c\mathfrak{n}$.
\end{itemize}
Strictly speaking, \cite{Wi86} only proved the case of weight two. However it is easy to see that the same argument works for higher weight. Suppose there exists some automorphic form $\mathbf{F}\in \pi(k_0,\psi_1,\psi_2)\cap M_{k_0}(\mathfrak{n}p^{c},\tilde\theta)$ such that $\epsilon_0(\mathbf{F})\neq 0$. It follows from $\mathbf{e}(\mathbf{F})\neq 0$ that $\mathbf{F}$ is an ordinary form and $\psi_1|_{F_v^\times}$ is unramified for $v|p$. Moreover since $\epsilon_1(\mathbf{F})\neq 0$, we must have 
\begin{itemize}
\item $\psi_1+\psi_2\equiv \theta_1+\theta_2\mod\varpi$.
\item $\psi_1(\varpi_v)\equiv 1 \mod\varpi,~v|p$.
\end{itemize}
From the first identity, it is easy to see that $\psi_1\equiv \theta_1$ or $\theta_2\mod \varpi$. In either case, we may find a $v|p$ such that $\psi_i|_{G_{F_v}}\not\equiv 1\mod\varpi$. This contradicts the second equality above. Thus $\epsilon_0(\mathbf{F})=0$ and $\mathbf{E}^\ord=\epsilon_0(\mathbf{E})$ has to be a non-zero cusp form.

Thus $\varpi, T_v-(\theta_1(\Frob_v)+\theta_2(\Frob_v)),v\nmid p\mathfrak{n}$ and $\langle \gamma \rangle-1,\gamma\in O_{F,p}\cap (F\otimes\Q_p)^\times$ in fact generate a maximal ideal $\km_2$ of $\T^{\ord}(\mathfrak{n}p^c,\tilde\theta)$, the ordinary Hecke algebra of cusp forms. Hence we may find an ordinary cuspidal eigenform $\mathbf{F}'\in S_{k_0}(\mathfrak{n}p^c,\tilde\theta,\cO)_{\km_2}$. Using the Jacquet-Langlands correspondence, we can transfer $\mathbf{F}'$ from $\GL_2/F$ to $D^\times$ and get an ordinary eigenform $\mathbf{F}''\in S^{\ord}_{(\vec{k_0},\vec{0}),\tilde\theta_p}(U^p_e(c),\cO)$ with the same $T_v,\langle \gamma \rangle$ eigenvalues. Here $U^p_e=\prod_{v\nmid p}U_{e,v}$ is an open subgroup of $\prod_{v\nmid p}\GL_2(O_{F_v})$ such that $U_{e,v}=\GL_2(O_{F_v}),v\nmid p\mathfrak{n}$ and $\tilde\theta_p:\AFi/F^\times_{>>0}\to\cO^\times$ is the $p$-adic character associated to $\tilde\theta$: 
\[\tilde\theta_p(g)=\tilde\theta(g)N_{F/\Q}(g_p)^{k_0-2}.\]

Shrink $U^p_e$ if necessary, we may assume $\mathbf{F}''\otimes \theta_1^{-1}: \DAi\to \cO$,
\[g\mapsto \mathbf{F}''(g)\theta_1^{-1}(N_{D/F}(g))\]
is an element in $S^{\ord}_{(\vec{k_0},\vec{0}),\tilde\theta_p\theta_1^{-1}}(U^p_e(c),\cO)$. It follows from the definition of $\theta_1$ that $\varpi$, $T_v-(1+\chi(\Frob_b)),v\nmid \mathfrak{n}p$ and $\psi_{v,1}(\gamma)-\tilde{\psi}_{v,1}(\gamma),\gamma\in F_v^\times,v|p$ (as in the proposition)  generate a maximal ideal of $\T^{\ord}_{(\vec{k_0},\vec{0}),\tilde\theta_p\theta_1^{-1}}(U^p_e(c))$, hence also a maximal ideal of $\T^{\ord}_{\tilde\theta_p\theta_1^{-1}}(U^p_e)$. Using the weight independence result (Thm 2.3 of \cite{Hida89}), we conclude that these elements generate a maximal ideal of $\T^{\ord}_{\psi}(U^p_e)$.
\end{proof} 

\subsection{Proof of Theorem \ref{thmB}} \label{PotthmB}
We will freely use the notation introduced in subsection \ref{sasotmr}. The strategy is roughly as follows: we will make certain soluble base change and invoke proposition \ref{propB} and soluble base change results. First we do some reduction work.

\begin{para} \label{firstreduction}
Let $\tilde{\bar{\chi}},\tilde\omega:G_{F,S}\to\cO^\times$ be the Teichm\"uller lifting of $\bar{\chi},\omega$. Then we may write 
\[\chi=\tilde{\bar{\chi}}\varepsilon^k\tilde\omega^{-k}\theta^2\]
for some integer $k$ and some character $\theta:G_{F,S}\to\cO^\times$ of finite $p$-power order. It is easy to see that we only need to prove Theorem \ref{thmB} with $\chi$ replaced by $\chi\theta^{-2}$. Hence we may assume $\chi$ and $\bar{\chi}$ ramify at the same set of places away from $p$.

It suffices to prove Theorem \ref{thmB} with $R^{\ps,\ord}$ replaced by $R^{\ps,\ord}/\kp$ for any minimal prime $\kp$ of $R^{\ps,\ord}$. Fix a minimal prime $\kp$. Let $\rho(\kp):G_{F,S}\to\GL_2(k(\kp))$ be the associated semi-simple representation and let $\Sigma^{o,2}_F$ be the set of places $v$ above $p$ such that $\rho(\kp)|_{G_{F_v}}\cong \begin{pmatrix}\psi^{univ}_{v,2}\,\modd\kp&*\\0& \psi^{univ}_{v,1}\,\modd \kp\end{pmatrix}$, i.e. $\psi^{univ}_{v,2}\,\modd\kp$ is a sub-representation of $\rho(\kp)|_{G_{F_v}}$. 
We denote the set of places of $F$ above $p$ by $\Sigma_{F,p}$. Let $\Sigma^o_F=\Sigma_{F,p}\setminus \Sigma^{o,2}_F$. Then 
\[\rho(\kp)|_{G_{F_v}}\cong \begin{pmatrix}\psi^{univ}_{v,1}\,\modd\kp&*\\0& \psi^{univ}_{v,2}\,\modd \kp\end{pmatrix}\]
is a \textit{non-split} extension for every $v\in\Sigma^o_F$. If not, $\psi^{univ}_{v,2}\,\modd\kp$ would be a sub-representation and $v$ has to belong to $\Sigma_F^{o,2}$ by the construction of $\Sigma_F^{o,2}$. After twisting the deformation problem in \ref{sotmr1} with $\chi^{-1}$ if necessary, we may assume 
\begin{eqnarray} \label{onehalf}
|\Sigma^o_F|\leq\frac{1}{2}|\Sigma_{F,p}|.
\end{eqnarray}
Later on we will use this to bound the Selmer group in \ref{selmext}.

We denote the trace of $\rho(\kp)$ by $T(\kp)$. Then $T(\kp)$ is ($\psi_{v,1}^{univ}\mod\kp$)-ordinary (resp. ($\psi_{v,2}^{univ}\mod\kp$)-ordinary) if $v\in\Sigma^o_F$ (resp. $v\in\Sigma_{F,p}\setminus \Sigma^o_F$).
\end{para}

The following simple lemma and corollary will be quite useful.
\begin{lem} \label{simlem}
Let $R\in C_{\cO}$ be a domain with maximal ideal $\km$ and fraction field $K$. Let $\rho:G_{F,S}\to\GL_2(R)$ be a continuous representation such that 
\[\tr \rho\equiv 1+\bar{\chi}\mod\km.\]
Then for any $v\in S\setminus\Sigma_{F,p}$, there exists a finite extension $K'/K$ such that either
\begin{itemize}
\item $\rho\otimes K'|_{G_{F_v}}$ is reducible, or
\item $\rho\otimes K'|_{G_{F_v}}\cong \Ind_{G_{F_{v^2}}}^{G_{F_v}}\theta$, where $F_{v^2}$ denotes the unramified quadratic extension of $F_v$ and $\theta:G_{F_{v^2}}\to k(\kp)^\times$ is a character.
\end{itemize}
\end{lem}
\begin{proof}
If $\bar{\chi}|_{I_{F_v}}$ is non-trivial, then $\rho\otimes K|_{I_{F_v}'}$ is a direct sum of distinct characters where $I_{F_v}'$ denotes the prime to $p$ subgroup of $I_{F_v}$. Hence $\rho\otimes K|_{G_{F_v}}$ is reducible in this case. Now suppose $\bar{\chi}|_{I_{F_v}}=\mathbf{1}$. In this case, $\rho|_{I_{F_v}}$ has to factor through its pro-$p$ quotient $I_{F_v}(p)\cong \Z_p$. Hence there exists a quadratic extension $K'/K$ such that $\rho\otimes K'|_{I_{F_v}}$ is reducible. If $\rho\otimes K'|_{I_{F_v}}$ is a non-split extension, then $\rho\otimes K'|_{G_{F_v}}$ is reducible. Otherwise $\rho\otimes K'|_{I_{F_v}}\cong \theta_1\oplus\theta_2$ and $\rho\otimes K'|_{G_{F_v}}$ is either reducible or dihedral.
\end{proof}

\begin{cor} \label{simcor}
Let $R\in C_{\cO}$ be a domain with maximal ideal $\km$ and fraction field $K$. Let $T:G_{F,S}\to R$ be a continuous $2$-dimensional pseudo-representation with determinant $\chi$ such that $T\equiv 1+\bar{\chi}\mod\km$. Then for any $v\in S\setminus\Sigma_{F,p}$, there exists a finite extension $K'/K$ and two characters $\theta_1,\theta_2:I_{F_v}\to (K')^\times$ of finite orders  such that 
\[T|_{I_{F_v}}=\theta_1+\theta_2.\]
\end{cor}
\begin{proof}
Let $R^{\ps}$ be the universal deformation ring which parametrizes all two-dimensional pseudo-representations with determinant $\chi$ that lift $1+\bar{\chi}$. Denote the universal pseudo-representation by $T^{univ}$. It suffices to treat the case $R=R^{\ps}/\kq, T=T^{univ}\mod \kq$ for some $\kq\in\Spec R^{\ps}$. Our assertion is clear if the semi-simple representation $\rho(\kq)$ associate to $\kq$ is reducible. Now assume $\rho(\kq)$ is irreducible. Then arguing as in the proof of corollary \ref{corB}, we conclude that $\kq$ is in the image of $\Spec R_b\to \Spec R^{\ps}$ where $R_b$ parametrizes all deformations of $\bar{\rho}_b:G_{F,S}\to \GL_2(\F)$ with determinant $\chi$ for some representation $\bar{\rho}_b$ of the form $\begin{pmatrix}1 & * \\ 0& \bar{\chi}\end{pmatrix},*\neq 0$. Our assertion follows from the previous lemma.
\end{proof}

\begin{para}[\textbf{Existence of Eisenstein ideal}] \label{conF1}
Arguing as in (4.16), (4.17) of \cite{SW99}, we can find a finite totally real field extension $F_1/F$ such that
\begin{itemize}
\item $F_1/\Q$ is abelian with even degree.
\item $F(\bar{\chi})\cap F_1=F$.
\item $p$ is unramified in $F_1$.
\item $\mathrm{ord}_{\varpi} \mathrm{L}_p(F_1,-1,\widetilde{\bar{\chi}\omega})>0$ as in proposition \ref{exoemi}.
\end{itemize}
Let $\Sigma^o_{F_1}$ be the set of places of $F_1$ above $\Sigma^o_F$. Then it follows from proposition \ref{exoemi} that there exists an open subgroup $U_e^p=\prod_{v\nmid p}U_{e,v}\subseteq \prod_{v\nmid p}\GL_2(O_{F_{1,v}})$ such that $\T^{\ord}_{\psi}(U^p_e)$ has the desired maximal ideal (see \ref{1asspropB} below for details here). Denote by $S_1$ the union of places of $F_1$ above $S$ and places $v$ such that $U_{e,v}\neq\GL_2(O_{F_{1,v}})$.
\end{para}

\begin{para} \label{conF2}
Let $F_2/F_1$ be a finite totally real field extension with following properties:
\begin{itemize}
\item $F_2/\Q$ is abelian.
\item $p$ is unramified in $F_2$.
\item $F_1(\bar{\chi})\cap F_2=F_1$.
\item For any place $w\nmid p$ above $v\in S_1$, we have $2|[k(w):k(v)]$ and $p|N(w)-1$ and
\[\bar{\chi}|_{G_{F_{2,w}}}=\mathbf{1}.\]
In particular, $\bar{\chi}|_{G_{F_2}}$ and ${\chi}|_{G_{F_2}}$ are unramified outside of places above $p$.
\end{itemize}
The only non-trivial requirement is the last one under the first condition. This is possible since we assume $\bar{\chi}:G_F\to\F^\times$ can be extended to $G_\Q$.
\end{para}

\begin{para} \label{Washington}
Finally, let $l_0$ be a rational prime congruent to $1$ modulo the order of $\bar{\chi}$ and larger than the norm of any place in $S_1$. We choose a finite extension $F_3/F_2$ contained in the cyclotomic $\Z_{l_0}$-extension of $F_2$ of sufficiently large $l_0$-power degree such that:
\begin{itemize}
\item $|S_3|\leq \frac{1}{12}[F_3:\Q]$, where $S_3$ denotes the set of places of $F_3$ above $S_1$.
\item The $p$-part of class group of $F_3(\bar{\chi})$ satisfies $|\mathrm{Cl}(F_3(\bar{\chi}))[p]|\leq \frac{1}{12}[F_3:\Q]$.
\end{itemize}
This is possible in view of the result of Washington \cite{Wa78}.

Denote the set of primes of $F_3$ above $p$ (resp. $\Sigma^o_F$) by $\Sigma_p$ (resp. $\Sigma^o$). Then $\frac{|\Sigma^o|}{|\Sigma_p|}=\frac{|\Sigma^o_F|}{|\Sigma_{F,p}|}\leq \frac{1}{2}$ by our assumption \eqref{onehalf}. In particular, $|\Sigma^o|\leq|\Sigma_p|-1$. We introduced the following Selmer group as in \ref{selmext}:
\[H^1_{\Sigma^o}(F_3):=\ker(H^1(G_{F_3,S_3},\F(\bar{\chi}^{-1}))\stackrel{\mathrm{res}}{\longrightarrow}\bigoplus_{v\in \Sigma_p\setminus \Sigma^o}H^1(G_{F_{3,v}},\F(\bar{\chi}^{-1}))).\]
\end{para}

\begin{lem} \label{ubdimh1}
$F_3$ is an abelian extension of $\Q$. Moreover it satisfies the following properties:
\begin{itemize}
\item $F_2(\bar{\chi})\cap F_3=F_2$.
\item $p$ is unramified in $F_3$ and $\bar{\chi}|_{G_{F_{3,v}}}\neq \mathbf{1}$ for any $v|p$.
\item $\dim_\F H^1_{\Sigma^o}(F_3)\leq \frac{2}{3}[F_3:\Q]-1$.
\end{itemize}
\end{lem}
\begin{proof}
The first two claims are clear as $l_0$ is larger than the order of $\bar{\chi}$ and $p$. For the last one, let $K$ be the kernel of
\[H^1_{\Sigma^o}(F_3)\stackrel{\mathrm{res}}{\longrightarrow}\bigoplus_{v\in \Sigma^o}H^1(G_{F_{3,v}},\F(\bar{\chi}^{-1}))\oplus \bigoplus_{v\in S_3\setminus \Sigma_p}H^1(G_{F_{3,v}},\F(\bar{\chi}^{-1}))/H^1(G_{k(v)},\F(\bar{\chi}^{-1})).\]
Since $p$ is prime to the order of $\bar{\chi}$, we have a natural injection 
\[K\hookrightarrow H^1(G_{F_3(\bar{\chi})},\F(\bar{\chi}^{-1}))\cong \Hom (G_{F_3(\bar{\chi})},\F)(\bar{\chi}^{-1}).\]
It is easy to see that the image of $K$ lies inside the subspace of characters which are unramified everywhere. Hence its dimension is bounded by $\dim_{\F_p} |\mathrm{Cl}(F_3(\bar{\chi}))[p]|$. Note that $\bar{\chi}|_{G_{F_{3,v}}}$ is trivial for $v\in S_3\setminus \Sigma_p$, hence
\[H^1(G_{F_{3,v}},\F(\bar{\chi}^{-1}))/H^1(G_{k(v)},\F(\bar{\chi}^{-1}))\cong \Hom(O_{F_{3,v}}^\times,\F),\]
which is one-dimensional. For $v\in \Sigma_p$, it follows from our second claim in the lemma that
\[\dim_\F H^1(G_{F_{3,v}},\F(\bar{\chi}^{-1}))=[F_{3,v}:\Q_p]+e_\chi\leq [F_{3,v}:\Q_p]+1,\]
where $e_\chi=1$ if $\bar{\chi}|_{G_{F_{3,v}}}=\omega^{-1}$ and $0$ otherwise. Put all these results together:
\begin{eqnarray*}
\dim_\F H^1_{\Sigma^o}(F_3) &\leq& \dim K+\sum_{v\in\Sigma^o} \dim H^1(G_{F_{3,v}},\F(\bar{\chi}^{-1}))+\sum_{v\in S_3\setminus\Sigma_p} \dim \Hom(\cO^\times_{F_{3,v}},\F)\\
&\leq&  \dim_{\F_p} |\mathrm{Cl}(F_3(\bar{\chi}))[p]| + \sum_{v \in \Sigma^o} [F_{3,v}:\Q_p]+ |\Sigma^o|+|S_3|-|\Sigma_p|\\
&\leq& \frac{1}{12}[F_3:\Q]+\frac{|\Sigma^o|}{|\Sigma_p|}[F_3:\Q]+\frac{1}{12}[F_3:\Q]-1\\
&\leq& \frac{2}{3}[F_3:\Q]-1.
\end{eqnarray*}
\end{proof}

\begin{para}
For $v\in S_3\setminus \Sigma_p$, we define a $p$-power character $\xi_v:k(v)^\times\to\cO^\times$ as follows: let $\kp$ be the prime of $R^{\ps,\ord}$ defined in the beginning of this subsection. By corollary \ref{simcor} and our construction of $F_2$ in \ref{conF2}, enlarging $E$ if necessary, there exists a character $\xi_v:k(v)^\times\to \cO^\times$, viewed as a character of $I_{F_{3,v}}$ by the class field theory, such that
\[T(\kp)|_{I_{F_{3,v}}}=\xi_v+\xi_v^{-1}.\]
If $p\in\kp$, then we can simply take $\xi_v$ to be trivial. Then using the data $\{\xi_v\}_v,\Sigma^o$, we can define $R^{\ps,\ord,\{\xi_v\}}_{\Sigma^o}$ as in \ref{rpsordxiv} with $F=F_3,S=S_3$. We are going to apply corollary \ref{corB}. First we need check the assumptions in proposition \ref{propB} (which is written in corollary \ref{rcmod1c}).

Now the second assumption of \ref{rcmod1c}  becomes:
\[[F_3:\Q]-4|S_3|+4|\Sigma_p|-3>\delta_{F_3}+\dim_\F H^1_{\Sigma^o}(F_3).\]
Since $F_3$ is abelian over $\Q$, Leopoldt's conjecture is known in this case. Hence $\delta_{F_3}=0$. Then the inequality follows easily from lemma \ref{ubdimh1} and our assumption on $|S_3|$.
\end{para}

\begin{para}\label{1asspropB}
To see the first assumption (Assumption in \ref{ass1}) is satisfied, it follows from the existence of the Eisenstein maximal ideal as in \ref{conF1} that we can find a two-dimensional Galois representation (after possibly replacing $E$ by some extension) $\rho_1:G_{F_1,S_1}\to\GL_2(\cO)$ of determinant $\chi$ which comes from an automorphic representation of $\GL_2(\A_{F_1})$ such that 
\begin{itemize}
\item $\tr \rho_1 \equiv 1+\bar{\chi}\mod\varpi$.
\item For any $v|p$, $\rho_1|_{G_{F_{1,v}}}\cong\begin{pmatrix}\psi_{v,1}&*\\0&\psi_{v,2} \end{pmatrix}$ with $\psi_{v,1}\equiv \mathbf{1}\mod\varpi$ if $v\in\Sigma^o_{F_1}$ and $\psi_{v,1}\equiv \bar{\chi}|_{G_{F_{1,v}}}\mod\varpi$ otherwise. Moreover, $\psi_{v,1}$ is Hodge-Tate and has strictly less Hodge-Tate number than $\psi_{v,2}$ for any embedding $F_{1,v}\hookrightarrow \overbar{\Q_p}$.
\end{itemize}

Note that $\rho_1|_{G_{F_3}}$ has to be \textit{irreducible}. Otherwise suppose $G_{F_3}$ acts as a character $\theta$ on some one-dimensional subspace $L$. Then the reduction of $\theta$ modulo $\varpi$ is  $\mathbf{1}$ or $\bar{\chi}$. Hence $\theta(c)$ has to be a constant for any complex conjugation $c\in G_{F_3}$. This implies that $G_{F_1}$ also fixes $L$ and $\rho_1$ is reducible, which contradicts the cuspidal condition. 

Now it follows from the soluble base change (see for example 1.4 \cite{BLGHT11}) that $\rho_1|_{G_{F_3}}$ also comes from an automorphic representation of $\GL_2(\A_{F_3})$.  By lemma \ref{simlem}, for any $w\in S_3\setminus \Sigma_p$, there exists a character $\xi'_w:k(w)^\times\to \cO^\times$ of $p$-power order such that $\rho_1|_{I_{F_{3,w}}}=\xi'_w+(\xi'_w)^{-1}$. Hence the elements in Assumption \ref{ass1} generate a maximal ideal of $\T^{\ord}_{\psi,\xi'}(U^p)$ with $\xi'=\prod \xi'_v$, a character of $U^p$. See \ref{ass1} for the precise meaning of these notations. Arguing as in the proof of lemma \ref{SWbc}, we conclude that these elements also generate a maximal ideal of $\T^{\ord}=\T^{\ord}_{\psi,\xi}(U^p)$.

Thus we can apply corollary \ref{corB} to $R^{\ps,\ord,\{\xi_v\}}_{\Sigma^o}$. 
\end{para}

\begin{para}
It is easy to see that $T(\kp)|_{G_{F_3}}$ (defined in the beginning of this subsection) induces a natural map by the universal property:
\[R^{\ps,\ord,\{\xi_v\}}_{\Sigma^o}\to R^{\ps,\ord}/\kp.\]
\end{para}

\begin{lem} \label{finitemap}
This is a \textit{finite} map.
\end{lem}
\begin{proof}
It suffices to prove $R^{\ps,\ord}/(\kp,\km_1)$ is of finite length, where $\km_1$ is the maximal ideal of $R^{\ps,\ord,\{\xi_v\}}_{\Sigma^o}$. Let $\kq$ be a prime ideal $R^{\ps,\ord}/(\kp,\km_1)$. We need to show that $\kq$ is in fact the maximal ideal. Denote by $\rho(\kq):G_{F}\to\GL_2(k(\kq))$ the associated two-dimensional semi-simple representation. Then
\[\tr\rho(\kq)|_{G_{F_3}}=1+\bar{\chi}|_{G_{F_3}}.\]
Hence $\rho(\kq)|_{G_{F_3}}$ is reducible. Arguing as in the second paragraph of \ref{1asspropB}, we see that $\rho(\kq)$ is also reducible. It follows easily from this that $\tr\rho(\kq)=1+\bar{\chi}$. Therefore $\kq$ is the maximal ideal of $R^{\ps,\ord}/(\kp,\km_1)$.
\end{proof}

\begin{para}
Now we can prove Theorem \ref{thmB}. As in \ref{rpsordxiv}, there is a map 
\[\iota_{F_3}:\Lambda_{F_3}:=\widehat{\bigotimes}_{v|p}\cO[[O_{F_{3,v}}^\times(p)]]\to R^{\ps,\ord,\{\xi_v\}}_{\Sigma^o}.\]
Similarly there is a map $\iota_{F}:\Lambda_{F}\to R^{\ps,\ord}/\kp$ coming from the universal deformations of $\mathbf{1}|_{G_{F_v}}$ if $v\in\Sigma^o_F$ and $\bar{\chi}|_{G_{F_v}}$ if $v\in\Sigma_{F,p}\setminus \Sigma^o_F$. Note that this map is different from the one in Theorem \ref{thmB}, which comes from the universal deformations of $\mathbf{1}|_{G_{F_v}}$ for any $v|p$. However, it is easy to see that both maps have the same images in $R^{\ps,\ord}/\kp$.

By the universal property, we have the commutative diagram:
\[\begin{tikzcd}
\Lambda_{F_3} \arrow[r,"\iota_{F_3}"] \arrow[d,"N_{F_3/F}"] & R^{\ps,\ord,\{\xi_v\}}_{\Sigma^o} \arrow[d]\\
\Lambda_F \arrow[r,"\iota_F"] & R^{\ps,\ord}/\kp 
\end{tikzcd}\]
where the left vertical map arises from the norm map and the right vertical map is the one we considered in the previous lemma. Note that both maps are finite: the left one is in fact surjective and the other one follows from lemma \ref{finitemap}. By corollary \ref{corB}, $\iota_{F_3}$ is also finite. Therefore $\iota_F$ also has to be finite. This proves the first part of Theorem \ref{thmB}.

For the second part of Theorem \ref{thmB}, let $\rho(\kp)$ be a Galois representation as in the Theorem. Then $\rho(\kp)|_{G_{F_3}}$ is irreducible by the same argument as in the second paragraph of \ref{1asspropB}. Hence we can apply the third part of corollary \ref{corB} to $\rho(\kp)|_{G_{F_3}}$. Theorem \ref{thmB} now follows from the theory of soluble base change.
\end{para}

\section{A generalization of results of Skinner-Wiles II} \label{AgoroS-W2}
In this section, we treat the case where $\bar{\chi}|_{G_{F_v}}=\mathbf{1}$ for any $v|p$. We note that this case was excluded in the work of Skinner-Wiles \cite{SW99}. The main difficulty here is that we cannot bound the dimension of reducible locus as what we did in proposition \ref{redloc}. To solve this problem, we need to put additional conditions on the liftings (see \ref{admpairs} below).

\subsection{Statement of the main results} 
\begin{para} \label{sotmr2}
In this subsection, $F$ denotes an \textit{abelian} totally real extension of $\Q$ in which $p$ is \textit{unramified}. Let $S$ be a finite set of finite places containing all places above $p$. Let $\chi:G_{F,S}\to \cO^\times$ be a continuous character such that
\begin{itemize}
\item $\chi(c)=-1$ for any complex conjugation $c\in G_{F,S}$.
\item $\bar{\chi}$, the reduction of $\chi$ modulo $\varpi$, can be extended to a character of $G_\Q$.
\item $\bar{\chi}|_{G_{F_v}}= \mathbf{1}$ for every $v|p$.
\item $\chi|_{G_{F_v}}$ is de Rham for any $v|p$. In other words, $\chi=\varepsilon^k\psi_0$ with $k$ an integer and $\psi_0$ a character of finite order. 
\end{itemize}

Consider the universal deformation ring $R^{\ps,\ord}_1$ which pro-represents the functor from $\cOf$ to the category of sets assigning $R$ to the set of tuples $(T;\{\psi_{v,1}\}_{v|p})$ where
\begin{itemize}
\item $T:G_{F,S}\to R$ is a two-dimensional pseudo-representation lifting $1+\bar\chi$ with determinant $\chi$.
\item For any $v|p$, $\psi_{v,1}:G_{F_v}\to R^\times$ is a lifting of the trivial character. Moreover
\[T|_{G_{F_v}}=\psi_{v,1}+\psi_{v,2}\]
with $\psi_{v,2}=\chi|_{G_{F_v}}(\psi_{v,1})^{-1}$ such that $T$ is $\psi_{v,1}$-ordinary in the sense of \ref{psiord}.
\end{itemize}

Let $\psi^{univ}_{v,1}:G_{F_v}\to (R^{\ps,\ord}_1)^\times$ be the universal characters. From the class field theory, this induces a homomorphism $\cO[[O_{F_v}^\times(p)]]\to R^{\ps,\ord}_1$ for any $v|p$. Taking the completed tensor product over $\cO$ for all $v|p$, we get a map:
\[\Lambda_F:=\widehat{\bigotimes}_{v|p}\cO[[O_{F_v}^\times(p)]]\to R^{\ps,\ord}_1.\]
\end{para}

The main results of this section are:
\begin{thm} \label{thmC}
Under the assumptions for $F,\chi$ as above, we have
\begin{enumerate}
\item $R^{\ps,\ord}_1$ is a finite $\Lambda_F$-algebra.
\item For any maximal ideal $\kp$ of $R^{\ps,\ord}_1[\frac{1}{p}]$, we denote the associated semi-simple representation $G_{F,S}\to\GL_2(k(\kp))$ by $\rho(\kp)$ (see \ref{tar}). Assume 
\begin{itemize}
\item $\rho(\kp)$ is irreducible.
\item For any $v|p$, $\rho(\kp)|_{G_{F_v}}\cong\begin{pmatrix}\psi_{v,1} & *\\ 0 & \psi_{v,2}\end{pmatrix}$ such that $\psi_{v,1}$ is de Rham and has strictly less Hodge-Tate number than $\psi_{v,2}$ for any embedding $F_v\hookrightarrow \overbar{\Q_p}$.
\end{itemize}
Then $\rho(\kp)$ comes from a twist of a Hilbert modular form.
\end{enumerate}
\end{thm}

\subsection{Some local ordinary deformation rings}
In this subsection, $F$ will be a totally real field and $\chi:G_F\to\cO^\times$ is a continuous character such that the reduction of $\chi|_{G_{F_v}}$ modulo $\varpi$ is trivial for any $v|p$. We will introduce some local ordinary deformation rings and prove some basic properties of them. Fix a place $v$ above $p$. Let $\bar{\rho}:G_{F_v}\to\GL_2(\F)$ be a representation whose semi-simplification is the trivial representation.

\begin{para} \label{locorddefcl}
Let $\cL_0\subseteq \F^{\oplus 2}$ be a line fixed by $G_{F_v}$ under the action $\bar\rho$. Consider the functor from $\cOf$ to the category of sets sending $R$ to the triple $(\rho,\psi_{v,1},\cL)$, where
\begin{itemize}
\item $\rho:G_{F_v}\to\GL_2(R)$ is a lifting of $\bar{\rho}$ with determinant $\chi|_{G_{F_v}}$;
\item $\psi_{v,1}:G_{F_v}\to R^\times$ is a lifting of the trivial character;
\item $\cL\subseteq R^{\oplus 2}$ is a $G_{F_v}$-stable, rank-one $R$-direct summand lifting $\cL_0$ with $G_{F_v}$ acting via $\psi_{v,1}$.
\end{itemize}
This is pro-representable by a complete local noetherian ring, which we denote by $R^{\Delta,\cL_0}_{v,\bar\rho}$. 
\end{para}

\begin{lem} \label{RDeltaL}
Suppose that $F_v$ does not contain a primitive $p$-th root of unity. Then $R^{\Delta,\cL_0}_{v,\bar\rho}$ is formally smooth over $\cO$ of relative dimension $2[F_v:\Q_p]+3$.
\end{lem}

\begin{proof}
We may assume $\cL_0$ is $\F\cdot\begin{pmatrix} 1\\ 0 \end{pmatrix}$. Then for any such lifting $(\rho,\psi_{v,1},\cL)$ over $R$, there exists a unique $n\in R$ such that $\begin{pmatrix} 1 & 0\\ n & 1 \end{pmatrix}\cdot \cL=R\cdot\begin{pmatrix} 1\\ 0 \end{pmatrix}$. The conjugation of $\rho$ by $\begin{pmatrix} 1 & 0\\ n & 1 \end{pmatrix}$ has the form $\begin{pmatrix} \psi_{v,1} & *\\ 0 & \chi \psi_{v,1}^{-1}\end{pmatrix}$. Note that since $\rho$ lifts a representation with trivial semi-simplification, it must factor through the pro-$p$ quotient of $G_{F_v}$, which is a free pro-$p$ group of rank $[F_v:\Q_p]+1$ by Theorem 7.5.11 of \cite{NSW08}. Now the lemma is clear.
\end{proof}

\begin{para} \label{gerlocorddef}
We also need a local ordinary deformation ring without the chosen line $\cL_0$. This is defined by Geraghty in his thesis \cite{Ger10}. We give a slightly different construction, which in our case agrees with Geraghty's construction in view of the lemma below. Let $R^{\square}_{\tilde{\Lambda}_v}$ be the universal object in $C_\cO$ which parametrizes pairs $(\rho,\psi_{v,1})$ with $\rho:G_{F_v}\to\GL_2(R)$, a lifting of $\bar{\rho}$ with determinant $\chi|_{G_{F_v}}$, and $\psi_{v,1}:G_{F_v}\to R^\times$, a lifting of the trivial character. Consider the closed subset $\mathcal{G}^{\Delta}$ of primes $\kp\in\Spec R^{\square}_{\tilde{\Lambda}_v}$ such that the induced lifting $\rho(\kp):G_{F_v}\to\GL_2(k(\kp))$ has a one-dimensional subrepresentation given by the induced character $\psi_{v,1}(\kp):G_{F_v}\to k(\kp)^\times$. We denote by $\Spec R^{\Delta}_{v,\bar\rho}$ the reduced closed subscheme induced by $\mathcal{G}^{\Delta}$. I thank Ziquan Zhuang for showing me the following argument.
\end{para}

\begin{lem} \label{RDelta}
Suppose that $F_v$ does not contain a primitive $p$-th root of unity. 
\begin{enumerate}
\item $R^{\Delta}_{v,\bar\rho}$ is $\cO$-flat hence agrees with the definition of Geraghty. Moreover, $R^{\Delta}_{v,\bar\rho}$ is normal of dimension $1+2[F_v:\Q_p]+3$.
\item  Let $(\tilde\rho,\tilde\psi_{v,1})$ be the universal pair on $R^{\square}_{\tilde{\Lambda}_v}$. Write $\tilde\psi_{v,2}=\chi|_{G_{F_v}}(\tilde\psi_{v,1})^{-1}$. Then kernel of $R^{\square}_{\tilde{\Lambda}_v}\twoheadrightarrow R^{\Delta}_{v,\bar\rho}$ is generated by the entries of $(\tilde\rho(\sigma)-\tilde\psi_{v,1}(\sigma))(\tilde\rho(\tau)-\tilde\psi_{v,2}(\tau)),\sigma,\tau\in G_{F_v}$.
\end{enumerate}
\end{lem}

\begin{proof}
First we give an explicit description of $R^{\square}_{\tilde{\Lambda}_v}$. As we mentioned in the proof of lemma \ref{RDeltaL}, under our assumption on $F_v$, the pro-$p$ quotient of $G_{F_v}$ is a free pro-$p$ group of rank $r+1$ with $r=[F_v:\Q_p]$. Choose a set of generators $\sigma_0,\cdots,\sigma_{r}$ and write $\tilde\rho(\sigma_i)=\begin{pmatrix} a_i+1&b_i\\c_i&d_i+1\end{pmatrix}$ and $\tilde\psi_{v,1}(\sigma_i)=n_i+1$. Then it is clear that $R^{\square}_{\tilde{\Lambda}_v}$ is topologically generated by $a_i,b_i,c_i,d_i,n_i$ with $i=0,\cdots,r$ modulo relations $\det\tilde\rho(\sigma_i)=\chi(\sigma_i)$. 

We will only prove the case when $\bar\rho$ is the trivial representation. In fact, it will be clear that the situation is much simpler if $\bar\rho$ is a non-split extension. The closed subset $\mathcal{G}^{\Delta}$ is the locus where kernel of $\tilde\rho(\sigma_i)-\tilde\psi_{v,1}(\sigma_i)$ has a common non-zero vector for any $i$. Equivalently, the following matrix has at most rank one:
\[\begin{pmatrix}a_0-n_0 & b_0\\ c_0 & d_0-n_0 \\ \vdots& \vdots \\ a_i-n_i & b_i\\ c_i & d_i-n_i \\ \vdots& \vdots \\ a_{r}-n_{r} & b_{r}\\ c_{r} & d_{r}-n_{r}\end{pmatrix}.\]
This is same as saying all $2\times 2$-submatrices have determinant zero. Denote the entries of this matrix by $x_{ij},i=1,\cdots,2r+2,j=1,2$. It is easy to see that the determinant condition $\det\tilde\rho(\sigma_i)=\chi(\sigma_i)$ is equivalent with expressing $n_i$ in terms of $x_{(2i)j},x_{(2i+1)j},j=1,2$. Hence $R^{\square}_{\tilde{\Lambda}_v}\cong \Spec \cO[[x_{ij},i=1,\cdots,2r+2,j=1,2]]$ and $\mathcal{G}^{\Delta}$ is the underlying space defined by the ideal $I=(x_{i1}x_{j2}-x_{i2}x_{j1})$ for any $i,j$. This is the completion at the origin of the cone of the Segre embedding $\mathbf{P}^1\times\mathbf{P}^{2r+1}\hookrightarrow \mathbf{P}^{4r+3}$. By proposition 2.7 of \cite{Sho16}, $\cO[[x_{ij}]]/I$ is $\cO$-flat and Cohen-Macaulay. To see that it is actually normal, one can check easily that $\cO[[x_{ij}]]/I$ has dimension $1+2r+3$ and the singular locus is defined by the ideal generated by all $x_{ij}$, which has codimension at least $2$. We omit the details here.

To see that the kernel of $R^{\square}_{\tilde{\Lambda}_v}\twoheadrightarrow R^{\Delta}_{v,\bar\rho}$ has the form as described in the lemma, first we note that clearly these elements are contained in the kernel. On the other hand, the elements $x_{i1}x_{j2}-x_{i2}x_{j1}$ are in fact entries of $(\tilde\rho(\sigma)-\tilde\psi_{v,1}(\sigma))(\tilde\psi_{v,2}(\tau)-\tilde\rho(\tau))$ for some $\sigma,\tau$. This finishes the proof of the lemma.

\end{proof}

\subsection{Patching at a one-dimensional prime: ordinary case II}
In this subsection, we are going to prove a result similar to proposition \ref{propA} when $\bar{\chi}|_{G_{F_v}}=\mathbf{1},v|p$ (see below for the precise statement).

\begin{para}
\noindent \underline{\textbf{Setup}} The setup is almost the same as \ref{Paao-dproc1} except we require $\bar{\chi}|_{G_{F_v}}=\mathbf{1},v|p$ now. More precisely, let $F$ be a totally real field of even degree over $\Q$  in which $p$ is \textit{unramified} and $D$ be a totally definite quaternion algebra over $F$ which splits at all finite places. We denote the set of places above $p$ by $\Sigma_p$. Let $S$ be a finite set of places of $F$ containing $\Sigma_p$ such that for any $v\in S\setminus \Sigma_p$,
\[N(v)\equiv 1\mod p.\]
Let $\xi_v:k(v)^\times\to\cO^\times$ be a character of $p$-power order for each $v\in S\setminus \Sigma_p$. As before, we will view $\xi_v$ as characters of $I_{F_v}$ by the class field theory and fix a complex conjugation $\sigma^*\in G_{F}$. 

Fix a continuous character $\chi:G_{F,S}\to\cO^\times$ such that
\begin{itemize}
\item $\chi$ is unramified at places outside of $\Sigma_p$.
\item $\chi(\Frob_v)\equiv1\mod \varpi$ for $v\in S\setminus \Sigma_p$.
\item $\chi(c)=-1$ for any complex conjugation $c\in G_F$.
\item For any $v|p$, $\bar{\chi}|_{G_{F_v}}= \mathbf{1}$ and $\chi|_{G_{F_v}}$ is Hodge-Tate. Here $\bar{\chi}$ denotes the reduction of $\chi$ modulo $\varpi$.
\end{itemize}

We can define tame levels $U^p$ and ordinary Hecke algebra $\T^{\ord}:=\T^{\ord}_{\psi,\xi}(U^p)$ as in \ref{ass1}. We make the following assumption: 
\begin{itemize}
\item \textbf{Assumption}: $T_v-(1+\chi(\Frob_v)),v\notin S$ and $\psi_{v,1}(\gamma)-1,\gamma\in F_v^\times,v|p$ and $\varpi$ generate a maximal ideal $\km$ of $\T^\ord$.
\end{itemize}
\end{para}

\begin{para} \label{Rpsordxiv1}
As in \ref{rpsordxiv}, we consider the universal deformation ring $R^{\ps,\ord,\{\xi_v\}}_{1}$ which pro-represents the functor from $\cOf$ to the category of sets sending $R$ to the set of tuples $(T,\{\psi_{v,1}\}_{v\in\Sigma_p})$ where
\begin{itemize}
\item $T$ is a two-dimensional pseudo-representation of $G_{F,S}$ over $R$ lifting $1+\bar\chi$ with determinant $\chi$ and for any $v\in S\setminus \Sigma_p$,
\[T|_{I_{F_v}}=\xi_v+\xi_v^{-1} .\]
\item For any $v|p$, $\psi_{v,1}:G_{F_v}\to R^\times$ is a lifting of the trivial character and satisfies
\[T|_{G_{F_v}}=\psi_{v,1}+\psi_{v,1}^{-1}\chi.\]
Moreover $T$ is $\psi_{v,1}$-ordinary.
\end{itemize}

Then there is a natural surjective map $R^{\ps,\ord,\{\xi_v\}}_{1}\to \T^{\ord}_\km$. We say a prime of $R^{\ps,\ord,\{\xi_v\}}_{1}$ pro-modular if it comes from a prime of $\T^{\ord}_\km$. As in \ref{ordnice}, a prime $\kq\in\T^{\ord}_\km$ is called nice if it satisfies the conditions there. In particular, we drop the original third condition in \ref{nice}. In view of the second part of lemma \ref{nirred} and remark \ref{nicerem}, this follows from our assumptions that $p$ is unramified in $F$ and $\bar{\chi}|_{G_{F_v}}=\mathbf{1}$ for any $v|p$.

The main result of this subsection is:
\end{para}

\begin{prop} \label{propC}
Under all the assumptions in this subsection, if $\kq\in\Spec\T^{\ord}_\km$ is a nice prime, then
\[(R^{\ps,\ord,\{\xi_v\}}_{1})_{\kq^\ps}\to \T^{\ord}_\kq\]
has nilpotent kernel. Here $\kq^{\ps}=\kq\cap R^{\ps,\ord,\{\xi_v\}}_{1}$.
\end{prop}

\begin{para}
The proof is almost identical to the proof of proposition \ref{propA}, which is essentially the same as the proof of Theorem \ref{thmA} while replacing completed cohomology by Hida family and unrestricted deformation rings at places above $p$ by ordinary deformation rings. Compared with \ref{rdxbq}, there are certain technical differences:
\begin{enumerate}
\item In the first part of  \ref{rdxbq}, we replace the ordinary deformation ring $R^{\Delta}_v$ by $R^{\Delta}_{v,\bar\rho_b}$ introduced in \ref{gerlocorddef}.
\item In the second part, we consider the framed global deformation rings:
\[R^{\Delta_P,\{\xi_v\}}_{\bar{\rho}_b,Q}:=R^{\square_P,\{\xi_v\}}_{\bar{\rho}_b,Q}\widehat{\otimes}_{(\bigotimes_{v|p}R^{\square}_v)}(\bigotimes_{v|p}R^{\Delta}_{v,\bar\rho_b})\]
and unframed deformation rings
\[R^{\Delta,\{\xi_v\}}_{\bar{\rho}_b,Q}:=R^{\{\xi_v\}}_{\bar{\rho}_b,Q}\widehat{\otimes}_{(\bigotimes_{v|p}R^{\square}_v)}(\bigotimes_{v|p}R^{\Delta}_{v,\bar\rho_b}).\]
As in \ref{rdxbq}, the kernel of the natural surjective map
\[R^{\{\xi_v\}}_{\bar{\rho}_b}\widehat{\otimes}_{R^{\ps,\{\xi_v\}}} R_{1}^{\ps,\ord,\{\xi_v\}}\to R^{\Delta,\{\xi_v\}}_{\bar{\rho}_b}\]
is killed by the element $c$ in proposition \ref{rpsrc}. This follows from the second part of lemma \ref{RDelta}.

\end{enumerate}
\end{para}

\subsection{Admissible pairs} \label{admpairs}
\begin{para}
Keep the notations as in the previous subsection. Fix a complex conjugation $\sigma^*\in G_F$. 

Note in the case $\bar{\chi}|_{G_{F_v}}\neq \mathbf{1},v|p$, we can impose natural ordinary conditions (i.e. choice of $\Sigma^o$ in \ref{setupord1}) to bound the dimension of reducible locus of the deformation ring $R_b$. See proposition \ref{redloc} and lemma \ref{ubdimh1} for more details. In the case $\bar{\chi}|_{G_{F_v}}= \mathbf{1}$, such an ordinary condition is replaced by fixing a line for $G_{F_v}$ to lift. 

More precisely, fix a place $v_0|p$ of $F$.
\end{para}

\begin{defn}
Consider a pair $(\bar{\rho},\cL)$ where $\bar{\rho}$ is a two-dimensional $\F$-representation of $G_{F,S}$ and $\cL$ is a line in the representation space of $\bar\rho$. We say $(\bar{\rho},\cL)$ is \textit{admissible} if 
\begin{itemize}
\item $\bar\rho$ is a (possibly split) extension of $\bar{\chi}$ by $\mathbf{1}$ such that $\bar\rho|_{G_{F_{v_0}}}$ is \textit{split}.
\item $\cL$ is not $\rho(G_F)$-stable.
\end{itemize}
\end{defn}

\begin{rem}
As pointed out by one referee, the idea of introducing an extra line in the non $p$-distinguished case (i.e. $\bar\chi|_{G_{F_v}}$ is trivial for some $v|p$) at least goes back to the work of Calegari-Emerton \cite{CE05}. Our deformation problem below \ref{defbL} essentially appears as $\mathrm{Def}$ in their introduction. They also refer to some work of Dickinson in the local case.
\end{rem}

\begin{para}
Let $H^1(F)$ be $H^1(G_{F,S},\F(\bar{\chi}^{-1}))\cong \Ext^1_{\F[G_{F,S}]}(\bar\chi,\mathbf{1})$ and $H^1_{v_0}(F)$ be the kernel of
\[H^1(G_{F,S},\F(\bar{\chi}^{-1}))\to H^1(G_{F_{v_0}},\F(\bar{\chi}^{-1})).\] 
Then each class $b\in H^1(F)$ gives rise to an extension of $\bar{\chi}$ by $\mathbf{1}$, which we denote by $\bar{\rho}_b$. Clearly $\bar{\rho}_b|_{G_{F_{v_0}}}$ is split if and only if $b\in H^1_{v_0}(F)$. By abuse of notation, we say $(b,\cL), b\in H^1_{v_0}(F_1),\cL\in\bar\rho_b$ admissible if so is $(\bar{\rho}_b,\cL)$.

Suppose $(b,\cL)$ is admissible. It is clear that if $b=0$, i.e. $\bar\rho_b=\mathbf{1}\oplus \bar\chi$, then $\cL$ is unique up to isomorphisms. 

We are going to define certain deformation rings associated to an admissible pair. This will replace the role of $R_b$ in \ref{Rbpsi}.
\end{para}

\begin{defn}  \label{defbL}
Let $(b,\cL)$ be an admissible pair and $T$ be a finite set of places. Fix an ordered basis $\beta_b$ of $\bar\rho_b$. Let $\mathrm{Def}^{\square_T}_{b,\cL}$  be the functor from $\cOf$ to the category of sets sending $R$ to the set of isomorphism classes of $(M_R,\iota_R,\rho_R,\cL_R,\{\psi_{v,1}\}_{v|p},\{\beta_v\}_{v\in T})$, where
\begin{itemize}
\item $M_R$ is a projective $R$-module of rank $2$ with a continuous $R$-linear action $\rho_R$ of $G_{F,S}$.
\item $\det \rho_R=\chi$.
\item $\tr \rho_R|_{I_{F_v}}=\xi_v+\xi_v^{-1}$ for $v\in S\setminus \Sigma_p$.
\item $\iota_R:M_R/\km_R M_R\stackrel{\sim}{\to}\bar\rho_b$ is an isomorphism of $\F[G_{F,S}]$-modules. Here $\km_R$ denotes the maximal ideal of $R$.
\item $\psi_{v,1}:G_{F_v}\to R^\times$ is a lifting of the trivial character for every $v|p$.
\item $\cL_R\subseteq M_R$ is an $R$-direct summand of $M_R$ of rank one which lifts $\cL$ under $\iota_R$. It is $G_{F_{v_0}}$-stable with $G_{F_{v_0}}$ acting via $\psi_{v_0,1}$.
\item For $v|p,v\neq v_0$, the representation $\rho|_{G_{F_v}}$ is $\psi_{v,1}$-ordinary in the sense that choosing an ordered basis of $M_R$ which lifts $\beta_b$, then the natural map $R^{\square}_{\tilde{\Lambda}_v}\to R$ given by the universal property factors through $R^\Delta_{v,\bar\rho_b}$. See \ref{gerlocorddef} for the notations here. Intrinsically, by the second part of lemma \ref{RDelta}, this is the same as requiring 
\[(\rho_R(\sigma)-\psi_{v,1}(\sigma))(\rho_R(\tau)-\psi_{v,2}(\tau))=0\]
for any $\sigma,\tau\in G_{F_v}$. Here as usual $\psi_{v,2}=\chi|_{G_{F_v}}\psi_{v,1}^{-1}$.
\item For $v\in T$, $\beta_v$ is an ordered basis of $M_R$ which lifts $\beta_b$ under $\iota_R$.
\end{itemize}
If $T$ is empty, we will simply write $\mathrm{Def}_{b,\cL}$.
\end{defn}

\begin{lem}[Definition of $R_{b,\cL}^{\square_T}$] \label{defRbcLT}
$\mathrm{Def}_{b,\cL}^{\square_T}$ is pro-representable by a complete local noetherian $\cO$-algebra which we denote by $R_{b,\cL}^{\square_T}$.  Moreover, the natural map $R_{b,\cL}\to R_{b,\cL}^{\square_T}$ induces a non-canonical isomorphism $R_{b,\cL}^{\square_T}\cong R_{b,\cL}[[y_1,\cdots,y_{4|T|-1}]]$.
\end{lem}

\begin{proof}
If $b\neq 0$, then the result is classical. Now we assume $b=0$. Given $R\in \cOf$ and a tuple $(M_R,\iota_R,\rho_R,\cL_R,\{\psi_{v,1}\}_{v|p},\{\beta_v\}_{v\in T})\in \mathrm{Def}^{\square_T}_{b,\cL}(R)$, we claim that up to a scalar in $R^\times$, there is a unique isomorphism $M_R\stackrel{\sim}{\to} R^{\oplus 2}=R\cdot\mathbf{e}_1\oplus R\cdot\mathbf{e}_2$, such that under this isomorphism, the image of $\cL_R$ is generated by $\mathbf{e}_1+\mathbf{e}_2$ and $\rho(\sigma^*)$ has the form $\begin{pmatrix} 1 & 0\\ 0 & -1 \end{pmatrix}$. Recall that $\sigma^*\in G_F$ is a fixed complex conjugation. This is because $G_{F}$ does not fix $\cL$ hence $\cL_R$ cannot be generated by $\mathbf{e}_1$ or $\mathbf{e}_2$. Therefore $\mathrm{Def}^{\square_T}_{b,\cL}(R)$ can be identified with the set of $(\rho_R,\{\psi_{v,1}\}_{v|p},\{\alpha_v\}_{v\in T})$ where 
\begin{itemize}
\item $\rho_R:G_{F,S}\to \GL_2(R)$ is a lifting of $\begin{pmatrix} \mathbf{1} & 0\\ 0 &\bar\chi\end{pmatrix}$ with determinant $\chi$ and satisfies $\tr \rho_R|_{I_{F_v}}=\xi_v+\xi_v^{-1}$ for $v\in S\setminus \Sigma_p$ and $\rho_R(\sigma^*)=\begin{pmatrix} 1 & 0\\ 0 & -1 \end{pmatrix}$;
\item $\psi_{v,1}:G_{F_v}\to R^\times$ is a lifting of the trivial character for $v|p$ such that if $v\neq v_0$, the representation $\rho|_{G_{F_v}}$ is $\psi_{v,1}$-ordinary in the above sense. Moreover $R\cdot\begin{pmatrix}1 \\ 1 \end{pmatrix}$ is $\rho(G_{F_{v_0}})$-stable with $G_{F_{v_0}}$ acting via $\psi_{v_0,1}$;
\item $\alpha_v\in 1+M_2(\km_R),v\in T$;
\end{itemize}
modulo the equivalence relation $(\rho_R,\{\psi_{v,1}\}_{v|p},\{\alpha_v\}_{v\in T}) \sim (\rho'_R,\{\psi'_{v,1}\}_{v|p},\{\alpha'_v\}_{v\in T})$ if $\rho_R=\rho'_R,\psi_{v,1}=\psi'_{v,1}$ and there exists a scalar $a_R\in1+\km_R$ such that $\alpha_v=a_R\alpha'_v$ for any $v\in T$. Using this description, all the statements in the lemma are clear. 
\end{proof}

\begin{para}
We prove several basic properties of $R_{b,\cL}$. First we give a lower bound on the connectedness dimension of $R_{b,\cL}$. See \ref{connectednessdim} for the definition.

Let $(b,\cL)$ be an admissible pair. Fix an ordered basis $\beta_b$ of $\bar\rho_b$. 
Let $R^{\ord,\cL}_{\loc}$ be the following completed tensor product over $\cO$:
\[R^{\Delta,\cL}_{v_0,\bar\rho_b}\widehat\otimes(\widehat{\bigotimes}_{v\in \Sigma_p\setminus \{v_0\}}R^{\Delta}_{v,\bar\rho_b})\widehat{\otimes}(\widehat{\bigotimes}_{v\in S\setminus \Sigma_p}R^{\square,\xi_v}_v),\]
where the ordinary deformation rings $R^{\Delta,\cL}_{v_0,\bar\rho_b}$ and $R^{\Delta}_{v,\bar\rho_b}$ were introduced in \ref{locorddefcl} and  \ref{gerlocorddef}, and for $v\in S\setminus \Sigma_p$, the deformation ring $R^{\square,\xi_v}_v$ associated to $\bar\rho_b|_{G_{F_v}}$ was defined in \ref{ldfr}.

It is well-known (corollary 2.3.5 of \cite{CHT08}) that $R_{b,\cL}^{\square_S}$ can be written as the form
\[R_{b,\cL}^{\square_S}\cong R^{\ord,\cL}_{\loc}[[x_1,\cdots,x_{g_1}]]/(h_1,\cdots,h_{r_1})\]
where $g_1=\dim_\F H^1(G_{F,S},\ad^0\bar{\rho}_b(1))+|S|-1-[F:\Q]-\dim_\F H^0(G_{F,S},\ad^0\bar{\rho}_b(1))$ and $r_1=\dim_\F H^1(G_{F,S},\ad^0\bar{\rho}_b(1))$. From this, we get  
\end{para}

\begin{lem} \label{cRbcL}
The connectedness dimension $c(R_{b,\cL}/(\varpi))$ of $R_{b,\cL}/(\varpi)$ is at least 
\[[F:\Q]-|S|+|\Sigma_p|-1.\]
\end{lem}

\begin{proof}
Same argument as in the proof of \ref{crord}. The only difference is that since $p$ is unramified in $F$ and $\bar\chi|_{G_{F_v}}=\mathbf{1}$ for any $v|p$, we have $H^0(G_{F,S},\ad^0\bar{\rho}_b(1))=0$.
\end{proof}

Similarly, we have
\begin{lem} \label{dimRbcL}
Each irreducible component of $R_{b,\cL}$ has dimension at least $[F:\Q]+1$.
\end{lem}

\begin{para}
Given an admissible pair $(b,\cL)$, it follows from the universal property that there is a natural map from $R^{\ps,\ord,\{\xi_v\}}_{1}$ (defined in \ref{Rpsordxiv1}) to $R_{b,\cL}$. The following result can be viewed as a weak version of corollary \ref{psccomp} when $b=0$. 
\end{para}

\begin{prop} \label{psdefcL}
Let $\kQ\in\Spec R_{0,\cL}$ such that $\rho(\kQ)$, the pushforward of the universal deformation to $k(\kQ)$, is irreducible. Let $\kQ^{\ps}=\kQ\cap R^{\ps,\ord,\{\xi_v\}}_{1}$. Then 
\[\dim R^{\ps,\ord,\{\xi_v\}}_{1}/\kQ^{\ps}\geq \dim R_{0,\cL}/\kQ-1\]
\end{prop}

\begin{proof}
Recall that $R_{0,\cL}$ parametrizes tuples $(\rho_R,\{\psi_{v,1}\}_{v|p})$ which satisfies certain properties as in the proof of lemma \ref{defRbcLT}. Let $(\tilde\rho,\{\tilde\psi_{v,1}\}_{v|p})$ be the universal tuple. In particular, $\tilde\rho(\sigma^*)=\begin{pmatrix} 1 & 0\\ 0 & -1 \end{pmatrix}$ and $\tilde\rho(\sigma)\begin{pmatrix}1 \\ 1 \end{pmatrix}=\tilde\psi_{v_0,1}(\sigma)\begin{pmatrix}1 \\ 1 \end{pmatrix}$ for $\sigma\in G_{F_{v_0}}$. Write $\tilde\rho=\begin{pmatrix} \tilde a & \tilde b \\ \tilde c & \tilde d\end{pmatrix}$. Then $\tilde a(\sigma), \tilde d(\sigma)$ lie inside the image of $R^{\ps,\ord,\{\xi_v\}}_{1}$ for any $\sigma \in G_F$ (cf. \ref{adrep}). The values of $\tilde\psi_{v,1}$ are also in the image. Therefore $R_{0,\cL}$ is topologically generated by $\tilde b(\sigma), \tilde c(\sigma) ,\sigma\in G_{F,S}$ over $R^{\ps,\ord,\{\xi_v\}}_{1}$. Note that in fact finitely many elements $\sigma_1,\cdots,\sigma_r\in G_{F,S}$ would be enough.

Since $\rho(\kQ)$ is irreducible, we may find a prime $\kq\in\Spec R_{0,\cL}$ containing $\kQ$ such that $R_{0,\cL}/\kq$ is one-dimensional and $\rho(\kq):= \tilde\rho \otimes k(\kq)$ is irreducible. After possibly reordering $\sigma_1,\cdots,\sigma_r$, we may assume $\tilde{b}(\sigma_1)\notin \kq$. Consider the following homomorphism of $R^{\ps,\ord,\{\xi_v\}}_{1}$-algebras:
\[(R^{\ps,\ord,\{\xi_v\}}_{1}/\kQ^{\ps})[[x]]\to R_{0,\cL}/\kQ \]
sending $x$ to $\tilde b(\sigma_1)$. Here $(R^{\ps,\ord,\{\xi_v\}}_{1}/\kQ^{\ps})[[x]]$ denotes the formal power series ring over $R^{\ps,\ord,\{\xi_v\}}_{1}/\kQ^{\ps}$ with variable $x$. Note that $\tilde{b}(\sigma_1)$ is contained in the maximal ideal of $R_{0,\cL}$ as $\tilde{\rho}$ is a lifting of $\begin{pmatrix} \mathbf{1} & 0\\ 0 &\bar\chi\end{pmatrix}$, cf. proof of lemma \ref{defRbcLT}.

For simplicity, we denote $(R^{\ps,\ord,\{\xi_v\}}_{1}/\kQ^{\ps})[[x]]$ by $R_1$ and $R_{0,\cL}/\kQ$ by $R_2$. By abuse of notation, $\kq$ can be viewed as a prime ideal of $R_2$. Let $\kq_x=\kq\cap R_1$. Then $R_1/\kq_x$ is also one-dimensional. We claim that for any $\sigma\in G_{F,S}$ and any integer $n>0$,
\begin{enumerate}
\item images of $\tilde{b}(\sigma),\tilde c(\sigma)$ in $R_2/\kq^n$ are integral over $R_1/\kq_x^n$;
\item the image of the map $(R_1)_{\kq_x}\to (R_2)_{\kq_x}$ contains the images of $\tilde b (\sigma),\tilde c(\sigma)$ in $(R_2)_{\kq_x}$. Note that here all the localizations are at prime $\kq_x$ as $R_1$-modules.
\end{enumerate}

The first claim implies that $R_2/\kq^n$ is a finite $R_1/\kq_x^n$-module generated by $\tilde b (\sigma_i),\tilde c(\sigma_i),i=1,\cdots,r$. Localizing at $\kq_x$, it follows from the second claim that the natural map $(R_1/\kq_x^n)_{\kq_x}\to (R_2/\kq^n)_{\kq_x}$ is surjective. Thus $(R_2/\kq^n)_{\kq_x}$ is local Artinian and $(R_2/\kq^n)_{\kq_x}=(R_2/\kq^n)_{\kq}$. Hence by passing to the limit, we have a natural surjection $\widehat{(R_1)_{\kq_x}}\to \widehat{(R_2)_{\kq}}$. Thus 
\[\dim R_1= \dim \widehat{(R_1)_{\kq_x}}+1\geq \dim \widehat{(R_2)_{\kq}}+1=\dim R_2.\]
From this, we easily deduce the proposition.

To see the first claim, it suffices to prove this when $n=1$. But this is clear as both $R_1/\kq_x$ and $R_2/\kq$ are one-dimensional complete local Noetherian domains with same residue fields.

To see the second claim, note that for any $\sigma\in G_{F,S}$, element $\tilde b (\sigma_1)\tilde c (\sigma)\in R_{0,\cL}$ is in the image of $R^{\ps,\ord,\{\xi_v\}}_{1}$ (cf. \ref{adrep}). Since $x\in R_1$ maps to $\tilde b (\sigma_1)\notin \kq$, it is clear that all the $\tilde c(\sigma)$ are in the image of the map $(R_1)_{\kq_x}\to (R_2)_{\kq_x}$. Since we assume $\rho(\kq)$ is irreducible, there exists some $\tilde c(\sigma_0)\notin \kq$. Hence by the same argument, all $\tilde b(\sigma)$ are also in the image. This finishes the proof of the proposition.
\end{proof}

\begin{para}[Reducible locus of $R_{b,\cL}$] \label{redlocbcl}
As in \ref{redlocrc}, we give an upper bound on the dimension of the reducible locus of $R_{b,\cL}$. More precisely, let $\tilde\rho: G_{F,S}\to \GL_2(R_{b,\cL})$ be a universal deformation. The subset $\kp\in\Spec R_{b,\cL}$ such that $\tilde\rho\otimes k(\kp)$ is reducible is closed in $\Spec R_{b,\cL}$ and we denote its reduced subscheme by $\Spec R_{b,\cL}^{\mathrm{red}}$.

As before, we denote the $\Z_p$-rank of the maximal pro-$p$ abelian quotient of $G_{F,S}$ by $\delta_F+1$ and we choose a set of elements $\tau_0,\cdots,\tau_{\delta_F}\in G_{F(\bar\chi)}$ whose images in $G_{F,S}^{\mathrm{ab}}(p)\otimes_{\Z_p} \Q_p$ form a basis. Recall that $H^1_{v_0}(F)$ is the kernel of $H^1(G_{F,S},\F(\bar{\chi}^{-1}))\to H^1(G_{F_{v_0}},\F(\bar{\chi}^{-1}))$. Choose $\sigma_1,\cdots,\sigma_{r_s}\in G_{F(\bar{\chi})}$ such that they form a basis of $H^1_{v_0}(F)^\vee$ under the canonical pairing $H^1_{v_0}(F)\times G_{F(\bar{\chi})}\to\F$. Here $r_s=\dim_{\F} H^1_{v_0}(F)$. See \ref{redlocrc} for more details here.

Write $\tilde\rho=\begin{pmatrix} \tilde a & \tilde b \\ \tilde c & \tilde d\end{pmatrix}$. We may always assume  $\tilde\rho(\sigma^*)=\begin{pmatrix} 1 & 0\\ 0 & -1 \end{pmatrix}$. This also implies that $\bar\chi$ is non-trivial as $\bar\chi(\sigma^*)=-1$.
\end{para}

\begin{prop} \label{reddimRbcL}
\hspace{2em}
\begin{enumerate}
\item If $b\neq 0$, then $\dim R^{\mathrm{red}}_{b,\cL}/(\varpi)\leq \delta_F+\dim_\F H^1_{v_0}(F)+2$.
\item If $b=0$, assume $\bar\chi$ is quadratic. Let $I$ be the ideal of $R^{\mathrm{red}}_{0,\cL}$ generated by $\varpi$ and the image of elements 
\begin{itemize}
\item $\tilde a(\tau_i)+\tilde d(\tau_i)-2$, $i=0,\cdots,\delta_F$,
\item $\tilde{b}(\sigma_i)+\tilde{c}(\sigma_i)$, $i=1,\cdots, r_s$.
\end{itemize}
Then $R^{\mathrm{red}}_{0,\cL}/I$ is Artinian. In particular, $\dim R^{\mathrm{red}}_{0,\cL}/(\varpi)\leq \delta_F+\dim_\F H^1_{v_0}(F)+1$.
\end{enumerate}
\end{prop}

\begin{proof}
When $b\neq 0$, the argument is the same as the proof of proposition \ref{redloc} except that there is one more dimension coming from the lifting of $\cL$.

Now assume $b=0$. Let $\kq\in\Spec R^{\mathrm{red}}_{0,\cL}/I$. Note that $\tilde\rho(\sigma^*)=\begin{pmatrix} 1 & 0\\ 0 & -1 \end{pmatrix}$, hence by the same argument as in \ref{redloc}, it is easy to see that $\rho(\kq):=\tilde\rho\otimes k(\kq)$ is either of the form $\begin{pmatrix} 1 & * \\ 0 & \bar\chi \end{pmatrix}$ or $\begin{pmatrix} 1 & 0 \\ * & \bar\chi \end{pmatrix}$. We will only treat the first case. The second case is similar as we assume $\bar\chi$ is quadratic. By our assumption, $\rho(\kq)|_{G_{F_{v_0}}}$ is split, hence $*$ can be viewed as an element in $H^1_{v_0}(F)\otimes k(\kq)$. Since $\kq$ contains all $\tilde{c}(\sigma)$, from the construction of $I$, we see that $\tilde{b}(\sigma_i)=0,i=1,\cdots, r_s$, hence $*=0$. This implies that $\kq$ is the maximal ideal of $R^{\mathrm{red}}_{0,\cL}$ and $R^{\mathrm{red}}_{0,\cL}/I$ is Artinian.
\end{proof}

\subsection{\texorpdfstring{$R_{0,\cL}$}{Lg} is pro-modular}
\begin{para}
Keep the notations as in the previous subsection. We first show the existence of Eisenstein maximal ideal of the ordinary Hecke algebra $\T^{\ord}$ under certain assumptions.
\end{para}

\begin{lem} \label{quaeximai}
Assume $\bar\chi$ is quadratic. Then $T_v-(1+\chi(\Frob_v)),v\notin S$ and $\psi_{v,1}(\gamma)-1,\gamma\in F_v^\times,v|p$ and $\varpi$ generate a maximal ideal of $\T^\ord$.
\end{lem}

\begin{proof}
After possibly enlarging $\cO$, it follows from the global class field theory that there exists a character $\psi_0:G_{F(\bar\chi)}\to\cO^\times$ such that 
\begin{itemize}
\item $\psi_0$ is unramified outside of places above $S$.
\item For any place $w\in S\setminus \Sigma_p$, there exists a place $w'$ of $F(\bar\chi)$ above $w$ such that 
\[\psi_0|_{I_{F(\bar\chi)_{w'}}}=\xi_w.\]
Here since $w$ splits in $F(\bar\chi)$, we identify $I_{F(\bar\chi)_{w'}}$ with $I_{F_w}$ and $\xi_w$ as a character of $I_{F(\bar\chi)_{w'}}$.
\item $\psi_0\equiv 1\mod \varpi$.
\item $\psi_0$ is de Rham at all places above $p$. Moreover for any embedding $\iota:F(\bar\chi)\hookrightarrow \overbar{\Q_p}$, the $\iota$-Hodge-Tate weight is not equal to the $(\iota\circ c)$-Hodge-Tate weight. Here $c\in G_{F}$ is a complex conjugation.
\item $\psi_0\cdot \psi_0^c=\chi|_{G_{F(\bar\chi)}}$, where $\psi_0^c$ is the complex conjugate of $\psi_0$.
\end{itemize}

We call a character regular de Rham if it satisfies these conditions. Since we are going to need a similar result later, we provide details for its existence here. Let $\tilde{S}$ be the set of places of $F(\bar\chi)$ above $S$. Note that any place $v$ in $S$ splits in $F(\bar\chi)$. Choose a place $v'$ of $F(\bar\chi)$ above $v$. Let $\Sigma'$ be the union of all such $v'$ and $\Sigma'_p$ be the subset of places above $p$. If $v\in S\setminus\Sigma_p$, we may identify $\xi_v$ as a character of $I_{F(\bar\chi)_{v'}}$ as above. 

Consider the universal deformation ring $R^{univ}_1$ which parametrizes all liftings $\psi$ of the trivial character $\mathbf{1}:G_{F(\bar\chi),\tilde{S}}\to\F^\times$ such that $\psi\cdot \psi^c=\chi|_{G_{F(\bar\chi)}}$ and $\psi|_{I_{F(\bar\chi)_{v'}}}=\xi_v,v\in S\setminus\Sigma_p$. Write $\psi^{univ}:G_{F(\bar\chi),\tilde{S}}\to (R^{univ}_1)^\times$ as the universal character. By the local class field theory, $\psi^{univ}|_{I_{F(\bar\chi)_{v'}}}$ induces a map $\cO[[O_{F(\bar\chi)_{v'}}^\times(p)]]\to R^{univ}_1$, where as usual $O_{F(\bar\chi)_{v'}}^\times(p)$ denotes the pro-$p$ completion of $O_{F(\bar\chi)_{v'}}^\times$. Let $\Lambda_{\Sigma'_p}=\hat{\otimes}_{v'\in\Sigma'_p}\cO[[O_{F(\bar\chi)_{v'}}^\times(p)]]$, the completed tensor product over $\cO$. We get a natural map
\[\Lambda_{\Sigma'_p}\to R^{univ}_1.\]

We claim that this map is finite and $\Spec R^{univ}_1 \to \Spec \Lambda_{\Sigma'_p}$ is surjective. Clearly this will imply the existence of $\psi_0$ as we claimed above. To see the claim, by twisting a character of $G_F$, we may assume $\chi|_{G_{F(\bar\chi)}}$ is trivial. Let $H'$ be the maximal pro-$p$ abelian  quotient of $G_{F(\bar\chi),\tilde{S}}$ and $H$ be the quotient of $H'$ by elements of the form $\tau\tau^c$. Here $\tau^c$ denotes the complex conjugation of $\tau$. Then the local class field theory induces natural maps $\iota_{v'}:O_{F(\bar\chi)_{v'}}^\times(p)\to H$. Note that $R^{univ}_1$ is simply the quotient of $\cO[[H]]$ by elements of the form $\iota_{v'}(u)-\xi_v(u)$ with $v\in S\setminus\Sigma_p$ and $u\in O_{F(\bar\chi)_{v'}}^\times(p)$. Here by abuse of notation, we again view $\xi_v$ as characters of $O_{F(\bar\chi)_{v'}}^\times(p)$. Now to deduce our claim, it is enough to prove

\begin{enumerate}
\item The natural map $\prod_{v'\in \Sigma'\setminus \Sigma'_p} O_{F(\bar\chi)_{v'}}^\times(p)\to H$ is injective.
\item The natural map $\prod_{v'\in \Sigma'_p} O_{F(\bar\chi)_{v'}}^\times(p)\to H$ is injective and the cokernel is finite.
\end{enumerate}

By the global class field theory, $H'$ has a finite index subgroup $O^\times_{F(\bar\chi),p}(p)/J$ for some subgroup $J$ of $O_{F(\bar\chi)}^\times$ of finite index. Note that $H$ is $H'$ modulo elements fixed by the complex conjugation. Using this description and the fact that $J\cap O_F^\times$ is of finite index in $J$, it is easy to see that the kernel and cokernel of the map $\prod_{v'\in \Sigma'_p} O_{F(\bar\chi)_{v'}}^\times(p)\to H$ are finite. The kernel has to be trivial since $p$ is unramified in $F(\bar\chi)$. This finishes the proof of our second claim. The first claim can be proved similarly and we omit the details here.

Now let $\psi_0$ be a character as in the beginning of the proof and let $\rho_0=\Ind_{G_{F(\bar\chi)}}^{G_F}\psi_0$. Then $\det \rho_0=\chi$. Since any place of $F$ above $p$ splits in $F(\bar\chi)$, the representation $\rho_0|_{G_{F_v}}$ is completely reducible. Moreover since we assume $\bar\chi$ is quadratic, the semi-simplification of the residual representation of $\rho_0$ is $\mathbf{1}\oplus \bar\chi$. By the theory of CM forms, there exists a regular algebraic ordinary cuspidal automorphic representation of $\GL_2(\A_F)$ corresponding to $\rho_0$. Clearly this automorphic representation gives rise to the maximal ideal of the Hecke algebra in the lemma. 
\end{proof}

\begin{para}
Keep the assumption as in the lemma that $\bar\chi$ is quadratic. We denote the maximal ideal in the lemma \ref{quaeximai} by $\km$. Then we have a non-zero surjective map $R^{\ps,\ord,\{\xi_v\}}_{1}\to \T^{\ord}$. Let $(0,\cL)$ be an admissible pair as in the previous subsection. Consider the natural map $R^{\ps,\ord,\{\xi_v\}}_{1}\to R_{0,\cL}$. We say a prime of $R_{0,\cL}$ is \textit{pro-modular} if its pull-back to $R^{\ps,\ord,\{\xi_v\}}_{1}$ is pro-modular, i.e. in the image of $\Spec \T^{\ord}_\km\to \Spec R^{\ps,\ord,\{\xi_v\}}_{1}$. The main result of this subsection is
\end{para}

\begin{prop} \label{R0cLmod}
Assume $\bar\chi$ is quadratic and
\begin{itemize}
\item $[F:\Q]-4|S|+4|\Sigma_p|-2>\delta_F+\dim_\F H^1_{v_0}(F)$.
\end{itemize}
Then $R_{0,\cL}$ is pro-modular.
\end{prop}

\begin{proof}
By lemma \ref{quaeximai}, we know that $\T^{\ord}_\km\neq 0$. First we claim that there exists a nice prime (in the sense of \ref{ordnice}) in the image of $\Spec R_{0,\cL}\to \Spec R^{\ps,\ord,\{\xi_v\}}_{1}$. In view of proposition \ref{propC}, this will imply that at least one irreducible component of $R^{\ps,\ord,\{\xi_v\}}_{1}$ is pro-modular.

Such a nice prime can be found on a CM component of $\T^{\ord}_\km$. More precisely, let $R^{univ}_1$ be the deformation ring defined in the proof of lemma \ref{quaeximai}. Then taking the trace of the induction of the universal character $\psi^{univ}$ over $R^{univ}_1$ from $G_{F(\bar\chi)}$ to $G_F$ induces a natural map $R^{\ps,\ord,\{\xi_v\}}_{1}\to R^{univ}_1$. Also we proved in the proof of lemma \ref{quaeximai} that there exists an irreducible component $C$ of $R^{univ}_1$ whose map to $\Spec \Lambda_{\Sigma'}$ is finite and surjective (see notations there). Therefore, regular de Rham points are dense in $C$. It follows from the theory of CM forms that the image of $C$ in $\Spec R^{\ps,\ord,\{\xi_v\}}_{1}$ is pro-modular.

Now let $\kq'$ be any one-dimensional prime of $C$ containing $p$ and $\psi^{univ}(\Frob_{v'})-1,v'\in\Sigma'\setminus\Sigma'_p$. It is easy to check that its image $\kq$ in $\Spec \T^{\ord}_\km$ is a nice prime. Moreover, the semi-simple representation associated to $\kq$ is of the form $\Ind_{G_{F(\bar\chi)}}^{G_F}\psi_1$, for some character $\psi_1:G_{F(\bar\chi)}\to (\T^{\ord}_\km/\kq)^\times$. This implies that $\kq$ is also in the image of $\Spec R_{0,\cL}$. Hence we find a nice prime as claimed in the beginning.

After showing that at least one irreducible component of $R^{\ps,\ord,\{\xi_v\}}_{1}$ is pro-modular, the rest of the proof goes exactly the same as the proof of proposition \ref{rcmodc}: it follows from our assumption, lemma \ref{cRbcL} and proposition \ref{reddimRbcL} that 
\[c(R_{0,\cL}/(\varpi)) > 3(|S|-|\Sigma_p|) + \dim R^{\mathrm{red}}_{0,\cL}/(\varpi).\]
This implies that any irreducible component of $R^{\ps,\ord,\{\xi_v\}}_{1}$ has to be pro-modular since different irreducible components can be ``connected" by nice primes. We omit the details here.
\end{proof}

\subsection{\texorpdfstring{$R_{b,\cL}$}{Lg} is pro-modular}
In this subsection, we extend the result of the previous subsection to any admissible pairs $(b,\cL)$ with $b\neq 0$. As before, we say a prime of $R_{b,\cL}$ pro-modular if its image in $\Spec R^{\ps,\ord,\{\xi_v\}}_{1}$ is pro-modular.

\begin{prop} \label{RbcLpro-modular}
Assume $\bar\chi$ is quadratic and 
\begin{itemize}
\item $[F:\Q]-6|S|+6|\Sigma_p|-3>\delta_F+\dim_\F H^1_{v_0}(F)$.
\end{itemize}
Then $R_{b,\cL}$ is pro-modular for any admissible pair $(b,\cL)$ with $b\neq 0$.
\end{prop}

\begin{para}
The strategy is the same as the proof of proposition \ref{R0cLmod}: we first find a nice prime in the image of $\Spec R_{b,\cL}\to \Spec R^{\ps,\ord,\{\xi_v\}}_{1}$ and use the numerology $c(R_{b,\cL}/(\varpi)) > 3(|S|-|\Sigma_p|) + \dim R^{\mathrm{red}}_{b,\cL}/(\varpi)$ to deduce that any irreducible component of $R_{b,\cL}$ is pro-modular. The question is how to find such a nice prime.

The basic idea is the same as the proof of Proposition \ref{propB}: we carefully choose a prime ($\kQ^{\ps}$ below) in the image of $\Spec R_{0,\cL}\to\Spec R^{\ps,\ord,\{\xi_v\}}_{1}$ and then specialize it to a nice prime in the image of $\Spec R_{b,\cL}\to\Spec R^{\ps,\ord,\{\xi_v\}}_{1}$. The main difficulty here is to find such a specialization whose associated representation has certain $G_{F_{v_0}}$-stable line lifting $\cL$. For this purpose,  we will  distinguish $3$ possibilities of $\kQ^{\ps}$ in \ref{case1bcLnot}, \ref{case2a}, \ref{case3d}.

\end{para}

\begin{para}[\textbf{Choice of $\kQ^{\ps}$}] \label{ChoicekQ}
Now we introduce some notations. Note that $b$ gives rise to a representation $\bar\rho_b:G_{F,S}\to\GL_2(\F)$ of the form
\[\sigma\mapsto \begin{pmatrix} 1 & \bar{b}(\sigma) \\ 0 & \bar\chi(\sigma) \end{pmatrix}.\]
We may assume $\bar{b}(\sigma^*)=0$ where $\sigma^*\in G_F$ is our fixed complex conjugation.  

Recall that $H^1_{v_0}(F)$ is defined as the kernel of $H^1(G_{F,S},\F(\bar\chi^{-1}))\to H^1(G_{F_{v_0}},\F(\bar\chi^{-1}))$. Then as in \ref{redlocrc},  we have an injection $H^1_{v_0}(F)\to \Hom (G_{F(\bar\chi),\tilde{S}},\F)$. Let $r_s=\dim_{\F} H^1_{v_0}(F)$. Choose elements $\sigma_1,\cdots,\sigma_{r_s}\in G_{F(\bar\chi),\tilde{S}}$ such that they form a dual basis of $H^1_{v_0}(F)^\vee$ and $\bar{b}(\sigma_1)=1$ and $\bar{b}(\sigma_i)=0,i=2,\cdots,r_s$. For any $v\in S\setminus \Sigma_p$, the maximal pro-$p$ quotient $G_{F_v}(p)$ of $G_{F_v}$ is topologically generated by $2$ elements. Let $\tau_{v,1},\tau_{v,2}$ be such topological generators.

Let $(0,\cL_0)$ be an admissible pair and $(\tilde\rho,\{\tilde\psi_{v,1}\}_{v|p})$ be the universal tuple over $R_{0,\cL_0}$ as in the proof of proposition \ref{psdefcL}. Write $\tilde\rho=\begin{pmatrix} \tilde a & \tilde b \\ \tilde c & \tilde d\end{pmatrix}$. Consider the following deformation:
\[\rho_{1}:G_{F,S}\to \GL_2(\F[[T]]),~\sigma\mapsto \begin{pmatrix}1 & T\bar{b}(\sigma)\\ 0 & \bar{\chi}(\sigma) \end{pmatrix}.\]
If we take $T=1$ (at least formally), then this becomes $\bar\rho_b$. Though we couldn't do this in $R_{0\,\cL_0}$, we will try to make this happen in $R^{\ps,\ord,\{\xi_v\}}_{1}$.

In view of the description of $R_{0\,\cL_0}$ in the proof of lemma \ref{defRbcLT}, the tuple $(\rho_1,\{\mathbf{1}\}_{v|p})$ defines a prime ideal $\kq_1$ of $R_{0,\cL_0}$. Let $I$ be the ideal of $R_{0,\cL_0}$ generated by $\varpi$ and $\tilde{b}(\sigma_2),\cdots,\tilde{b}(\sigma_{r_s})$ and for $v\in S\setminus\Sigma_p$ and for $i\in\{1,2\}$,
\begin{itemize}
\item $\tilde{c}(\tau_{v,i}),\tilde{a}(\tau_{v,i})-1$,
\item $\tilde{b}(\tau_{v,i})-\tilde{b}(\sigma_1)\widetilde{\bar{b}(\tau_{v,i})}$, where $\widetilde{\bar{b}(\tau_{v,i})}\in\cO$ is a lifting of $\bar{b}(\tau_{v,i})\in\F$.
\end{itemize}
It is easy to check that $I$ is contained in $\kq_1$ and $\tilde{d}(\tau_{v,i})-1\in I$ for $v\in S\setminus\Sigma_p,i\in\{1,2\}$. 

\begin{rem}
Roughly  speaking, the reason for adding $\tilde{b}(\sigma_2),\cdots,\tilde{b}(\sigma_{r_s})$ is to force the deformation we obtain later having reduction $\bar\rho_b$ once we know it is trivial at $v_0$. Similarly, adding these entries at places $v\in S\setminus \Sigma_p$ is to make sure the deformation is trivial at $v\in S\setminus\Sigma_p$, which is part of our definition of a nice prime in \ref{ordnice}. See the last two paragraphs of the proof of lemma \ref{case1bcL} below.
\end{rem}

\begin{rem}
The deformation $\rho_1$ is not new: it is nothing but $\bar{\rho}_{b_1}$ in the proof of proposition \ref{propB} (taking $b_1=0$), which first appeared in the work of Skinner-Wiles \cite{SW99}. Though it is a non-trivial deformation of $1\oplus\bar\chi$, its trace is a trivial deformation as a pseudo-representation. We will use this fact in the proof of lemma \ref{xynex} below.
\end{rem}

Let $\kQ$ be a minimal prime of $R_{0,\cL_0}/I$ contained in $\kq_1$. By lemma \ref{dimRbcL}, the dimension of $R_{0,\cL_0}/\kQ$ is at least 
\[[F:\Q]+1-\dim_{\F} H^1_{v_0}(F)-6(|S|-|\Sigma_p|).\]
By our assumption, this is at least $3+\delta_F$. We claim that $\kQ\notin \Spec R^{\mathrm{red}}_{0,\cL_0}$. Suppose not, then $\tilde{b}(\sigma)\tilde{c}(\tau)\in\kQ$ for any $\sigma,\tau\in G_{F,S}$. Since $\tilde{b}(\sigma_1)\notin\kq_1$, we must have $\tilde{c}(\sigma)\in \kQ$ hence $\tilde{b}(\sigma_i)+\tilde{c}(\sigma_i)\in\kQ,i=2,\cdots,r_s$. Now using proposition \ref{reddimRbcL} and arguing as in the proof of proposition \ref{propB}, we deduce that $\dim R_{0,\cL_0}/\kQ$ is at most $2+\delta_F$. Contradiction. Hence $\kQ$ is not in the reducible locus and we may apply proposition \ref{psdefcL} and conclude that
\[\dim R^{\ps,\ord,\{\xi_v\}}_{1}/\kQ^{\ps}\geq 2+\delta_F.\]
where $\kQ^{\ps}=\kQ\cap R^{\ps,\ord,\{\xi_v\}}_{1}$. Then $\kQ^{\ps}$ is pro-modular by proposition \ref{R0cLmod}. Our goal is to find a nice prime which contains $\kQ^{\ps}$ and lies in the image of $\Spec R_{b,\cL}\to \Spec R^{\ps,\ord,\{\xi_v\}}_{1}$. There are three cases.
\begin{enumerate}
\item $\tilde{a}(\sigma)-\tilde{\psi}_{v_0,1}(\sigma),\tilde{d}(\sigma)-\tilde{\psi}_{v_0,1}(\sigma)\in \kQ$ for any $\sigma\in G_{F_{v_0}}$.
\item $\tilde{a}(\sigma)-\tilde{\psi}_{v_0,1}(\sigma)\notin \kQ$ for some $\sigma\in G_{F_{v_0}}$.
\item $\tilde{d}(\sigma)-\tilde{\psi}_{v_0,1}(\sigma)\notin \kQ$ for some $\sigma\in G_{F_{v_0}}$.
\end{enumerate}
\end{para}

\begin{para} \label{case1bcLnot}
We first deal with the first case. Let $\kq$ be a prime ideal of $R^{\ps,\ord,\{\xi_v\}}_{1}$ containing $\kQ^{\ps}$ such that the associated $2$-dimensional representation is irreducible. Then $\kq$ is pro-modular. We claim that $\kq$ is a nice prime and lies in the image of $\Spec R_{b,\cL}\to \Spec R^{\ps,\ord,\{\xi_v\}}_{1}$. In fact, we will even show that it is in the image of $\Spec R_{b,\cL}$ for any $\cL$.

Let $A=\tilde{\F}[[T]]$ be the normalization of $R^{\ps,\ord,\{\xi_v\}}_{1}/\kq$, where $\tilde{\F}$ is a finite extension of $\F$. By Ribet's lemma, we may find a lattice of $\rho(\kq)$: $\rho^o:G_{F,S}\to\GL_2(A)$ such that the reduction $\bar\rho_\kq$ of $\rho^o\mod T$ has the form
\[\begin{pmatrix}1 & * \\ 0 & \bar{\chi} \end{pmatrix},~*\neq 0,\]
and $\rho^o(\sigma^*)=\begin{pmatrix}1 & 0\\ 0 & -1 \end{pmatrix}$. Write $\rho^o(g)=\begin{pmatrix} a'(g) & b'(g) \\ c'(g) & d'(g) \end{pmatrix}$. It suffices to prove
\end{para}

\begin{lem} \label{case1bcL}
\hspace{2em}
\begin{enumerate}
\item $\rho^o|_{G_{F_{v_0}}}$ is trivial. In particular, any line in $\rho^o$ is fixed by $G_{F_{v_0}}$.
\item Up to a scalar, $\bar\rho_\kq$ corresponds to the extension class $b$. Hence we may conjugate $\rho^o$ and assume $b'(\sigma_1)=1$.
\item Assume $b'(\sigma_1)=1$. For any $v\in S\setminus\Sigma_p$, the entries of $\rho^o(G_{F_v})$ belong to $\F$. Hence $\rho^o|_{G_{F_v}}$ is equal to $\bar\rho_\kq|_{G_{F_v}}$. Here both are viewed as taking values in $\GL_2(A)$.
\end{enumerate}
\end{lem}

\begin{proof}
We need some notations as in the proof of proposition \ref{propB}. Let $(T^{univ},\{\psi_{v,1}\}_{v|p})$ be the universal tuple over $R^{\ps,\ord,\{\xi_v\}}_{1}$ (see \ref{Rpsordxiv1}). As in \ref{adx}, we can define functions $a,d,x(\cdot,\cdot)$ associated to $T^{univ}$. Denote the natural map $R^{\ps,\ord,\{\xi_v\}}_{1}\to R_{0,\cL_0}$ by $\varphi$. Then we have $\tilde{a}=\varphi\circ a, a\equiv a'\mod\kq$ and similar results hold for $d,x,\psi_{v,1}$.

Because $\rho^o$ is irreducible, we can choose $\tau\in G_{F,S}$ such that $c'(\tau)\neq 0$. Since $G_{F_{v_0}}$ acts on $\begin{pmatrix} 1 \\ 1 \end{pmatrix}$ via $\tilde\psi_{v_0,1}$, for any $\sigma\in G_{F_{v_0}}$, we have $\tilde{b}(\sigma)=-(\tilde{a}(\sigma)-\tilde{\psi}_{v_0,1}(\sigma))\in\kQ$. Hence $\varphi (x(\sigma,\tau))=\tilde{b}(\sigma)\tilde{c}(\tau)\in\kQ$. This implies $b'(\sigma)c'(\tau)=0$. Thus $b'(\sigma)=0$. Similarly, we can show that $c'(\sigma)=0$. Also by our assumptions, $a'(\sigma)-\psi_{v_0,1}(\sigma)\equiv d'(\sigma)-\psi_{v_0,1}(\sigma)\equiv 0\mod\kq$. Hence $\rho^o|_{G_{F_{v_0}}}$ is scalar and has to be trivial as we fix the determinant and $p\neq 2$. This proves our first claim.

For the second part of the lemma, since $\rho^o|_{G_{F_{v_0}}}$ is trivial, the reduction $\bar\rho_\kq$ is trivial at $v_0$ and hence corresponds to an extension class in $H^1_{v_0}(F)$. Recall that $\tilde{b}(\sigma_2),\cdots,\tilde{b}(\sigma_{r_s})\in \kQ$ by our construction, we can argue as in the previous paragraph and get $b'(\sigma_i)=0$ for $i=2,\cdots,r_s$. As $\bar\rho_\kq$ is a non-split extension, our choice of these $\sigma_i$ in \ref{ChoicekQ} implies that the reduction of $b'(\sigma_1)\mod T$ has to be non-zero and $b'\mod T$ is a non-zero multiple of $\bar{b}$. 

To see the last claim, we note that $\rho^o|_{G_{F_{v}}}$ factors through the pro-$p$ quotient of $G_{F_v}$ for $v\in S\setminus\Sigma_p$. Recall that this quotient is topologically generated by $\tau_{v,1},\tau_{v,2}$. By our construction of $\kQ$,  for $i=1,2$, we have $a'(\tau_{v,i})-1=d'(\tau_{v,i})-1=c'(\tau_{v,i})=0$ and
\[b'(\tau_{v,i})=b'(\sigma_1)\bar{b}(\tau_{v,i})=\bar{b}(\tau_{v,i})\in\F.\]
From these identities, we conclude that $\rho^o(G_{F_v})\subseteq \GL_2(\tilde{\F})\subseteq \GL_2(A)$. 
\end{proof}

\begin{para} \label{case2a}
Now we treat the case that $\tilde{a}(\sigma)-\tilde{\psi}_{v_0,1}(\sigma)\notin \kQ$ for some $\sigma\in G_{F_{v_0}}$. 
 Fix one such $\sigma$.  The idea is to use this element $\tilde{a}(\sigma)-\tilde{\psi}_{v_0,1}(\sigma)$ to find points in the image of $\Spec R_{b,\cL}$.

We will freely use the notation introduced in the first paragraph of the proof of lemma \ref{case1bcL}. Then we have $a(\sigma)-\psi_{v_0,1}(\sigma)\notin \kQ^{\ps}$. Choose $\tau\in G_{F,S}$ such that $\tilde{c}(\tau)\notin\kQ$. This is possible as $\kQ\notin \Spec R^{\mathrm{red}}_{0,\cL_0}$. Note that $\tilde{b}(\sigma)=-\tilde{a}(\sigma)+\tilde{\psi}_{v_0,1}(\sigma)\notin \kQ$. Hence $\varphi(x(\sigma,\tau))=\tilde{b}(\sigma)\tilde{c}(\tau)\notin\kQ$. It follows that $x(\sigma,\tau)\notin\kQ^{\ps}$. Similarly we have $(-a(\sigma)+\psi_{v_0,1}(\sigma))x(\sigma_1,\tau)\notin \kQ^{\ps}$.

Let $x,y$ be the images of $(-a(\sigma)+\psi_{v_0,1}(\sigma))x(\sigma_1,\tau)$ and $x(\sigma,\tau)$ in $R^{\ps,\ord,\{\xi_v\}}_{1}/\kQ^{\ps}$ under the natural quotient map. Let $A_1=R^{\ps,\ord,\{\xi_v\}}_{1}/\kQ^{\ps}$ and $A_2$ be the $A_1$-subalgebra of the fraction field of $A_1$ generated by $\frac{y}{x}$. Denote the maximal ideal of $A_1$ by $\km_1$. The following lemma is inspired by the result of Skinner-Wiles (see Proposition \ref{ht1def}).
\end{para}

\begin{lem} \label{xynex}
$(\km_1,\frac{y}{x}-n)$ is a maximal ideal of $A_2$ for any $n\in \F$.
\end{lem}

\begin{proof}
Write $t=\frac{y}{x}$. Suppose $(\km_1,\frac{y}{x}-n)$ is not a maximal ideal. Then we can find $f_1(X),f_2(X)\in A_1[X]$ such that all the coefficients of $f_1(X)$ belong to $\km_1$ and
\[f_1(t)+(t-n)f_2(t)=1.\]

Note that $f_1(n)\in\km_1$. Hence after possibly replacing $f_2(X)$ by another polynomial, we may assume $f_1(t)=0$, i.e. $(t-n)f_2(t)=1$. Suppose $f_2(X)$ has degree $k$. Then 
\[(1-\frac{n}{t})g_2(\frac{1}{t})=\frac{1}{t^{k+1}}\]
for some $g_2(X)\in A_1[X]$ of degree $k$.

Consider the natural inclusion $\varphi_1:R^{\ps,\ord,\{\xi_v\}}_{1}/\kQ^{\ps}\hookrightarrow R_{0,\cL_0}/\kQ$. Then $\frac{\varphi_1(x)}{\varphi_1(y)}$ is the reduction of $\tilde{b}(\sigma_1)$ modulo $\kQ$. Therefore
\[(1-n\tilde{b}(\sigma_1))g_2(\tilde{b}(\sigma_1))\equiv \tilde{b}(\sigma_1)^{k+1} \mod \kQ.\]

Recall that there are a prime ideal $\kq_1$ of $R_{0,\cL_0}/\kQ$ and a natural isomorphism $R_{0,\cL_0}/\kq_1\stackrel{\sim}{\longrightarrow} \F[[T]]$ sending $\tilde{b}(\sigma_1)$ to $T$. Moreover, the pull-back of $\kq_1$ to $R^{\ps,\ord,\{\xi_v\}}_{1}$ is the maximal ideal. Now we modulo the above equation by $\kq_1$. We get
\[(1-nT)g_3(T)=T^{k+1}\]
for some polynomial $g_3(t)\in\F[T]$ of degree at most $k$. But this is clearly impossible. Thus $(\km_1,\frac{y}{x}-n)$ is a maximal ideal.
 \end{proof}

Write $\cL$ as $\F\cdot\begin{pmatrix} n \\ 1\end{pmatrix}$ for some $n\in \F$. We define $\km_2\in\Spec A_2$ to be the maximal ideal $(\km_1,\frac{y}{x}-n)$ and denote by $\widehat{(A_2)_{\km_2}}$ the $\km_2$-adic completion of $A_2$. Let $V_1\subseteq \Spec A_1$ be the closed subset defined by $(x)$ and $V_2\subseteq \Spec A_1$ be the closed subset corresponding to reducible deformations (see \ref{redlrps}). Then their union $V$ is a proper closed subset of $\Spec A_1$. We denote its complement by $U$.

Consider $\Spec \widehat{(A_2)_{\km_2}}\to \Spec (A_2)_{\km_2}\to \Spec A_1$. The first map is surjective and the generic point of $A_1$ is in the image of the second map. Hence the preimage $U'$ of $U$ in $\Spec \widehat{(A_2)_{\km_2}}$ is a non-empty open set. Let $\kq'$ be a one-dimensional prime in $U'$ and $\kq=\kq'\cap A_1$. This is a pro-modular prime as it contains $\kQ^{\ps}$. Also it is not in reducible locus by our construction. We claim that $\kq$ is a nice prime and in the image of $\Spec R_{b,\cL}\to\Spec R^{\ps,\ord,\{\xi_v\}}_{1}$.

As in \ref{case1bcLnot}, we can find a representation $\rho^o:G_{F,S}\to\GL_2(A)$ with $A=\tilde{\F}[[T]]$ the normalization of $R^{\ps,\ord,\{\xi_v\}}_{1}/\kq$ such that $\rho^o(\sigma^*)=\begin{pmatrix}1 & 0\\ 0 & -1 \end{pmatrix}$ and the mod $T$ reduction $\bar{\rho}_\kq$ of $\rho^o$ is of the form $\begin{pmatrix}1 & * \\ 0 & \bar{\chi} \end{pmatrix},~*\neq 0$. Moreover if we write $\rho^o(g)=\begin{pmatrix} a'(g) & b'(g) \\ c'(g) & d'(g) \end{pmatrix}$, then $a'\equiv a\mod \kq, d'\equiv d\mod\kq$.

\begin{lem} \label{case2bcL}
\hspace{2em}
\begin{enumerate}
\item $\rho^o(G_{F_{v_0}})$ acts on $A\cdot \begin{pmatrix} n_0\\1\end{pmatrix}$ via $\psi_{v_0,1}$ for some $n_0\in A$ and  $n_0\equiv b'(\sigma_1)n\mod (T)$.
\item $\bar{\rho}_\kq|_{G_{F_{v_0}}}$ is trivial.
\end{enumerate}
\end{lem}

\begin{proof}
It follows from our construction that $x=-(a(\sigma)-\psi_{v_0,1}(\sigma))x(\sigma_1,\tau)\notin \kq$. Hence 
\[(a'(\sigma)-\psi_{v_0,1}(\sigma))b'(\sigma_1)c'(\tau)\neq 0.\]
Since $\kq$ is a pull-back of a prime ideal of $\widehat{(A_2)_{\km_2}}$, it is clear that
\[\frac{y}{x}-n\mod \kq=-\frac{b'(\sigma)}{(a'(\sigma)-\psi_{v_0,1}(\sigma))b'(\sigma_1)}-n\in T\tilde{\F}[[T]].\]
Therefore $\frac{b'(\sigma)}{a'(\sigma)-\psi_{v_0,1}(\sigma)}\equiv -b'(\sigma_1)n\mod (T)$. Note that $a'(\sigma)-\psi_{v_0,1}(\sigma)\neq 0$. The matrix $\rho^o(\sigma)-\psi_{v_0,1}(\sigma)$ has rank one and its kernel is generated by $A\cdot \begin{pmatrix} n_0\\1\end{pmatrix}$  with $n_0\in b'(\sigma_1)n+(T)$. Hence $\rho^o(G_{F_{v_0}})$ acts on $A\cdot \begin{pmatrix} n_0\\1\end{pmatrix}$ via $\psi_{v_0,1}$. Note that reduction of this line modulo $T$ cannot be $\F\cdot \begin{pmatrix} 1 \\ 0 \end{pmatrix}$. Hence $\bar{\rho}_\kq|_{G_{F_{v_0}}}$ is trivial.
\end{proof}

Now we can argue as in the proof of lemma \ref{case1bcL} and conclude that $\bar{\rho}_\kq$ is the extension class $b$ up to a scalar. After renormalizing $\rho^o$, we may assume $b'(\sigma_1)=1$ and $\bar{\rho}_\kq=\bar{\rho}_b$. Then lemma \ref{case2bcL} implies that $\kq$ is in the image of $\Spec R_{b,\cL}\to\Spec R^{\ps,\ord,\{\xi_v\}}_{1}$. Arguing as in the last part of the proof of lemma \ref{case1bcL}, it is easy to see that $\kq$ is nice. Thus we found a nice prime in the image of $\Spec R_{b,\cL}\to\Spec R^{\ps,\ord,\{\xi_v\}}_{1}$ in this case.

\begin{para} \label{case3d}
The last case where $\tilde{d}(\sigma)-\tilde{\psi}_{v_0,1}(\sigma)\notin \kQ$ for some $\sigma\in G_{F_{v_0}}$ can be proved similarly. We omit the details here. Thus we have finished the proof of proposition \ref{RbcLpro-modular}.
\end{para}

\begin{cor} \label{corC}
Assume $\bar\chi$ is quadratic and 
\begin{itemize}
\item $[F:\Q]-6|S|+6|\Sigma_p|-3>\delta_F+\dim_\F H^1_{v_0}(F)$.
\end{itemize}
We have 
\begin{enumerate}
\item A prime $\kp\in\Spec R^{\ps,\ord,\{\xi_v\}}_{1}$ is pro-modular if $\rho(\kp)$ is irreducible.
\item $R^{\ps,\ord,\{\xi_v\}}_{1}$ is a finite $\Lambda_F$-algebra. See the discussion above Theorem \ref{thmC} for the definition of this map.
\item If $\kp$ is a maximal ideal of $R^{\ps,\ord,\{\xi_v\}}_{1}[\frac{1}{p}]$ such that 
\begin{itemize}
\item $\rho(\kp)$ is irreducible.
\item Write $\rho(\kp)|_{G_{F_v}}\cong \begin{pmatrix} \psi_{v,1} & * \\ 0 & \psi_{v,2}\end{pmatrix}$. We assume $\psi_{v,1}$ is Hodge-Tate and has strictly less Hodge-Tate number than $\psi_{v,2}$ for any $v|p$ and any embedding $F_v\hookrightarrow \overbar{\Q_p}$.
\end{itemize}
Then $\rho(\kp)$ comes from a twist of a Hilbert modular form.
\end{enumerate}
\end{cor}

\begin{proof}
The proof is almost the same as the proof of corollary \ref{corB}. We will show that up to twist, each prime $\kp$  is in the image of $\Spec R_{b,\cL}$ for some admissible pair $(b,\cL)$ if $\rho(\kp)$ is irreducible and $R^{\ps,\ord,\{\xi_v\}}_{1}/\kp$ has dimension one. 

Denote by $A$ the normalization of $R^{\ps,\ord,\{\xi_v\}}_{1}/\kp$ and by $T$ a uniformizer of $A$. We define a basis of the representation space of $\rho(\kp)$ in the following way: let $\mathbf{e}_2$ be a non-zero eigenvector of $\sigma^*$ with eigenvalue $-1$. The lattice $A[G_{F,S}]\cdot \mathbf{e}_2$ can be written as $A\cdot\mathbf{e}_1\oplus A\cdot \mathbf{e}_2$ for some vector $\mathbf{e}_1$ fixed by $\sigma^*$. Then $A[G_{F,S}]\cdot \mathbf{e}_1=A\cdot\mathbf{e}_1\oplus A\cdot T^n\mathbf{e}_2$ for some integer $n>0$. Let $\cL_0$ be a line where $\rho(\kp)(G_{F_{v_0}})$ acts via $\psi_{v,1}$. Suppose $\cL_0$ is generated by $f_1\mathbf{e}_1+f_2\mathbf{e}_2$ with $f_1,f_2\in A$. For an element $f\in A$, denote its $T$-adic valuation by $v_T(f)$. There are three possibilities:
\begin{enumerate}
\item $v_T(f_1)\geq v_T(f_2)$.
\item $v_T(f_2)-n<v_T(f_1)< v_T(f_2)$.
\item $v_T(f_1)\leq v_T(f_2)-n$.
\end{enumerate}

In the first case, we consider the $G_{F,S}$-stable lattice $A\cdot\mathbf{e}_1\oplus A\cdot \mathbf{e}_2$. Its reduction modulo $T$ is a non-split extension of $\bar{\chi}$ by $\mathbf{1}$. Moreover, the reduction of $\cL_0$ is not $\mathbf{e}_1$ modulo $T$. Hence $\kp$ is in the image of $\Spec R_{b,\cL}$ for some admissible pair $(b,\cL),b\neq 0$.

In the second case, we can consider the $G_{F,S}$-stable lattice $A\cdot f_1\mathbf{e}_1\oplus A\cdot f_2\mathbf{e}_2$. The mod $T$ reduction of this lattice is $\mathbf{1}\oplus\bar\chi$ and the reduction of $\cL_0$ is clearly not $G_{F,S}$-stable. Hence $\kp$ is in the image of $\Spec R_{0,\cL}$ for some admissible pair $(0,\cL)$.

In the last case, we can twist everything by $\chi^{-1}$ and this case is reduced to the first case. Thus in any case we can apply the results in proposition \ref{R0cLmod}, \ref{RbcLpro-modular} and conclude that $\kp$ is pro-modular.
\end{proof}

\subsection{Proof of Theorem \ref{thmC}}
\begin{para}
This part is almost the same as \ref{PotthmB}. Note that as in \ref{conF2}, after base change, we may always assume 
\begin{itemize}
\item $\bar\chi$ is quadratic and $\bar\chi|_{G_{F_v}}=\mathbf{1}$ for any $v\in S$, 
\item $\chi$ is unramified everywhere outside of places above $p$.
\end{itemize} 
One key step in  \ref{PotthmB} is to find a totally real field extension $F_1/F$ where we can apply corollary \ref{corC}. This will follow from lemma \ref{bcbsel} below. Once we have lemma \ref{bcbsel}, the rest of the proof of Theorem \ref{thmC} is identical to \ref{PotthmB}. We omit the details here.
\end{para}

\begin{lem} \label{bcbsel}
There exists a totally real field extension $F_1/F$ such that 
\begin{itemize}
\item $F_1$ is abelian over $\Q$.
\item $p$ is unramified in $F_1$.
\item Let $S_1$ (resp. $\Sigma_{p,1}$) be the set of places of $F_1$ above $S$ (resp. $p$) and $v_0$ be a place of $F_1$ above $p$. Then 
\[[F_1:\Q]-6|S_1|+6|\Sigma_{p,1}|-3>\delta_{F_1}+\dim_\F H^1_{v_0}(F_1),\] 
where $H^1_{v_0}(F_1)$ is the kernel of $H^1(G_{F_1,S_1},\F(\bar\chi^{-1}))\to H^1(G_{F_{1,v_0}},\F(\bar\chi^{-1}))$. 
\end{itemize}
\end{lem}

\begin{proof}
As in \ref{Washington}, choose a prime $l_0$ which is larger than the norm of any place in $S$. Consider the cyclotomic $\Z_{l_0}$-extension of $F$. Let $L/F$ be a finite extension in this tower. We claim that we can take $F_1=L$ as long as $[L:F]$ is sufficiently large. 

To see this, first note that $|S_L|$ has a uniform bound $C_1$ independent of $L$. Here $S_L$ denotes the set of places of $L$ above $S$. Next, we follow  the proof of lemma \ref{ubdimh1} and give an upper bound of $\dim_\F H^1_{v_L}(L)$, where $v_L$ is a place of $L$ above $p$ and $H^1_{v_L}(L)$ is defined similarly as before. Let $K$ be the kernel of 
\[H^1_{v_L}(L)\stackrel{\mathrm{res}}{\longrightarrow}\bigoplus_{v\in S_L\setminus \{v_L\}}H^1(G_{L_{v}},\F(\bar{\chi}^{-1}))/H^1(G_{k(v)},\F(\bar{\chi}^{-1})).\]
Then $K$ can be embedded into $H^1(G_{L(\bar\chi)},\F)=\Hom (G_{L(\bar\chi)},\F)$ and the image lies inside the subspace of characters unramified everywhere. Hence $\dim_\F K\leq \dim_{\F_p} \mathrm{Cl}(L(\bar{\chi}))[p]$. By Washington's result \cite{Wa78}, $\dim_{\F_p} \mathrm{Cl}(L(\bar{\chi}))[p]$ has a uniform upper bound $C_2$. On the other hand, since $\bar\chi|_{G_{L_v}}$ is trivial for any $v\in S_L$, 
\[H^1(G_{L_{v}},\F(\bar{\chi}^{-1}))/H^1(G_{k(v)},\F(\bar{\chi}^{-1}))=\Hom(O_{L_v}^\times,\F),\]
whose dimension is at most $[L_v:\Q_p]+1$ if $v|p$ and is at most one otherwise. Putting all these together, we have
\[\dim_\F H^1_{v_L}(L)\leq  \dim_{\F_p} \mathrm{Cl}(L(\bar{\chi}))[p]+ [L:\Q]-[L_{v_L}:\Q_p]+|S_L|.\]
Note that $[L_{v_L}:\Q_p]\geq \frac{1}{|S_L|}[L:\Q]\geq \frac{1}{C_1}[L:\Q]$. Hence 
\[\dim_\F H^1_{v_L}(L)\leq C_2+[L:\Q]-\frac{1}{C_1}[L:\Q]+C_1.\]
Since $L$ is abelian over $\Q$ and Leopoldt's conjecture is known in this case, hence $\delta_L=0$. From the inequality above, it is clear that when $[L:\Q]$ is large enough, $F_1=L$ satisfies the properties in the lemma.
\end{proof}

\section{The main Theorem} \label{Tmt}
\subsection{The main Theorem}
\begin{thm} \label{mainthm}
Let $p>2$ be an odd prime number. Let $F$ be a totally real abelian extension of $\Q$ in which $p$ completely splits. Suppose 
\[\rho:\Gal(\overbar{F}/F)\to \GL_2(\cO)\] 
is a continuous irreducible representation with the following properties
\begin{itemize}
\item $\rho$ ramifies at only finitely many places.
\item Let $\bar{\rho}$ be the reduction of $\rho$ modulo $\varpi$. We assume its semi-simplification has the form $\bar{\chi}_1\oplus\bar{\chi}_2$ and ${\bar{\chi}_1}/\bar{\chi}_2$ can be extended to a character of $G_\Q$.
\item $\rho|_{G_{F_v}}$ is irreducible and de Rham of distinct Hodge-Tate weights for any $v|p$ and any embedding $F_v\hookrightarrow \overbar{\Q_p}$. Moreover, if $p=3$, we assume
\[({\bar{\chi}_1}/\bar{\chi}_2)|_{G_{F_v}}\neq \omega^{\pm 1}.\]
\item $({\bar{\chi}_1}/{\bar{\chi}_2})(c)=-1$ for any complex conjugation $c\in \Gal(\overbar{F}/F)$.
\end{itemize}
Then $\rho$ arises from a twist of a Hilbert modular form, i.e. a regular algebraic cuspidal automorphic representation of $\GL_2(\A_F)$.
\end{thm}

\begin{para}
The rest of this section is devoted to the proof this Theorem. There will be three different cases depending on the shape of $({\bar{\chi}_1}/\bar{\chi}_2)|_{G_{F_v}},v|p$ and the proof of each case will be given in the following 3 subsections. We first do some reduction work.

We may always twist $\rho$ with a finite order character and assume $\bar{\chi}_1$ is trivial. Let $S$ be a finite set of primes containing all the ramified finite places of $\rho$ and let $\Sigma_p$ be the set of primes above $p$. By soluble base change and Theorem 5 of chapter 10 of \cite{AT68}, we may always assume 
\[|S\setminus\Sigma_p|+2< [F:\Q].\]
So in particular $[F:\Q]>2$.

As before, we will fix a complex conjugation $\sigma^*\in G_{F,S}$. Let $\chi$ be the determinant of $\rho$ and $\bar{\chi}$ be the reduction of $\chi$ modulo $\varpi$. Consider the functor from $\cOf$ to the category of sets sending $R$ to the set of $2$-dimensional pseudo-representations of $G_{F,S}$ which lift $1+\bar{\chi}$ with determinant $\chi$. This is pro-represented by a complete local noetherian ring $R^{\ps}$ and $\tr\rho$ gives rise to a prime ideal $\kp\in\Spec R^{\ps}$.
\end{para}

\begin{lem} \label{dimrpskp}
$\dim R^{\ps}_\kp\geq 2[F:\Q]$.
\end{lem}
\begin{proof}
By corollary \ref{dr0}, $\widehat{(R^{\ps})_\kp}$ parametrizes all deformations with determinant $\chi$ of $\rho_E=\rho\otimes E:G_{F,S}\to\GL_2(E)$. See the precise deformation problem in \ref{dr0} except here we put an extra condition on the determinant. It follows from this description that  
\[\widehat{(R^{\ps})_\kp}\cong E[[x_1,\cdots,x_{h^1}]]/(f_1,\cdots,f_{h^2}),\]
where $h^i=\dim_E H^i(G_{F,S},\ad^0 \rho_E),i=1,2$ and $\ad^0 \rho_E$ as usual denotes the subspace of $\End_E(\rho_E)$ with trace $0$. Then by the global Euler characteristic formula (for example see lemma 9.7 of \cite{Kis03}),
\[h^1-h^2=(\dim_E \ad^0 \rho_E)[F:\Q]-\sum_{v|\infty} \dim_E (\ad^0 \rho_E)^{G_{F_v}} +h^0=2[F:\Q] \]
since we assume $\rho_E$ is irreducible and $\dim_E (\ad^0 \rho_E)^{G_{F_v}} =1$ for any $v|\infty$. This proves the dimension inequality in the lemma.
\end{proof}

\begin{para} \label{kPC}
Choose an irreducible component $C$ of $\Spec R^{\ps}$ that contains $\kp$ and has dimension at least $2[F:\Q]+1$. We denote its generic point by $\kP$. Let $T^{univ}:G_{F,S}\to R^{\ps}$ be the universal pseudo-character. By lemma \ref{simcor},
 for any $v\in S\setminus \Sigma_p$, we can write $T^{univ}|_{I_{F_v}}\equiv\theta_{v,1}+\theta_{v,2}\mod\kP$ for some characters $\theta_{v,1},\theta_{v,2}$ of finite orders after possible enlarging $\cO$. 

Now consider the following map given by the universal property:
\[R^{\ps}_p\to R^{\ps}\]
where $R^{\ps}_p$ denotes the completed tensor product $\widehat\bigotimes_{v|p}R^{\ps}_v$ over $\cO$ and $R^{\ps}_v$ denotes the universal deformation ring which parametrizes all $2$-dimensional pseudo-representations of $G_{F_v}$ that lift $(1+\bar{\chi})|_{G_{F_v}}$ with determinant $\chi|_{G_{F_v}}$. Let $R^{\ps,\ord}_v$ be the quotient of $R^{\ps}_v$ that parametrizes all reducible liftings (i.e. liftings that are sum of two characters) and $R^{\ps,\ord}_p$ be the completed tensor product of $R^{\ps,\ord}_v,v|p$. We denote
\[R^{\ps,\ord}=R^{\ps}\otimes_{R^{\ps}_p}R^{\ps,\ord}_p;~C^{\ord}=\Spec R^{\ps,\ord}\cap C.\]

According to the shape of $\bar{\chi}|_{G_{F_v}},v|p$, we can separate into three cases:
\begin{enumerate}
\item (Generic case) $\bar{\chi}|_{G_{F_v}}\neq \mathbf{1},\omega^{\pm1}$ for any $v|p$.
\item $ {\chi}|_{G_{F_v}}= \mathbf{1}$ for any $v|p$.
\item $\bar\chi|_{G_{F_v}}=\omega^{\pm1}$ for any $v|p$.
\end{enumerate}
\end{para}

\subsection{Case 1: Generic case} \label{case1:generic}
In this subsection, we prove Theorem \ref{mainthm} when $\bar{\chi}|_{G_{F_v}}\neq \mathbf{1},\omega^{\pm1}$ for any $v|p$. We will keep this assumption throughout this subsection. We note that in this case, a direct definition of $R^{\ps,\ord}$ was given in \ref{sotmr1} with the same notation.

\begin{lem} \label{Corddim}
$\dim C^{\ord}\geq [F:\Q]+1$.
\end{lem} 
\begin{proof}
For any $v|p$, since we assume $\bar{\chi}|_{G_{F_v}}\neq \mathbf{1},\omega^{\pm1}$, it follows from corollary B.20 of \cite{Pas13} that the kernel of $R^{\ps}_v\to R^{\ps,\ord}_v$ is a principal ideal. Hence 
\[\dim C^{\ord}\geq \dim C-|\Sigma_p|\geq [F:\Q]+1.\]
\end{proof}

\begin{rem} 
This is the essential reason that we rule out the case $\bar{\chi}|_{G_{F_v}}=\omega^{\pm 1}$ as in this case, the kernel is generated by two elements.
\end{rem}

On the other hand, there is a map $\Lambda_F\to R^{\ps,\ord}$ as in \ref{sotmr1}, which comes from the universal deformation of $\mathbf{1}|_{I_{F_v}}, v|p$. Theorem \ref{thmB} tells us that this is a \textit{finite} map. Combining this with the previous lemma, we see that
\begin{cor} \label{cord1}
There exists an irreducible component $C^{\ord}_1$ of $C^{\ord}$ such that 
\begin{enumerate}
\item $\dim C_1^{\ord}= [F:\Q]+1$. 
\item $C^{\ord}_1\to \Spec \Lambda_F$ is a finite surjective map.
\item Let $C^{\ord,\mathrm{aut}}_1$ be the set of \textit{regular de Rham} primes in $C^{\ord}_1$. Here a prime $\kq$ is called \textit{regular de Rham} if
\begin{itemize}
\item $p\notin \kq$ and $R^{\ps,\ord}/\kq$ is one-dimensional.
\item Denote the associated semi-simple representation $G_{F,S}\to\GL_2(k(\kq))$ by $\rho(\kq)$ (see \ref{tar}). We require $\rho(\kq)$ to be irreducible.
\item For any $v|p$, $\rho(\kq)|_{G_{F_v}}\cong\begin{pmatrix}\psi_{v,1} & *\\ 0 & \psi_{v,2}\end{pmatrix}$ such that $\psi_{v,1}$ is Hodge-Tate and has strictly less Hodge-Tate number than $\psi_{v,2}$ for any embedding $F_v\hookrightarrow \overbar{\Q_p}$.
\end{itemize}
Then $C^{\ord,\mathrm{aut}}_1$ is dense in $C^{\ord}_1$.
\end{enumerate}
\end{cor}

\begin{proof}
The existence of an irreducible component that satisfies the first condition is clear. Fix one and denote it by $C^{\ord}_1$. The second claim follows from the first claim and the fact that $\Lambda_F$ is a domain. As for the last one, note that the subset of $\kq\in\Spec R^{\ps}$ such that $\rho(\kq)$ is reducible has dimension at most $2+\delta_F=2<1+[F:\Q]$ (see \ref{redlrps}). Hence we may ignore the second condition in the definition of $C^{\ord,\mathrm{aut}}_1$.

Let $\psi^{univ}_{v,1},\psi^{univ}_{v,2}:G_{F_v}\to (R^{\ps,\ord})^\times$ be the liftings of $\mathbf{1},\bar{\chi}|_{G_{F_v}}$ respectively. It follows from lemma \ref{ordpseudo} that for any $v|p$, there exists $n_v\in\{1,2\}$ such that for any $\kq\in C^{\ord}_1$,
\[\rho(\kq)|_{G_{F_v}}\cong \begin{pmatrix} \psi^{univ}_{v,n_v}\modd\kq & *\\ 0 & \psi^{univ}_{v,3-n_v}\modd \kq\end{pmatrix}.\]
Using this and the surjectivity of the map $C_1^{\ord}\to \Spec \Lambda_F$, we see that the image of $C^{\ord,\mathrm{aut}}_1$ in $\Spec \Lambda_F$ is dense. Thus $C^{\ord,\mathrm{aut}}_1$ is also dense in $C^{\ord}_1$ as the map to $\Spec \Lambda_F$ is finite.
\end{proof}

\begin{para} \label{potnice}
Now we choose a prime $\kq\in C^{\ord}_1$ that is ``potentially nice'' in the following sense:
\begin{itemize}
\item $p\in\kq$ and $R^{\ps}/\kq$ is one-dimensional. In particular, the image of $\kq$ in $\Spec\Lambda_F$ is not the maximal ideal so that we can apply the third part of lemma \ref{nirred}.
\item $\rho(\kq)$ is irreducible.
\item For any $v\in S\setminus\Sigma_p$, $\rho(\kq)|_{G_{F_v}}$ has finite image.
\end{itemize}
To see the existence of such a prime, denote the generic point of $C^{\ord}_1$ by $\kQ\in\Spec R^{\ps}$. Let  $I_S\subseteq R^{\ps}$ be the ideal generated by $T^{univ}(\Frob_v)-1-\chi(\Frob_v),v\in S\setminus\Sigma_p$. Here $\Frob_v$ is not well defined and is just a choice of lift (geometric) Frobenius. Consider $R^{\ps}/(\varpi, \kQ,I_S)$. By our assumptions, its dimension is larger than $1=\delta_F+1$. Hence we can choose a prime $\kq$ of $R^{\ps}/(\varpi, \kQ,I_S)$ such that $\rho(\kq)$ is irreducible. We claim that $\kq$ is potentially nice, i.e. $\rho(\kq)|_{G_{F_v}}$ has finite image for any $v\in S\setminus\Sigma_p$.

By lemma \ref{simlem}, $\rho(\kq)|_{G_{F_v}}$ is either reducible or induced from a character $\theta$ of $G_{F_{v^2}}$. In the second case, it suffices to prove $\theta$ has finite orders. This follows from $\theta(\Frob_v^2)=-\bar{\chi}(\Frob_v)\in\F^\times$. In the first case, note that $\tr\rho(\kq)(\Frob_v)=1+\bar{\chi}(\Frob_v)$. It is easy to see that the semi-simplification of $\rho(\kq)|_{G_{F_v}}$ is a sum of two characters of finite orders. Hence $\rho(\kq)|_H$ is unipotent for some subgroup $H$ of finite index in $G_{F_v}$. Clearly $\rho(\kq)(H)$ is finite. Therefore in either case, our claim is clear.

We fix such a choice of potentially nice prime $\kq$.
\end{para}

\begin{para} \label{laststep}
Finally choose a finite totally real soluble extension $F_1$ of $F$ in which $p$ splits completely such that $[F_1:\Q]$ is even and for any place $w$ of $F_1$ above some place $v\in S\setminus \Sigma_p$,
\begin{enumerate}
\item $\rho(\kq)|_{G_{F_{1,w}}}$ is trivial,
\item $N(w)\equiv 1\mod p$,
\item $T^{univ}|_{I_{F_{1,w}}}\equiv 2\mod\kP$. Recall that $\kP$ is a minimal prime of $R^{\ps}$ introduced in \ref{kPC} and $T^{univ}|_{I_{F_{1,w}}}$ is a sum of two characters of finite orders modulo $\kP$.
\end{enumerate}
This is possible by lemma 2.2 of \cite{Ta03}. Note that we can define a Hecke algebra $\T:=\T_{\psi,\xi}(U^p)$ (of a quaternion algebra over $F_1$) as in \ref{autlev} with $S$ in \ref{defrps1} taken as the set of places of $F_1$ above $S$ here and $\xi=\mathbf{1}$, the trivial character. The assumption in \ref{autlev} holds for $\T$ as we can take a prime $\kq_1\in C^{\ord,\mathrm{aut}}_1$ and consider $\rho(\kq_1)|_{G_{F_1}}$. By Theorem \ref{thmB} and soluble base change, this comes from a regular algebraic cuspidal automorphic representation of $\GL_2(\A_{F_1})$ and gives rise to the existence of our desired maximal ideal of Hecke algebra.

Let $S_1$ be the set of places of $F_1$ above $S$. Consider $R^{\ps,1}$ defined in \ref{defrps1}. More precisely, it pro-represents the functor from $\cOf$ to the category of sets sending $R$ to the set of two-dimensional pseudo-representations $T$ of $G_{F_1,S_1}$ over $R$ such that $T$ is a lifting of $1+\bar\chi|_{G_{F_1,S_1}}$ with determinant $\chi|_{G_{F_1,S_1}}$ and
\[T|_{I_{F_{1,v}}}=2 \]
for any $v\in S_1,v\nmid p$. 

It follows from our construction that there is a map $R^{\ps,1}\to R^{\ps}/\kP$. Let $\kp',\kq'\in\Spec R^{\ps,1}$ be the pull-backs of $\kp,\kq$. We claim that $\kq'$ is a nice prime in the sense of \ref{nice}. All the conditions are clear except that $\rho(\kq)|_{G_{F_1}}$ is irreducible and $\kq'$ is pro-modular. 

The irreducibility can be proved in the same way as in the second paragraph of \ref{1asspropB}. To see that $\kq'$ is pro-modular, we will actually prove that the image of $C^{\ord}_1$ in $\Spec R^{\ps,1}$ is pro-modular. First notice that by Theorem \ref{thmB} for any $\kq_1\in C^{\ord,\mathrm{aut}}_1$, $\rho(\kq_1)$ comes from a regular algebraic cuspidal automorphic representation of $\GL_2(\A_{F_1})$. Hence by soluble base change and the same irreducibility argument, the image of any prime of $C^{\ord,\mathrm{aut}}_1$ in $\Spec R^{\ps,1}$ is modular. By the density result in corollary \ref{cord1}, we see that the image of $C^{\ord}_1$ is in fact pro-modular. In particular, $\kq'$ is pro-modular and hence nice.

Note that we can find an irreducible component of $R^{\ps,1}$ that contains the image of $C$, the irreducible component of $\Spec R^{\ps}$ defined by $\kP$. Therefore we can apply corollary \ref{corA} with $\kp=\kp'$ and $\kq=\kq'$ and  conclude that $\rho(\kp')=\rho|_{G_{F_1}}$ comes from a regular algebraic cuspidal automorphic representation of $\GL_2(\A_{F_1})$.  Theorem \ref{mainthm} now follows from soluble base change.
\end{para}

\subsection{Case 2: \texorpdfstring{$\bar\chi|_{G_{F_v}}=\mathbf{1}$}{Lg}}
In this subsection, we deal with the case where $\bar\chi|_{G_{F_v}}=\mathbf{1}$ for any $v|p$. We will prove lemma \ref{Corddim} in this case and show that there are enough automorphic points on $C^{\ord}$ as in corollary \ref{cord1}. To do this, we need work on some covering of $C^{\ord}$.

Let $R^{\ps,\ord}_1$ be the universal deformation ring introduced in \ref{sotmr2}. It is clear that there is a natural map $R^{\ps,\ord}\to R^{\ps,\ord}_1$. We denote by $C^{\ord,1}$ the pull-back of $C^{\ord}$ to $\Spec R^{\ps,\ord}_1$.

\begin{lem}
$\dim C^{\ord,1}\geq[F:\Q]+1$.
\end{lem}

\begin{proof}
The proof is similar to the proof of lemma \ref{Corddim}. For any $v|p$, let $R^{\ps,\ord}_{v,1}$ be the universal deformation ring which parametrizes all characters $\psi_{v,1}$ of $G_{F_v}$ lifting the trivial character $\mathbf{1}$. Then it is clear that this ring is (non-canonically) isomorphic to $\cO[[x_1,x_2]]$. Also assigning each such a character $\psi_{v,1}$ to $\psi_{v,1}+\chi\psi_{v,1}^{-1}$ induces a natural finite map $\phi:R^{\ps}_v\to R^{\ps,\ord}_{v,1}$. Consider the map:
\[R^{\ps}_v\hat{\otimes}_{\cO} R^{\ps,\ord}_{v,1}\to R^{\ps,\ord}_{v,1}:~a\otimes b\mapsto \phi(a)b.\]
We claim that the kernel of this surjective map is generated by \textbf{three} elements. This is clear as in this case $R^{\ps}_v$ is isomorphic to $\cO[[t_1,t_2,t_3]]$ by Corollary 9.13. of \cite{Pas13}. Let $R^{\ps,\ord}_{p,1}$ be the completed tensor products of all $R^{\ps,\ord}_{v,1},v|p$ over $\cO$. Then $R^{\ps,\ord}_1=R^{\ps}\otimes_{R^{\ps}_p}R^{\ps,\ord}_{p,1}$. Recall that $C=\Spec R^{\ps}/\kP$ has dimension at least $1+2[F:\Q]$. Hence $C^{\ord,1}$ is the underlying space of the spectrum of $(R^{\ps}/\kP)\otimes_{R^{\ps}_p}R^{\ps,\ord}_{p,1}$, which we may rewrite as 
\[(R^{\ps}/\kP\hat{\otimes}_{\cO} R^{\ps,\ord}_{p,1})\otimes_{(R^{\ps}_p\hat{\otimes}_{\cO} R^{\ps,\ord}_{p,1})}R^{\ps,\ord}_{p,1}\]
From what we just discussed, $C^{\ord,1}$ has dimension at least 
\[1+2[F:\Q]+2[F:\Q]-3[F:\Q]=1+[F:\Q].\] 
This is exactly what we want.
\end{proof}

There is a natural map $C^{\ord,1}\to \Spec R^{\ps,\ord}_1\to \Spec \Lambda_F$ as in \ref{sotmr2}. The following corollary can be proved in exactly the same way as corollary \ref{cord1} using Theorem \ref{thmC}.

\begin{cor} 
There exists an irreducible component $C^{\ord,1}_1$ of $C^{\ord,1}$ such that 
\begin{enumerate}
\item $\dim C_1^{\ord}= [F:\Q]+1$. 
\item $C^{\ord,1}\to \Spec \Lambda_F$ is a finite surjective map.
\item Let $C^{\ord,1,\mathrm{aut}}_1$ be the set of regular de Rham primes in $C^{\ord,1}_1$ defined in corollary \ref{cord1}. Then $C^{\ord,1,\mathrm{aut}}_1$ is dense in $C^{\ord,1}_1$.
\end{enumerate}
\end{cor}

The rest of the proof is the same as in the generic case: we may find a potentially nice prime $\kq'\in C^{\ord,1}_1$. Let $\kq$ be its image in $C^{\ord}$. Then after possibly a soluble base change, $\kq$ becomes a nice prime and we can apply corollary \ref{corA} and conclude that $\rho$ comes from a regular algebraic cuspidal automorphic representation of $\GL_2(\A_F)$. For more details, see the arguments in \ref{potnice}, \ref{laststep}. 

\subsection{Case 3: \texorpdfstring{$\bar\chi|_{G_{F_v}}=\omega^{\pm1}$}{Lg}} \label{case3}
In this subsection, we treat the last case where $\bar\chi|_{G_{F_v}}=\omega^{\pm1}$ for any $v|p$. We will keep the assumption that $p\geq 5$. Then after possibly some twist, we may assume $\bar\chi|_{G_{F_v}}=\omega\neq\omega^{-1}$. In this case, lemma \ref{Corddim} might not hold in general. Hence we need some careful study of $C^{\ord}$. Roughly speaking, our strategy is to `connect' $\kp$ with some pro-modular primes by nice primes. More precisely, we make the following definition.

\begin{defn} 
We say an irreducible component $C'$ of $R^{\ps}$ is \textit{good} if there exist irreducible components $C_1,\cdots,C_t=C'$ for some integer $t>0$ and potentially nice primes (defined in \ref{potnice}) $\kq_1,\cdots,\kq_t$ such that
\begin{itemize}
\item $\kq_i\in C_i\cap C_{i-1},i=2,\cdots,t$, and $\kq_1\in C_1$.
\item $\kq_1$ is in the closure of the set of regular de Rham primes of $\Spec R^{\ps,\ord}$ (introduced in corollary \ref{cord1}). In particular, this set of regular de Rham primes in $\Spec R^{\ps,\ord}$ is non-empty.
\end{itemize}
By abuse of notation, we will also say the minimal prime corresponding to $C'$ is \textit{good}.
\end{defn}

\begin{lem} 
If $C$ is good, then Theorem \ref{mainthm} holds.
\end{lem} 

\begin{proof}
The argument is very similar to \ref{laststep}. We choose a soluble totally real field extension $F_1/F$ in which $p$ completely splits such that $[F_1:\Q]$ is even and for any place $w$ of $F_1$ above some place in $S\setminus\Sigma_p$,
\begin{itemize}
\item $\rho(\kq_i)|_{G_{F_{1,w}}}=\mathbf{1},i=1\cdots,t$.
\item (2)(3) as in \ref{laststep}.
\end{itemize}
As in \ref{laststep}, we can define a Hecke algebra $\T_{\psi,\xi}(U^p)$ (of a quaternion algebra over $F_1$) which satisfies the assumption in \ref{autlev}. Also over $F_1$, we have the deformation ring $R^{\ps,1}$. Let $\kq'_i$ be the pull-back of $\kq_i$. It suffices to prove that $\kq'_t$ is nice by corollary \ref{corA}. By the same argument as in \ref{laststep}, we see that $\kq'_1$ is nice. Hence the image of $C_1$ in $\Spec R^{\ps,1}$ is pro-modular by Theorem \ref{thmA}. In particular, $\kq'_{2}$ is nice. Repeating this argument, we see that $\kq'_i$ is nice for any $i$.
\end{proof} 

Therefore it rests to prove

\begin{prop}\label{generalprop}
$C$ is good. 
\end{prop}

To find $C_i,\kq_i$ as in the definition, we need a careful study of $R^{\ps}_v,v|p$ in this case. 

\begin{lem} \label{nongenRps}
Assume $p>3$. Let $v$ be a place of $F$ above $p$.
\begin{enumerate}
\item $R^{\ps}_v\cong \cO[[x_0,x_1,y_0,y_1]]/(x_0y_1-x_1y_0)$.
\item Under this isomorphism, $R^{\ps,\ord}_v=R^{\ps}_v/(x_0,x_1)$.
\item The prime ideal $(x_0,x_1,y_0,y_1)$ corresponds to the pseudo-representation $\psi_1+\psi_2$ where $\psi_i:G_{F_v}\to\cO^\times$ are characters such that $\psi_1/\psi_2=\varepsilon$.
\end{enumerate}
\end{lem}

\begin{proof}
These claims follow easily from the results in Appendix B of \cite{Pas13}. Let $\bar\rho_p$ be a non-split extension of $\omega$ by $\mathbf{1}$ as $\F[G_{F_v}]$-modules. Consider the universal deformation ring $R^{\chi}_{\bar\rho_p}$ of $\bar\rho_p$ with determinant $\chi|_{G_{F_v}}$. Then corollary B.16 of \cite{Pas13} says that the natural map $R^{\ps}_v\to R^{\chi}_{\bar\rho_p}$ given by evaluating the traces is an isomorphism. The first part of the lemma now follows from lemma corollary B.5 of \cite{Pas13} with $x_0=c_0,x_1=c_1,y_0=-d_0,y_1=d_1+p$. The second part can be proved using the explicit description above proposition B.2 ibid. For the last part, it is clear that $(x_0,x_1,y_0,y_1)$ is the only non-smooth point of $\Spec R^{\chi}_{\bar\rho_p}[\frac{1}{p}]$. Then by standard obstruction theory, the associated pseudo-representation must have the form as in the lemma.
\end{proof}

\begin{cor}
$\dim C^{\ord}\ge 1$.
\end{cor}

\begin{proof}
It follows from the previous lemma that the kernel of $R^{\ps}\to R^{\ps,\ord}$ is generated by $2[F:\Q]$ elements. Since $\dim C\geq 1+2[F:\Q]$, the dimension of $C^{\ord}$ is at least one.
\end{proof}

\begin{para}
Choose a prime $\kq\in C^{\ord}$ such that $\dim R^{\ps}/\kq=1$. Enlarge $\cO$ if necessary so that we may assume the normalization of $R^{\ps}/\kq$ is either $\cO$ or isomorphic to $\F[[T]]$. There are three possibilities for the associated semi-simple representation $\rho(\kq)$:
\begin{enumerate}
\item $\rho(\kq)$ is irreducible.
\item $\rho(\kq)\cong\psi_1\oplus \psi_2$ is reducible  and $\psi_1/\psi_2$ is not of the form $\varepsilon^{\pm 1}\theta$, where $\theta$ is a finite order character of $G_{F}$. We call this case generic reducible.
\item $\rho(\kq)\cong\psi_1\oplus \psi_2$ and $\psi_1/\psi_2=\varepsilon\theta$, for some finite order character $\theta$ of $G_{F}$. We call this case non-generic reducible.
\end{enumerate}
\end{para}

\begin{para} \label{case(1)genprop}
\noindent \underline{\textbf{Case (1): $\rho(\kq)$ is irreducible}.} 
\end{para}

\begin{lem} \label{condimRpskq}
The connectedness dimension $c((R^{\ps})_\kq)$ of $(R^{\ps})_\kq$ is at least $2[F:\Q]-1$.
\end{lem}

\begin{proof}
For the definition of connectedness dimension, see \ref{connectednessdim}. Since $c((R^{\ps})_\kq)\geq c(\widehat{(R^{\ps})_\kq})$, it suffices to prove $c(\widehat{(R^{\ps})_\kq})\geq 2[F:\Q]-1$.


Choose a lattice $\rho(\kq)^o$ of $\rho(\kq)$ so that the residual representation $\bar\rho_b$ of $\rho(\kq)^o$ is non-split. Let $R_b$ be the universal deformation ring of $\bar\rho_b:G_{F,S}\to\GL_2(\F)$ with determinant $\chi$. Hence $\rho(\kq)^o$ gives rise to a prime ideal $\kq_b$ of $R_b$. Then it follows from the first part of corollary \ref{psccomp} that $\widehat{(R^{\ps})_\kq}\cong \widehat{(R_b)_{\kq_b}}$. On the other hand, it is standard that $R_b$ can be written as the form $\cO[[x_1,\cdots,x_{h^1}]]/(f_1,\cdots,f_{h^2})$ with $h^i=\dim_\F H^i(G_{F,S},\ad^0\bar\rho_b)$. Since $\bar\chi$ is totally odd, by the global Euler characteristic formula, we have $h^1-h^2=h^0+2[F:\Q]=2[F:\Q]$ as $\bar\rho_b$ is non-split. Hence $\widehat{(R_b)_{\kq_b}}\cong S_1/(f_1,\cdots,f_{h^2})$ where $S_1$ is the completion of $\cO[[x_1,\cdots,x_{h^1}]]$ at the pull-back of $\kq_b$. Using proposition \ref{cndimformula}, we get the desired lower bound on $c(\widehat{(R_b)_{\kq_b}})$.
\end{proof}

\begin{lem} \label{exnicecom}
There exists an irreducible component of $R^{\ps}$ that contains $\kq$ and a potentially nice prime in the closure of the set of regular de Rham primes of $\Spec R^{\ps,\ord}$.
\end{lem}

\begin{proof}
Let $\psi^{univ}_{v,1},\psi^{univ}_{v,2}:G_{F_v}\to (R^{\ps,\ord})^\times$ be the liftings of $\mathbf{1},\bar{\chi}|_{G_{F_v}}$ respectively. Let $\Sigma^o$ be the set of places $v|p$ such that
\[\rho(\kq)|_{G_{F_v}}\cong \begin{pmatrix} \psi^{univ}_{v,1}\modd\kq & *\\ 0 & \psi^{univ}_{v,2}\modd \kq\end{pmatrix}\]
and $\psi^{univ}_{v,1}$ is a lifting of $\mathbf{1}$.
We can consider the quotient $R^{\ps,\ord}_{\Sigma^o}$ of $R^{\ps,\ord}$ which parametrizes pseudo-representations that are $\psi_{v,1}^{univ}$-ordinary (resp. $\psi_{v,2}^{univ}$-ordinary) if $v\in\Sigma^o$ (resp. $v\in\Sigma_p\setminus \Sigma^o$) in the sense of \ref{psiord}.  Then $\kq$ can be viewed as a prime of $R^{\ps,\ord}_{\Sigma^o}$. 

On the other hand, as in the proof of the previous lemma, we may choose a lattice $\rho(\kq)^o$ of $\rho(\kq)$ such that its residual representation $\bar\rho_b$ is a non-split extension of $\mathbf{1}$ by $\bar\chi$. Consider the universal ordinary deformation $R^{\Delta}_b$ which parametrizes all deformations $\rho_b$ of $\bar\rho_b:G_{F,S}\to\GL_2(\F)$ with determinant $\chi$ such that for any $v|p$, $\rho_b|_{G_{F_v}}\cong \begin{pmatrix} \psi_{v,1} & *\\ 0 & *\end{pmatrix}$, where $\psi_{v,1}$ is a lifting of $\mathbf{1}$ (resp. $\bar\chi|_{G_{F_v}}$) if $v\in\Sigma^o$ (resp. $v\in\Sigma_p\setminus \Sigma^o$). Let $\kq_b$ be the prime of $R^{\Delta}_b$ corresponding to $\rho(\kq)^o$. Arguing as in the second part of \ref{rdxbq}, we have an isomorphism:
\[\widehat{(R^{\ps,\ord}_{\Sigma^o})_{\kq}}\cong \widehat{(R^{\Delta}_b)_{\kq_b}}.\]

Since $\bar\rho_b$ is a non-split extension of $\mathbf{1}$ by $\bar\chi$ and $\bar\chi|_{G_{F_v}}\neq\omega^{-1},v|p$, it is easy to deduce that $H^0(G_{F,S},\ad^0(\bar\rho_b(1)))=0$. Hence by lemma \ref{grrcord}, each irreducible component of $R^{\Delta}_b$ has dimension at least $1+[F:\Q]$. In particular,
\[\dim (R^{\ps,\ord}_{\Sigma^o})_{\kq} = \dim (R^{\Delta}_b)_{\kq_b}\geq [F:\Q].\]
Thus there exists an irreducible component $C^{\ord}_b$ of $R^{\ps,\ord}_{\Sigma^o}$ containing $\kq$ with dimension at least $1+[F:\Q]$ \footnote{In fact by the finiteness result over $\Lambda_F$ (Theorem \ref{thmB}), its dimension is exactly $[F:\Q]+1$.}.

Now we can argue as in the proof of  corollary \ref{cord1} and \ref{potnice} to show that the regular de Rham points are dense in $C^{\ord}_b$ and $C^{\ord}_b$ contains a potentially nice prime. For the lemma, we can just take any irreducible component of $R^{\ps}$ containing $C^{\ord}_b$.
\end{proof}

\begin{para}  \label{proofgenpropcase1}
Now we prove proposition \ref{generalprop} in case (1). Let $C_{nice}$ be an irreducible component as in lemma \ref{exnicecom}. Let $Z_1$ be the union of \textit{good} irreducible components  $C'$ of $R^{\ps}$ which contain $\kq$. Note that $C_{nice}\in Z_1$. It suffices to prove that $Z_1$ contains all the minimal primes of $(R^{\ps})_{\kq}$. Suppose not, let $Z_2$ be the union of irreducible components which contain $\kq$ and are not contained in $Z_1$. Then by lemma \ref{condimRpskq}, 
\[\dim Z_1\cap Z_2\geq 2[F:\Q]-1\geq |S|-|\Sigma_p|+3.\]
Arguing as in \ref{potnice}, we can find a potentially nice prime $\kq'$ in $Z_1\cap Z_2$. Then any irreducible component of $Z_2$ containing $\kq'$ would be good. We thus get a contradiction. 
\end{para}

\begin{para} \noindent \underline{\textbf{Case (2): $\rho(\kq)\cong\psi_1\oplus \psi_2$ generic}.}  \label{Case2generic}
Recall that generic here means  $\psi_1/\psi_2$ is not of the form $\varepsilon^{\pm 1}\theta$,  where $\theta$ is a finite order character of $G_{F}$. 

For any $v|p$, let $\kq_v\in \Spec R^{\ps,\ord}_v$ be the pull-back of $\kq$. By abuse of notation, we also view it as a prime ideal of $R^{\ps}_v$. We claim that $\kq_v$ is one-dimensional for any $v$. This is trivial if $p\notin\kq$. Suppose the residue field of $\kq$ has characteristic $p$. Then the image of $G_F$ under $\psi_1$ is infinite. Note that Leopoldt's conjecture is known for $F$ as $F$ is abelian over $\Q$. This implies that the image of the composite map
\[ O_{F_v}^{\times}\to G_{F_v}^{\mathrm{ab}}\to G_{F,S}^{\mathrm{ab}}(p)\]
has finite index inside $G_{F,S}^{\mathrm{ab}}(p)$, where the first map is given by local class field theory. Therefore the image of $\psi_1|_{G_{F_v}}$ is infinite as well. Hence $\kq_v$ is not the maximal ideal.

Moreover, it follows from the same argument that $(\frac{\psi_1}{\psi_2})|_{G_{F_v}}\neq \varepsilon ^{\pm1}$. Otherwise $\psi_1/\psi_2$ would be equal to $\varepsilon^{\pm 1}$ up to a finite order character, which contradicts our generic assumption.
\end{para}

\begin{lem}
For any $v|p$, the kernel of $(R^{\ps}_v)_{\kq_v}\to (R^{\ps,\ord}_v)_{\kq_v}$ is principal.
\end{lem}

\begin{proof}
Under the isomorphism in lemma \ref{nongenRps}, $\kq_v$ contains $x_0,x_1$. By our generic assumption, $\kq_v\neq (x_0,x_1,y_0,y_1)$. Hence without loss of generality, we may assume $y_0\notin \kq_v$. Then the kernel of $(R^{\ps}_v)_{\kq_v}\to (R^{\ps,\ord}_v)_{\kq_v}$ is generated by $x_0$ as $x_1=x_0 y_1 y_0^{-1}$ in $(R^{\ps}_v)_{\kq_v}$.
\end{proof}

\begin{cor}
$\dim C^{\ord}\geq [F:\Q]+1$.
\end{cor}

\begin{proof}
It follows from the previous lemma that the kernel of $(R^{\ps})_{\kq}\to (R^{\ps,\ord})_{\kq}$ is generated by $[F:\Q]$ elements. We can localize $C$ and $C^{\ord}$ at $\kq$ and obtain the desired lower bound.
\end{proof}

Note that this is lemma \ref{Corddim} in this case. Theorem \ref{mainthm} now can be proved in exactly the same way as \ref{case1:generic}.

\begin{para}\noindent \underline{\textbf{Case (3): $\rho(\kq)\cong\psi_1\oplus \psi_2$ non-generic.}}
In this case, $\psi_1/\psi_2=\varepsilon\theta$ for a finite order character $\theta$ of $G_{F,S}$. Note that $\kq$ cannot contain $p$, otherwise $\psi_1,\psi_2$ would be of finite orders. See \ref{Case2generic} for more details. Therefore $k(\kq)$ is a finite extension of $E$. Enlarging $E$ if necessary, we may assume $k(\kq)=E$.

The rough idea here is to obtain a lower bound on the connectedness dimension of $(R^{\ps})_{\kq}$ first and argue as in the case \ref{case(1)genprop}. As in \cite{SW99} (see also \ref{Rordbimf1b}, \ref{Rordbmodab}), the geometry of $\Spec (R^{\ps})_{\kq}$ is studied by comparing with some other Galois deformation rings. More precisely, we need a careful study of the universal deformation rings of some non-split extensions of $\psi_1$ by $\psi_2$. One key step is to bound the dimension of the reducible locus. In this case, this is essentially a result of Lichtenbaum \cite{Lich72}, which implies the reducible deformation of $\rho(\kq)$ has dimension $[F:\Q]$. I would like to thank Professor Richard Taylor for suggesting looking at such deformation rings. First we compute the extension group $\Ext^1_{E[G_{F,S}]}(\psi_1,\psi_2)\cong H^1(G_{F,S},E(\psi_2/\psi_1))$.
\end{para}

\begin{lem} \label{Ext1char0}
The natural restriction map
\[H^1(G_{F,S},E(\psi_2/\psi_1))\to \bigoplus_{v|p} H^1(G_{F_v},E(\psi_2/\psi_1))\]
is an isomorphism. In particular, $\dim_E H^1(G_{F,S},E(\psi_2/\psi_1))=[F:\Q]$.
\end{lem}

\begin{proof}
Recall that $\psi_1/\psi_2=\theta\varepsilon$. It follows from the long exact sequence of Poitou-Tate (see for example 8.6.10 of \cite{NSW08}) that there is a short exact sequence:
\[H^2(G_{F,S},E/\cO(\theta\varepsilon^2))^\vee\otimes E\to H^1(G_{F,S},E(\psi_2/\psi_1))\to\bigoplus_{v\in S\mbox{ or }v|\infty}H^1(G_{F_v},E(\psi_2/\psi_1)).\]
Let $F'=F(\theta)$ and $S'$ be the set of places of $F'$ above $S$. Then $G_{F',S'}$ has finite index in $G_{F,S}$. Using the usual restriction-corestriction sequence and the fact that $E$ is of characteristic zero, we see that the restriction map induces a surjective map:
\[H^2(G_{F',S'},E/\cO(2))^\vee\otimes E\twoheadrightarrow H^2(G_{F,S},E/\cO(\theta\varepsilon^2))^\vee\otimes E.\]
Since $\theta\varepsilon$ is totally odd, $F'$ is a totally real field. It follows from proposition 9.6 of \cite{Lich72} that $H^2(G_{F',S'},E/\cO(2))=0$. Strictly speaking, in \cite{Lich72}, this is only proved when $S'$ consists of all the places above $p$. However, it is not too hard to deduce this from Lichtenbaum's result. Hence $H^2(G_{F,S},E/\cO(\theta\varepsilon^2))^\vee\otimes E$ vanishes and we get an inclusion
\[ H^1(G_{F,S},E(\psi_2/\psi_1))\hookrightarrow\bigoplus_{v\in S\mbox{ or }v|\infty}H^1(G_{F_v},E(\psi_2/\psi_1)).\]

Since $\frac{\psi_2}{\psi_1}|_{G_{F_v}}\neq \mathbf{1}, \varepsilon$, we have $\dim_E H^1(G_{F_v},E(\psi_2/\psi_1))=1$ if $v|p$ and $0$ otherwise. Thus $\dim_E H^1(G_{F,S},E(\psi_2/\psi_1))\leq [F:\Q]$.  On the other hand, the global Euler characteristic formula (lemma 9.7 of \cite{Kis03}) gives a lower bound $\dim_E H^1(G_{F,S},E(\psi_2/\psi_1)) \geq [F:\Q]$. From this, we easily deduce all the claims in the lemma.
\end{proof}

\begin{para}
Any non-zero class $B\in \Ext^1_{E[G_{F,S}]}(\psi_1,\psi_2)$ corresponds to a non-split extension $\rho_B: G_{F,S}\to\GL_2(E)$ of $\psi_1$ by $\psi_2$. Using the basis given by the eigenvectors of $\sigma^*$, our fixed complex conjugation, we may assume $\rho_B$ is of the form
\[\begin{pmatrix} \psi_2 & b_B \\ 0 & \psi_1\end{pmatrix}.\] 
We denote the universal deformation ring of $\rho_B$ with fixed determinant $\chi$ by $R_B$. More precisely, it pro-represents the functor from the category of Artinian local $E$-algebras with residue field $E$ (equipped with $p$-adic topology) to the category of sets, assigning $A$ to the set of continuous deformations over $A$ of $\rho_B$ with determinant $\chi$.

First we bound the dimension of the reducible locus of $\Spec R_B$. More precisely, the subset of $\Spec R_B$ corresponding to reducible deformations is closed in $\Spec R_B$ and we denote its reduced closed subscheme by $\Spec R_B^{\mathrm{red}}$.
\end{para}

\begin{lem} \label{RBreddim}
$\dim \Spec R_B^{\mathrm{red}}\leq[F:\Q]+1$.
\end{lem}

\begin{proof}
Arguing as in proposition \ref{redloc}, we see that 
\[\dim \Spec R_B^{\mathrm{red}}\leq 1+\delta_F+ \dim_E H^1(G_{F,S},E(\psi_2/\psi_1)).\]
Since $F$ is abelian over $\Q$, we have $\delta_F=0$. Our claim follows from the previous lemma.
\end{proof}

\begin{para} \label{Rbkpb}
Let $\rho_B^o:G_{F,S}\to\GL_2(\cO)$ be a lattice of $\rho_B$ such that its mod $\varpi$ reduction $\bar\rho_b:G_{F,S}\to\GL_2(\F)$ is non-split. We may consider the universal deformation ring $R_{b}$ of $\bar\rho_b$ with fixed determinant $\chi$. Then $\rho_B^o$ naturally gives rise to a prime $\kp_B$ of $R_b$. There is a natural map $R_B\to \widehat{(R_b)_{\kp_B}}$ given by the universal property. By proposition 9.5 of \cite{Kis03}, we have
\end{para}

\begin{lem} \label{RBRb}
The natural map $R_B\to \widehat{(R_b)_{\kp_B}}$ is an isomorphism.
\end{lem}

\begin{cor} \label{corRB}
\hspace{2em}
\begin{enumerate}
\item $\dim R_B/\kP\geq 2[F:\Q]$ for any minimal prime $\kP$ of $R_B$.
\item The connectedness dimension $c(R_B)$ is at least $2[F:\Q]-1$.
\item By evaluating the trace of the universal deformation, we get a natural map $R^{\ps}\to R_B$. Suppose $\kQ\in\Spec R_B\setminus \Spec R^{\mathrm{red}}_B$. Let $\kQ^{\ps}=\kQ\cap R^{\ps}$. Then
\[\dim R^{\ps}/\kQ^{\ps}\geq 1+\dim R_B/\kQ.\]
\item The natural map $\Spec R_B\to \Spec R^{\ps}$ maps minimal primes to minimal primes.
\end{enumerate}
\end{cor}

\begin{proof}
Under the isomorphism in the previous lemma, it is enough to prove all the statements with $R_B$ replaced by $\widehat{(R_b)_{\kp_B}}$. Note that we can write $R_b$ as $\cO[[x_1,\cdots,x_g]]/(f_1,\cdots,f_r)$ with $g-r=2[F:\Q]$ (cf. proof of lemma \ref{condimRpskq}). The first two parts of the corollary follow from this and proposition \ref{cndimformula}.

For the third part, let $\kp=\kQ\cap (R_b)_{\kp_B}$. Then by the going-down property of flat maps, we have $\dim (R_b)_{\kp_B}/\kp\geq \dim R_B/\kQ$. Note that $\kQ^{\ps}=\kp\cap R^{\ps}$. Our claim follows from the second part of corollary \ref{psccomp}.

Finally, let $\kP$ be a minimal prime of $R_B$. By lemma \ref{RBreddim} and the first part of the corollary, $\kP\notin \Spec R^{\mathrm{red}}_B$. Using the going-down property of flat maps, we have $\kP\cap (R_b)_{\kp_B}$ is a minimal prime of $R_b$ and corresponds to an irreducible deformation. Our claim again follows from the second part of corollary \ref{psccomp}.
\end{proof}

\begin{defn}
For any $B\in \Ext^1_{E[G_{F,S}]}(\psi_1,\psi_2)$, we define $Z_B$ as the set of irreducible components of $R^{\ps}$ whose generic points lie in the image of $\Spec R_{B}\to \Spec R^{\ps}$. It follows from the last part of the previous corollary that $\bigcup_{C'\in Z_B} C'$ contains the image of $\Spec R_{B}\to \Spec R^{\ps}$.
\end{defn}

\begin{cor} \label{anycomisgood}
For any $B$, any irreducible component in  $Z_B$ is good as long as one component is good.
\end{cor}

\begin{proof}
Let $C_B=\bigcup_{C'\in Z_B} C'$. Arguing as in \ref{proofgenpropcase1}, it is enough to prove that the connectedness dimension of $C_B$ is at least $|S|-|\Sigma_p|+3$. Suppose $C_B=Z_1\cup Z_2$, where $Z_i,i=1,2,$ is a non-empty union of irreducible components in $C_B$. We need to show $\dim Z_1\cap Z_2\geq |S|-|\Sigma_p|+3$. Let $C_i$ be the inverse image of $Z_i$ in $\Spec R_B$. Then $C_1,C_2$ are non-empty unions of irreducible components of $\Spec R_B$ and form a covering of $\Spec R_B$. Hence by the second part of corollary \ref{corRB}, $\dim C_1\cap C_2\geq 2[F:\Q]-1$. Choose a point $\kQ\in C_1\cap C_2$ with $\dim R_B/\kQ\geq 2[F:\Q]-1>[F:\Q]+1$. Then $\kQ\notin \Spec R^{\mathrm{red}}_B$ by lemma \ref{RBreddim}. It follows from the third part of corollary \ref{corRB} that $R^{\ps}/(\kQ\cap R^{\ps})$ has dimension at least $2[F:\Q]>|S|-|\Sigma_p|+3$. Moreover $\kQ\cap R^{\ps}\in Z_1\cap Z_2$. Thus we get the desired lower bound on $\dim Z_1\cap Z_2$.
\end{proof}

\begin{proof}[Proof of proposition \ref{generalprop} in case (3)]
We are going to prove proposition \ref{generalprop} in $2$ steps: after possibly enlarging $E$, we can 
\begin{enumerate}
\item Find a good irreducible component in $Z_{B_0}$ for some $B_0$.
\item Find an extension class $B_1$ such that $C\in Z_{B_1}$ and there exist irreducible components $C_{B_0}\in Z_{B_0}, C_{B_1}\in Z_{B_1}$ containing a common potentially nice prime.
\end{enumerate}
By corollary \ref{anycomisgood}, the first claim implies that any irreducible component in $Z_{B_0}$ is good. Using the second claim and corollary \ref{anycomisgood} again, we see that $C$ is good. It rests to prove above two claims.

\begin{proof}[Proof of Step 1]
Fix a place $v_0$ of $F$ above $p$. Let $B_0\in H^1(G_{F,S},E(\psi_2/\psi_1))$ be a non-zero class such that under the restriction map, its image in $H^1(G_{F_v},E(\psi_2/\psi_1)),v|p$ is zero unless $v=v_0$. Such a class exists by lemma \ref{Ext1char0}. We are going to find a good component in $Z_{B_0}$.

As in \ref{Rbkpb}, we may find a lattice $\rho^o_{B_0}:G_{F,S}\to \GL_2(\cO)$ of $\rho_{B_0}$ with non-split mod $\varpi$ reduction $\bar\rho_{b_0}:G_{F,S}\to\GL_2(\F)$. Let $\bar\psi_i,i=1,2$ be the mod $\varpi$ reduction of $\psi_i$. Then $\bar\rho_{b_0}$ is an extension of $\bar\psi_1$ by $\bar\psi_2$. It follows from the choice of $B_0$ that $\bar\rho_{b_0}|_{G_{F_{v}}}$ is completely reducible for $v\in\Sigma_p\setminus \{v_0\}$. Let $R_{b_0}$ be the universal deformation ring of $\bar\rho_{b_0}$ with determinant of $\chi$. Consider the quotient $R_{\{v_0\}}^{\Delta}$ which parametrizes deformations $\rho_b$ such that for any $v|p$, $\rho_b|_{G_{F_v}}\cong \begin{pmatrix} \psi_{v,1} & *\\ 0 & *\end{pmatrix}$, where $\psi_{v,1}$ is a lifting of $\bar\psi_1$ (resp. $\bar\psi_2$) if $v\neq v_0$ (resp. $v=v_0$). By lemma \ref{grrcord}, each irreducible component of $R^{\Delta}_{\{v_0\}}$ has dimension at least $[F:\Q]$.

By the universal property, $\rho^o_{B_0}$ gives rise to a prime $\kp_0$ of $R^{\Delta}_{\{v_0\}}$. Take a minimal prime $\kP$ of $(R^{\Delta}_{\{v_0\}})_{\kp_0}$. We claim that the push-forward $\rho(\kP)$ of the universal deformation to $R^{\Delta}_{\{v_0\}}/\kP$ is irreducible. The argument is almost the same as the proof of \ref{redloc}: Suppose not, we may write
\[\rho(\kP)\cong\begin{pmatrix} \tilde{\psi_2} & *\\ 0 & \tilde{\psi_1}\end{pmatrix}.\]
Since $F$ is abelian over $\Q$ and Leopoldt's conjecture is known in this case, we may choose $\tau_0\in G_{F,S}$ whose image in $G_{F,S}^{\mathrm{ab}}(p)\otimes \Q_p$ forms a basis. Let $\kP'$ be a minimal prime of $R^{\Delta}_{\{v_0\}}/\kP$ containing $\psi_{2}(\tau_0)- \tilde{\psi_2}(\tau_0)$. Then $\psi_i\equiv\tilde{\psi_i}\mod\kP',i=1,2$. Hence $\rho(\kP'){:=}\rho(\kP)\mod\kP'$ is an extension of $\psi_1$ by $\psi_2$. On the other hand, it follows from the definition of $R^{\Delta}_{\{v_0\}}$ that $\rho(\kP')|_{G_{F_v}}$ is completely reducible for $v\in\Sigma_p\setminus\{v_0\}$. By lemma \ref{Ext1char0}, such an extension is unique up to a scalar. Thus $\kP'=\kp_0$. But this implies that $\dim R^{\Delta}_{\{v_0\}}/\kP\leq 2$ which contradicts our previous lower bound.

Choose a one-dimensional prime $\kq''$ of $R^{\Delta}_{\{v_0\}}$ containing $\kP$ such that $\rho(\kq'')$ is irreducible. Let $\kq'=\kq''\cap R^{\ps}$. Now we can use what we have proved in case (1) (\ref{case(1)genprop}) with $\kq$ replaced by $\kq'$. In particular, proposition \ref{generalprop} implies that any irreducible component containing $\kq'$ is good. Hence let $\kQ$ be a minimal prime of $R_{b_0}$ contained in $\kP\subseteq \kp_0$. Its image in $\Spec R^{\ps}$ is good as it is a generalization of $\kq'$. Moreover, by lemma \ref{RBRb}, $\kQ$ is in the image of $\Spec R_{B_0}\to \Spec R_{b_0}$. Thus we have found a good component in $Z_{B_0}$.
\end{proof}

\begin{proof}[Proof of Step 2]
Enlarging $E$ if necessary, we can assume $C$ is geometrically irreducible. 

Choose a prime ideal $\kq_0$ of $R^{\ps}$ such that
\begin{enumerate}
\item $\kq_0\subset \kq$ and $\dim R^{\ps}/\kq_0=2$.
\item $C$ contains $\kq_0$ and is the only irreducible component having this property.
\item $\rho(\kq_0)$ is irreducible.
\end{enumerate}
To see its existence, note that the second and third conditions are open conditions in $\Spec (R^{\ps})_{\kq}$. Also they cut out non-empty open subset as the generic point of $C$ satisfies both conditions. Hence we can find such a prime by Corollary 10.5.8 of \cite{EGA}. 

Consider the $\kq$-adic completion $A_0$ of $(R^{\ps}/\kq_0)_{\kq}$. This is a one-dimensional complete local noetherian ring with residue field $E$. Choose a minimal prime $\kQ_{A_0}$ of $A_0$ and let $A_1=A_0/\kQ_{A_0}$. Finally we denote the normal closure of $A_1$ by $A_2$. Then $A_2\cong\tilde{E}[[T]]$ for some finite extension $\tilde{E}$ of $E$. Using that $C$ is geometrically irreducible, it is easy to see that we may enlarge $E$ and assume $\tilde{E}=E$. There is a natural inclusion $R^{\ps}/\kq_0\hookrightarrow A_2$. By abuse of notation, $\rho(\kq_0)$ can be considered as a representation over the fraction field of $A_2$. We choose a lattice of $\rho(\kq_0)$:
\[\rho(\kq_0)^o:G_{F,S}\to\GL_2(A_2)\]
such that its mod $T$ reduction is a non-split extension of $\psi_1$ by $\psi_2$. Denote this extension class by $B_1$. Then $\kq_0$ lies in the image of $\Spec R_{B_1}\to\Spec R^{\ps}$. Since $\kq_0$ only contains one minimal prime, $C\in Z_{B_1}$. It suffices to find irreducible components $C_{B_0},C_{B_1}$ as in the claim.

The following argument is very similar to the proof of proposition \ref{propB}. Let $\tilde{F}=\bar{F}^{\ker(\psi_2/\psi_1)}$. Since $\psi_1/\psi_2=\varepsilon\theta$, it is clear that $\Gal(\tilde{F}/F)$ has an open subgroup isomorphic to $\Z_p$. Hence 
\[H^i(\Gal(\tilde{F}/F),E(\psi_2/\psi_1))=0,~i\geq 1.\]
Denote the kernel of $\psi_2/\psi_1:G_{F,S}\to E^\times$ by $H$. Using Hochschild-Serre spectral sequence, we have:
\[H^1(G_{F,S},E(\psi_2/\psi_1))=H^1(H,E(\psi_2/\psi_1))^{\Gal(\tilde{F}/F)}=\Hom_{\Gal(\tilde{F}/F)}(H,E(\psi_2/\psi_1))\]
In other words, there is a natural pairing: $H\times H^1(G_{F,S},E(\psi_2/\psi_1))\to E(\psi_2/\psi_1)$. If $E\cdot B_1=E\cdot B_0$, then $Z_{B_1}=Z_{B_0}$ and our claim is clear. Hence we may assume $E\cdot B_1\neq E\cdot B_0$. Extend $B_0,B_1$ to a basis $B_0,\cdots, B_{[F:\Q]-1}$ of $H^1(G_{F,S},E(\psi_2/\psi_1))$. After possibly replacing $B_0,\cdots,B_{[F:\Q]-1}$ by some non-zero multiples, we can choose $\sigma_0,\cdots,\sigma_{[F:\Q]-1}\in H$ that form a dual basis of $B_0,\cdots, B_{[F:\Q]-1}$ under the previous pairing.

Let $\rho^{univ}_{B_0}:G_{F,S}\to\GL_2(R_{B_0})$ be a lifting given by the universal property. We write 
\[\rho^{univ}_{B_0}(\sigma)=\begin{pmatrix} a_0(\sigma) & b_0(\sigma) \\ c_0(\sigma) & d_0(\sigma) \end{pmatrix},\sigma\in G_{F,S}.\]
Moreover, we may assume $b_0(\sigma^*)=c_0(\sigma^*)=0$ for our fixed complex conjugation $\sigma^*\in G_{F,S}$. Consider the following deformation of $\rho_{B_0}$:
\[\rho_{B_0}+T\rho_{B_1}:G_{F,S}\to\GL_2(E[[T]]),~\sigma\mapsto \begin{pmatrix} \psi_2(\sigma) & b_{B_0}(\sigma)+T b_{B_1}(\sigma) \\ 0 & \psi_1(\sigma)\end{pmatrix}.\]
This gives rise to a prime $\kq_{01}$ of $R_{B_0}$. Clearly $b_0(\sigma_i)\in\kq_{01},i=2,\cdots, [F:\Q]-1$. Let $\kQ_{01}$ be a minimal prime of $(R_{B_0})_{\kq_{01}}/(b_0(\sigma_2),\cdots,b_0(\sigma_{[F:\Q]-1}))$. By abuse of notation, we also view it as a prime of $R_{B_0}$. Then by the first part of corollary \ref{corRB}, we have
\[\dim R_{B_0}/\kQ_{01}\geq [F:\Q]+2.\]
Hence $\kQ_{01}\notin \Spec R^{\mathrm{red}}_{B_0}$ by lemma \ref{RBreddim}. Let $\kQ'_{01}=\kQ_{01}\cap R^{\ps}$. Consider the ideal $I_0$ of $R^{\ps}/\kQ_{01}'$ generated by $x(\sigma_0,\tau),\tau\in G_{F,S}$ (see the notations in \ref{adx}). 

\begin{lem}
$\dim R^{\ps}/(\kQ'_{01},I_0)\geq 2+[F:\Q]$.
\end{lem}

\begin{proof}
It follows from proposition \ref{Dkp} that $\widehat{(R^{\ps})_{\kq}}$ parametrizes the deformations of the pseudo-representation $\psi_1+\psi_2$. We have a natural map $\widehat{(R^{\ps})_{\kq}}\to R_{B_0}$ given by evaluating the trace of the universal deformation. Let $\kQ''_{01}=\kQ_{01}\cap \widehat{(R^{\ps})_{\kq}}$ and $I'_0$ be the ideal of $\widehat{(R^{\ps})_{\kq}}/\kQ''_{01}$ generated by $x(\sigma_0,\tau),\tau\in G_{F,S}$. By remark \ref{ht1Rem}, 
\[\hht(I'_0)\leq 1.\]
Note that $\kQ'_{01}=\kQ''_{01}\cap R^{\ps}$. Since $R^{\ps}\to \widehat{(R^{\ps})_{\kq}}$ is flat, by the going-down property, 
\[\hht(I_0)\leq 1.\]
It follows from the third part of corollary \ref{corRB} that 
\[\dim R^{\ps}/\kQ'_{01}\geq 1+ \dim R_{B_0}/\kQ_{01}\geq [F:\Q]+3.\]
From this, our claim is clear.
\end{proof}

Therefore we may find a prime $\kP_{01}$ of $R^{\ps}$ containing $\kQ'_{01},I_0$ such that 
\[\dim R^{\ps}/\kP_{01}\geq [F:\Q]+2.\] 
It is clear that $\rho(\kP_{01})$ is irreducible. Choose a prime $\kq_1$ with the following properties:
\begin{enumerate}
\item $\kP_{01}\subseteq \kq_1 \subseteq \kq$.
\item $\dim R^{\ps}/\kq_1=2$.
\item $\rho(\kq_1)$ is irreducible.
\item Any irreducible component of $R^{\ps}$ containing $\kq_1$ also contains $\kP_{01}$.
\end{enumerate}
Again, since the last two properties are open non-empty conditions in $(R^{\ps}/\kP_{01})_{\kq}$, the existence of such a prime follows from corollary 10.5.8 of \cite{EGA}. Arguing as in the last three paragraphs of the proof of proposition \ref{propB}, we can show that $\kq_1$ belongs to the image of $\Spec R_{B_1}\to \Spec R^{\ps}$. In particular, there exists an irreducible component $C_{B_1}\in Z_{B_1}$ that contains $\kq_1$. By our construction, $\kP_{01}\in C_{B_1}$.

On the other hand, since $\kQ'_{01}$ is in the image of $\Spec R_{B_0}\to \Spec R^{\ps}$, we can find an irreducible component $C_{B_0}\in Z_{B_0}$ containing $\kQ'_{01}$, hence also containing $\kP_{01}$. Now note that
\[\dim C_{B_0}\cap C_{B_1}\geq \dim R^{\ps}/\kP_{01}\geq [F:\Q]+2.\]
We can argue as in \ref{potnice} to find a potentially nice prime in $C_{B_0}\cap C_{B_1}$. This is exactly what we want in our claim.
\end{proof}

\end{proof}

\section{Fontaine-Mazur conjecture in the residually irreducible case (I)} \label{FMcitric1}
In this subsection, we prove the Fontaine-Mazur conjecture of $\GL_2/\Q$ in the residually irreducible case (under some mild condition). This was previously known by the work of Kisin \cite{Kis09a}, Hu-Tan \cite{HT15} and Emerton \cite{Eme1}. Our treatment is more uniform by establishing a patching argument for completed homology directly. More precisely, we will prove

\begin{thm} \label{ThmirredA}
Let $p>2$ be an odd prime. Let $F$ be a totally real extension of $\Q$ in which $p$ completely splits. Suppose 
\[\rho:\Gal(\overbar{F}/F)\to \GL_2(\cO)\] 
is a continuous irreducible representation with the following properties
\begin{enumerate}
\item $\rho$ ramifies at only finitely many places.
\item $\bar\rho|_{G_{F(\zeta_p)}}$ is absolutely irreducible, where $\bar{\rho}$ denotes the reduction of $\rho$ modulo $\varpi$. 
\item \label{extracond} For any $v|p$, $\rho|_{G_{F_v}}$ is absolutely irreducible and de Rham of distinct Hodge-Tate weights. If $p=3$, then $\bar{\rho}|_{G_{F_v}}$ is not of the form $\begin{pmatrix} \eta & * \\ 0 & \eta\omega\end{pmatrix}$ or $\begin{pmatrix} \eta\omega & * \\ 0 & \eta\end{pmatrix}$.
\item $\det \rho(c)=-1$ for any complex conjugation $c\in \Gal(\overbar{F}/F)$.
\item $\bar\rho$ arises from a regular algebraic cuspidal automorphic representation $\pi_0$ of $\GL_2(\A_F)$. 
\end{enumerate}
Then $\rho$ arises from a regular algebraic cuspidal automorphic representation of $\GL_2(\A_F)$.
\end{thm}

\begin{rem}
In his thesis \cite{Tung18}, Shen-Ning Tung proved this result without our extra assumption in \eqref{extracond} of the theorem when $p=3$. 
\end{rem}

\begin{rem}
If $\rho|_{G_{F_v}}$ is reducible for any $v|p$, i.e. ordinary, the Theorem can be proved by using Hida families \cite{SW01}. If we are in the `mixed' situation, i.e. $\rho|_{G_{F_v}}$ is reducible for some $v|p$ and irreducible for some other place $v$, it should be possible to work with some `mixture of completed cohomology and Hida family'. More precisely, the space of $p$-adic automorphic forms involved would be the ordinary parts of the completed cohomology at places where $\rho|_{G_{F_v}}$ is reducible. It is conceivable that one can define Hecke algebras and develop certain structure theories in this situation. We leave the details to the interested readers.
\end{rem}

\begin{rem} \label{KW}
When $F=\Q$, the last condition automatically holds by the work of Khare-Wintenberger \cite{KW09a}, \cite{KW09b} on Serre's conjecture.
\end{rem}

\begin{para} \label{solublebc}
We do some standard reduction work here. By soluble base change, we may assume  (cf. lemma 2.2 of \cite{Ta03})
\begin{itemize}
\item $\rho|_{I_{F_v}}$ is unipotent for any $v\nmid p$. 
\item If $\rho|_{G_{F_v}}$ is ramified and $v\nmid p$, then $N(v)\equiv1\mod p$ and $\bar\rho|_{G_{F_v}}$ is trivial.
\item $[F:\Q]$ is even.
\item $\pi_0$ is unramified everywhere except places above $p$.
\end{itemize}

After possibly replacing $\cO$ by some unramified extension, we can assume that $\bar\rho|_{G_{F_v}}$ is reducible if it is absolutely reducible. We will keep all these assumptions throughout this section. For notation, we will use $\Sigma_p$ to denote all the places of $F$ above $p$ and $S$ to denote all the places where $\rho$ is ramified. Finally, let $\chi=\det \rho$ and $\bar\chi=\det \bar\rho$.
\end{para}

\subsection{Galois deformation rings}
\begin{para} \label{Gdrs}
Fix characters $\xi_v:k(v)^\times\to\cO^\times$ of $p$-power orders for $v\in S\setminus \Sigma_p$. By local class field theory, we will view them as characters of $I_{F_v}$. At the very end, we will be in the case where all $\xi_v$ are trivial. However later on, it is necessary to consider the case where none of $\xi_v$ is trivial in the proof (in order to apply Taylor's trick \cite{Ta08}). Hence we decide to allow this generality in the setup.

Suppose $Q,T$ are finite sets of primes. We introduced a Galois deformation ring $R^{\square_T,\{\xi_v\}}_{\bar\rho,Q}$ in \ref{globaldefr}. Recall that it pro-represents the functor $\mathrm{Def}^{\square_T,\{\xi_v\}}_{\bar{\rho},Q}$ from $C_{\cO}$ to the category of sets sending $R$ to the set of tuples $(\rho_R;\alpha_v)_{v\in T}$ modulo the equivalence relation $\sim$ where
\begin{itemize}
\item $\rho_R:G_{F,S\cup Q}\to\GL_2(R)$ is a lifting of $\bar{\rho}$ to $R$ with determinant $\chi$ such that $\tr (\rho_R)|_{I_{F_v}}=\xi_v+\xi_v^{-1}$ for any $v\in S\setminus \Sigma_p$.
\item $\alpha_v\in 1+M_2(\km_R)$. Here $\km_R$ is the maximal ideal of $R$.
\item $(\rho_R;\alpha_v)_{v\in T}\sim (\rho'_R;\alpha'_v)_{v\in T}$ if there exists an element $\beta\in 1+M_2(\km_R)$ with $\rho'_R=\beta\alpha_v\beta^{-1},\alpha'_v=\beta\alpha_v$ for any $v\in T$.
\end{itemize}
We will drop $Q$ (resp. $\square_T$) if $Q$ (resp. $T$) is empty. Since $\bar\rho$ is fixed in this section, we will also drop it in the subscript.

On the local side, for a finite place $v$, we denote by $R^{\square}_v$ the unrestricted universal lifting ring of $\bar{\rho}|_{G_{F_v}}:G_{F_v}\to\GL_2(\F)$ with determinant $\chi|_{G_{F_v}}$. For $v\in S\setminus \Sigma_p$, we denote by $R^{\square,\xi_v}_v$ the quotient of $R^{\square}_v$ corresponding to the liftings with traces equal to $\xi_v+\xi_v^{-1}$ when restricted on $I_{F_v}$. More details can be found in \ref{ldfr}.

Let $R^{\{\xi_v\}}_{\loc}$ be the following completed tensor product over $\cO$:
\[(\widehat{\bigotimes}_{v\in \Sigma_p}R^{\square}_v)\widehat{\otimes}(\widehat{\bigotimes}_{v\in S\setminus \Sigma_p}R^{\square,\xi_v}_v).\]
\end{para}

\begin{lem}
For $v|p$, $R^{\square}_v$ is a normal domain of dimension $1+6$ and flat over $\cO$. 
\end{lem}

\begin{proof}
If $\bar\rho|_{G_{F_v}}$ is reducible, this follows from the second part of lemma \ref{varlocdim}. If $\bar\rho|_{G_{F_v}}$ is irreducible and $H^2(G_{F_v},\ad^0 \bar\rho)=0$, then this is clear as $R^{\square}_v$ is smooth. If $\bar\rho|_{G_{F_v}}$ is irreducible and $H^2(G_{F_v},\ad^0 \bar\rho)\neq 0$, then $p=3$ and B\"ockle explicitly computed $R^{\square}_v$ (without framing and fixing determinant) in Theorem 5.1 \cite{Boc10}. Note that after changing of variables, B\"ockle's result implies that $R^{\square}_v\cong\cO[[b,c,d,d',x_1,x_2,x_3]]/(bc+d^3)$, which is clearly normal.
\end{proof}

\begin{lem} \label{iharaavoid}
For $v\in S\setminus \Sigma_p$, $R^{\square,\xi_v}_v$ is equidimensional of dimension $1+3$. The generic point of each irreducible component has characteristic zero. Moreover,
\begin{enumerate}
\item If $\xi_v$ is non-trivial, then $R^{\square,\xi_v}_v$ is integral.
\item In general, each minimal prime of $R^{\square,\xi_v}_v/(\varpi)$ contains a unique minimal prime of $R^{\square,\xi_v}_v$.
\end{enumerate}
\end{lem}
\begin{proof}
This follows from proposition 3.1 of \cite{Ta08}. 
\end{proof}

\begin{cor} \label{rlocdim}
$R^{\{\xi_v\}}_{\loc}$ is equidimensional of dimension $1+3[F:\Q]+3|S|$. The generic point of each irreducible component has characteristic zero. Moreover,
\begin{enumerate}
\item If all $\xi_v$ are distinct, then $R^{\{\xi_v\}}_{\loc}$ is integral.
\item In general, each minimal prime of $R^{\{\xi_v\}}_{\loc}/(\varpi)$ contains a unique minimal prime of $R^{\{\xi_v\}}_{\loc}$.
\end{enumerate}
\end{cor}

\begin{proof}
This follows from the previous lemmas and lemma 3.3 of \cite{BLGHT11}.
\end{proof}

\begin{para}
Given a tuple $(\rho_R;\alpha_v)_{v\in S}$ as in the definition of $R^{\square_S,\{\xi_v\}}_Q$, $\alpha_v^{-1}\rho_R\alpha_v|_{G_{F_v}}$ is a well-defined lifting of $\bar{\rho}|_{G_{F_v}}$. This induces a natural map $R^{\{\xi_v\}}_{\loc}\to R^{\square_S,\{\xi_v\}}_Q$.
\end{para}

\begin{prop} \label{extwpirr}
Set $g=\dim_\F H^1(G_{F,S},\ad^0\bar{\rho}(1))+|S|-1-[F:\Q]$. Then for any positive integer $N$, there exists a finite set of primes  $Q_N$ of $F$, disjoint from $S$, such that
\begin{itemize}
\item $|Q_N|=\dim_\F H^1(G_{F,S},\ad^0\bar{\rho}(1))$.
\item If $v\in Q_N$, then $N(v)\equiv 1\mod p^N$ and $\bar{\rho}(\Frob_v)$ has distinct eigenvalues.
\item As an $R^{\{\xi_v\}}_{\loc}$-algebra, $R^{\square_S,\{\xi_v\}}_{Q_N}$ is topologically generated by $g$ elements.
\end{itemize}
\end{prop}
\begin{proof}
This is proposition 3.2.5 of \cite{Kis09b} if $S=\Sigma_p$. In general, the same argument still works.
\end{proof}

\subsection{Completed homology with auxiliary levels}\label{Chwal}
\begin{para} \label{Healg}
By the global class field theory, we may view $\psi=\chi\varepsilon$ as a character of $\AFi/F^\times_{>>0}$. Let $D$ be the quaternion algebra over $F$ which is ramified exactly at all infinite places. Fix an isomorphism between $\DAi$ and $\GL_2(\A_F^\infty)$. We define a tame level $U^p=\prod_{v\nmid p}U_v$ as follows: $U_v=\GL_2(O_{F_v})$ if $v\notin S$ and 
\[U_v=\mathrm{Iw}_v:=\{g\in\GL_2(O_{F_v}),g\equiv \begin{pmatrix}*&*\\0&*\end{pmatrix}\mod \varpi_v\}\]
otherwise. For any $v\notin S\setminus \Sigma_p$, the map $\begin{pmatrix}a&b\\c&d\end{pmatrix}\mapsto \xi_v(\frac{a}{d}\mod \varpi_v)$ defines a character of $U_v$. The product of $\xi_v$ can be viewed as a character $\xi$ of $U^p$ by projecting to $\prod_{v\notin S\setminus \Sigma_p}U_v$. 

Using this, we can define a completed homology $M_{\psi,\xi}(U^p)$ and a Hecke algebra $\T:=\T_{\psi,\xi}(U^p)$ as in section \ref{varch}. The existence of $\pi_0$ in Theorem \ref{ThmirredA} implies that $\bar{\rho}$ is \textit{modular} in the sense that $T_v-\tr\bar{\rho}(\Frob_v),v\notin S$ and $\varpi$ generate a maximal ideal $\km$ of $\T$. Note that   the determinant of $\pi_0$ might be different from $\psi$, but they become the same after modulo $\varpi$. So $\pi_0$ still gives arise to a maximal of $\T$.

Since $\bar{\rho}$ is absolutely irreducible, there is a two-dimensional representation $\rho_\km:G_{F,S}\to\GL_2(\T_\km)$ with determinant $\chi$ such that the trace of $\Frob_v$ is $T_v$ for $v\notin S$. Note that $\tr\rho_\km|_{I_{F_v}}=\xi_v+\xi_v^{-1}$. Hence $\rho_\km$ induces a natural map $R^{\{\xi_v\}}\to \T_\km$, which is surjective as $\T_\km$ is topologically generated by $T_v,v\notin S$. These can all be checked easily on the finite levels.
\end{para}

\begin{para}
For any positive integer $N$, we fix a set of primes $Q_N$ as in proposition \ref{extwpirr}. Denote the unique quotient of $k(v)^\times$ of order $p^N$ by $\Delta_v$ and $\bigoplus_{v\in Q_N}\Delta_v$ by $\Delta_N$. Define tame levels $U^p_{Q_N,0}=\prod_{v\nmid p}U_{Q_N,0,v} ,U^p_{Q_N}=\prod_{v\nmid p}U_{Q_N,v}$ as follows: 
\begin{itemize}
\item $U_{Q_N,0,v}=U_{Q_N,v}=U_v$ if $v\notin Q_N$.
\item $U_{Q_N,0,v}=\mathrm{Iw}_v$ if $v\in Q_N$.
\item $U_{Q_N,v}=\ker(\mathrm{Iw}_v\stackrel{\varphi_v}{\to}\Delta_v)$ if $v\in Q_N$, where $\varphi_v$ is the composite of $\mathrm{Iw}_v\to k(v)^\times: \begin{pmatrix}a&b\\c&d\end{pmatrix}\mapsto \frac{a}{d}\mod \varpi_v$ and the natural quotient map $k(v)^\times\to \Delta_v$.
\end{itemize}

It is clear that $\Delta_v=\mathrm{Iw}_v/U_{Q_N,v}$ acts naturally on the completed homology $M_{\psi,\xi}(U^p_{Q_N})$ and $S_{\psi,\xi}(U^p_{Q_N}U_p,\cO/\varpi^n)$ via the right translation of $\mathrm{Iw}_v$ for any open compact subgroup $U_p\subseteq K_p$ and positive integer $n$. Hence all these spaces are $\cO[\Delta_N]$-modules. Let $\mathfrak{a}_{Q_N}$ be the augmentation ideal of $\cO[\Delta_N]$.
\end{para}

\begin{lem}  \label{finiteflat}
Suppose $U_p\subseteq K_p$ is an open compact subgroup such that $\psi|_{U_p\cap O_{F,p}^\times}$ is trivial modulo $\varpi^n$ for some $n$ and $U^pU_p$ is sufficiently small. Then 
\begin{enumerate}
\item $S_{\psi,\xi}(U^p_{Q_N}U_p,\cO/\varpi^n)$ and $S_{\psi,\xi}(U^p_{Q_N}U_p,\cO/\varpi^n)^\vee$ are finite flat $\cO/\varpi^n[\Delta_N]$-modules.
\item The natural map $S_{\psi,\xi}(U^p_{Q_N}U_p,\cO/\varpi^n)^\vee\to S_{\psi,\xi}(U^p_{Q_N,0}U_p,\cO/\varpi^n)^\vee$ induces a natural isomorphism:
\[S_{\psi,\xi}(U^p_{Q_N}U_p,\cO/\varpi^n)^\vee/\ka_{Q_N}S_{\psi,\xi}(U^p_{Q_N}U_p,\cO/\varpi^n)^\vee\cong S_{\psi,\xi}(U^p_{Q_N,0}U_p,\cO/\varpi^n)^\vee.\]
\end{enumerate}
\end{lem}
\begin{proof}
This is the Pontryagin dual of lemma 2.1.4 of \cite{Kis09a}.
\end{proof}

\begin{para}
We define $\T_{Q_N,0}\subseteq\End (M_{\psi,\xi}(U^p_{Q_N,0}))$ (resp. $\T_{Q_N}\subseteq\End (M_{\psi,\xi}(U^p_{Q_N}))$) to be the $\T_{\psi,\xi}(U^p_{Q_N,0})$-subalgebra generated by $U_{\varpi_v},v\in Q_N$ (resp. the $\T_{\psi,\xi}(U^p_{Q_N})$-subalgebra generated by $U_{\varpi_v},v\in Q_N$). Here $U_{\varpi_v}$ acts via the double coset action $[U_{Q_N,0,v}\begin{pmatrix}\varpi_v&0\\ 0&1\end{pmatrix}U_{Q_N,0,v}]$ (resp. $[U_{Q_N,v}\begin{pmatrix}\varpi_v&0\\ 0&1\end{pmatrix}U_{Q_N,v}]$). There is a natural map $\T_{Q_N}\to\T_{Q_N,0}$.

For any $v\in Q_N$, we denote the eigenvalues of $\bar{\rho}(\Frob_v)$ by $\alpha_v,\beta_v\in\F$ (after possibly enlarging $\F$), which are distinct by our construction. Therefore using Hensel's lemma, we can find a unique root $A_v\in\T_\km$ of $X^2-T_vX+\chi(\Frob_v)=0$ that lifts $\alpha_v\in\T/\km$. The natural map $\T_{\psi,\xi}(U^p_{Q_N,0})\to\T$ now can be extended uniquely to a map $\T_{Q_N,0}\to \T_\km$ sending $U_{\varpi_v}$ to $A_v$. Hence $\km\cap \T_{Q_N,0}$ defines a maximal ideal $\km_{Q_N,0}$ of $\T_{Q_N,0}$, which is nothing but $(\km\cap \T_{\psi,\xi}(U^p_{Q_N,0}),U_v-\tilde\alpha_v,v\in Q_N)$ where $\alpha_v\in\cO$ is any lifting of $\alpha_v$.
\end{para}

\begin{lem}
Under the same assumption as in the previous lemma, the natural map
\[S_{\psi,\xi}(U^pU_p,\cO/\varpi^n)_\km\to S_{\psi,\xi}(U^p_{Q_N,0}U_p,\cO/\varpi^n)_{\km_{Q_N,0}}\]
is an isomorphism of $(\T_{Q_N,0})_{\km_{Q_N,0}}$-modules.
\end{lem}

\begin{proof}
Let $\tilde\psi$ be the Teichm\"uller lifting of $\psi\mod \varpi$ and write $\tilde\psi=\psi\theta^2$ for some continuous character $\theta:(\A^\infty_F)^\times/F^\times_{>>0}\to \cO^\times$. By twisting with $\theta$ and arguing as in \ref{exgal},  we may assume $\psi=\tilde\psi$. Now it suffices to show that $S_{\psi,\xi}(U^pU_p,\cO)_\km\to S_{\psi,\xi}(U^p_{Q_N,0}U_p,\cO)_{\km_{Q_N,0}}$ is an isomorphism. This follows from directly from lemma 2.1.7 of \cite{Kis09a}.
\end{proof}

\begin{para} \label{Sinfty}
Denote $\km_{Q_N,0}\cap\T_{Q_N}$ by $\km_{Q_N}$. Then under the same assumption as in lemma \ref{finiteflat}, it follows that $S_{\psi,\xi}(U^p_{Q_N}U_p,\cO/\varpi^n)^\vee_{\km_{Q_N}}$ is a finite flat $\cO/\varpi^n[\Delta_N]$-module and we have natural $(\T_{Q_N})_{\km_{Q_N}}$-module isomorphisms
\begin{eqnarray} \label{fflat2}
~S_{\psi,\xi}(U^p_{Q_N}U_p,\cO/\varpi^n)^\vee_{\km_{Q_N}}/(\ka_{Q_N})\cong S_{\psi,\xi}(U^p_{Q_N,0}U_p,\cO/\varpi^n)^\vee_{\km_{Q_N,0}}\cong S_{\psi,\xi}(U^p,\cO/\varpi^n)^\vee_{\km}.
\end{eqnarray}

We recall that there is a Galois-theoretic meaning of $\Delta_N$ in the following sense. There is a natural map $R^{\{\xi_v\}}_{Q_N}\to (\T_{Q_N})_{\km_{Q_N}}$ sending the trace of $\Frob_v,v\notin S\cup Q_N$ on the tautological $R^{\{\xi_v\}}_{Q_N}$-representation $\rho_{Q_N}$ of $G_{F,S\cup Q_N}$ to $T_v$. Hence for any $U_p,n$, we may view $S_{\psi,\xi}(U^p_{Q_N}U_p,\cO/\varpi^n)^\vee_{\km_{Q_N}}$ as a $R^{\{\xi_v\}}_{Q_N}$-module. Moreover there is a natural $\cO[\Delta_{Q_N}]$-algebra structure on $R^{\{\xi_v\}}_{Q_N}$ (which comes from $\rho_{Q_N}|_{G_{F_v}},v\in Q_N$) such that its induced action of $\cO[\Delta_{Q_N}]$ on $S_{\psi,\xi}(U^p_{Q_N}U_p,\cO/\varpi^n)^\vee_{\km_{Q_N}}$ agrees with the one considered before. Under the natural map $R^{\{\xi_v\}}_{Q_N}\to R^{\{\xi_v\}}$, the image of $\ka_{Q_N}$ in $R^{\{\xi_v\}}$ is zero. See lemma 2.1 of \cite{Ta06} for more details.
\end{para}

\subsection{Patching of completed homology}\label{Pofch}
\begin{para}
We will patch the completed homology and conclude with an `$R=\T$' Theorem under certain conditions. The patched completed homology first appeared in the work \cite{CEG+}. Following \cite{Sch15}, we adopt the language of ultrafilters (see also \cite{GN16}). 

Let $\cI$ be the set of positive integers  and $\fR=\prod_\cI \cO$. From now on we fix a non-principal ultrafilter $\kF$ on $\cI$. Then $\kF$ gives rise to a multiplicative set $S_\kF\subseteq \fR$ which contains all idempotents $e_{I}$ with $I\in\kF$ and $e_I(i)=1$ if $i\in I$, $e_I(i)=0$ otherwise. 

We define $\fR_\kF=S_\kF^{-1}\fR$. This is a quotient of $\fR$ and $\cdot\otimes_\fR \fR_\kF$ is an exact functor. Since $\kF$ is non-principal, for any finite set $T\subseteq \cI$, we have $\fR_T\otimes_\fR \fR_\kF\cong \fR_\kF$, where $\fR_T$ is the quotient of $\fR$ by elements of the form $(a_i)_{i\in \cI}$ with $a_i=0$ for $i\notin T$.
\end{para} 

\begin{para}
First we set
\begin{itemize}
\item $\Delta_\infty:=\Z_p^{\oplus r}$, where $r=\dim_\F H^1(G_{F,S},\ad^0\bar{\rho}(1))=|Q_N|$. Fix surjective maps $\Z_p\to \Delta_v$ for all $v\in Q_N$ and thus surjective maps $\Delta_\infty\to \Delta_{Q_N}$. 
\item $\cO_\infty=\cO[[y_1,\cdots,y_{4|S|-1}]]$ with maximal ideal $\kb$. 
\item $S_\infty=\cO_\infty[[\Delta_\infty]]\cong \cO_{\infty}[[s_1,\cdots,s_r]]$. This is a local $\cO_\infty$-algebra with maximal ideal $\ka$. 
\item $\ka_0=\ker(S_\infty\to\cO_\infty)=(s_1,\cdots,s_r)$ the augmentation ideal, and $\ka_1=\ker(S_\infty\to\cO)=(y_1,\cdots,y_{4|S|-1},s_1,\cdots,s_r)$. 
\end{itemize}

For any positive integer $n$, recall that $S_{\psi,\xi}(U^p_{Q_N}U_p,\cO/\varpi^n)_{\km_{Q_N}}$ is an $\cO[\Delta_{Q_N}]$-module. Hence $S_{\psi,\xi}(U^p_{Q_N}U_p,\cO/\varpi^n)_{\km_{Q_N}}^\vee\otimes_\cO \cO_\infty$ is a natural $S_\infty$-module. For simplicity, we denote it by $M(U_p,N,n)$. Then $\prod_{N\in\cI}M(U_p,N,n)$ has a natural $\fR$-module structure. We define

\begin{itemize}
\item $M^{\patch,\{\xi_v\}}_n:=\varprojlim_{U_p}(\prod_{N\in\cI} M(U_p,N,n)/\ka^n M(U_p,N,n)\otimes_\fR \fR_\kF)$.
\item $M^{\{\xi_v\}}_{n}:=\varprojlim_{U_p}(\prod_{\cI} (S_{\psi,\xi}(U^pU_p,\cO/\varpi^n)_\km^\vee\otimes_\cO \cO_\infty/\kb^n) \otimes_\fR \fR_\kF)$.
\end{itemize}
Both are equipped with the topology given by the projective limits. Here $U_p$ runs through all open compact subgroups of $D_p^\times=\prod_{v|p}\GL_2(F_v)$ in all projective limits above. We note that $M^{\{\xi_v\}}_{n}$ is nothing but 
\[\varprojlim_{U_p}(S_{\psi,\xi}(U^pU_p,\cO/\varpi^n)_\km^\vee\otimes_\cO \cO_\infty/\kb^n)=M_{\psi,\xi}(U^p)_\km \otimes_\cO \cO_\infty/\kb^n\]
since $S_{\psi,\xi}(U^pU_p,\cO/\varpi^n)_\km^\vee\otimes_\cO \cO_\infty/\kb^n$ has finite cardinality.
\end{para}

\begin{para}
In the definition of $M^{\patch,\{\xi_v\}}_n$, we can replace the limits taken over all open compact subgroups of $D_p^\times$ by over all open subgroups of $K_p=\prod_{v|p}\GL_2(O_{F_v})$. From this description, it is clear that both patched completed homologies are natural $\cO[[K_p]]$-modules. The action of $K_p$ can be extended to $D_p^\times=\prod_{v|p}\GL_2(F_v)$ in the usual way: the action of $g\in D_p^\times$ induces isomorphisms $M(U_p,N,n)\stackrel{\sim}{\longrightarrow}M(g^{-1}U_pg,N,n)$.

We prove some simple properties of $M^{\patch,\{\xi_v\}}_n$. Note that the diagonal action of  $S_\infty$ on $\prod_{N\in\cI} M(U_p,N,n)/\ka^n M(U_p,N,n)$ defines a natural $S_\infty$-module structure on $M^{\patch,\{\xi_v\}}_n$.
\end{para}

\begin{lem} \label{mnflat2}
$M^{\patch,\{\xi_v\}}_n$ is a flat $S_\infty/\ka^n$-module. Moreover the natural maps $M^{\patch,\{\xi_v\}}_n$ to $M^{\patch,\{\xi_v\}}_{n-1}$ and \eqref{fflat2} induce isomorphisms
\begin{eqnarray*}
M^{\patch,\{\xi_v\}}_n/\ka^{n-1}M^{\patch,\{\xi_v\}}_n&\cong& M^{\patch,\{\xi_v\}}_{n-1}.\\
M^{\patch,\{\xi_v\}}_n/\ka_0 M^{\patch,\{\xi_v\}}_n&\cong& M^{\{\xi_v\}}_{n}\cong M_{\psi,\xi}(U^p)_\km \otimes_\cO \cO_\infty/\kb^n.
\end{eqnarray*}
In particular, $(M^{\patch,\{\xi_v\}}_n)^\vee$ are admissible $D_p^\times$-representations as $(M_{\psi,\xi}(U^p)_\km)^\vee$ is admissible.
\end{lem}

\begin{proof}
This follows easily from lemma \ref{finiteflat} and its variant in \eqref{fflat2} and lemma \ref{flmf}.
\end{proof}

\begin{para}
Finally we define the patched completed homology to be
\[M^{\{\xi_v\}}_\infty=\varprojlim_n M^{\patch,\{\xi_v\}}_n.\]
It follows from the previous lemma that $M^{\{\xi_v\}}_\infty$ is a flat $S_\infty$-module and $\ka$-adically complete and separated. There is a natural isomorphism 
\begin{eqnarray} \label{Minftyma1}
M^{\{\xi_v\}}_\infty/\ka_1M^{\{\xi_v\}}_\infty\cong M_{\psi,\xi}(U^p)_\km.
\end{eqnarray}
Moreover there is a natural action of $\cO[[K_p]]$ and $\cO[D_p^\times]$ such that both actions agree on $\cO[K_p]$. It follows from the construction that $M^{\{\xi_v\}}_\infty$ is an object of $\kC_{D_p^\times,\psi}(\cO)$.
\end{para}

\begin{para}
On the other hand,
\begin{itemize}
\item Set $R^{\{\xi_v\}}_\infty=R^{\{\xi_v\}}_{\loc}[[x_1,\cdots,x_g]]$, where $g=r+|S|-1-[F:\Q]$ as in proposition \ref{extwpirr}.
\item Fix isomorphisms $R^{\square_S,\{\xi_v\}}_{Q_N}\cong R^{\{\xi_v\}}_{Q_N}\hat\otimes_\cO \cO_\infty$ and surjective $R^{\{\xi_v\}}_{\loc}$-algebra homomorphisms $R^{\{\xi_v\}}_\infty\twoheadrightarrow R^{\square_S,\{\xi_v\}}_{Q_N}$ which exist by proposition \ref{extwpirr}. 
\end{itemize}
The natural action of $R^{\{\xi_v\}}_{Q_N}$ on $S_{\psi,\xi}(U^p_{Q_N}U_p,\cO/\varpi^n)_{\km_{Q_N}}^\vee$ extends $\cO_\infty$-linearly to an action of $R^{\{\xi_v\}}_{Q_N}\otimes_\cO \cO_\infty$ on $M(U_p,N,n)$. Hence we get an action of $R^{\{\xi_v\}}_{\infty}$ on $M(U_p,N,n)$ via the maps fixed above. The diagonal action of $R^{\{\xi_v\}}_{\infty}$ on $\prod_{N\in\cI} M(U_p,N,n)/\ka^n M(U_p,N,n)$ thus makes $M^{\patch,\{\xi_v\}}_n$ and $M^{\{\xi_v\}}_\infty$ into $R^{\{\xi_v\}}_{\infty}$-modules. 

Using the fact that $M(U_p,N,n)/\ka^n M(U_p,N,n)$ has bounded cardinalities for any $N$ and $R^{\{\xi_v\}}_{Q_N}$ is topologically finitely generated, we conclude from the discussion in the second paragraph of \ref{Sinfty} that the image of $R^{\{\xi_v\}}_{\infty}$ in $\End (M^{\{\xi_v\}}_\infty)$ contains the image of $S_\infty$. Hence there exists a lifting $S_\infty\to R^{\{\xi_v\}}_{\infty}$ which makes the action of $S_\infty$ on $M^{\{\xi_v\}}_\infty$ factors through $R^{\{\xi_v\}}_{\infty}$. Moreover, it follows from the discussion in \ref{Sinfty} that we can choose the lifting such that the image of $\ka_1$ is in the kernel of $R^{\{\xi_v\}}_{\infty}\twoheadrightarrow R^{\{\xi_v\}}$.
\end{para}

\subsection{An application of Pa\v sk\=unas' theory and the local-global compatibility result}

\begin{para}
In subsection \ref{sblgc}, we attach a block $\kB_\km$ of $D_p^\times$ to the maximal ideal $\km$ of Hecke algebra. Under our assumptions, it follows from Theorem \ref{lgc} that $M_{\psi,\xi}(U^p)_\km\in \kC_{D_p^\times,\psi}(\cO)^{\kB_{\km}}$. Hence $M^{\{\xi_v\}}_\infty$ also belongs to $\kC_{D_p^\times,\psi}(\cO)^{\kB_{\km}}$ as  $M^{\{\xi_v\}}_\infty/\ka_1M^{\{\xi_v\}}_\infty\cong M_{\psi,\xi}(U^p)_\km$. Let $P_{\kB_\km}$ be the projective generator of $\kB_\km$. See \ref{sblgc} for the notations here.

We define
\begin{itemize}
\item $\mm^{\{\xi_v\}}_0:=\Hom_{\kC_{D_p^\times,\psi}(\cO)}(P_{\kB_\km},M_\psi(U^p)_\km)$.
\item $\mm^{\{\xi_v\}}_\infty:=\Hom_{\kC_{D_p^\times,\psi}(\cO)}(P_{\kB_\km},M^{\{\xi_v\}}_\infty)$.
\end{itemize}
Recall that $\Hom_{\kC_{D_p^\times,\psi}(\cO)}(P_{\kB_\km},\cdot)$ is an exact functor. Hence $\mm^{\{\xi_v\}}_\infty$ is a flat $S_\infty$-module. It follows from corollary \ref{mclgc} that $\mm^{\{\xi_v\}}_0$ is a \textit{finitely generated, faithful} $\T_\km$-module. Also by \eqref{Minftyma1}, we have
\[\mm^{\{\xi_v\}}_0\cong\mm^{\{\xi_v\}}_\infty/\ka_1\mm^{\{\xi_v\}}_\infty\]
and $\mm^{\{\xi_v\}}_\infty$ is $\ka$-adically complete and separated. Note that $\mm^{\{\xi_v\}}_\infty$ is a natural $R^{\{\xi_v\}}_{\infty}$-module. The action of $R^{\{\xi_v\}}_{\infty}$ on $\mm^{\{\xi_v\}}_\infty/\ka\mm^{\{\xi_v\}}_\infty\cong\mm^{\{\xi_v\}}_0/\varpi\mm^{\{\xi_v\}}_0$ factors through its quotient $\T_\km$. Hence 
\end{para}

\begin{lem}
$\mm^{\{\xi_v\}}_\infty$ is a finitely generated $R^{\{\xi_v\}}_{\infty}$-module.
\end{lem}

\begin{para}
Now we have the following commutative diagram
\[\begin{tikzcd}
S_\infty \arrow[r,dashed]  & R^{\{\xi_v\}}_{\infty} \arrow[ld, two heads,"\pi_R"]  \arrow [d, two heads] \arrow[r] & \End_{S_\infty} (\mm^{\{\xi_v\}}_\infty) \arrow[d] \\
R^{\{\xi_v\}} \arrow[r,two heads]  &{\T_{\km}} \arrow[r,hook] & \End_{\cO} (\mm_0^{\{\xi_v\}}),
\end{tikzcd}\]
such that the image of $\ka_1\subseteq S_\infty$ in $R^{\{\xi_v\}}_{\infty}$ is in fact in the kernel of $\pi_R$.
 
Recall that $\ka_1=(y_1,\cdots,y_{4|S|-1},s_1,\cdots,s_r)$. It follows from the flatness of $\mm^{\{\xi_v\}}_\infty$ over $S_\infty$ that $y_1,\cdots,y_{4|S|-1},s_1,\cdots,s_r$ form a regular sequence for $\mm^{\{\xi_v\}}_\infty$. Hence
\[\dim_{R^{\{\xi_v\}}_{\infty}} (\mm^{\{\xi_v\}}_\infty) = \dim_{R^{\{\xi_v\}}_{\infty}} (\mm^{\{\xi_v\}}_\infty/\ka_1\mm^{\{\xi_v\}}_\infty)+4|S|-1+r=\dim_{\T_\km} (\mm_0^{\{\xi_v\}})+4|S|-1+r.\]
Note that $\mm_0^{\{\xi_v\}}$ is a faithful $\T_\km$-module, hence $\dim_{\T_\km} (\mm_0^{\{\xi_v\}})=\dim {\T_\km}$. Thus $\dim_{\T_\km} (\mm_0^{\{\xi_v\}})\geq1+2[F:\Q]$ by Theorem \ref{dh}. We get
\[\dim_{R^{\{\xi_v\}}_{\infty}} (\mm^{\{\xi_v\}}_\infty)\geq 4|S|+2[F:\Q]+r.\]

On the other hand, since $R^{\{\xi_v\}}_{\infty}=R^{\{\xi_v\}}_{\loc}[[x_1,\cdots,x_g]]$, it follow from corollary \ref{rlocdim} that 
\[\dim R^{\{\xi_v\}}_{\infty}=1+3[F:\Q]+3|S|+g=4|S|+2[F:\Q]+r.\]
Combined with the previous inequality, it is easy to see 
\end{para}

\begin{lem} \label{suppmin0}
The support of $\mm_\infty^{\{\xi_v\}}$ over $R^{\{\xi_v\}}_{\infty}$ contains an irreducible component of $R^{\{\xi_v\}}_{\infty}$. Moreover any minimal prime of $\Supp_{R^{\{\xi_v\}}_{\infty}} (\mm_\infty^{\{\xi_v\}})$ has characteristic zero.
\end{lem}
\begin{proof}
The second claim follows from the fact that $\mm_\infty^{\{\xi_v\}}$ is $\cO$-flat.
\end{proof}

\begin{cor} \label{nondihia}
If all $\xi_v$ are non-trivial, then $\mm_0^{\{\xi_v\}}$ has full support on $\Spec R^{\{\xi_v\}}_{\infty}$.
\end{cor}
\begin{proof}
It follows from lemma \ref{rlocdim} that $R^{\{\xi_v\}}_{\infty}$ is irreducible in this case.
\end{proof}

\begin{para}
To treat the general case, we use Taylor's trick in \cite{Ta08}.Let $\xi'_v:k(v)\to \cO^\times$ be non-trivial characters of $p$-power order for $v\notin S\setminus \Sigma_p$. Then the product of $\xi'_v$ can be viewed as a character $\xi'$ of $U^p$ and we can define completed cohomology $S_{\psi,\xi'}(U^p)$, $S_{\psi,\xi'}(U^p,E/\cO)$ and Hecke algebra $\T':=\T_{\psi,\xi'}(U^p)$ as in subsection \ref{varch}. The following result is lemma \ref{SWbc}.
\end{para}

\begin{lem}
$T_v-\tr\bar{\rho}(\Frob_v),v\notin S$ and $\varpi$ generate a maximal ideal $\km'$ of $\T'$.
\end{lem}

\begin{para}
Therefore we get a non-zero surjective map $R^{\{\xi'_v\}}\to\T'_{\km'}$. Note that $R^{\ps,\{\xi'_v\}}/(\varpi)\cong R^{\ps,\{\xi_v\}}/(\varpi)$ as both rings represent the same universal problem. Similarly, we have a natural isomorphism $R^{\{\xi_v\}}_{\loc}/(\varpi)\cong R^{\{\xi'_v\}}_{\loc}/(\varpi)$.

It is easy to see that we can use the same set of Taylor-Wiles primes $Q_N$ in proposition \ref{extwpirr} with $\xi_v$ replaced by $\xi'_v$ and choose the map $R^{\{\xi'_v\}}_{\loc}[[x_1,\cdots,x_g]]\to R^{\square_S,\{\xi'_v\}}_{Q_N}$ to be the same as $R^{\{\xi_v\}}_{\loc}[[x_1,\cdots,x_g]]\to R^{\square_S,\{\xi_v\}}_{Q_N}$ after reducing mod $\varpi$ under the above isomorphisms. Thus we can use the same primes to patch our completed homology and get a diagram:
\[\begin{tikzcd}
S_\infty \arrow[r,dashed]  & R^{\{\xi'_v\}}_{\infty} \arrow[ld, two heads,"\pi_R"]  \arrow [d, two heads] \arrow[r] & \End_{S_\infty} (\mm^{\{\xi'_v\}}_\infty) \arrow[d] \\
R^{\{\xi'_v\}} \arrow[r,two heads]  &{\T'_{\km'}} \arrow[r,hook] & \End_{\cO} (\mm_0^{\{\xi'_v\}}).
\end{tikzcd}\]
Moreover, we have the following commutative diagram
\[\begin{tikzcd}
R^{\{\xi'_v\}}_{\infty}/(\varpi)   \arrow [d,"\cong"] \arrow[r] & \End_{S_\infty} (\mm^{\{\xi'_v\}}_\infty/\varpi \mm^{\{\xi'_v\}}_\infty) \arrow[d,"\cong"] \\
R^{\{\xi_v\}}_{\infty}/(\varpi)  \arrow[r] & \End_{S_\infty} (\mm^{\{\xi_v\}}_\infty/\varpi \mm^{\{\xi_v\}}_\infty) \\
\end{tikzcd}\]

Since $\xi'_v$ are all non-trivial, $\mm^{\{\xi'_v\}}_\infty$ has full support on $\Spec R^{\{\xi'_v\}}_{\infty}$ by corollary \ref{nondihia}. Hence $\mm^{\{\xi_v\}}_\infty/\varpi \mm^{\{\xi_v\}}_\infty$ also has full support on $\Spec R^{\{\xi'_v\}}_{\infty}/(\varpi)$ by the above diagram. It follows from lemma \ref{rlocdim} that any minimal prime of $R^{\{\xi_v\}}_{\infty}/(\varpi)$ contains a unique characteristic zero prime in the support of $\mm^{\{\xi_v\}}_\infty$. By lemma \ref{suppmin0}, this means all minimal primes of $R^{\{\xi_v\}}_\infty$ are in the support. Therefore
\end{para}

\begin{prop}
$\mm^{\{\xi_v\}}_\infty$ has full support on $\Spec R^{\{\xi_v\}}_{\infty}$.
\end{prop}

\begin{cor}\label{R=T}
$\mm^{\{\xi_v\}}_0$ has full support on $\Spec R^{\{\xi_v\}}$. In particular, the kernel of 
\[R^{\{\xi_v\}}\twoheadrightarrow\T_\km\]
is nilpotent.
\end{cor}
\begin{proof}
Since $\mm^{\{\xi_v\}}_\infty$ is a finitely generated $R^{\{\xi_v\}}_{\infty}$-module, this follows from the previous proposition and the fact that the kernel of $R^{\{\xi_v\}}_{\infty}\to R^{\{\xi_v\}}$ contains $\ka_1$.
\end{proof}

\begin{para} \label{Proof of thmirredA}
Now we can prove our main Theorem.
\begin{proof}[Proof of Theorem \ref{ThmirredA}]
Take $\xi_v$ to be the trivial characters for all $v\in S\setminus \Sigma_p$. Then $\rho$ gives rise to a maximal prime ideal $\kp$ of $R^{\{\mathbf{1}\}}[\frac{1}{p}]$, which can be viewed as a maximal ideal of $\T_\km[\frac{1}{p}]$ by our previous result. Theorem \ref{ThmirredA} follows from corollary \ref{classicality}.
\end{proof}
\end{para}

\section{Fontaine-Mazur conjecture in the residually irreducible case (II)} \label{FMcitric2}

In this section, we prove Theorem \ref{ThmirredA} without assuming $\bar\rho|_{G_{F(\zeta_p)}}$ is irreducible.
\begin{thm} \label{ThmirredB}
Let $F$ be a totally real extension of $\Q$ in which $p$ completely splits. Suppose 
\[\rho:\Gal(\overbar{F}/F)\to \GL_2(\cO)\] 
is a continuous irreducible representation with the following properties
\begin{enumerate}
\item $\rho$ ramifies at only finitely many places.
\item $\bar\rho$ is absolutely irreducible, where $\bar{\rho}$ denotes the reduction of $\rho$ modulo $\varpi$. 
\item For any $v|p$, $\rho|_{G_{F_v}}$ is absolutely irreducible and de Rham of distinct Hodge-Tate weights. If $p=3$, then $\bar{\rho}|_{G_{F_v}}$ is not of the form $\begin{pmatrix} \eta & * \\ 0 & \eta\omega\end{pmatrix}$ or $\begin{pmatrix} \eta\omega & * \\ 0 & \eta\end{pmatrix}$
\item $\det \rho(c)=-1$ for any complex conjugation $c\in \Gal(\overbar{F}/F)$.
\item $\bar\rho$ arises from a regular algebraic cuspidal automorphic representation $\pi_0$ of $\GL_2(\A_F)$. 
\end{enumerate}
Then $\rho$ arises from a regular algebraic cuspidal automorphic representation of $\GL_2(\A_F)$.
\end{thm}
 
\begin{rem}
If we are in the ordinary case, i.e. $\rho|_{G_{F_v}}$ is reducible for any $v|p$, this result is known mainly by the work of Skinner-Wiles \cite{SW01}. Note that in \cite{SW01}, they need to assume that $\bar\rho|_{G_{F_v}}$ is not of the form $\begin{pmatrix} \eta & * \\ 0 & \eta\end{pmatrix}$ for any $v|p$. However, if this happens, then $\bar\rho|_{G_{F(\zeta_p)}}$ cannot be reducible as we assume $p$ splits completely in $F$. In this case, the result was known by the work of \cite{Kis09a}, \cite{HT15}.
\end{rem}

\begin{rem}
Again the last condition holds when $F=\Q$. See remark \ref{KW}.
\end{rem}

\begin{para}
For the proof, we follow the strategy of \cite{SW01}. We will use notations introduced in \ref{solublebc}. By soluble base change, we may assume everything in \ref{solublebc}. and
\[[F:\Q]-4|S|+4|\Sigma_p|\geq2.\]

As in \ref{Gdrs}, we have a Galois deformation ring $R^{\{\mathbf{1}\}}$ parametrizing all deformations of $\bar\rho:G_{F,S}\to\GL_2(\F)$ which are unipotent when restricted to $I_{F_v},v\in S\setminus \Sigma_p$ and have determinant $\det \rho$.  On the automorphic side, we have a Hecke algebra $\T$ and a maximal ideal $\km\in\Spec \T$ corresponding to $\bar\rho$ as in \ref{Healg} with all $\xi_v$ chosen to be the trivial characters. There is a natural surjective map $R^{\{\mathbf{1}\}}\to\T_\km$. Arguing as in \ref{Proof of thmirredA}, it suffices to prove: 
\end{para}

\begin{prop} \label{kerRtoTnil}
The kernel of $R^{\{\mathbf{1}\}}\to\T_\km$ is nilpotent.
\end{prop}

As before, we say a prime of $R^{\{\mathbf{1}\}}$ is \textit{pro-modular} if it comes from a prime of $\T_\km$.

\begin{defn}
A pro-modular prime $\kq$ of $R^{\{\mathbf{1}\}}$ is \textit{nice} if 
\begin{enumerate}
\item $R^{\{\mathbf{1}\}}/\kq$ is one-dimensional with characteristic $p$.
\item $\rho(\kq)$ is absolutely irreducible and not induced from a quadratic extension of $F$. Here $\rho(\kq)$ denotes the push-forward of the universal deformation on $R^{\{\mathbf{1}\}}$ to $k(\kq)$.
\item $\rho(\kq)|_{G_{F_v}}$ is trivial for $v\in S\setminus\Sigma_p$.
\end{enumerate}
\end{defn}

\begin{prop}
Let $\kq\in\Spec R^{\{\mathbf{1}\}}$ be a nice prime. By abuse of notation, we also view it as a prime of $\T_\km$. Then the kernel of $(R^{\{\mathbf{1}\}})_\kq\to \T_\kq$ is nilpotent.
\end{prop}

\begin{proof}
The proof is essentially the same as that of Theorem \ref{thmA}. In fact, since $\bar\rho$ is irreducible, the argument is much simpler in this case. The auxiliary primes in proposition \ref{exttwp} can be chosen so that $\bar\rho(\Frob_v)$ has distinct eigenvalues, $v\in Q_N$. Hence it is enough to use the construction of patched completed homology in \ref{Chwal}, \ref{Pofch}.
\end{proof}

\begin{para}
This proposition basically says that an irreducible component of $\Spec R^{\{\mathbf{1}\}}$ is pro-modular if it contains a nice prime. To show that there are a lot of nice primes in $\Spec R^{\{\mathbf{1}\}}$, we need a lower bound on the connectedness dimension of $R^{\{\mathbf{1}\}}$ (see \ref{connectednessdim} for its definition).
\end{para}

\begin{prop} \label{cndimR1}
$c(R^{\{\mathbf{1}\}}/(\varpi))\geq 2[F:\Q]-|S|+|\Sigma_p|-1$.
\end{prop}

\begin{proof}
Denote by $R_{\bar\rho}$ the unrestricted universal deformation ring of $\bar\rho:G_{F,S}\to\GL_2(\F)$ with fixed determinant $\det \rho$. Then it is standard that $R_{\bar\rho}$ can be written as the form $\cO[[x_1,\cdots,x_{h^1}]]/(f_1,\cdots,f_{h^2})$ with $h^1-h^2=2[F:\Q]$ (see the proof of lemma \ref{condimRpskq}). There is a natural quotient map $R_{\bar\rho}\twoheadrightarrow R^{\{\mathbf{1}\}}$ corresponding to the unipotent conditions at $v\in S\setminus \Sigma_p$. Its kernel is generated by $\tr \rho^{univ}(t_v)-2,v\in S\setminus\Sigma_p$, where $\rho^{univ}$ denotes the universal deformation over $R_{\bar\rho}$ and $t_v$ is a topological generator of the pro-$p$ quotient of $I_{F_v}$. Thus we can write
\[R^{\{\mathbf{1}\}}\cong \cO[[x_1,\cdots,x_g]]/(f_1,\cdots,f_r)\]
with $g-r=2[F:\Q]-|S|+|\Sigma_p|$. Our claim follows from proposition \ref{cndimformula}.
\end{proof}

\begin{lem} \label{existnice}
Let $\kQ\in\Spec R^{\{\mathbf{1}\}}$ be a pro-modular prime containing $p$ such that 
\[\dim R^{\{\mathbf{1}\}}/\kQ\geq 2[F:\Q]-|S|+|\Sigma_p|-1.\] 
Then there exists a nice prime $\kq$ containing $\kQ$.
\end{lem}

\begin{proof}
For $v\in S\setminus \Sigma_p$, consider the local deformation ring $R^{\square,\mathbf{1}}_v$ introduced in \ref{Gdrs} (with $\xi_v=\mathbf{1}$). Then by lemma \ref{iharaavoid}, $\dim R^{\square,\mathbf{1}}_v/(\varpi)=3$. Hence we can choose $k=3(|S|-|\Sigma_p|)$ elements $a_1,\cdots,a_{k}$ in the maximal ideal of $R^{\{\mathbf{1}\}}$ such that $\rho(\kp)|_{G_{F_v}}$ is trivial for any $v\in S\setminus\Sigma_p$ and $\kp\in\Spec R^{\{\mathbf{1}\}}/(a_1,\cdots,a_k)$. See proof of proposition \ref{rcmodc} for more details. Let $\kQ'$ be a minimal prime of $R^{\{\mathbf{1}\}}/(\kQ,a_1,\cdots,a_k)$. Then it follows from our assumption that
\[\dim R^{\{\mathbf{1}\}}/\kQ'\geq 2[F:\Q]-4|S|+4|\Sigma_p|-1\geq [F:\Q]+1.\]
By the global class field theory, the dihedral locus in $\Spec R^{\{\mathbf{1}\}}/(\varpi)$ has dimension at most $[F:\Q]$. Hence $\rho(\kQ')$ is not dihedral and we can choose a one-dimensional prime $\kq$ containing $\kQ'$ such that $\rho(\kq)$ is not dihedral. Clearly $\kq$ is nice by our construction.
\end{proof}

\begin{proof}[Proof of proposition \ref{kerRtoTnil}]
By Theorem \ref{dh}, $\T_\km$ has dimension at least $2[F:\Q]+1$. By lemma \ref{existnice}, there exists a nice prime in $\Spec R^{\{\mathbf{1}\}}$. Hence at least one irreducible component of $\Spec R^{\{\mathbf{1}\}}$ is pro-modular. Using the results we just proved in this section, we can argue as in the last paragraph of the proof of proposition \ref{rcmodc} to show that any irreducible component of $\Spec R^{\{\mathbf{1}\}}$ has a nice prime. Proposition \ref{kerRtoTnil} follows immediately.
\end{proof}

\bibliographystyle{amsalpha}

\bibliography{bib}

\newcommand{\etalchar}[1]{$^{#1}$}
\providecommand{\bysame}{\leavevmode\hbox to3em{\hrulefill}\thinspace}
\providecommand{\MR}{\relax\ifhmode\unskip\space\fi MR }
\providecommand{\MRhref}[2]{%
  \href{http://www.ams.org/mathscinet-getitem?mr=#1}{#2}
}
\providecommand{\href}[2]{#2}
\begin{thebibliography}{BLGHT11}

\bibitem[AT68]{AT68}
E.~Artin and J.~Tate, \emph{Class field theory}, W. A. Benjamin, Inc., New
  York-Amsterdam, 1968. \MR{0223335}

\bibitem[B\"01]{Boc01}
Gebhard B\"{o}ckle, \emph{On the density of modular points in universal
  deformation spaces}, Amer. J. Math. \textbf{123} (2001), no.~5, 985--1007.
  \MR{1854117}

\bibitem[B\"10]{Boc10}
\bysame, \emph{Deformation rings for some mod 3 {G}alois representations of the
  absolute {G}alois group of {$\bold Q_3$}}, Ast\'{e}risque (2010), no.~330,
  529--542. \MR{2642411}

\bibitem[BB10]{BB10}
Laurent Berger and Christophe Breuil, \emph{Sur quelques repr\'esentations
  potentiellement cristallines de {${\rm GL}_2(\bold Q_p)$}}, Ast\'erisque
  (2010), no.~330, 155--211. \MR{2642406}

\bibitem[BE10]{BE10}
Christophe Breuil and Matthew Emerton, \emph{Repr\'esentations {$p$}-adiques
  ordinaires de {${\rm GL}_2(\bold Q_p)$} et compatibilit\'e local-global},
  Ast\'erisque (2010), no.~331, 255--315. \MR{2667890}

\bibitem[BH06]{BH06}
Colin~J. Bushnell and Guy Henniart, \emph{The local {L}anglands conjecture for
  {$\rm GL(2)$}}, Grundlehren der Mathematischen Wissenschaften [Fundamental
  Principles of Mathematical Sciences], vol. 335, Springer-Verlag, Berlin,
  2006. \MR{2234120}

\bibitem[BH15]{BH15}
Christophe Breuil and Florian Herzig, \emph{Ordinary representations of
  {$G(\Bbb{Q}_p)$} and fundamental algebraic representations}, Duke Math. J.
  \textbf{164} (2015), no.~7, 1271--1352. \MR{3347316}

\bibitem[BLGHT11]{BLGHT11}
Tom Barnet-Lamb, David Geraghty, Michael Harris, and Richard Taylor, \emph{A
  family of {C}alabi-{Y}au varieties and potential automorphy {II}}, Publ. Res.
  Inst. Math. Sci. \textbf{47} (2011), no.~1, 29--98. \MR{2827723}

\bibitem[BR86]{BR86}
Markus Brodmann and Josef Rung, \emph{Local cohomology and the connectedness
  dimension in algebraic varieties}, Comment. Math. Helv. \textbf{61} (1986),
  no.~3, 481--490. \MR{860135}

\bibitem[CDP14]{CDP14}
Pierre Colmez, Gabriel Dospinescu, and Vytautas Pa{\v s}k\=unas, \emph{The
  {$p$}-adic local {L}anglands correspondence for {${\rm GL}_2(\Bbb Q_p)$}},
  Camb. J. Math. \textbf{2} (2014), no.~1, 1--47. \MR{3272011}

\bibitem[CE05]{CE05}
Frank Calegari and Matthew Emerton, \emph{On the ramification of {H}ecke
  algebras at {E}isenstein primes}, Invent. Math. \textbf{160} (2005), no.~1,
  97--144. \MR{2129709}

\bibitem[CE12]{CE12}
\bysame, \emph{Completed cohomology---a survey}, Non-abelian fundamental groups
  and {I}wasawa theory, London Math. Soc. Lecture Note Ser., vol. 393,
  Cambridge Univ. Press, Cambridge, 2012, pp.~239--257. \MR{2905536}

\bibitem[CEG{\etalchar{+}}16]{CEG+}
Ana Caraiani, Matthew Emerton, Toby Gee, David Geraghty, Vytautas
  Pa\v{s}k\={u}nas, and Sug~Woo Shin, \emph{Patching and the {$p$}-adic local
  {L}anglands correspondence}, Camb. J. Math. \textbf{4} (2016), no.~2,
  197--287. \MR{3529394}

\bibitem[CHT08]{CHT08}
Laurent Clozel, Michael Harris, and Richard Taylor, \emph{Automorphy for some
  {$l$}-adic lifts of automorphic mod {$l$} {G}alois representations}, Publ.
  Math. Inst. Hautes \'Etudes Sci. (2008), no.~108, 1--181, With Appendix A,
  summarizing unpublished work of Russ Mann, and Appendix B by Marie-France
  Vign\'eras. \MR{2470687}

\bibitem[CS19]{CS19}
Frank Calegari and Joel Specter, \emph{Pseudorepresentations of weight one are
  unramified}, Algebra Number Theory \textbf{13} (2019), no.~7, 1583--1596.
  \MR{4009671}

\bibitem[DdSMS99]{DdSMS99}
J.~D. Dixon, M.~P.~F. du~Sautoy, A.~Mann, and D.~Segal, \emph{Analytic
  pro-{$p$} groups}, second ed., Cambridge Studies in Advanced Mathematics,
  vol.~61, Cambridge University Press, Cambridge, 1999. \MR{1720368}

\bibitem[DDT94]{DDT95}
Henri Darmon, Fred Diamond, and Richard Taylor, \emph{Fermat's last theorem},
  Current developments in mathematics, 1995 ({C}ambridge, {MA}), Int. Press,
  Cambridge, MA, 1994, pp.~1--154. \MR{1474977}

\bibitem[Del73]{De72}
P.~Deligne, \emph{Formes modulaires et repr\'esentations de {${\rm GL}(2)$}},
  Modular functions of one variable, {II} ({P}roc. {I}nternat. {S}ummer
  {S}chool, {U}niv. {A}ntwerp, {A}ntwerp, 1972), Springer, Berlin, 1973,
  pp.~55--105. Lecture Notes in Math., Vol. 349. \MR{0347738}

\bibitem[DT94]{DT94}
Fred Diamond and Richard Taylor, \emph{Nonoptimal levels of mod {$l$} modular
  representations}, Invent. Math. \textbf{115} (1994), no.~3, 435--462.
  \MR{1262939}

\bibitem[Eme06]{Eme06}
Matthew Emerton, \emph{A local-global compatibility conjecture in the
  {$p$}-adic {L}anglands programme for {${\rm GL}_{2/{\Bbb Q}}$}}, Pure Appl.
  Math. Q. \textbf{2} (2006), no.~2, Special Issue: In honor of John H. Coates.
  Part 2, 279--393. \MR{2251474}

\bibitem[Eme10]{Eme10a}
\bysame, \emph{Ordinary parts of admissible representations of {$p$}-adic
  reductive groups {II}. {D}erived functors}, Ast\'erisque (2010), no.~331,
  403--459. \MR{2667883}

\bibitem[Eme11]{Eme1}
\bysame, \emph{Local-global compatibility in the p-adic langlands programme for
  {$GL_2/Q$}}, 2011.

\bibitem[FM95]{FM93}
Jean-Marc Fontaine and Barry Mazur, \emph{Geometric {G}alois representations},
  Elliptic curves, modular forms, \& {F}ermat's last theorem ({H}ong {K}ong,
  1993), Ser. Number Theory, I, Int. Press, Cambridge, MA, 1995, pp.~41--78.
  \MR{1363495}

\bibitem[Ger10]{Ger10}
David~James Geraghty, \emph{Modularity lifting theorems for ordinary {G}alois
  representations}, ProQuest LLC, Ann Arbor, MI, 2010, Thesis (Ph.D.)--Harvard
  University. \MR{2941425}

\bibitem[GN20]{GN16}
Toby Gee and James Newton, \emph{Patching and the completed homology of locally
  symmetric spaces}, Journal of the Institute of Mathematics of Jussieu (2020),
  1--64.

\bibitem[Gro66]{EGA}
A.~Grothendieck, \emph{\'el\'ements de g\'eom\'etrie alg\'ebrique. {IV}.
  \'etude locale des sch\'emas et des morphismes de sch\'emas. {III}}, Inst.
  Hautes \'Etudes Sci. Publ. Math. (1966), no.~28, 255. \MR{0217086}

\bibitem[Hid89a]{Hida88}
Haruzo Hida, \emph{Nearly ordinary {H}ecke algebras and {G}alois
  representations of several variables}, Algebraic analysis, geometry, and
  number theory ({B}altimore, {MD}, 1988), Johns Hopkins Univ. Press,
  Baltimore, MD, 1989, pp.~115--134. \MR{1463699}

\bibitem[Hid89b]{Hida89}
\bysame, \emph{On nearly ordinary {H}ecke algebras for {${\rm GL}(2)$} over
  totally real fields}, Algebraic number theory, Adv. Stud. Pure Math.,
  vol.~17, Academic Press, Boston, MA, 1989, pp.~139--169. \MR{1097614}

\bibitem[Hid06]{Hida06}
\bysame, \emph{Hilbert modular forms and {I}wasawa theory}, Oxford Mathematical
  Monographs, The Clarendon Press, Oxford University Press, Oxford, 2006.
  \MR{2243770}

\bibitem[HT15]{HT15}
Yongquan Hu and Fucheng Tan, \emph{The {B}reuil-{M}\'ezard conjecture for
  non-scalar split residual representations}, Ann. Sci. \'Ec. Norm. Sup\'er.
  (4) \textbf{48} (2015), no.~6, 1383--1421. \MR{3429471}

\bibitem[Kis03]{Kis03}
Mark Kisin, \emph{Overconvergent modular forms and the {F}ontaine-{M}azur
  conjecture}, Invent. Math. \textbf{153} (2003), no.~2, 373--454. \MR{1992017}

\bibitem[Kis09a]{Kis09a}
\bysame, \emph{The {F}ontaine-{M}azur conjecture for {${\rm GL}_2$}}, J. Amer.
  Math. Soc. \textbf{22} (2009), no.~3, 641--690. \MR{2505297}

\bibitem[Kis09b]{Kis09b}
\bysame, \emph{Moduli of finite flat group schemes, and modularity}, Ann. of
  Math. (2) \textbf{170} (2009), no.~3, 1085--1180. \MR{2600871}

\bibitem[KW09a]{KW09a}
Chandrashekhar Khare and Jean-Pierre Wintenberger, \emph{Serre's modularity
  conjecture. {I}}, Invent. Math. \textbf{178} (2009), no.~3, 485--504.
  \MR{2551763}

\bibitem[KW09b]{KW09b}
\bysame, \emph{Serre's modularity conjecture. {II}}, Invent. Math. \textbf{178}
  (2009), no.~3, 505--586. \MR{2551764}

\bibitem[Lic72]{Lich72}
Stephen Lichtenbaum, \emph{On the values of zeta and {$L$}-functions. {I}},
  Ann. of Math. (2) \textbf{96} (1972), 338--360. \MR{0360527}

\bibitem[LvO96]{LvO96}
Huishi Li and Freddy van Oystaeyen, \emph{Zariskian filtrations},
  $K$-Monographs in Mathematics, vol.~2, Kluwer Academic Publishers, Dordrecht,
  1996. \MR{1420862}

\bibitem[Mat80]{Mat2}
Hideyuki Matsumura, \emph{Commutative algebra}, second ed., Mathematics Lecture
  Note Series, vol.~56, Benjamin/Cummings Publishing Co., Inc., Reading, Mass.,
  1980. \MR{575344}

\bibitem[Mat89]{Mat1}
\bysame, \emph{Commutative ring theory}, second ed., Cambridge Studies in
  Advanced Mathematics, vol.~8, Cambridge University Press, Cambridge, 1989,
  Translated from the Japanese by M. Reid. \MR{1011461}

\bibitem[Maz89]{Maz87}
B.~Mazur, \emph{Deforming {G}alois representations}, Galois groups over {${\bf
  Q}$} ({B}erkeley, {CA}, 1987), Math. Sci. Res. Inst. Publ., vol.~16,
  Springer, New York, 1989, pp.~385--437. \MR{1012172}

\bibitem[NSW08]{NSW08}
J\"urgen Neukirch, Alexander Schmidt, and Kay Wingberg, \emph{Cohomology of
  number fields}, second ed., Grundlehren der Mathematischen Wissenschaften
  [Fundamental Principles of Mathematical Sciences], vol. 323, Springer-Verlag,
  Berlin, 2008. \MR{2392026}

\bibitem[NT20]{NT20}
James Newton and Jack~A. Thorne, \emph{Adjoint selmer groups of automorphic
  galois representations of unitary type}, 2020.

\bibitem[Pa{\v s}09]{Pas09}
Vytautas Pa{\v s}k\={u}nas, \emph{On some crystalline representations of {${\rm
  GL}_2(\Bbb Q_p)$}}, Algebra Number Theory \textbf{3} (2009), no.~4, 411--421.
  \MR{2525557}

\bibitem[Pa{\v s}10]{Pas10}
Vytautas Pa{\v s}k\=unas, \emph{Extensions for supersingular representations of
  {${\rm GL}_2(\Bbb Q_p)$}}, Ast\'erisque (2010), no.~331, 317--353.
  \MR{2667891}

\bibitem[Pa{\v s}13]{Pas13}
\bysame, \emph{The image of {C}olmez's {M}ontreal functor}, Publ. Math. Inst.
  Hautes \'Etudes Sci. \textbf{118} (2013), 1--191. \MR{3150248}

\bibitem[Pa{\v s}14]{Pas14}
\bysame, \emph{Blocks for {${\rm mod}\,p$} representations of {${\rm GL}_2(\Bbb
  Q_p)$}}, Automorphic forms and {G}alois representations. {V}ol. 2, London
  Math. Soc. Lecture Note Ser., vol. 415, Cambridge Univ. Press, Cambridge,
  2014, pp.~231--247. \MR{3444235}

\bibitem[Pa{\v s}16]{Pas16}
\bysame, \emph{On 2-dimensional 2-adic {G}alois representations of local and
  global fields}, Algebra Number Theory \textbf{10} (2016), no.~6, 1301--1358.
  \MR{3544298}

\bibitem[Pa{\v s}18]{Pas18}
\bysame, \emph{{On some consequences of a theorem of J. Ludwig}}, ArXiv
  e-prints (2018).

\bibitem[PT21]{PT21}
Vytautas Pa{\v s}k\=unas and Shen-Ning Tung, \emph{Finiteness properties of the
  category of mod $p$ representations of $\mathrm{GL}_2(\mathbb{Q}_p)$}, 2021.

\bibitem[Rib76]{Rib76}
Kenneth~A. Ribet, \emph{A modular construction of unramified {$p$}-extensions
  of {$Q(\mu _{p})$}}, Invent. Math. \textbf{34} (1976), no.~3, 151--162.
  \MR{419403}

\bibitem[Sch]{Sch15}
Peter Scholze, \emph{On the p-adic cohomology of the lubin-tate tower}, Annales
  de l'ENS, to appear.

\bibitem[Shi78]{Shi78}
Goro Shimura, \emph{The special values of the zeta functions associated with
  {H}ilbert modular forms}, Duke Math. J. \textbf{45} (1978), no.~3, 637--679.
  \MR{507462}

\bibitem[Sho16]{Sho16}
Jack Shotton, \emph{Local deformation rings for {${\rm GL}_2$} and a
  {B}reuil-{M}\'ezard conjecture when {$\ell\ne p$}}, Algebra Number Theory
  \textbf{10} (2016), no.~7, 1437--1475. \MR{3554238}

\bibitem[Ski09]{Ski09}
Christopher Skinner, \emph{A note on the {$p$}-adic {G}alois representations
  attached to {H}ilbert modular forms}, Doc. Math. \textbf{14} (2009),
  241--258. \MR{2538615}

\bibitem[SS16]{SS16}
Tobias Schmidt and Matthias Strauch, \emph{Dimensions of some locally analytic
  representations}, Represent. Theory \textbf{20} (2016), 14--38. \MR{3455080}

\bibitem[{Sta}18]{stacks-project}
The {Stacks Project Authors}, \emph{\textit{Stacks Project}},
  \url{https://stacks.math.columbia.edu}, 2018.

\bibitem[SW99]{SW99}
C.~M. Skinner and A.~J. Wiles, \emph{Residually reducible representations and
  modular forms}, Inst. Hautes \'Etudes Sci. Publ. Math. (1999), no.~89, 5--126
  (2000). \MR{1793414}

\bibitem[SW01]{SW01}
C.~M. Skinner and Andrew~J. Wiles, \emph{Nearly ordinary deformations of
  irreducible residual representations}, Ann. Fac. Sci. Toulouse Math. (6)
  \textbf{10} (2001), no.~1, 185--215. \MR{1928993}

\bibitem[Tay91]{Ta91}
Richard Taylor, \emph{Galois representations associated to {S}iegel modular
  forms of low weight}, Duke Math. J. \textbf{63} (1991), no.~2, 281--332.
  \MR{1115109}

\bibitem[Tay03]{Ta03}
\bysame, \emph{On icosahedral {A}rtin representations. {II}}, Amer. J. Math.
  \textbf{125} (2003), no.~3, 549--566. \MR{1981033}

\bibitem[Tay06]{Ta06}
\bysame, \emph{On the meromorphic continuation of degree two {$L$}-functions},
  Doc. Math. (2006), no.~Extra Vol., 729--779. \MR{2290604}

\bibitem[Tay08]{Ta08}
\bysame, \emph{Automorphy for some {$l$}-adic lifts of automorphic mod {$l$}
  {G}alois representations. {II}}, Publ. Math. Inst. Hautes \'Etudes Sci.
  (2008), no.~108, 183--239. \MR{2470688}

\bibitem[Tho15]{Tho15}
Jack~A. Thorne, \emph{Automorphy lifting for residually reducible {$l$}-adic
  {G}alois representations}, J. Amer. Math. Soc. \textbf{28} (2015), no.~3,
  785--870. \MR{3327536}

\bibitem[Tun18]{Tung18}
Shen-Ning Tung, \emph{On the automorphy of 2-dimensional potentially
  semi-stable deformation rings of $g_{\mathbb{q}_p}$}, 2018.

\bibitem[Tun19]{Tung19}
\bysame, \emph{On the modularity of 2-adic potentially semi-stable deformation
  rings}, 2019.

\bibitem[Was78]{Wa78}
Lawrence~C. Washington, \emph{The non-{$p$}-part of the class number in a
  cyclotomic {${\bf Z}_{p}$}-extension}, Invent. Math. \textbf{49} (1978),
  no.~1, 87--97. \MR{511097}

\bibitem[Wil86]{Wi86}
A.~Wiles, \emph{On {$p$}-adic representations for totally real fields}, Ann. of
  Math. (2) \textbf{123} (1986), no.~3, 407--456. \MR{840720}

\bibitem[Wil88]{Wi88}
\bysame, \emph{On ordinary {$\lambda$}-adic representations associated to
  modular forms}, Invent. Math. \textbf{94} (1988), no.~3, 529--573.
  \MR{969243}

\end{thebibliography}

\end{document}